\documentclass{amsbook}

\usepackage{amsmath}
\usepackage{amssymb}
\usepackage{graphicx}
\usepackage{amsfonts}

\usepackage[colorlinks = true,
            linkcolor = red,
            urlcolor  = blue,
            citecolor = blue,
            anchorcolor = blue]{hyperref}
            
\newtheorem{theorem}{Theorem}[chapter]
\newtheorem{lemma}[theorem]{Lemma}
\newtheorem{proposition}[theorem]{Proposition}
\newtheorem{corollary}[theorem]{Corollary}
\usepackage{enumerate}
\theoremstyle{definition}
\newtheorem{definition}[theorem]{Definition}
\newtheorem{example}[theorem]{Example}

\newtheorem{remark}[theorem]{Remark}
\newtheorem{remarks}[theorem]{Remarks}

\numberwithin{equation}{chapter}

\newcommand{\lab}{\label}

\newcommand{\ben}{\begin{enumerate}}
\newcommand{\een}{\end{enumerate}}

\newcommand{\bea}{\begin{eqnarray}}
\newcommand{\ba}{\begin{array}}
\newcommand{\bean}{\begin{eqnarray*}}
\newcommand{\ea}{\end{array}}
\newcommand{\eea}{\end{eqnarray}}
\newcommand{\eean}{\end{eqnarray*}}
\newcommand{\beq}{\begin{equation}}
\newcommand{\eeq}{\end{equation}}
\newcommand{\bthm}{\begin{thm}}
\newcommand{\ethm}{\end{thm}}
\newcommand{\blem}{\begin{lem}}
\newcommand{\elem}{\end{lem}}
\newcommand{\bprop}{\begin{prop}}
\newcommand{\eprop}{\end{prop}}
\newcommand{\bcor}{\begin{cor}}
\newcommand{\ecor}{\end{cor}}
\newcommand{\bdfn}{\begin{dfn}}
\newcommand{\edfn}{\end{dfn}}
\newcommand{\brem}{\begin{rem}}
\newcommand{\erem}{\end{rem}}
\newcommand{\bpf}{\begin{proof}}
\newcommand{\epf}{\end{proof}}
\newcommand{\bfact}{\begin{fact}}
\newcommand{\efact}{\end{fact}}
\newcommand{\bobs}{\begin{obs}}
\newcommand{\eobs}{\end{obs}}
\newcommand{\bexm}{\begin{exm}}
\newcommand{\dcc}{d_{cc}}
\newcommand{\eexm}{\end{exm}}
\newcommand{\bdf}{\begin{fact}}
\newcommand{\edf}{\end{fact}}

\newcommand{\stm}{\setminus}
\alph{enumii} \roman{enumiii}

\newtheorem{thm}{Theorem}[section]
\newtheorem{prop}[thm]{Proposition}
\newtheorem{lem}[thm]{Lemma}

\newtheorem{cor}[thm]{Corollary}
\newtheorem{dfn}[thm]{Definition}
\newtheorem{rem}[thm]{Remark}
\newtheorem{fact}[thm]{Fact}
\newtheorem{obs}[thm]{Observation}

\def\cA{\mathcal A}             \def\cB{\mathcal B}       
\def\cH{\mathcal H}             \def\cF{\mathcal F}       
\def\cL{{\mathcal L}}                   \def\cP{{\mathcal P}}
                    \def\cJ{\mathcal J}
\def\cS{\mathcal S}             
\def\cW{\mathcal W}

\def\cF{\mathcal F}

\def\fg{{\mathfrak g}}

\def\fv{{\mathfrak v}}

\def\N{{\mathbb N}}                \def\Z{{\mathbb Z}}      \def\R{{\mathbb R}}
\def\C{{\mathbb C}}

\newcommand{\od}{\overline{d}}
\newcommand{\G}{{\mathbb G}}
\newcommand{\K}{{\mathbb K}}
\newcommand{\Qua}{{\mathbb H}}
\newcommand{\Oct}{{\mathbb O}}
\newcommand{\oG}{{\overline{\G}}}
\newcommand{\Heis}{{{\mathbf{Heis}}}}
\newcommand{\Sph}{{\mathbb S}}
\newcommand{\deriv}[1]{{\frac{\partial}{\partial #1}}}

\newcommand{\B}{{\mathbb B}}

\DeclareMathOperator{\diamcc}{diam_{cc}}
\DeclareMathOperator{\distcc}{dist_{cc}}
\def\a{\alpha}                \def\b{\beta}             
\def\De{\Delta}               \def\e{\varepsilon}          \def\f{\phi}
\def\g{\gamma}                \def\Ga{\Gamma}           \def\l{\lambda}
\def\La{\Lambda}              \def\om{\omega}           \def\Om{\Omega}
               \def\sg{\sigma}
               \def\th{\theta}           
\def\ka{\kappa}

\newcommand{\ep}{\varepsilon}

\newcommand{\ga}{\gamma}
\newcommand{\eps}{\epsilon}

\newcommand{\gz}{\G(\Z)}

\newcommand{\be}{{\mathbf e}}
\newcommand{\bi}{{\mathbf i}}
\newcommand{\bj}{{\mathbf j}}
\newcommand{\bk}{{\mathbf k}}

\def\1{1\!\!1}

\def\and{\text{ and }}

\def\h{{\text h}}
\def\hmu{\h_\mu}           
\def\H{\text{{\rm H}}}

\def\Int{\text{{\rm Int}}}
         \def\P{\text{{\rm P}}}

 \def\osc{\rm osc}

\DeclareMathOperator{\Conf}{Conf}

\DeclareMathOperator{\Isom}{Isom}
\DeclareMathOperator{\Sim}{Sim}
\DeclareMathOperator{\Fin}{Fin}
\DeclareMathOperator{\Sp}{Sp}
\DeclareMathOperator{\PSp}{PSp}
\DeclareMathOperator{\PSU}{PSU}

\def\ba{\bigcap}              \def\bu{\bigcup}
\def\({\bigl(}                \def\){\bigr)}
\def\lt{\left}                \def\rt{\right}

\def\ld{\ldots}               \def\bd{\partial}         \def\^{\tilde}
\def\ve{\varepsilon}

\def\es{\emptyset}            \def\sms{\setminus}
\def\sbt{\subset}             \def\spt{\supset}

      \def\imp{\Rightarrow}

           \def\downto{\searrow}
\def\sp{\medskip}             \def\fr{\noindent}        
\def\ov{\overline}            \def\un{\underline}
\def\ess{{\rm ess}}
\def\om{\omega}

\def\du{\bigoplus}

\newcommand{\ra}{\rightarrow}

\newcommand{\loc}{{\scriptstyle{loc}}}

\newcommand{\pf}{{\mathcal{L}}}

\DeclareMathOperator{\diam}{diam} \DeclareMathOperator{\spa}{span}

\DeclareMathOperator{\dist}{dist} \DeclareMathOperator{\Imag}{Im}
\DeclareMathOperator{\Real}{Re}
\DeclareMathOperator{\Spin}{Spin}
\newcommand{\bom}{\overline{\om}}
\def\HD{\dim_{\cH}}  
\def\HDl{\underline{\dim}_{\cH}}  
\def\PD{\dim_{\cP}}
\def\PDl{\underline{\dim}_{\cP}}  
\makeindex

\begin{document}
\frontmatter
\title{Conformal graph directed Markov systems on Carnot groups}

\author{Vasilis Chousionis}
\address{Department of Mathematics \\
University of Connecticut \\
196 Auditorium Road U-3009 \\
Storrs, CT 06269-3009}
%\curraddr{Department of Mathematics and Statistics,
%University of Helsinki}
\email{vasileios.chousionis@uconn.edu}

\author{Jeremy T. Tyson}
\address{Department of Mathematics \\ University of
Illinois at Urbana-Cham\-paign \\ 1409 West Green Street \\ Urbana, IL 61801}
\email{tyson@math.uiuc.edu}

\author{Mariusz Urba\'nski}
\address{Department of Mathematics \\ University of North Texas\\ General Academics Building 435
  \\ 1155 Union Circle 311430 \\ Denton, TX 76203-5017}
\email{urbanski@unt.edu}

\date{\today}
\thanks{VC was supported by the Academy of Finland Grant SA 267047. JTT was supported by NSF grant DMS 1201875 and Simons Collaborative Grant \#353627. MU was supported by NSF grant DMS 1361677.}

\begin{abstract}
We develop a comprehensive theory of conformal graph directed Markov systems in the non-Riemannian setting of Carnot groups equipped with a sub-Riemannian metric. In particular, we develop the thermodynamic formalism and show that, under natural hypotheses, the limit set of an Carnot conformal GDMS has Hausdorff dimension given by Bowen's parameter. We illustrate our results for a variety of examples of both linear and nonlinear iterated function systems and graph directed Markov systems in such sub-Riemannian spaces. These include the Heisenberg continued fractions introduced by Lukyanenko and Vandehey as well as Kleinian and Schottky groups associated to the non-real classical rank one hyperbolic spaces.
\end{abstract}

\maketitle

\setcounter{page}{4}
\tableofcontents

%-----------------------------------------------------------------------
% Beginning of chap1.tex
%-----------------------------------------------------------------------
%
%  AMS-LaTeX sample file for a chapter of a monograph, to be used with
%  an AMS monograph document class.  This is a data file input by
%  chapter.tex.
%
%  Use this file as a model for a chapter; DO NOT START BY removing its
%  contents and filling in your own text.
%
%%%%%%%%%%%%%%%%%%%%%%%%%%%%%%%%%%%%%%%%%%%%%%%%%%%%%%%%%%%%%%%%%%%%%%%%

%\frontmatter
%\pagenumbering{roman}

\chapter*{Introduction}\label{chap:introduction}

In this monograph we lay the foundation for a theory of
conformal dynamical systems in nilpotent stratified Lie groups (Carnot groups) \index{Carnot group} equipped with a sub-Riemannian metric. \index{sub-Riemannian metric}
In particular, we develop a thermodynamic formalism \index{thermodynamic formalism} for conformal graph directed Markov systems \index{conformal graph directed Markov system} which permits us to identify the Hausdorff dimension of the limit set as the zero of a suitable pressure function. \index{pressure function} We also consider the structure of conformal measures and \index{conformal measure} investigate when the Hausdorff and packing measures \index{Hausdorff measure} \index{packing measure} of the limit set are either positive or finite. Finally, we extend to the sub-Riemannian context a result, first proved by Schief and later extended by Peres--Rams--Simon--Solomyak, on the equivalence between the open set condition, strong open set condition, \index{open set condition} \index{strong open set condition} and positivity of the Hausdorff measure of the limit set. The well established theory of Euclidean conformal graph directed Markov systems (see \cite{MUGDMS}) serves as a model for our investigations.

We formulate the preceding theory in the general setting of a Carnot group equipped with a left-invariant homogeneous metric. However,
our results are of particular interest in the special situation of Iwasawa groups. \index{Iwasawa group} Such groups, a particular class of Carnot groups of step at most two, arise as
boundaries at infinity of the classical rank one symmetric spaces, \index{rank one symmetric space} or
alternatively as nilpotent components in the Iwasawa decomposition of real rank one simple Lie groups.

Provided the dimension of the ambient Iwasawa group is at least
three, a version of Liouville's rigidity theorem holds: \index{Liouville's theorem} every locally
defined conformal self-map is the restriction of a M\"obius map,
acting on the (conformally equivalent) one-point compactification
equipped with a suitable spherical metric. As in the Euclidean case,
the space of M\"obius maps is finite-dimensional, and may be
identified with a group of matrices acting by isometries on the
corresponding hyperbolic space. The action of this group is
sufficiently rich, for instance, it is two-point transitive with
nontrivial stabilizer subgroups.

A recent rigidity theorem of Cowling and Ottazzi (see Theorem \ref{th:cowling-ottazzi}) asserts that every conformal mapping defined on domains in a Carnot group which is not of Iwasawa type is necessarily an affine similarity mapping. \index{affine similarity} Moreover, the theory which we develop applies to families of contractive similarities of any homogeneous metric on any Carnot group, which need not be conformal, see also Remark \ref{abusingGDMS}. Our results therefore encompass two settings, graph directed Markov systems (GDMS) consisting of either contractive similarities or contractive conformal mappings, in any Carnot group.
We stress that a number of our results are new in the setting of similarity GDMS in general Carnot groups, especially in the case of countably infinite systems.
Finite self-similar iterated function systems \index{iterated function system} in Carnot groups have
previously been studied in detail by the second author and his collaborators \cite{bhit:horizfractals}, \cite{bt:horizdim}, \cite{btw:dimcomp}. The results of the present monograph apply equally well either to finite or countably infinite
self-similar iterated function systems in arbitrary Carnot groups. See Chapter \ref{chap:examples} for further information.
Moreover, all of the results of this monograph apply to non-affine conformal GDMS in Iwasawa groups; such a setting is entirely new.

We briefly indicate the structure of this monograph.
Chapter \ref{chap:carnot-groups} briefly reviews the algebraic, metric and geometric
structure of general Carnot groups and the primary morphisms of interest: the contact mappings.
In Chapter \ref{chap:iwasawa-and-conformal} we describe in detail the class of Iwasawa Carnot groups which plays a particularly important role throughout this monograph. We also recall the definition of conformal mappings in Carnot groups, and describe both the classification of conformal maps on Iwasawa groups as well as the Cowling--Ottazzi rigidity result mentioned above.
In Chapter \ref{chap:conformal-metric-and-geometric-properties} we develop a series of metric and geometric properties of conformal mappings. Throughout the remainder of this monograph, we will primarily only make use of such properties in our consideration of Carnot conformal GDMS.

Chapter \ref{chap:CGDMS} introduces abstract graph directed Markov systems, as well as conformal graph directed Markov systems in Iwasawa
groups. A brief interlude (Chapter \ref{chap:examples}) describes several examples of Iwasawa conformal GDMS. In Chapters
\ref{chap:IGDMS-dimensions} and \ref{chap:geometric-properties} we establish formulas for the Hausdorff and packing dimensions of
invariant sets of Iwasawa conformal GDMS and investigate the structure
of conformal measures on such sets. We revisit the examples of Chapter
\ref{chap:examples} in Chapter \ref{chap:examples-2}, where we apply
the results of earlier chapters to compute or estimate the dimensions
of several invariant sets. In Chapter \ref{chap:measure-properties}
we address a finer problem, namely, the finiteness or positivity of
Hausdorff or packing measure on the invariant set.  In Chapter \ref{chap:separation-equiv} we establish the equivalence between the open set condition, the strong open set condition, and the positivity of the Hausdorff measure of the limit set of finite GDMS. A brief concluding
chapter (Chapter \ref{chap:invariant-measures}) discusses the Hausdorff
dimensions of invariant measures.

\

\paragraph{\bf Acknowledgements.} Research for this paper was completed during visits of various subsets of the authors to the University of Illinois, the University of Bern, and the University of North Texas. We wish to acknowledge the hospitality of all of these institutions.

\mainmatter
\pagenumbering{arabic}
\chapter{Carnot groups}\label{chap:carnot-groups}

In this chapter we introduce Carnot groups and their sub-Riemannian geometry. We discuss the differential geometric, metric, and measure-theoretic structure of such groups. We also describe the class of contact mappings of a Carnot group, that is, mappings which preserve the inherent stratified structure. Typical examples of contact mappings include left translations, dilations, and automorphisms, which together generate the family of orientation-preserving similarities. We recall one of the most basic examples of a Carnot group: the (complex) Heisenberg group. In the following chapter we introduce the so-called Iwasawa groups, which comprise the Heisenberg groups over the complex numbers, over the quaternions and over the Cayley numbers. Each Iwasawa group is endowed with a natural conformal inversion mapping which ensures that the full class of conformal self-maps is larger than the group of similarities.

We conclude this chapter by reviewing the Dimension Comparison Theorem which relates the spectra of Hausdorff measures in a Carnot group defined with respect to either a sub-Riemannian metric or a Euclidean metric. The Dimension Comparison Theorem will be used later to deduce Euclidean dimension estimates for limit sets of conformal graph directed Markov systems in Carnot groups.

\section{Carnot groups}\label{sec:carnot-groups}

A {\it Carnot group} \index{Carnot group} is a connected and simply
connected nilpotent Lie group $\G$ whose Lie algebra $\fg$ admits a
stratification
\begin{equation}\label{fg}
\fg = \fv_1 \oplus \cdots \oplus \fv_\iota
\end{equation}
into vector subspaces satisfying the commutation rules
\begin{equation}\label{commutators}
[\fv_1,\fv_i]=\fv_{i+1}
\end{equation}
for $1\le i<\iota$ and $[\fv_1,\fv_{\iota}]=(0)$. In particular, the
full Lie algebra is generated via iterated Lie brackets of elements of
the lowest layer $\fv_1$ of the stratification. This layer is known as
the {\it horizontal layer} \index{horizontal layer} of the Lie
algebra, and elements of $\fv_1$ are known as {\it horizontal tangent
  vectors}. \index{horizontal tangent vector} As usual, we interpret
the elements of $\fg$ as either tangent vectors to $\G$ at the neutral
element $o$, or as left invariant vector fields on $\G$. The integer
$\iota$ is known as the {\it step} \index{step} of the group $\G$.

Via the above identification between interpretations of the Lie
algebra, the horizontal layer $\fv_1$ corresponds to a distribution
$H\G$ in the tangent bundle $T\G$ given by the rule $H_p\G = \{ X(p)
\, : \, X \in \fv_1 \}$. The distribution $H\G$ is known as the {\it
  horizontal distribution}. \index{horizontal distribution} The
bracket generating condition \eqref{commutators} implies that $H\G$ is
completely nonintegrable. According to a fundamental theorem of Chow
and Rashevsky, \index{Chow--Rashevsky theorem} see e.g.\
\cite{mont:tour}, the complete nonintegrability of $H\G$ ensures that
any two points of $\G$ can be joined by \index{horizontal curve} a
horizontal curve, i.e., a piecewise smooth curve $\gamma$ such that
$\gamma'(s) \in H_{\gamma(s)}\G$ whenever $\gamma'(s)$ is defined.

Since $\G$ is connected, simply connected and nilpotent, the
exponential map $\exp:\fg\to\G$ is a global diffeomorphism and $\G$ is
naturally identified with a Euclidean space $\fg = \R^N$. Nilpotence
implies that the group law on $\G$ is given by polynomial operations
in the coordinates of the corresponding Euclidean space $\R^N$. The precise formula for the group law
can be derived from the {\it Baker--Campbell--Hausdorff formula} \index{Baker-Campbell-Hausdorff formula}
\begin{equation}\label{BCHformula}
\exp(U) * \exp(V) = \exp( U + V + \tfrac12 [U,V] + \tfrac1{12} ([U,[U,V]] + [V,[V,U]]) + \cdots),
\end{equation}
valid for $U,V \in \fg$. (See, for instance, \cite[Theorem 2.2.13]{BLU}.) Note that nilpotence ensures that the series occurring on the right hand side of \eqref{BCHformula} terminates.
For instance, in the case when $\G$ has step two, \eqref{BCHformula} reads
\begin{equation}\label{BCHformula2}
\exp(U) * \exp(V) = \exp( U + V + \tfrac12 [U,V] ).
\end{equation}
Introducing {\it exponential coordinates of the first kind} \index{exponential coordinates of the first kind} $p = (U_1,U_2)$, where $p = \exp(U_1 + U_2)$ with $U_1 \in \fv_1$ and $U_2 \in \fv_2$, we deduce from \eqref{BCHformula2} that
\begin{equation}\label{step2grouplaw}
p*q = (U_1+V_1,U_2+V_2+\tfrac12[U_1,V_1])
\end{equation}
if $p=(U_1,U_2)$ and $q=(V_1,V_2)$.

For $r>0$ we define the {\it dilation} with {\it scale factor} $r$,
\index{dilation} \index{scale factor} $\delta_r$, to be the automorphism \index{automorphism} of $\fg$ which is
given on the subspace $\fv_i$ by the rule $\delta_r(X) = r^i X$, $X
\in \fv_i$. Conjugation with the exponential map transfers this map
$\delta_r$ to an automorphism of the group $\G$ which we continue to
call a {\it dilation} and continue to denote by $\delta_r$, thus
\begin{equation}\label{carnot-dilation}
\delta_r(p) = \exp(\delta_r(\exp^{-1}(p))).
\end{equation}

\subsection{Homogeneous metrics}\label{sec:homogeneous-metrics}

Each Carnot group can be equipped with a geodesic metric, the {\it Carnot--Carath\'eodory metric}
$d_{cc}$. \index{Carnot Caratheodory metric@Carnot--Carath\'eodory metric} To define this metric, we fix
an inner product $\langle\cdot,\cdot\rangle$ on the horizontal layer
$\fv_1$ of the Lie algebra, which we promote to a left invariant
family of inner products via left invariance of the elements of $\fg$.
Relative to this inner product, we fix an orthonormal basis of vector
fields $X_1,\ldots,X_m \in \fv_1$, and we define the
{\it horizontal norm} \index{horizontal norm} \index{norm!horizontal}
of a vector $v \in H_p\G$ to be
$$
|v|_{p,\G} = \left( \sum_{j=1}^{m} \langle v,X_j(p) \rangle^2
\right)^{1/2}.
$$
The {\it horizontal length} $\ell_{cc}(\gamma)$ of a horizontal curve
$\gamma:[a,b]\to\G$ is computed by integrating the horizontal norm of
the tangent vector field along $\gamma$:
$$
\ell_{cc}(\gamma) = \int_a^b |\gamma'(s)|_{\gamma(s),\G} \, ds.
$$
Infimizing $\ell_{cc}(\gamma)$ over all horizontal curves $\gamma$
joining $p$ to $q$ defines the {\it Carnot--Carath\'eodory distance}
$d_{cc}(p,q)$. It is well known that $d_{cc}$ is a geodesic metric
on~$\G$. We record that a \textit{CC-geodesic} \index{geodesic} connecting two points $p,q \in \G$ is a length minimizing horizontal curve $\gamma:[0,T] \ra \G$,  such that $\gamma(0)=p$, $\gamma(T)=q$.

Explicit formulas for the Carnot-Carath\'eodory metric are difficult
to come by and are known only in very specific cases. For many
purposes it suffices to consider any bi-Lipschitz equivalent metric on
$\G$. A large class of such metrics is given by the homogeneous
metrics. A metric $d$ on $\G$ is said to be {\it homogeneous}
\index{homogeneous metric} if $d: \R^N \times \R^N \ra [0,\infty)$ is
continuous with respect to the Euclidean topology, is left invariant
and is $1$-homogeneous with respect to the dilations
$(\delta_r)_{r>0}$. The $1$-homogeneity of $d$ means that
$$
d(\delta_r(p),\delta_r(q)) = r\, d(p,q)
$$
for all $p,q\in\G$ and all
$r>0$. We note in particular that the Carnot--Carath\'eodory metric
$d_{cc}$ is a homogeneous metric. The Kor\'anyi (gauge) metric on an
Iwasawa group (defined in subsection \ref{subsec:cygan}) provides
another example of a homogeneous metric. Any two homogeneous metrics
$d_1$ and $d_2$ on a given Carnot group $\G$ are equivalent in the
sense that there exists a constant $C>0$ so that
$$
C^{-1}d_1(p,q) \le d_2(p,q) \le Cd_1(p,q)
$$
for all $p,q\in\G$; this is an easy
consequence of the assumptions. In fact, if $d$ denotes any
homogeneous metric, then there exists a constant $L>0$ so that
\begin{equation}\label{quasiconvexity}
d(p,q) \le d_{cc}(p,q) \le L\,d(p,q)
\end{equation}
for all $p,q \in \G$. The fact that there is no additional
multiplicative factor in the left hand inequality in
\eqref{quasiconvexity} stems from the observation that the Carnot--Carath\'eodory metric
is the path metric associated to the homogeneous metric $d$. That is, $d_{cc}(p,q)$ is equal to the infimum of the lengths of horizontal paths joining $p$ to $q$, where length is computed with respect to the homogeneous metric $d$. It follows that the length of a horizontal curve $\gamma$ is the same when computed with respect to any homogeneous metric.

\subsection{Haar measure and Hausdorff dimension}\label{sec:haar-and-hausdorff}

By a theorem of Mitchell (see also \cite{mont:tour}), the Hausdorff
dimension of $\G$ in any homogeneous metric is equal to
the {\it homogeneous dimension} \index{homogeneous dimension}
\index{dimension!homogeneous}
\begin{equation}\label{Q}
Q = \sum_{i=1}^\iota i \, \dim \fv_i.
\end{equation}
In all nonabelian examples (i.e., when the step $\iota$ is strictly
greater than one), we have $Q>N$ where $N$ denotes the topological
dimension of $\G$. It follows that the Carnot--Carath\'eodory metric
$d_{cc}$ is never bi-Lipschitz equivalent to any Riemannian metric on
the underlying Euclidean space $\R^N$.

We denote the Haar measure \index{Haar measure} of a set $E$ in a
Carnot group $\G$ by $|E|$. The Haar measure in $\G$ is proportional
to the Lebesgue measure in the underlying Euclidean space $\R^{N}$.
Moreover, there exists a constant $c_0$ so that for any $p\in\G$ and
$r>0$ we have
\begin{equation}\label{measure-of-bcc}
|B(p,r)|= c_0r^Q,
\end{equation}
where $Q$ is as in \eqref{Q}. One way to see this is to note that
$B(o,1)$ is mapped onto $B(p,r)$ by the composition of left
translation by $p$ and the dilation $\delta_r$ with scale factor
$r>0$. Observe that the Jacobian determinant of the (smooth) map
$\delta_r$ is everywhere equal to $r^Q$. It follows that
\eqref{measure-of-bcc} holds with $c_0=|B(o,1)|$.

\subsection{The Heisenberg group}\label{sec:heisenberg}

We briefly record a simple example of a nonabelian Carnot group: the first (complex) Heisenberg group. This group is the lowest-dimensional example in a class of groups, the Iwasawa groups, which we will describe in detail in the following chapter.

The underlying space for the first Heisenberg group $\Heis$ is $\R^3$, which we also view as $\C\times\R$. We endow $\C\times\R$ with the group law
$$
(z;t) \ast (z';t')=(z+z';t+t'+2\Imag (z\overline{z'})),
$$
where we denote elements of $\Heis$ by either $(z;t) \in \C\times\R$ or $(x,y;t) \in \R^3$. The identity element in $\Heis$ is the origin in $\R^3$, and the group inverse of $p \in \Heis$ coincides with its Euclidean additive inverse $-p$.

The vector fields
$$
X=\deriv{x}+2y\deriv{t} \quad \mbox{and} \quad
Y=\deriv{y}-2x\deriv{t},
$$
together with
$$
T = \deriv{t},
$$
form a left invariant basis for the tangent bundle of $\Heis$. The horizontal bundle $H\Heis$ is the non-integrable subbundle \index{horizontal bundle} ({\it horizontal bundle}) which is spanned, at each point $p \in \Heis$, by the values of $X$ and $Y$ at $p$.
%$$
%H_p\Heis^n = \spa \{ X_1(p),\ldots,X_n(p),Y_1(p),\ldots,Y_n(p) \} = \spa \{ X_1(p), \ldots,X_{2n}(p) \}.
%$$
Consequently, $\Heis$ has the structure of a Carnot group of step two, with two-dimensional horizontal space $\fv_1$ and one-dimensional vertical space (center) $\fv_2$. For each $r>0$ the dilation $\delta_r$ of $\Heis$ takes the form
\begin{equation}\label{dilations}
\delta_r(x,y;t) = (rx,ry;r^2t).
\end{equation}

\section{Contact mappings}\label{sec:contact}

Let $\Omega$ and $\Omega'$ be domains in a Carnot group $\G$. A diffeomorphism $f:\Omega \to \Omega'$ is said to be {\it contact}\index{contact mapping}\index{mapping!contact} if its differential preserves the horizontal bundle (and hence preserves each stratum in the decomposition \eqref{fg}). More precisely, $f$ is contact if $f_{*p}$ maps $H_p\G$ bijectively to $H_{f(p)}\G$ for each $p \in \Omega$. Since $f$ is a diffeomorphism, its action on vector fields is well defined, and the preceding statement implies that $f_*$ maps $\fv_1$ to itself (recall that $H_p\G = \{ X(p) \, : \, X \in \fv_1 \}$. Moreover, $f_*$ preserves the Lie bracket and hence maps $\fv_j$ to itself for each $j=1,\ldots,\iota$. Hence when $f$ is contact we obtain
$$
f_{*p} : H_p^j\G \to H_{f(p)}^j \G \qquad \forall\,p \in \Omega, \, j = 1,\ldots,\iota,
$$
where $H_p^j\G = \{ X(p) \, : \, X \in \fv_j \}$.

Examples of contact mappings of Carnot groups include left translations, dilations and homogeneous automorphisms. For instance, denoting by $\ell_q$ the left translation of $\G$ by a point $q$, $\ell_q(p) = q*p$, we note that $(\ell_q)_{*p}(Y_p) = Y_{q*p}$ for all left invariant vector fields $Y$ and any $p \in \G$. In particular, $(\ell_q)_*$ acts as the identity on each stratum $\fv_j$ of $\fg$. Similarly, the dilation $\delta_r$ acts on the level of the Lie algebra as follows:
$$
(\delta_r)_{*p}(Y_1+\cdots+Y_\iota) = r(Y_1)_{\delta_r(p)} + \cdots + r^\iota(Y_\iota)_{\delta_r(p)}, \qquad \mbox{where $Y_j \in \fv_j$,}
$$
and hence $(\delta_r)_*$ acts on $\fv_j$ as multiplication by $r^j$. An automorphism of $\G$ is a bijective map $L$ which preserves the Lie group law: $L(p*q) = L(p)*L(q)$. We say that $L$ is homogeneous \index{automorphism!homogeneous} \index{homogeneous automorphism} if it commutes with dilations: $L(\delta_r(p)) = \delta_r(L(p))$ for all $p \in \G$ and $r>0$. The differential of a homogeneous automorphism is an automorphism of Lie algebras which is again homogeneous and strata-preserving.

\begin{example} \label{h1rotations}
In the Heisenberg group $\Heis$, the maps
$$
R_\theta:\Heis \to \Heis, \qquad R_\theta(z,t) = (e^{\bi\theta}z,t), \qquad \theta \in \R
$$
are homogeneous automorphisms. The action of the differential of $R_\theta$ on the Lie algebra has the matrix representation
$$
\begin{pmatrix} \cos\theta & \sin\theta & 0 \\ -\sin\theta & \cos\theta & 0 \\ 0 & 0 & 1 \end{pmatrix}
$$
when expressed in the basis $X,Y,T$.
\end{example}

\section{Dimension comparison in Carnot groups}\label{sec:DCP}

Each Carnot group is equipped with two distinct metric structures: the sub-Riemannian geometry defined by the Carnot--Carath\'eodory metric and the underlying Euclidean metric geometry. These two metrics are topologically---even bi-H\"older---equivalent, but are not bi-Lipschitz equivalent (except in the trivial case when the group is abelian). It is natural to ask for precise estimates describing the relationship between the Hausdorff dimensions defined by these two metrics. This {\it Dimension Comparison Problem} was first considered by Balogh, Rickly and Serra-Cassano in the Heisenberg group \cite{brsc:comparison} (see also \cite{bt:horizdim} for a continuation of this line of research) and later by Balogh, Warhurst and the second author in general Carnot groups \cite{btw:announce}, \cite{btw:dimcomp}.

We first state the Dimension Comparison Theorem in its most general form (on arbitrary Carnot groups), and then specialize to the case of step two groups. Denote by $\G$ a Carnot group of arbitrary step $\iota \ge 1$, with stratified Lie algebra $\fg = \fv_1 \oplus \cdots \oplus \fv_\iota$. Let $m_i$ be the dimension of the vector subspace $\fv_i$ in $\fg \simeq \R^N$. Recall that the topological dimension $N$ and the homogeneous dimension $Q$ of $\G$ satisfy
$$
N = \sum_{j=1}^\iota m_j
$$
and
$$
Q = \sum_{j=1}^\iota j\,m_j.
$$
For convenience, we also set $m_0 = 0$ and $m_{\iota+1} = 0$. We define the {\it upper} and {\it lower dimension comparison functions} $\beta_+=\beta_+^\G$ and $\beta_-=\beta_-^\G$ as follows: \index{dimension comparison function}
$$
\beta_-:[0,N] \to [0,Q], \qquad \beta_-(\alpha) = \sum_{j=0}^{\ell_-} j \, m_j + (1+\ell_-) \left( \alpha - \sum_{j=0}^{\ell_-} m_j\right)
$$
where $\ell_- = \ell_-(\alpha)$ is the unique integer in $\{0,\ldots,\iota-1\}$ such that
$$
\sum_{j=0}^{\ell_-} m_j < \alpha \le \sum_{j=0}^{1+\ell_-} m_j,
$$
and
$$
\beta_+:[0,N] \to [0,Q], \qquad \beta_+(\alpha) = \sum_{j=\ell_+}^{\iota+1} j \, m_j + (-1+\ell_+) \left( \alpha - \sum_{j=\ell_+}^{\iota+1} m_j\right)
$$
where $\ell_+ = \ell_+(\alpha)$ is the unique integer in $\{2,\ldots,\iota+1\}$ such that
$$
\sum_{j=\ell_+}^{\iota+1} m_j < \alpha \le \sum_{j=-1+\ell_+}^{\iota+1} m_j.
$$
The integers $\ell_+$ and $\ell_-$ can be interpreted as weighted versions of the usual greatest integer function $\lfloor x \rfloor$, $x \in \R$. The integer $\ell_-(\alpha)$ is the largest number of layers of the Lie algebra, starting from the lowest layer $\fv_1$, for which the cumulative dimension $\sum_{j=0}^\ell m_j$ is less than $\alpha$. The integer $\ell_+(\alpha)$ has a similar interpretation, starting from the highest layer $\fv_\iota$. The dimension comparison functions $\beta_\pm(\alpha)$ each involve two terms, one of which gives the sum of the weighted dimensions of the relevant subspaces $\fv_j$ determined by the value of $\ell_\pm(\alpha)$, and the other of which gives the fractional value of the weighted dimension of the `boundary' subspace $\fv_{1+\ell_-}$ or $\fv_{-1+\ell_+}$. Note that the formulas for the dimension comparison functions $\beta_\pm$ involve only the dimensions of the subspaces $\fv_j$, and do not depend in any way on the precise algebraic relationships (e.g., commutation relations) involving vector fields determining bases for these subspaces.

In the special case of step two groups (which will be of primary interest), these formulas simplify as follows. Let $\G$ be a step two Carnot group with topological dimension $N = m_1 + m_2$ and homogeneous dimension $Q = m_1 + 2m_2$, where $\fg = \fv_1 \oplus \fv_2$ and $m_j = \dim \fv_j$. The dimension comparison functions $\beta_\pm$ satisfy
\begin{equation}\label{eq:betaminus}
\beta_-(\alpha) = \begin{cases} \alpha, & \mbox{if $0\le \alpha \le m_1$,} \\ 2\alpha - m_1, & \mbox{if $m_1 \le \alpha \le N$,} \end{cases}
\end{equation}
and
\begin{equation}\label{eq:betaplus}
\beta_+(\alpha) = \begin{cases} 2\alpha, & \mbox{if $0\le \alpha \le m_2$,} \\ \alpha+m_2, & \mbox{if $m_2 \le \alpha \le N$.}
\end{cases}
\end{equation}
Note that $\beta_-$ and $\beta_+$ are continuous and strictly increasing piecewise linear functions from $[0,N]$ to $[0,Q]$ satisfying the symmetry relation $\beta_+(N-\alpha) = Q-\beta_-(\alpha)$, $0\le\alpha\le N$. (These facts are true for the dimension comparison functions of Carnot groups of arbitrary step.)

For instance, in the first Heisenberg group $\Heis$, where $\iota=2$, $m_1=2$ and $m_2=1$, we have
$$
\beta_-(\alpha) = \max\{\alpha,2\alpha-2\}
$$
and
$$
\beta_+(\alpha) = \min\{2\alpha,\alpha+1\}
$$
for $0\le\alpha\le 3$.

We are now ready to state the Dimension Comparison Theorem. See \cite{brsc:comparison} for a proof in the Heisenberg group and \cite{btw:dimcomp} for a proof in arbitrary Carnot groups.
In the statement of the theorem, we have denoted by $\dim_{\cH,cc}$, resp.\ $\dim_{\cH,E}$, the Hausdorff dimension with respect to the metric $d_{cc}$, resp.\ $d_E$ on a Carnot group $\G$.

\begin{theorem}[Dimension Comparison in Carnot groups]\label{th:DCP}
Let $\G$ be a Carnot group of topological dimension $N$ and homogeneous dimension $Q$. Let $d_{cc}$, resp.\ $d_E$, denote the Carnot--Carath\'eodory, resp.\ Euclidean, metrics on $\G$. For any set $S \subset \G$,
\begin{equation}\label{eq:DCP}
\beta_-(\dim_{\cH,E} S) \le \dim_{\cH,cc} S \le \beta_+(\dim_{\cH,E} S).
\end{equation}
\end{theorem}

The estimates in \eqref{eq:DCP} are sharp in the following sense: for each ordered pair $(\alpha,\beta) \in [0,N]\times[0,Q]$ such that $\beta_-(\alpha) \le \beta \le \beta_+(\alpha)$, there exists a compact subset $S_{\alpha,\beta} \subset \G$ such that $\dim_{\cH,cc} S_{\alpha,\beta} = \beta$ and $\dim_{\cH,E} S_{\alpha,\beta} = \alpha$. Invariant sets of self-similar iterated function systems provide a large class of examples of sets $S$ which often realize the lower bound $\dim_{\cH,cc} S = \beta_-(\dim_{\cH,E} S$. For more details and more precise statements, see \cite{btw:dimcomp}.

%In the special case of Iwasawa groups, we may substitute the Kor\'anyi metric $d$ in place of the Carnot--Carath\'eodory metric $d_{cc}$. %Recall that these two metrics are bi-Lipschitz equivalent and so the Hausdorff dimensions defined by these two metrics coincide. We thus %obtain the conclusion
%\begin{equation}\label{eq:DCP}
%\beta_-(\dim_{\cH,E} S) \le \dim_{\cH} S \le \beta_+(\dim_{\cH,E} S)
%\end{equation}
%for any set $S \subset \G$, where $\dim_\cH$ denotes Hausdorff dimension in the metric space $(\G,d)$.

\chapter{Carnot groups of Iwasawa type and conformal mappings}\label{chap:iwasawa-and-conformal}

In this chapter we consider Carnot groups of Iwasawa type equipped with a sub-Riemannian metric.
Such groups occur as the nilpotent components in the Iwasawa
decomposition of real rank one simple Lie groups. The one-point
compactifications of these groups, equipped with suitable
sub-Riemannian metrics, arise as boundaries at infinity of the
classical rank one symmetric spaces. The study of conformal (and more
generally, quasiconformal) mappings in Iwasawa groups dates back to
the foundational work of Mostow \cite{mos:hyperbolic},
\cite{mos:rigidity} on rigidity of hyperbolic manifolds, with
significant later contributions by Kor\'anyi and Reimann
\cite{kr:heisenberg}, \cite{kr:foundations} and Pansu
\cite{pan:metriques}. In this chapter we recall the definitions and basic
analytic and geometric properties of conformal mappings of Carnot groups, especially Iwasawa groups.

\section{Iwasawa groups of complex, quaternionic or octonionic type}\label{sec:Heisenbergs}

In this section we describe a collection of examples of Carnot groups,
the so-called Iwasawa groups of complex, quaternionic and octonionic type. \index{Iwasawa group}
These are precisely the Carnot groups which admit a sufficiently rich family of conformal self-mappings (see Theorem \ref{th:cowling-ottazzi} for further explanation).

We give a unified description which covers both the complex and
quaternionic cases. In subsection \ref{subsec:grouplaw} we recall the
definition of these objects, and in subsection
\ref{subsec:sub-riem-geom} we indicate how to view them as
sub-Riemannian Carnot groups. In subsection \ref{subsec:cygan} we
describe an explicit homogeneous metric on these groups (the so-called
gauge metric) which is closely tied to conformal geometry.

The octonionic case is more subtle. While the definitions and basic
formulas, e.g., for the gauge metric, are ostensibly identical to
those in the complex and quaternionic cases, the nonassociativity of
the octonions leads to many complexities in the derivation of those
formulas. We make brief remarks about the octonionic case in
subsection \ref{subsec:octonionic-Heisenberg}, where we also provide
copious references to the literature for those wishing to learn more
about this fascinating construction.

\subsection{The group law}\label{subsec:grouplaw}

Let $\K$ denote either the complex numbers $\C$ or the quaternions $\Qua$. \index{quaternions} We denote by $k = \dim_\R\K \in \{2,4\}$ the dimension of $\K$ as an $\R$-vector space. The {\it Heisenberg group} $\Heis^n_\K$ \label{Heisenberg group} is modeled by the space $\K^n \times \Imag(\K)$ equipped with the non-abelian group law
$$
(z;\tau) \ast (z';\tau')=(z+z';\tau+\tau'+2\Imag \sum_{\nu=1}^n \overline{z_\nu'}z_\nu ).
$$
Here $z=(z_1,\ldots,z_n)$ and $z'=(z_1',\ldots,z_n')$ lie in $\K^n$ and $\tau,\tau' \in \Imag(\K) := \{ \Imag(z) \, : \, z \in \K \}$. We recall that $\overline{z}$ denotes the conjugate of an element $z \in \K$, i.e., $\overline{z} = x-\bi y$ if $z = x+\bi y \in \C$ and $\overline{z} = a-\bi b - \bj c - \bk d$ if $z=a+\bi b+\bj c+\bk d \in \Qua$, while $\Imag(z)$ denotes the imaginary part\footnote{We define the imaginary part of $z$ to be $\bi y$ rather than just $y$ in order to be consistent with the quaternionic convention.} of $z \in \K$, i.e., $\Imag(z) = \bi y$ if $z = x+\bi y \in \C$ and $\Imag(z) = \bi b+\bj c + \bk d$ if $z = a+\bi b + \bj c + \bk d \in \Qua$. When it is important to do so, we call the space $\Heis^n_\K$ either the {\it complex Heisenberg group} or the {\it quaternionic Heisenberg group}. \index{complex Heisenberg group} \index{quaternionic Heisenberg group} We often abbreviate $\Heis^n = \Heis_\C^n$ and $\Heis = \Heis^1$. Note that $\Heis^n_\K$ is identified with the Euclidean space $\R^N$ where $N = kn+k-1$, $k = \dim_\R\K$.

We briefly pause to explicitly record the group law in real coordinates. In the complex Heisenberg group $\C^n\times\Imag(\C) \leftrightarrow \R^{2n+1}$ we obtain
$$
(x,y;t) \ast (x',y';t')=(x+x',y+y';t+t'+2(x'\cdot y - x\cdot y')),
$$
where $x,x',y,y' \in \R^n$ and $t,t' \in \R$. In the quaternionic Heisenberg group $\Qua^n \times \Imag(\Qua) \leftrightarrow \R^{4n+3}$ we obtain
$$
(x,y,z,w;t,u,v)\ast(x',y',z',w';t',u',v') = (x'',y'',z'',w'';t'',u'',v'')
$$
with $x''=x+x'$, $y''=y+y'$, $z''=z+z'$, $w''=w+w'$,
$$
t'' = t+t'+2(x'\cdot y - x\cdot y' + w' \cdot z - w \cdot z'),
$$
$$
u'' = u+u'+2(x'\cdot z - x\cdot z' + y' \cdot w - y \cdot w'),
$$
and
$$
v'' = v+v'+2(x'\cdot w - x\cdot w' + z' \cdot y - z \cdot y').
$$
In this case, $x,x',y,y',z,z',w,w' \in \R^n$ and $t,t',u,u',v,v' \in \R$.

The identity element in $\G$ is the origin of the corresponding Euclidean space $\R^{kn+k-1}$; we denote this element by $o$. The inverse of $p \in \G$ coincides with its Euclidean additive inverse $-p$, however, we denote this element by $p^{-1}$ to emphasize the nonabelian character of the group $\G$.

\subsection{Left invariant vector fields and the horizontal distribution}\label{subsec:sub-riem-geom}

In this section we indicate how to view the Heisenberg group $\Heis^n_\K$ as a Carnot group. We first consider the complex Heisenberg groups. The vector fields
$$
X_j=\deriv{x_j}+2y_{j}\deriv{t} \quad \mbox{and} \quad
Y_j=\deriv{y_{j}}-2x_j\deriv{t}, \quad j=1,\ldots,n,
$$
together with
$$
T = \deriv{t},
$$
form a left invariant basis for the tangent bundle of $\Heis^n$. For convenience we also write $X_{n+j}=Y_j$ for $j=1,\ldots,n$. The sub-Riemannian geometry of the Heisenberg group is defined by the non-integrable subbundle \index{horizontal bundle} ({\it horizontal bundle}) $H\Heis^n$, where
$$
H_p\Heis^n = \spa \{ X_1(p),\ldots,X_n(p),Y_1(p),\ldots,Y_n(p) \} = \spa \{ X_1(p), \ldots,X_{2n}(p) \}.
$$
The non-integrability of this distribution is clear, since $[X_j,Y_j]=-4T$ for every $j=1,\ldots,n$. The Carnot--Carath\'eodory metric \index{Carnot Caratheodory metric@Carnot--Carath\'eodory metric} on $\Heis^n$ is defined by fixing on $H\Heis^n$ a frame such that the vector fields $X_1,Y_1,\ldots,X_n,Y_n$ are orthonormal.

The Heisenberg group $\Heis^n$ is naturally equipped with the structure of a contact manifold. Define a $1$-form $\alpha$ on $\Heis^n$ as follows:
\begin{equation}\begin{split}\label{tau}
\alpha
= dt + 2\sum_{j=1}^n (x_j \, dy_{j} - y_{j}\,dx_j)
%= dx_{2n+1} + 2\sum_{j=1}^n (x_j \, dx_{n+j} - x_{n+j}\,dx_j).
\end{split}\end{equation}
The $1$-form $\alpha$ is a contact form, \index{contact form} which means that $\alpha \wedge (d\alpha)^n$ is a nonzero multiple of the volume form. Observe also that $d\alpha = 4  \sum_{j=1}^n dx_j\wedge dx_{n+j}$ is a multiple of the standard symplectic form in $\R^{2n}$, and that the horizontal space $H_p\Heis^n$ coincides with the kernel at $p$ of $\alpha$.

For each $r>0$ the dilation $\delta_r$ of $\Heis^n$ takes the form
\begin{equation}\label{dilations2}
\delta_r(x,y;t) = (rx,ry;r^2t).
\end{equation}

A similar story can be told in the quaternionic Heisenberg group.
The vector fields
$$
X_j=\deriv{x_j}+2y_{j}\deriv{t}+2z_j\deriv{u}+2w_j\deriv{v},
$$
$$
Y_j=X_{n+j}=\deriv{y_j}-2x_{j}\deriv{t}-2w_j\deriv{u}+2z_j\deriv{v},
$$
$$
Z_j=X_{2n+j}=\deriv{z_j}+2w_{j}\deriv{t}-2x_j\deriv{u}-2y_j\deriv{v},
$$
and
$$
W_j=X_{3n+j}=\deriv{w_j}-2z_{j}\deriv{t}+2y_j\deriv{u}-2x_j\deriv{v} \qquad
\qquad j=1,\ldots,n,
$$
form a left invariant basis for the horizontal distribution
$H\Heis_\Qua^n$. This distribution is not integrable, since the Lie span
of the vector fields $X_j,Y_j,Z_j,W_j$ contains the vertical
subspace spanned by $\partial/\partial t$, $\partial/\partial u$ and
$\partial/\partial v$. Again, the Carnot--Carath\'eodory metric on $\Heis_\Qua^n$ is defined by introducing a frame on $H\Heis_\Qua^n$ for which the vector fields $X_j,Y_j,Z_j,W_j$, $j=1,\ldots,n$, are orthonormal.

The horizontal space $H_p\Heis_\Qua^n$ coincides with the kernel at $p$ of the $\Imag\Qua$-valued quaternionic contact
form
\begin{equation}\begin{split}\label{tau2}
\alpha = \Biggl(
& dt + 2 \sum_{j=1}^n (x_j \, dy_{j} - y_{j}\,dx_j + z_j \, dw_j - w_j
\, dz_j), \\
& \qquad du + 2 \sum_{j=1}^n (x_j \, dz_{j} - z_{j}\,dx_j - y_j \, dw_j + w_j
\, dy_j), \\
& \qquad \qquad dv + 2 \sum_{j=1}^n (x_j \, dw_{j} -
  w_{j}\,dx_j + y_j \, dz_j - z_j \, dy_j) \Biggr).
\end{split}\end{equation}
We will not use the quaternionic contact structure of $\Heis_\Qua^n$ in what follows.

As for the complex Heisenberg groups, the dilation $\delta_r$ of
$\Heis_\Qua^n$ takes the form \index{dilation}
\begin{equation}\label{dilations3}
\delta_r(x,y,z,w;t,u,v) = (rx,ry,rz,rw;r^2t,r^2u,r^2v).
\end{equation}

\subsection{The first octonionic Heisenberg group}\label{subsec:octonionic-Heisenberg}

The octonionic Heisenberg group $\Heis_\Oct^1$ is defined in a manner
analogous to its complex or quaternionic cousins, however, the
nonassociativity of the octonions introduces significant complications
into the derivation of basic metric features of this space and its
connection to octonionic hyperbolic geometry. We give here a brief
description and refer the reader to \cite{all:octave-projective},
\cite{all:octave-hyperbolic}, \cite{baez:octonions} and
\cite{mp:octonions} for more detailed information and background
regarding the octonionic hyperbolic plane $H_\Oct^2$ and the
first\footnote{Note that while the octonionic Heisenberg group
$\Heis_\Oct^n$ can be defined for any $n$, it is only the first
octonionic Heisenberg group $\Heis_\Oct^1$ which arises as an Iwasawa
group. This is due to complexities stemming from the nonassociativity
of the octonions, which preclude the definition of octonionic
hyperbolic spaces of dimension strictly greater than two.}
octonionic Heisenberg group $\Heis_\Oct^1$.

Recall that the octonions \index{octonions} (also known as the
{\it Cayley numbers}) \index{Cayley numbers} can be defined as the
eight-dimensional real vector space $\Oct$ spanned by indeterminates
$\be_j$, $j=0,\ldots,7$, and equipped with a certain $\R$-linear
nonassociative multiplication rule $\cdot$. By convention we take the
first indeterminate $\be_0$ to be the identity element for this
multiplication rule, i.e., $\be_0\cdot\be_j=\be_j\cdot\be_0=\be_j$ for
all $j=0,\ldots,7$. We write $\be_0 = 1$ and summarize the remaining
relations in the following multiplication table:

\

\begin{center}
\begin{tabular}{||c||c|c|c|c|c|c|c||}
\hline\hline
$\cdot$ & $\be_1$ & $\be_2$ & $\be_3$ & $\be_4$ & $\be_5$ & $\be_6$ &
$\be_7$ \\
\hline\hline
$\be_1$ & $-1$ & $\be_4$ & $\be_7$ & $-\be_2$ & $\be_6$ & $-\be_5$ & $-\be_3$ \\
\hline
$\be_2$ & $-\be_4$ & $-1$ & $\be_5$ & $\be_1$ & $-\be_3$ & $\be_7$ &
$-\be_6$ \\
\hline
$\be_3$ & $-\be_7$ & $-\be_5$ & $-1$ & $\be_6$ & $\be_2$ & $-\be_4$ & $\be_1$ \\
\hline
$\be_4$ & $\be_2$ & $-\be_1$ & $-\be_6$ & $-1$ & $\be_7$ & $\be_3$ & $\be_5$ \\
\hline
$\be_5$ & $-\be_6$ & $\be_3$ & $-\be_2$ & $-\be_7$ & $-1$ & $\be_1$ & $\be_4$ \\
\hline
$\be_6$ & $\be_5$ & $-\be_7$ & $\be_4$ & $-\be_3$ & $-\be_1$ & $-1$ & $\be_2$ \\
\hline
$\be_7$ & $\be_3$ & $\be_6$ & $-\be_1$ & $-\be_5$ & $-\be_4$ & $-\be_2$ & $-1$ \\
\hline\hline
\end{tabular}
\end{center}

\

\noindent We note that this multiplication table can be conveniently
recalled via its connection to the Fano plane \index{Fano plane} as
indicated in Figure \ref{fig:fano}. Observe that each of the basis
elements $\be_j$, $1\le j\le 7$, satisfy $\be_j^2=-1$ and anticommute
($\be_j\cdot\be_k=-\be_k\cdot\be_j$, $1\le j,k\le 7$, $j\ne k$).

\begin{figure}[h]
\includegraphics[width=1.35in]{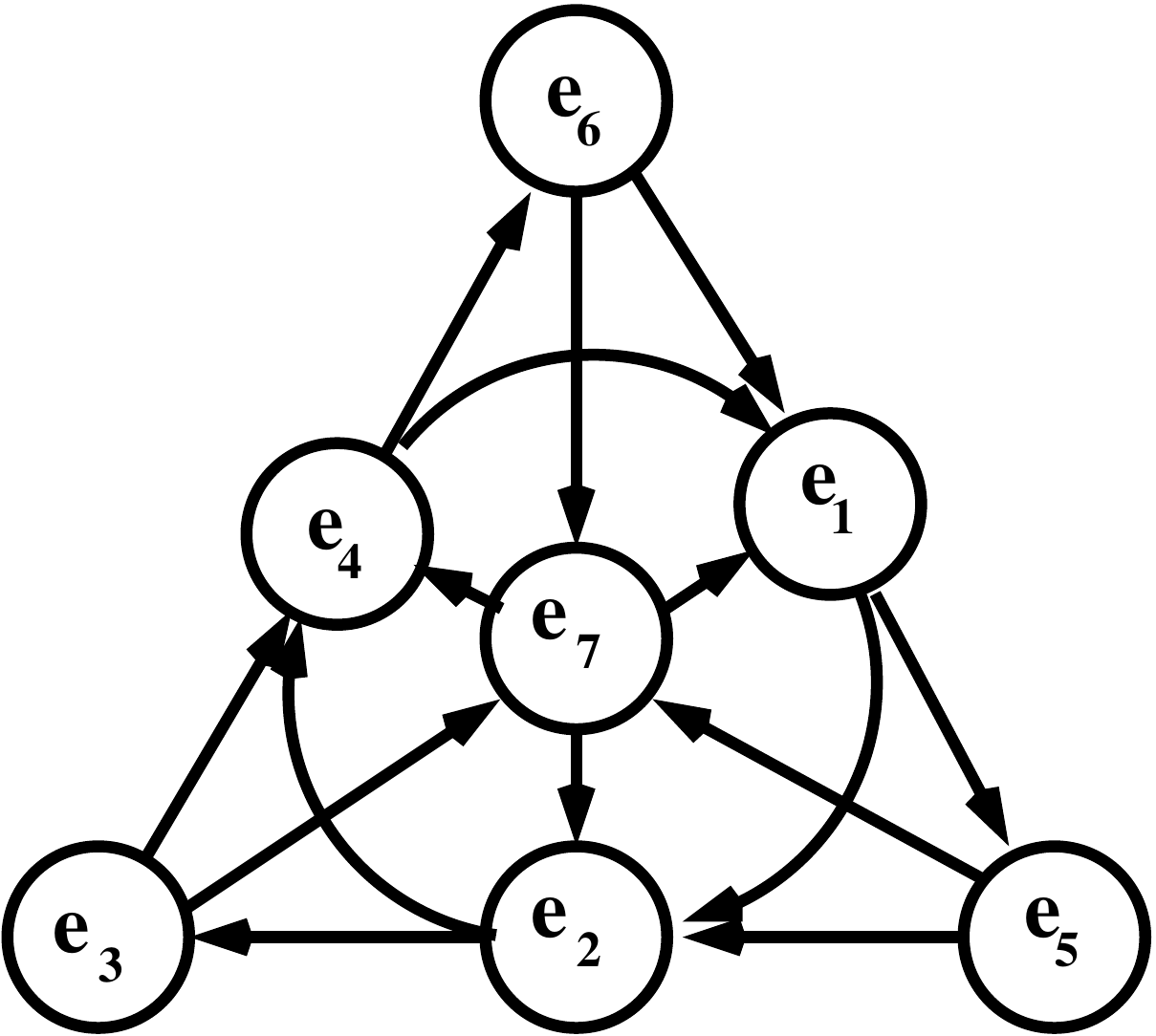}
\label{fig:fano}\caption{The Fano plane}
\end{figure}

Expressing a typical element $z \in \Oct$ in the form $z = z_0 +
\sum_{j=1}^7 z_j \be_j$ with $z_0,\ldots,z_7 \in \R$ we recall that
the real part of $z$ is $\Real(z) = z_0$ and the imaginary part is
$\Imag(z) = \sum_{j=1}^7 z_j \be_j$. The conjugate of $z = \Real(z) +
\Imag(z)$ is $\overline{z} = \Real(z) - \Imag(z)$ and the modulus is
$|z| = \sqrt{\overline{z}z} = \sqrt{z\overline{z}}$. Note that
$$
\Real(xy) = \Real(yx)
$$
for all $x,y\in\Oct$.

Since the octonions are not associative, one must be careful when performing computations.
It is useful to observe that certain associativity-type formulas hold. For instance, any product of octonions involving only two octonions is associative; this is a result of Artin. For instance, if $x \in \Oct$ and $\mu$ is an imaginary unit quaternion (so that $\overline\mu = -\mu$), then the triple product
\begin{equation}\label{eq:octonion-triple}
\mu x \overline\mu
\end{equation}
is unambiguously defined. Furthermore, the {\it Moufang identities} \index{Moufang identities}
\index{octonions!Moufang identities for}
\begin{equation}\label{eq:moufang1}
\Real((xy)z) = \Real(x(yz)),
\end{equation}
\begin{equation}\label{eq:moufang2}
(\mu x \overline\mu)(\mu y) = \mu (xy)
\end{equation}
and
\begin{equation}\label{eq:moufang3}
(x \mu) (\overline\mu y \mu) = (xy)\mu
\end{equation}
hold true for all $x,y,z \in \Oct$ and all imaginary unit quaternions $\mu$.
In view of \eqref{eq:moufang1} we may unambiguously denote by
$\Real(xyz)$ the quantity specified in that formula, despite the
nonassociativity of $\Oct$.

The first octonionic Heisenberg group is modeled as $\Heis_\Oct^1 =
\Oct \times \Imag\Oct$ (in real coordinates, as $\R^{15}$).
Here $\Imag\Oct$ denotes the space of imaginary octonions.
The group law, by analogy with the complex and quaternionic cases, is
$$
(z;\tau) * (z';\tau') = (z+z'; \tau + \tau' + 2\Imag(\overline{z'}z) ).
$$
The group $\Heis_\Oct^1$ is equipped with an $8$-dimensional
horizontal distribution $H\Heis_\Oct^1$ defined by a left-invariant
basis of horizontal vector fields $X_1,\ldots,X_8$; we omit the
explicit formulas for these vector fields in real coordinates.
By analogy with \eqref{dilations}, the dilation
automorphisms $\delta_r$ of $\Heis_\Oct^1$ are given by \index{automorphism}
$$
\delta_r(z;\tau) = (rz;r^2\tau), \qquad (z;\tau) \in \Oct.
$$

\subsection{The gauge metric of an Iwasawa group}\label{subsec:cygan}

The conformal geometry of an Iwasawa group $\G$ is naturally described
in terms of a more computationally friendly homogeneous metric called the {\it gauge metric} (also
known as the {\it Cygan} or {\it Kor\'anyi metric})
\index{gauge metric} \index{Cygan metric} \index{Koranyi metric@Kor\'anyi metric}
$$
d_H(p,q) = ||p^{-1}\cdot q||_H,
$$
where
$$
||(z;\tau)||_H = (|z|^4+|\tau|^2)^{1/4} \qquad \mbox{for $(z;\tau) \in
  \Heis_\K^n$.}
$$
The fact that $d_H$ is a homogeneous metric is well known. See \cite[p.\ 18]{cdpt:survey} for a proof in the case $\G=\Heis^n$. In the case $\G=\Heis_\Oct^1$ a proof can be found in \cite[Section 3.3]{mp:octonions}. However, since the presentation of $\Heis_\Oct^1$ in \cite{mp:octonions} is slightly different from ours, we give here a direct proof which is valid for any Iwasawa group.

\begin{proposition}\label{gauge-is-a-metric-on-first-octonionic-Heisenberg-group}
The expression $d_H$ defines a metric on any Iwasawa group $\G = \Heis_\K^n$.
\end{proposition}

\begin{proof}
It suffices to show that $||(z;\tau)*(z';\tau')||_H \le ||(z;\tau)||_H+||(z';\tau')||_H$ for all $(z;\tau)$ and $(z';\tau')$ in $\G$. We take advantage of the identity
$$
||(z;\tau)||_H^4 = \biggl|-|z|^2+\tau\biggr|^2.
$$
We compute
\begin{equation*}\begin{split}
||(z;\tau)*(z';\tau')||_H^4
&= ||(z+z';\tau+\tau'+\overline{z'}z-\overline{z}z')||_H^4 \\
&= \biggl| -|z+z'|^2 + \tau+\tau'+\overline{z'}z-\overline{z}z' \biggr|^2 \\
&= \biggl| (-|z|^2+\tau) - 2 \overline{z}z' + (-|z'|^2+\tau') \biggr|^2 \\
&\le \biggl| ||(z;\tau)||_H^2 + 2|z|\,|z'| + ||(z';\tau')||_H^2 \biggr|^2 \\
&\le (||(z;\tau)||_H+||(z';\tau')||_H)^4.
\end{split}\end{equation*}
\end{proof}

Throughout this monograph, we will denote by $d$ a general (unspecified) homogeneous metric in a Carnot group. We reserve the notations $d_{cc}$ and $d_H$ to denote the particular examples of the Carnot-Carath\'eodory metric (defined in any Carnot group) and the gauge metric (defined in any Iwasawa group). We will denote by $B(p,r)$, resp.\ $\overline{B}(p,r)$ the
open, resp.\ closed, ball with center $p$ and radius $r$ in the metric space $(\G,d)$, with similar notations for the metrics $d_{cc}$ and $d_H$.
Diameters of and distances between sets will be denoted $\diam$ and $\dist$, and suitably adorned with subscripts when necessary.

\begin{remarks}
Let $\G$ be an Iwasawa group. It follows from \eqref{quasiconvexity} that
\begin{equation}\label{ball-inclusions}
B_{cc}(p_0,r) \subset B_H(p_0,r) \subset B_{cc}(p_0,Lr)
\end{equation}
for all $p_0 \in \G$ and $r>0$. Here $L$ denotes the comparison constant from \eqref{quasiconvexity} relating the Carnot--Carath\'eodory and gauge metrics.

We also record the following elementary fact:
\begin{equation}\label{ball-diameter}
\diam_H \overline{B}_H(p,r) = 2r \qquad \mbox{for all $p \in \G$ and
  $r>0$,}
\end{equation}
Equation \eqref{ball-diameter} follows from the observation that through every point $p \in
\G$ there exist horizontal embedded lines $p * \exp(\R\,V)$, $V \in H_o\G$, $|V|_{o,\G}=1$. The restriction of $d_H$ to
such a line coincides with the Euclidean metric, whence $p * \exp(\pm rV) \in \overline{B}_H(p,r)$ are two points whose distance is equal to $2r$.
\end{remarks}

Among the many remarkable features of the gauge metric is the following {\it Ptolemaic inequality}: \index{Ptolemaic inequality}
\begin{equation}\label{eq:ptolemaic}
d_H(p_1,p_2) d_H(p_3,p_4) \le d_H(p_1,p_3) d_H(p_2,p_4) + d_H(p_2,p_3) d_H(p_1,p_4)
\end{equation}
for all $p_1,p_2,p_3,p_4$ in $\G$.
We refer the reader, for instance, to the paper \cite{pla:ptolemaic}
for a proof of the Ptolemaic inequality for the gauge metric in Iwasawa
groups. See also \cite{cdkr:h-type} and \cite{bs:mobius2} for more
information about Ptolemaic geometry and its relation to Iwasawa groups.

\subsection{Extended Iwasawa groups and cross ratios}\label{subsec:extended}

Let $\G$ be an Iwasawa group. We denote by $\oG = \G\cup\{\infty\}$ the one-point compactification of $\G$. Note that $\oG$ is homeomorphic to the sphere $\Sph^N$.
The gauge metric \index{gauge metric} $d_H$ induces a spherical metric $\od_H$ \index{spherical metric} on $\oG$, given by the formula
$$
\od_H(p,q) = \begin{cases} \frac{2d_H(p,q)}{\sqrt{1+d_H(p,o)^2} \,
\sqrt{1+d_H(q,o)^2}} & \mbox{if $p\ne\infty$ and $q\ne\infty$}, \\
\frac{2}{\sqrt{1+d_H(p,o)^2}} & \mbox{if $q=\infty$.} \end{cases}
$$
The quantity $\od_H$ is a natural analog of the usual spherical metric on the
Euclidean sphere $\Sph^n = \overline{\R^n}$. The fact that $\od_H$ is a
metric on $\oG$ follows from the Ptolemiac inequality \eqref{eq:ptolemaic}. \index{Ptolemaic inequality}

The extended Iwasawa group $M = \oG$ can be equipped with the structure of a sub-Riemannian manifold locally modeled on $\G$. In the complex case, this structure arises from the usual CR structure on the sphere $\Sph^{2n+1}$, while in the quaternionic case, it arises from the quaternionic CR structure on the sphere $\Sph^{4n+3}$. The Carnot--Carath\'eodory metric $d_{cc}$ on $\oG$ coming from this sub-Riemannian structure is the length metric generated by, and is bi-Lipschitz equivalent, to the spherical metric $\od_H$ defined above.

The {\it cross ratio} \index{cross ratio} of
four points $p_1,p_2,p_3,p_4 \in \oG$, at most two of which are equal, is
\begin{equation}\label{def:cross-ratio}
[p_1\colon\! p_2\colon\! p_3\colon\! p_4] :=
\frac{d_H(p_1,p_3)d_H(p_2,p_4)}{d_H(p_1,p_4)d_H(p_2,p_3)}.
\end{equation}
This quantity is well-defined as an element of $[0,\infty]$, where we
make the standard conventions $\frac{a}{0} = +\infty$ for $a>0$ and
$\frac{a}{+\infty} = 0$ for $a\ge 0$. If any one of the points $p_j$
is equal to the point at infinity, we modify the definition of
$[p_1\colon\! p_2\colon\! p_3\colon\! p_4]$ by deleting the terms in
the numerator and denominator containing that term. The
value of the cross ratio in \eqref{def:cross-ratio} is unchanged if
the spherical metric $\od_H$ is used in place of $d_H$.

\subsection{Summary}\label{subsec:summary}

To summarize, the Iwasawa groups $\G=\Heis_\K^n$ consist of the following:
\begin{itemize}
\item the {\it (complex) Heisenberg group} $\Heis^n=\Heis^n_\C$ for some $n\ge 1$,
\item the {\it quaternionic Heisenberg group} $\Heis_\Qua^n$ for some $n\ge 1$,
\item the {\it first octonionic Heisenberg group} $\Heis_\Oct^1$.
\end{itemize}
Each of these is a Carnot group of step two, with an even-dimensional
horizontal layer and a second layer of dimension $1$ (resp.\ $3$ or
$7$). The group $\G$ is equipped with a sub-Riemannian structure,
coming either from the Carnot--Carath\'eodory metric $d_{cc}$ or
from the bi-Lipschitz equivalent Kor\'anyi (gauge) metric $d_H$.
Moreover, the gauge metric $d_H$ induces a natural spherical metric $\od_H$
on the one-point compactification $\oG$. Later we will see that conformal mappings of $\G$ interact naturally
with the gauge metric $d_H$, extend to $\oG$ as M\"obius transformations, and preserve cross ratios.

Henceforth we denote by $m = kn$ the rank of the horizontal bundle $H\G$, by $N = kn+(k-1)$ the topological dimension of $\G$ (and also of $\oG$), and by $Q = kn+2(k-1)$ the homogeneous dimension of $\G$.

\section{Conformal mappings on Carnot groups}\label{sec:conformal}

Conformal mappings of Carnot groups admit several mutually equivalent descriptions. We adopt an abstract metric definition for conformality. All conformal mappings are smooth and contact, and a Liouville-type conformal rigidity theorem holds: every conformal map between domains in an Iwasawa group $\G$ is the restriction of a globally defined M\"obius map acting on the compactified space $\oG$. We give an explicit description of all such mappings and a list of basic geometric and analytic properties which will be used in subsequent chapters.

\subsection{Definition and basic properties}\label{subsec:conformal-definitions}

We adopt as our definition for conformal maps the standard {\it metric definition} for $1$-quasiconformal maps.

\begin{definition}\label{def:conf-defn}
Let $f:\Omega \to \Omega'$ be a homeomorphism between domains in a Carnot group $\G$. The map $f$ is said to be {\it conformal} \index{conformal mapping} if
\begin{equation}\label{metric1qc}
\lim_{r \to 0} \frac{\sup \{ d_{cc}(f(p),f(q)) \, : \, d_{cc}(p,q) = r \}}{\inf \{ d_{cc}(f(p),f(q')) \, : \, d_{cc}(p,q') = r \}} = 1
\end{equation}
for all $p \in \Omega$.
\end{definition}

\begin{remarks}\label{rem:conf-remarks}
(1) For topological reasons, the value on the left hand side of \eqref{metric1qc} is unchanged if the quotient is replaced by
\begin{equation}\label{metric1qc-alternate}
\frac{\sup \{ d_{cc}(f(p),f(q)) \, : \, d_{cc}(p,q) \le r \}}{\inf \{ d_{cc}(f(p),f(q')) \, : \, d_{cc}(p,q') \ge r \}}.
\end{equation}
The formulation of conformality using \eqref{metric1qc-alternate} is more suitable in the setting of general metric spaces, where the sphere $\{y:d(x,y)=r\}$ may be empty for certain choices of $x$ and $r>0$. See \cite{hk:quasi} for a detailed discussion of quasiconformality in metric spaces.

(2) It is clear that every Carnot--Carath\'eodory similarity map \index{similarity} \index{map!similarity} (i.e., every map $f$ which distorts all distances by a fixed scale factor $r>0$) is conformal. Examples of such maps include left translations, \index{left translation} dilations, \index{dilation} and certain automorphisms. %\index{automorphism}
In the following section we will give a complete classification of all conformal maps in Iwasawa groups (henceforth termed {\it Iwasawa conformal maps}).

(3) Definition \ref{def:conf-defn} is formulated in terms of the Carnot--Carath\'eodory metric. In Iwasawa groups, the same class of maps is obtained if the Carnot--Carath\'eodory metric is replaced throughout by the gauge metric $d_H$. In other words, $f$ is conformal if and only if
\begin{equation}\label{metric1qc2}
\lim_{r \to 0} \frac{\sup \{ d_H(f(p),f(q)) \, : \, d_H(p,q) = r \}}{\inf \{ d_H(f(p),f(q')) \, : \, d_H(p,q') = r \}} = 1
\end{equation}
for all $p \in \Omega$. See the introduction of \cite{kr:foundations} for a discussion of this matter in the complex Heisenberg groups; the rationale given there extends to the remaining Iwasawa groups without any complication.

(4) No a priori regularity is assumed in Definition \ref{def:conf-defn}. However, every homeomorphism $f$ satisfying \eqref{metric1qc} at all points $p$  in $\Omega$ is necessarily $C^\infty$.\footnote{Regularity is understood with respect to the underlying Euclidean structure.} This fact was proved by Capogna \cite{cap:regularity1}, \cite{cap:regularity2} in the case of the complex Heisenberg group, and by Capogna and Cowling \cite{cc:conformality} in arbitrary Carnot groups.

(5) Pansu \cite{pan:metriques} proved that if two domains $\Omega \subset \G$ and $\Omega' \subset \G'$ are equivalent by a conformal (or more generally, a quasiconformal) map, then $\G = \G'$. Moreover, every quasiconformal map between domains $\Omega,\Omega'$ in either the quaternionic Heisenberg group $\Heis_\H^n$ or the first octonionic Heisenberg group $\Heis_\Oct^1$ is $1$-quasiconformal and hence conformal in the sense of Definition \ref{def:conf-defn}. In the complex Heisenberg groups $\Heis^n_\C$ there exist plenty of nonconformal quasiconformal maps. An extensive theory of quasiconformal mappings in the complex Heisenberg group can be found in the papers \cite{kr:heisenberg} and \cite{kr:foundations} of Kor\'anyi and Reimann.

(6) According to a classical theorem of Liouville \index{Liouville's theorem}, conformal maps of a Euclidean space $\R^n$ of dimension at least three are the restrictions of M\"obius maps acting on $\Sph^n = \overline{\R^n}$. In Liouville's original formulation, the maps in question were assumed a priori to be of classs $C^4$. Gehring \cite{geh:rings} relaxed the smoothness assumption by showing that the same conclusion holds for $1$-quasiconformal maps (assumed initially only to lie in the local Sobolev space $W^{1,n}_\loc$). In particular, $1$-quasiconformal maps of Euclidean domains are $C^\infty$. Analogous rigidity results hold for Iwasawa conformal maps, see Theorem \ref{th:Iwasawa-Liouville}.
 %For smooth conformal maps on the complex Heisenberg groups $\Heis^n$, Liouville-type rigidity is due to Kor\'anyi and Reimann (see \cite[Theorem 8]{kr:heisenberg} for the case $n=1$ and %\cite{kr:foundations} for the general case).
\end{remarks}

%Let $\G$ denote an Iwasawa group. A diffeomorphism $f:\Omega \to \Omega'$ between domains in $\G$ is contact
%is said to be a {\it contact map} \index{contact mapping} \index{contactomorphism} (or, a {\it contactomorphism}) if it preserves the horizontal distribution $H\G$. More precisely, $f$ is %contact if $f_*(H_p\G) = H_{f(p)}\G$ for all $p \in \Omega$. In the complex case, $f$ is contact if
%\begin{equation}\label{contact-on-tau}
%f^*\alpha = \lambda_f \alpha,
%\end{equation}
%where $\alpha$ denotes the contact form defined in \eqref{tau}. Here
%$\lambda_f$ denotes a continuous nonvanishing function on $\Omega$; we
%say that $f$ is {\it orientation preserving} (resp.\ {\it orientation reversing}) if $\lambda_f>0$ (resp.\ $\lambda_f<0$) throughout $\Omega$.

For a proof of the following result, see \cite{kr:foundations} for the complex case and \cite{co:conformal-carnot} for the remaining cases.

\begin{proposition}
Every Iwasawa conformal map is a contact mapping.
\end{proposition}

In the following subsections, we give a number of examples of conformal mappings of Iwasawa groups. Later, we will see that the full group of orientation-preserving conformal mappings is generated by these mappings.

\subsection{Classification of Iwasawa conformal maps I: similarities}\label{subsec:Iwasawa-conformal-classification1}

\begin{example} \index{left translation} \index{dilation}
Left translations and dilations are conformal maps in any Carnot group.
\end{example}

\begin{example}[Rotations]\label{ex:rotation} \index{rotation}  \index{automorphism}
(1) The rotations of the (complex) Heisenberg group $\Heis^n$ are the automorphisms $R_A:\Heis^n\to\Heis^n$, $A \in U(n)$, given by
$$
R_A(z;t)=(Az;t).
$$
These were already discussed in $\Heis^1$ in Example \ref{h1rotations}. The complex Heisenberg group $\Heis^n$ is equipped with another mapping $\rho:\Heis^n\to\Heis^n$ given by $\rho(z,t) = (\overline{z},-t)$. The map $\rho$ is also an automorphism. The rotations of $\Heis^n$ are the mappings generated by $R_A$, $A \in U(n)$, and $\rho$.

\

(2) Rotations of the quaternionic Heisenberg group have a more complicated structure. The class of all rotations of $\Heis^n_\Qua$ include both the ``horizontal rotations'' $R_A:\Heis^n_\Qua\to\Heis^n_\Qua$, $A \in \Sp(n)$, given by
$$
R_A(z;\tau)=(Az;\tau)
$$
as well as the ``vertical rotations'' $R^V_B:\Heis^n_\Qua\to\Heis^n_\Qua$, $B \in \Sp(1)$, given by
$$
R^V_B(z;\tau)=(Bz\overline{B};B\tau\overline{B}).
$$
Here we identified $\Sp(1)$ with the group of unit quaternions. The full group of rotations is generated by these two types.

\

(3) For each imaginary unit octonion $\mu$, the map $R_\mu:\Heis^1_\Oct \to \Heis^1_\Oct$ given by
$$
R_\mu(z;\tau) = (z \overline\mu , \mu \tau \overline\mu )
$$
defines a rotation. Note that the expression $\mu \tau \overline\mu$ is well-defined; see \eqref{eq:octonion-triple}. The collection of maps $R_\mu$ generates the full rotation group of $\Heis^1_\Oct$. However, the function $\mu \mapsto R_\mu$ is not a homomorphism, and the full rotation group is larger than the space $\Sph^6$ of imaginary unit quaternions. In fact, the first octonionic rotation group is isomorphic to $\Spin(7)$. For more information, see \cite[Section 3.2]{mp:octonions}.
\end{example}

Rotations, left translations and dilations generate the full {\it similarity group} $\Sim(\G)$ of each Iwasawa group $\G$. \index{similarity group}

\subsection{Classification of Iwasawa conformal maps II: inversion}\label{subsec:Iwasawa-conformal-classification2}

Each Iwasawa group is equipped with a conformal inversion map which serves as an analog of the Euclidean inversion map $x \mapsto x/|x|^2$.

\begin{example}[Inversion]\label{ex:inverion}
Let $\G$ be an Iwasawa group $\G$. The {\it inversion} $\cJ:\G\setminus(o) \to \G\setminus(o)$ \index{inversion} is defined as follows.
On the complex Heisenberg group $\Heis^n$,
$$
\cJ(z;t) := \left( \frac{z}{|z|^2-\bi t}; \frac{-t}{||(z,t)||^4} \right),
$$
while on $\Heis^n_\Qua$ or $\Heis^1_\Oct$
$$
\cJ(z,\tau) := \left( z(|z|^2-\tau)^{-1}; - ||(z,\tau)||^{-4} \, \tau \right).
$$
Observe that $\cJ$ is an involution ($\cJ^2$ is the identity). It is an easy exercise (see also Remark \ref{rem:mp-remark} below) to check that
\begin{equation}\label{conformal-inversion-one}
d_H(\cJ(p),o) = \frac1{d_H(p,o)} \qquad \mbox{for all $p \in \G\setminus\{o\}$.}
\end{equation}
More generally,
\begin{equation}\label{conformal-inversion-two}
d_H(\cJ(p),\cJ(q)) = \frac{d_H(p,q)}{d_H(p,o)d_H(q,o)} \qquad \mbox{for
  all $p,q \in \G\setminus\{o\}$.}
\end{equation}
See Proposition \ref{MP} for a more general statement.
\end{example}

\begin{remark}\label{rem:mp-remark}
In the case of the first octonionic Heisenberg group, equations \eqref{conformal-inversion-one} and \eqref{conformal-inversion-two} are derived in Proposition 3.5 of \cite{mp:octonions}. However, as already mentioned, the presentation of $\Heis_\Oct^1$ is slightly different in that reference. Namely, points of $\Heis_\Oct^1$ are considered as pairs $(x,y) \in \Oct^2$ such that $x+x'+|y|^2=0$, with group law $(x,y)*(x',y') = (x+x'-\overline{y}y',y+y')$. In that presentation, the inversion map is given by $(x,y) \mapsto (|x|^{-2}\overline{x},-|x|^{-2}y\overline{x})$. We leave it as an exercise to the reader to verify that the choice $x = -|z|^2+\tau$, $y=\sqrt{2}\,z$ identifies our model of $\Heis^1_\Oct$ with that of \cite{mp:octonions}, and that the expressions for the conformal inversion coincide when this identification is taken into account.
\end{remark}

Conformality can also be defined for maps between domains in an extended Iwasawa group $\oG$, \index{extended Iwasawa group} \index{spherical metric} equipped with either the spherical metric $\od$ or the Carnot--Carath\'eodory metric $d_{cc}$. \index{Carnot-Caratheodory metric@Carnot--Carath\'eodory metric} As was the case in Definition \ref{def:conf-defn}, the class of conformal maps is the same. Each similarity of $\G$ extends to a map of $\oG$ preserving the point at infinity; the extended map is still conformal. In particular, the inversion $\cJ$ extends to a conformal self-map of $\oG$ interchanging the neutral element $o$ and~$\infty$.

\subsection{Classification of Iwasawa conformal maps III: Liouville's theorem}\label{subsec:Iwasawa-conformal-classification3}

Liouville's rigidity theorem in Iwasawa groups reads as follows. \index{Liouville's theorem}

\begin{theorem}\label{Liouville-theorem} \label{th:Iwasawa-Liouville}
Let $f:\Omega \to \Omega'$ be a conformal map between domains in an Iwasawa group $\G$. Then $f=F|_\Omega$ where $F$ is a conformal self-map of $\oG$. Moreover, the group of conformal maps of $\oG$ is generated by the similarities of $\G$ and the conformal inversion $\cJ$.
\end{theorem}

In complex Heisenberg groups, Theorem \ref{Liouville-theorem} is due to Kor\'anyi and Reimann \cite{kr:heisenberg}, \cite{kr:foundations} for smooth conformal maps and to Capogna \cite{cap:regularity1}, \cite{cap:regularity2} for general $1$-quasiconformal maps. In other Iwasawa groups, it follows from the regularity theorem of Capogna--Cowling \cite{cc:conformality} together with recent work of Cowling--Ottazzi \cite{co:conformal-carnot}.

In general, we will denote by $\Conf(\oG)$ the group of conformal self-maps of an extended Iwasawa group $\oG$. \index{extended Iwasawa group}

The decomposition of conformal maps in the following proposition
corresponds to the Bruhat decomposition of the corresponding matrix
group acting on the associated rank one symmetric space (cf.\
\cite[p.\ 312]{kr:heisenberg}). \index{Bruhat decompoisition} \index{rank one symmetric space}

\begin{proposition}\label{Bruhat}
If $f:\Omega \to \Omega'$ is a conformal map between domains in
an Iwasawa group $\G$, then $f = F|_\Omega$, where
\begin{equation}\label{G}
F = \ell_b \circ \delta_r \circ R \circ \cJ^\eps \circ \ell_{a^{-1}},
\end{equation}
where $a \in \G \setminus \Omega$, $b \in \G$, $r>0$, $\eps \in
\{0,1\}$, and $R$ is a rotation of $\G$.
\end{proposition}

The scaling factor $r$ is uniquely determined by the map $f$; we denote it by $r_f$. \index{scaling factor}

For the benefit of the reader we provide a short proof.

\begin{proof}
Let $f:\Omega \to \Omega'$ be a conformal map between domains in an
Iwasawa group $\G$. By Theorem \ref{Liouville-theorem}, $f=F|_\Omega$ for some $F$ which is conformal on $\oG$. If
$F(\infty)=\infty$ then $F$, hence also $f$, is a similarity. In this
case, \eqref{G} holds with $\eps=0$, $a=o$, $b=F(o)$, $r$ the scaling
factor of $F$, and for a suitable choice of a rotation $R$. If
$F(\infty) \ne \infty$, let $a:=F^{-1}(\infty) \in \G$. Then
$\tilde{F} := F  \circ \ell_a \circ \cJ \in \Conf(\oG)$ and
$\tilde{F}(\infty) = \infty$. It follows that $\tilde{F}$ is a
similarity of $\G$, whence, as in the first case, $\tilde{F} = \ell_b
\circ \delta_r \circ R$ for suitable $b \in \G$, $r>0$ and a rotation
$R$. The proof is complete.
\end{proof}

\begin{remark}
It follows from the preceding explicit description of $\Conf(\oG)$ that each conformal map $f$ preserves cross ratios of arbitrary quadruples of points, i.e., is a {\it M\"obius transformation}.
\index{Mobius transformation@M\"obius transformation} That is,
\begin{equation}\label{c-r-id}
[f(p_1)\colon\! f(p_2)\colon\! f(p_3)\colon\! f(p_4)]
= [p_1 \colon\! p_2 \colon\! p_3 \colon\! p_4]
\end{equation}
for all quadruples $p_1,p_2,p_3,p_4$ of points in $\oG$ for which the
expressions are defined.
\end{remark}

The extended Iwasawa group $\oG$ may be identified with the boundary at infinity of a rank one symmetric space. The classification of such spaces is well-known: every such space is either real hyperbolic space $H^{n+1}_\R$, complex hyperbolic space $H^{n+1}_\C$, quaternionic hyperbolic space $H^{n+1}_\H$, or the octonionic hyperbolic plane $H^2_\Oct$.\footnote{We omit real hyperbolic space $H^{n+1}_\R$ from the ensuing discussion as its boundary at infinity is the standard round sphere $\Sph^n$ equipped with its Euclidean metric and we are interested in nonabelian Iwasawa groups, i.e., Iwasawa groups of step two.} \index{hyperbolic space}
\index{complex hyperbolic space}
\index{quaternionic hyperbolic space}
\index{octonionic hyperbolic plane}

Elements of $\Conf(\oG)$ act as isometries on the corresponding hyperbolic space, conversely, each isometry of one of the preceding hyperbolic spaces extends to a conformal self-map of the corresponding extended Iwasawa group. Each of these isometry groups is finite dimensional, obtained via the action of a suitable matrix group. More precisely,
$$
\Conf(\overline{\Heis^n}) \simeq \Isom(\cH^{n+1}_\C) \simeq \PSU(n,1),
$$
$$
\Conf(\overline{\Heis^n_\Qua}) \simeq \Isom(\cH^{n+1}_\Qua) \simeq \PSp(n,1),
$$
and
$$
\Conf(\overline{\Heis^1_\Oct}) \simeq \Isom(\cH^2_\Oct) \simeq F_4^{(-20)}.
$$
Here $\PSU(n,1)$, resp.\ $\PSp(n,1)$, denotes the component of the
identity in the noncompact Lie group of $(n+1)\times(n+1)$ complex,
resp.\ quaternionic, matrices preserving an indefinite Hermitian form,
resp.\ quaternionic Hermitian form, of signature $(n,1)$. For the
identification of the exceptional Lie group $F_4^{(-20)}$ as the
isometry group of the octonionic hyperbolic plane, see for instance
\cite{all:octave-hyperbolic}.

\subsection{Classification of conformal mappings on Carnot groups}\label{subsec:Carnot-conformal-classification}

The following theorem of Cowling and Ottazzi \cite{co:conformal-carnot} indicates that Iwasawa groups are the only
Carnot groups which admit non-affine conformal maps. In view of this theorem, the theory of conformal dynamical systems which we develop in this paper provides for a nontrivial generalization of self-similar dynamical systems only in the Iwasawa group case. See the introduction for further information and discussion.

\begin{theorem}[Cowling--Ottazzi] \label{th:cowling-ottazzi}
Let $f:\Omega \to \Omega'$ be a conformal map between domains in a Carnot group $\G$. Then either $f$ is the restriction of a affine similarity of $\G$ or $\G$ is an Iwasawa group.
\end{theorem}

By an affine similarity \index{affine similarity} of a Carnot group $\G$ equipped with the metric $d_{cc}$, we mean the composition of a left translation, \index{left translation} a dilation, \index{dilation} and an isometric automorphism. \index{automorphism} Note that not all automorphisms of Carnot groups (even of the Heisenberg group) are necessarily isometries.

\chapter{Metric and geometric properties of conformal maps}\label{chap:conformal-metric-and-geometric-properties}

In subsequent chapters, we will only need to use a number of basic metric and geometric facts about Carnot conformal mappings. We collect these facts in the present chapter. Most of the results which we need derive from certain fundamental identities \eqref{MP1}, \eqref{MP2}, \eqref{MP3} which hold for arbitrary conformal mappings. We introduce the pointwise stretch factor \index{stretch factor} \index{pointwise stretch factor} $||Df(p)||$ of a conformal map $f$ at a point $p$. This quantity coincides with the norm of the horizontal differential of $f$ at $p$, \index{norm of the horizontal differential} \index{horizontal differential!norm} and also with the local Lipschitz constant. \index{local Lipschitz constant} Lemma \ref{harnack} states a relative continuity estimate for the function $p \mapsto ||Df(p)||$ which plays a key role in later chapters. In section \ref{sec:koebe} we use Lemma \ref{harnack} to derive estimates for Iwasawa conformal mappings in the spirit of the classical Koebe distortion theorem. \index{Koebe distortion theorem}

We reiterate that all of the results of the present chapter are valid in arbitrary Carnot groups, however, the main content comes in the Iwasawa group case. In particular, if the group in question is not of Iwasawa type, then all conformal mappings are similarities, and many of the estimates stated here become trivial. To simplify the exposition, we do not dwell on this point.

\begin{remark}\label{clarifyconformal}
Throughout this chapter, $\G$ denotes a general Carnot group and $f$ denotes a conformal map of $\G$. If $\G$ is an Iwasawa group, then $d$ will denote the gauge metric $d_H$. If $\G$ is not an Iwasawa group, then $d$ will denote the Carnot--Carath\'eodory metric $d_{cc}$.
\end{remark}

\section{Norm of the horizontal differential and local Lipschitz constant}

The following result can be derived from Proposition 2.1 of \cite{mp:octonions}. It can also be easily deduced from Proposition \ref{Bruhat}. Note that if $\G$ is not of Iwasawa type, then only \eqref{MP1} arises.

\begin{proposition}\label{MP}
Let $f \in \Conf(\oG)$ and let $r_f>0$ be the scaling factor from
Proposition \ref{Bruhat}. If $f(\infty)=\infty$ then
\begin{equation}\label{MP1}
d(f(p),f(q)) = r_f \, d(p,q) \qquad \mbox{for all $p,q\in\G$.}
\end{equation}
If $f(\infty) \ne \infty$ then
\begin{equation}\label{MP2}
d(f(p),f(q)) = \frac{r_f \, d(p,q)}{d(p,f^{-1}(\infty)) \,
  d(q,f^{-1}(\infty))} \quad \mbox{for all
  $p,q\in\G\setminus\{f^{-1}(\infty)\}$,}
\end{equation}
and
\begin{equation}\label{MP3}
d(f(p),f(\infty)) = \frac{r_f}{d(p,f^{-1}(\infty))} \qquad
\mbox{for all $p\in\G\setminus\{f^{-1}(\infty)\}$.}
\end{equation}
\end{proposition}

As a corollary, we observe that all elements of $\Conf(\oG)$ act conformally in a metric fashion, as follows.

\begin{corollary}\label{mob-is-conf}
Let $f \in \Conf(\oG)$ and let $p \in \G$, $p\ne f^{-1}(\infty)$. Then
\begin{equation}\label{eq:mob-is-conf}
\lim_{q\to p} \frac{d(f(p),f(q))}{d(p,q)}
\end{equation}
exists and is positive.
\end{corollary}

Note that existence of the limit in \eqref{eq:mob-is-conf} is a more restrictive condition than the $1$-quasiconformality at $p$ as in \eqref{metric1qc} or \eqref{metric1qc2}. Indeed, if the limit in \eqref{eq:mob-is-conf} exists, then
$$
\lim_{r \to 0} \frac{\sup \{ d(f(p),f(q)) \, : \, d(p,q) = r \}}{\inf \{ d(f(p),f(q')) \, : \, d(p,q') = r \}} =
\frac{\lim_{r\to 0} \sup \{ d(f(p),f(q))/r \, : \, d(p,q) = r \}}{\lim_{r\to 0} \inf \{ d(f(p),f(q'))/r \, : \, d(p,q') = r \}}
$$
which is equal to one.

The quantity in \eqref{eq:mob-is-conf} is the local stretching factor of $f$ at the point $p$. In Proposition \ref{Dfp} we will see that it agrees with the operator norm of $Df(p)$, the restriction of the differential $f_*$ to the horizontal tangent space at $p$. \index{horizontal differential} Note that the horizontal differential $D f(p)$ maps $H_p\G$ to $H_{f(p)}\G$. Recalling Subsection \ref{subsec:summary} we interpret $D f(p)$ as an automorphism \index{automorphism} of the vector space $\R^{kn}$ and denote its operator norm by
$$
||Df(p)||.
$$

Corollary \ref{mob-is-conf} follows from the identities
\eqref{MP1} and \eqref{MP2} upon dividing by $d(p,q)$ and letting
$q\to p$. In the case when $f(\infty)=\infty$ we obtain
$$
||Df(p)||=r_f \qquad \mbox{for all $p\in\G$,}
$$
while in the case when $f(\infty) \ne \infty$ we obtain
\begin{equation}\label{formula-for-norm-of-Df}
||Df(p)|| = \frac{r_f}{d(p,f^{-1}(\infty))^2} \qquad \mbox{for all
  $p \in \G \setminus \{f^{-1}(\infty)\}$.}
\end{equation}
It also follows from Corollary \ref{mob-is-conf} that the quantity
$||Df(\cdot)||$ satisfies the Leibniz rule \index{Leibniz rule}
\begin{equation}\label{leibniz}
||D(f\circ g)(p)|| = ||Df(g(p))||\,||Dg(p)||
\end{equation}
whenever $p \not\in \{ g^{-1}(\infty), (f\circ g)^{-1}(\infty) \}$.

Next, we provide an analytic formulation of conformality in Carnot groups as well as the promised relationship between the operator norm $||Df(p)||$ and the limit in \eqref{eq:mob-is-conf}.

\begin{theorem}\label{analytic-Carnot-conformal}
Let $f:\Omega\to\Omega'$ be a Carnot conformal map. Then
\begin{equation}\label{aCc1}
||Df(p)||^{kn} = \det Df(p)
\end{equation}
and
\begin{equation}\label{aCc2}
||Df(p)||^Q = \det f_*(p)
\end{equation}
for all $p \in \Omega$.
\end{theorem}

\begin{proposition}\label{Dfp}
Let $f:\Omega \to \Omega'$ be a Carnot conformal map. Then
\begin{equation}\label{Dfp1}
||Df(p)|| = \lim_{q\to p} \frac{d(f(p),f(q))}{d(p,q)}.
\end{equation}
\end{proposition}

We clarify that in the previous proposition $d$ denotes the gauge metric when $\G$ is an Iwasawa group and $d=\dcc$ when $\G$ is not an Iwasawa group. Recalling \eqref{MP1} we will also denote $r_f=\|D f (p)\|$ when $f$ is a metric similarity in $(\G,d)$ and $d$ is any homogeneous metric.

\begin{proof}[Proofs of Theorem \ref{analytic-Carnot-conformal} and Proposition \ref{Dfp}]
First assume that $f$ is a similarity mapping. (Recall that this is automatically the case if $\G$ is not an Iwasawa group). Then $f$ is a composition of left translations, dilations and automorphisms. In this case, \eqref{aCc1}, \eqref{aCc2} and \eqref{Dfp1} are trivial.

Suppose then that $\G$ is an Iwasawa group and that $f$ is not a similarity mapping. Since the Jacobian determinants $\det Df$ and $\det f_*$ as well as the stretch factor $||Df(p)||$ are multiplicative under composition, it suffices to verify \eqref{aCc1}, \eqref{aCc2} and \eqref{Dfp1} for the inversion mapping. This is an elementary computation which we leave to the reader.
\end{proof}

%\begin{remark}\label{heisderiv}
%If $f:\Omega\to\Omega'$ is conformal then the limit in \eqref{eq:mob-is-conf} exists and equals $||Df(p)||$ for each $p \in \Omega$.
%\index{local Lipschitz constant} \index{maximal stretching factor}
%We remark that the quantity $||Df(p)||$ coincides with the {\it local
%Lipschitz constant} (also known as the {\it maximal stretching factor}
%for the map $f$.
%\end{remark}

The following lemma states an inequality of Harnack type for the norm of the horizontal differential of a conformal mapping.

\begin{lemma}\label{harnack}
Let $S$ be a compact subset of a domain $\Omega \subset \G$. Then
there exists a constant $K_1=K_1(\delta)$ depending only on
$\delta=\diam(S)/\dist(S,\partial\Omega)$ so that
\begin{equation}\label{harnack-equation-1}
\left| \frac{||D f(p)||}{||D f(q)||} - 1 \right|
\le K_1 \frac{d(p,q)}{\diam(S)}
\end{equation}
whenever $p,q\in S$ and $f:\Omega \to \G$ is conformal. In
particular,
\begin{equation}\label{harnack-equation-2}
||Df(p)|| \le K \, ||Df(q)||
\end{equation}
for all $p,q\in S$, where $K=K(\delta)$ also depends only on $\delta$.
\end{lemma}

\begin{proof}
Let $f$ be conformal with $a = f^{-1}(\infty) \not\in \Omega$ and let
$p,q\in S \subset \Omega$. Then
$$
\frac{d(q,a)}{d(p,a)} \le
\frac{d(q,p)+d(p,a)}{d(p,a)} \le
1 + \frac{d(p,q)}{\dist(S,\partial\Omega)}
$$
and so
$$
\left( \frac{d(q,a)}{d(p,a)} \right)^2 \le
1 + \left( \frac{2}{\dist(S,\partial\Omega)} +
\frac{\diam(S)}{\dist(S,\partial\Omega)^2} \right) \, d(p,q).
$$
Reversing the roles of $p$ and $q$ yields
$$
\left| \left( \frac{d(q,a)}{d(p,a)} \right)^2 - 1\right|
\le K_1(\delta) \frac{d(p,q)}{\diam(S)}
$$
with $K_1(\delta)=2\delta+\delta^2$. An application of
\eqref{formula-for-norm-of-Df} completes the proof of
\eqref{harnack-equation-1}. Equation \eqref{harnack-equation-2}
follows from \eqref{harnack-equation-1} with
$K(\delta)=K_1(\delta)+1$.
\end{proof}

In view of \eqref{formula-for-norm-of-Df}, the maximal stretching factor $||Df||$ of a conformal map $f$ is continuous whenever it is defined. For $f:\Omega \to \G$ conformal and $S$ a compact subset of $\Omega$,
we denote by
\begin{equation}
\label{dfnorm}
||Df||_S := \max \{ ||Df(p)|| \, : \, p \in S \}.
\end{equation}

From \eqref{harnack-equation-2} we immediately deduce

\begin{corollary}\label{harnack-corollary}
Let $S$ be a compact subset of a domain $\Omega$. Then
$$
||Df||_S \le K \, ||Df(p)||
$$
for all $p \in S$, where $K=K(\delta)$ denotes the constant from Lemma \ref{harnack}.
\end{corollary}

For an arc length parameterized
$\gamma$ defined on $[a,b]$ and taking values in a metric space, we
denote by $\int_\gamma g \, ds := \int_a^b g(\gamma(t)) \, dt$ the
line integral of a real-valued Borel function $g$ along $\gamma$.

The following lemma is a standard fact. For technical reasons we
temporarily work with the Carnot--Carath\'eodory metric $d_{cc}$.

\begin{lemma}[$||Df||$ is an upper gradient for $f$]\label{upper-gradient}
Let $f:\Omega\to \Omega'$ be a conformal map between domains in a
Carnot group $\G$. Let $\gamma:[a,b] \to \Omega$ be a horizontal
curve with $\gamma(a)=p$ and $\gamma(b)=q$, parameterized with respect
to arc length in $(\Omega,d_{cc})$.
Then
$$
d_{cc}(f(p),f(q)) \le \int_\gamma ||Df(\gamma(s))|| \, ds.
$$
\end{lemma}

%Specifically, Lemma \ref{upper-gradient} follows from \cite[Lemma 5.2.8]{hkst:book}
%upon observing that $f$ is locally Lipschitz in
%$\G\setminus\{f^{-1}(\infty)\}$.

\begin{proof}
The curve $f \circ \gamma$ is a horizontal curve joining $f(p)$ to $f(q)$. Hence
$$
d_{cc}(f(p),f(q)) \le \ell_{cc}(f\circ\gamma) = \int_\gamma \langle Df(\gamma(s)),\gamma'(s) \rangle \, ds \le \int_\gamma ||Df(\gamma(s))|| \, ds.
$$
\end{proof}

\section{Koebe distortion theorems for Carnot conformal mappings}\label{sec:koebe}

The estimates of the previous section imply certain theorems in the spirit of the Koebe distortion theorem for conformal mappings of Carnot groups. We first state such theorems for the Carnot--Carath\'eodory metric, then derive corresponding results for general homogeneous metrics. \index{Koebe distortion theorem}

\begin{proposition}\label{upper-gradient-corollary}
Let $p_0 \in \Omega$ and let $r>0$ so that $S:=\overline{B}_{cc}(p_0,2r)
\subset \Omega$. Then
\begin{equation}\label{ugc1}
f(B_{cc}(p_0,r)) \subset B_{cc}(f(p_0),||Df||_S\,r)
\end{equation}
and
\begin{equation}\label{ugc2}
\diam f(B_{cc}(p_0,r)) \le ||Df||_S \, \diam(B_{cc}(p_0,r)).
\end{equation}
In particular, if $\overline{B}_{cc}(p_0,3r) \subset \Omega$, then
\begin{equation}\label{ugc3}
f(B_{cc}(p_0,r)) \subset B_{cc}(f(p_0),K\,||Df(p_0)||\,r)
\end{equation}
and
\begin{equation}\label{ugc4}
\diam_{cc} f(B_{cc}(p_0,r)) \le K \,||Df(p_0)||\,\diam_{cc}(B_{cc}(p_0,r))
\end{equation}
for some absolute constant $K$.
\end{proposition}

\begin{proof}
Given two points $p,q \in B_{cc}(p_0,r)$ there exists a Carnot--Carath\'eodory
geodesic $\gamma$ connecting $p$ to $q$ and contained in
$B_{cc}(p_0,2r)$. \index{geodesic} By Lemma \ref{upper-gradient},
$$
d_{cc}(f(p),f(q)) \le ||Df||_S\, d_{cc}(p,q).
$$
Choosing $q=p_0$ leads to \eqref{ugc1}. Taking the
supremum over all $p,q \in B(p_0,r)$ leads to \eqref{ugc2}.
Conclusions \eqref{ugc3} and \eqref{ugc4} follow via Corollary
\ref{harnack-corollary}, noting that the constant $K$ in that
corollary is uniformly bounded for points $p_0$ such that
$\overline{B}_{cc}(p_0,3r) \subset \Omega$ when $S=
\overline{B}_{cc}(p_0,2r)$.
\end{proof}

The following distortion estimate is related to the so-called {\it egg yolk principle} \index{egg yolk principle} for quasiconformal mappings, see for instance \cite[p.\ 93]{hei:lectures}. We will state the result first for the Carnot--Carath\'eodory metric, and derive a similar statement for the gauge metric on Iwasawa groups as a corollary. In what follows, $L$ will denote a comparison constant between the Carnot--Carath\'eodory and gauge metrics in case $\G$ is an Iwasawa group (cf.\ \eqref{quasiconvexity}). If $\G$ is not an Iwasawa group, we may take $L=1$.

\begin{proposition}[Koebe distortion theorem for conformal
  maps]\label{koebe}
Let $f:\Omega \to \Omega'$ be a conformal map between domains in an Iwasawa group $\G$, let $p_0\in\Omega$ and let $r>0$ be such that
$\overline{B}_{cc}(p_0,3r) \subset \Omega$. Then
\begin{equation}\label{koebe-equation}
B_{cc}(f(p_0),c^{-1}\,||Df(p_0)||r) \subset f(B_{cc}(p_0,r)) \subset
B_{cc}(f(p_0),c \,||Df(p_0)||\,r)
\end{equation}
for some constant $c$ depending only on the constant $K$ from Proposition \ref{upper-gradient-corollary} and the constant $L$ defined above.
\end{proposition}

We will make use of the following lemma, which provides a
quasisymmetry-type estimate for conformal maps in compact subsets of their domain.

\begin{lemma}\label{qsym}
Let $f:\Omega \to \Omega'$ be a conformal map between domains in a Carnot group $\G$, let $p_0\in\Omega$ and let $r>0$ be such that
$\overline{B}_{cc}(p_0,3r) \subset \Omega$. For $p_1,p_2\in
\overline{B}_{cc}(p_0,r) \subset \Omega$ define
$$
\tau := \frac{d_{cc}(f(p_0),f(p_1))}{d_{cc}(f(p_0),f(p_2))}.
$$
Then
\begin{equation}\label{qsym-eqn}
d_{cc}(p_0,p_2) \ge \frac1{2\tau L^3} d_{cc}(p_0,p_1).
\end{equation}
\end{lemma}

\begin{proof}
If $f$ is a similarity the result is obvious. Assume that $\G$ is an Iwasawa group and that $f$ is not a similarity.
We appeal to the preservation of cross-ratios by conformal maps. By assumption $q = f^{-1}(\infty) \ne \infty$. Since $f(\Omega) \subset \G$, we
must have $q \not \in \Omega$. Let $d_H$ denote the gauge metric in $\G$. Using the preservation of
cross-ratios by conformal maps \ref{c-r-id} applied to the quadruple
$p_0,p_1,p_2,q$, we deduce that
$$
\frac{d_H(f(p_0),f(p_1))}{d_H(f(p_0),f(p_2))} =
\frac{d_H(p_0,p_1)\,d_H(p_2,q)}{d_H(p_0,p_2)\,d_H(p_1,q)}
$$
and so, by \eqref{quasiconvexity},
$$
\tau = \frac{d_{cc}(f(p_0),f(p_1))}{d_{cc}(f(p_0),f(p_2))} \ge
\frac1{L^3} \,
\frac{d_{cc}(p_0,p_1)\,d_{cc}(p_2,q)}{d_{cc}(p_0,p_2)\,d_{cc}(p_1,q)}.
$$
Observe that %$d_{cc}(p_1,q) \ge 2r$ and%
$d_{cc}(p_2,q) \ge 2r$ since
$d_{cc}(p_0,q) \ge 3r$. Therefore,
\begin{equation*}
\begin{split}
\tau L^3 &\geq \frac{d_{cc}(p_2,q)}{d_{cc}(p_1,p_2)+d_{cc}(p_2,q)} \, \frac{d_{cc}(p_0,p_1)}{d_{cc}(p_0,p_2)}\\
&\geq \frac{d_{cc}(p_2,q)}{2r+d_{cc}(p_2,q)} \, \frac{d_{cc}(p_0,p_1)}{d_{cc}(p_0,p_2)} \geq \frac{d_{cc}(p_0,p_1)}{2\,d_{cc}(p_0,p_2)}.
\end{split}
\end{equation*}
\end{proof}

\begin{proof}[Proof of Proposition \ref{koebe}]
Since the right hand inclusion in \eqref{koebe-equation} has already
been proved in Proposition \ref{upper-gradient-corollary}, it suffices
to prove the left hand inclusion. We accomplish this by applying the
previous argument to the map $g=f^{-1}$.

Let $f$, $p_0$ and $r>0$ be as in the statement of the proposition.
From \eqref{leibniz} we conclude that $||Dg(f(p_0))|| =
||Df(p_0)||^{-1}$. Let $R>0$ be the maximal radius such that
\begin{equation}
\label{3Rr}
B_{cc}(f(p_0),3R) \subset f(B_{cc}(p_0,r)).
\end{equation}
For topological reasons, we conclude that
\begin{equation}
\label{top1}
\partial
f^{-1}(B_{cc}(f(p_0),3R)) \cap \partial B_{cc}(p_0,r) \ne \emptyset.
\end{equation}
Moreover, $\overline{B}_{cc}(f(p_0),3R) \subset
f(\overline{B}_{cc}(p_0,r)) \subset \Omega'$. Applying the conclusions
of Proposition \ref{upper-gradient-corollary} for $g=f^{-1}$ gives
$$
f^{-1}(B_{cc}(f(p_0),R)) \subset B_{cc}(p_0,K\,||Dg(f(p_0))||\,R) =
B_{cc}(p_0,K||Df(p_0)||^{-1}R)
$$
and so
\begin{equation}\label{inclusion-conclusion}
B_{cc}(f(p_0),R) \subset f(B_{cc}(p_0,K||Df(p_0)||^{-1}R)).
\end{equation}
In order to use this conclusion in combination with the choice of $R$
we employ Lemma \ref{qsym}. Let $s>0$ be the minimal radius such that
$$
B_{cc}(f(p_0),R) \subset f(B_{cc}(p_0,s)).
$$
Again for topological reasons we conclude that
\begin{equation}
\label{top2}
\partial B_{cc}(p_0,s)
\cap \partial f^{-1}(B_{cc}(f(p_0),R)) \ne \emptyset.
\end{equation}
By \eqref{top1} and \eqref{top2} it is possible to select
$p_1$ and $p_2$ so that
$$
d_{cc}(p_0,p_1)=r, \quad d_{cc}(p_0,p_2)=s, \quad
d_{cc}(f(p_0),f(p_1)) = 3R
$$
and
$$
d_{cc}(f(p_0),f(p_2)) = R.
$$
Applying Lemma \ref{qsym} with $\tau=3$ yields $s\ge\tfrac1{6L^3}r$.
By the choice of $s$ and \eqref{inclusion-conclusion}, we obtain $s\le
K||Df(p_0)||^{-1}R$, and so
$$
R \ge \frac1{6KL^3} ||Df(p_0)|| \, r.
$$
Therefore, using also \eqref{3Rr}
$$B_{cc}(f(p_0),\frac1{6KL^3} ||Df(p_0)|| \, r) \subset B_{cc}(f(p_0),R) \subset f(B_{cc}(p_0,r)).$$
This completes the proof.
\end{proof}

Finally, we convert the statement of the Koebe distortion theorem for conformal maps back to the gauge metric on Iwasawa groups, for convenience in later work.

\begin{corollary}[Koebe distortion for conformal maps,
  gauge metric version]\label{koebe2}
Let $f:\Omega \to \Omega'$ be a conformal map between domains in an Iwasawa group $\G$, let $p_0\in\Omega$ and let $r>0$ be such that
$\overline{B}_H(p_0,3Lr) \subset \Omega$. Then
\begin{equation}\label{koebe-equation2}
B_H(f(p_0),C^{-1}\,||Df(p_0)||r) \subset f(B_H(p_0,r)) \subset
B_H(f(p_0),C\,||Df(p_0)||\,r)
\end{equation}
for some constant $C$ depending only on the constants $K$ and $L$ as
before.
\end{corollary}

The corollary follows immediately from Proposition \ref{koebe} and the
inclusions in \eqref{ball-inclusions}.

\begin{remark}\label{Carnot-homogeneous-metrics-remark}
All of the distortion estimates discussed in this chapter continue to hold in an arbitrary Carnot group $\G$ equipped with an arbitrary homogeneous metric $d$, provided that the mapping $f:\Omega \ra \Omega'$ is assumed to be a metric similarity. In this case, \eqref{MP1} of Proposition \ref{MP} is just the definition of a metric similarity, and the remainder of the chapter follows. In particular $r_f=\|Df(p)\|=\|Df\|_\infty$, for $p \in \Omega$.  We will return to this point in connection with Definition \ref{Carnot-conformal-GDMS} and Remark \ref{abusingGDMS}, see also subsection \ref{sec:iifs1}.
\end{remark}

\chapter[Thermodynamic Formalism]{Countable alphabet symbolic dynamics: foundations of the thermodynamic formalism}\label{Chapter CASD:TFF}

In this chapter we introduce the concepts, notation, and terminology of countable alphabet symbolic dynamics. \index{symbolic dynamics} We provide the foundations of the corresponding thermodynamic formalism \index{thermodynamic formalism} with complete self-contained proofs. A more extensive exposition can be found in \cite{MUGDMS}. We stress that most of the results proved in this Chapter, e.g. Theorems \ref{t2.1.3}, \ref{t2.1.4}, \ref{t2.1.6},  and Corollary \ref{c2.7.5}, generalize results previously obtained in \cite{MUGDMS}. In \cite{MUGDMS} finite primitivity was frequently assumed, while we only need to assume finite irreducibility.

\section{Subshifts of finite type and topological pressure}\label{sec:subshifts}

Let $\mathbb{N}=\{1, 2, \ldots \}$  be  the set of all positive
integers  and let $E$ be  a  countable  set, either finite or infinite,
called in the sequel  an {\it alphabet}. \index{alphabet} Let
$$
\sg: E^\mathbb{N} \to E^\mathbb{N}
$$
be the  shift map, \index{shift map} i.e. cutting off the first coordinate. It is given by the
formula
$$
\sg\( (\om_n)^\infty_{n=1}  \) =  \( (\om_{n+1})^\infty_{n=1}  \).
$$
We also set
$$
E^*=\bigcup_{n=0}^\infty E^n.
$$
For every $\om \in E^*$, we denote by $|\om|$  the unique integer
$n \geq 0$ such that $\om \in E^n$. We  call $|\om|$ the length of
$\om$. We make the convention that $E^0=\{\es\}$. If $\om \in
E^\mathbb{N}$ and $n \geq 1$, we put
$$
\om |_n=\om_1\ldots \om_n\in E^n.
$$
If $\tau \in E^*$ and $\om \in E^* \cup E^\mathbb{N}$, we define
$$\tau\om=(\tau_1,\dots,\tau_{|\tau|},\om_1,\dots).$$
Given $\om,\tau\in E^{\mathbb N}$, we define $\omega\wedge\tau  \in
E^{\mathbb N}\cup E^*$ to be the longest
initial block common to both $\om$ and $\tau$. For each $\alpha >
0$, we define a metric $d_\alpha$ on
$E^{\mathbb N}$ by setting
\begin{equation}\label{d-alpha}
d_\alpha(\om,\tau) ={\rm e}^{-\alpha|\om\wedge\tau|}.
\end{equation}
All these metrics induce the same topology. A real or complex valued function defined on a subset of $E^\N$ is uniformly
continuous with respect to one of these metrics if and only if it is
uniformly continuous with respect to all. Also, a function is
H\"older with respect to one of these metrics if and only if it is
H\"older with respect to all; of course the H\"older exponent depends
on the metric. If no metric is specifically mentioned, we take it to
be $d_1$.

Now consider a $(0,1)$-valued matrix $A:E \times E \to \{0,1\}$. Set
$$
E^\mathbb{N}_A
:=\{\om \in E^\mathbb{N}:  \,\, A_{\om_i\om_{i+1}}=1  \,\, \mbox{for
  all}\,\,   i \in \N
\}.
$$
Elements \index{A-admissible matrices@$A$-admissible matrices} of $E^\mathbb{N}_A$ are called {\it $A$-admissible}. We also set
$$
E^n_A
:=\{w \in E^\mathbb{N}:  \,\, A_{\om_i\om_{i+1}}=1  \,\, \mbox{for
  all}\,\,  1\leq i \leq
n-1\}, \qquad \qquad n \in \N,
$$
and
$$
E^*_A:=\bigcup_{n=0}^\infty E^n_A.
$$
The elements of these sets are also  called $A$-admissible. For
every  $\om \in E^*_A$, we put
$$
[\om]:=\{\tau \in E^\mathbb{N}_A:\,\, \tau_{|_{|\om|}}=\om \}.
$$
The  following fact is obvious.

\begin{proposition}\label{p1j83}
The set $E^\mathbb{N}_A$ is  a closed subset of
$E^\mathbb{N}$, invariant under the shift map $\sg: E^\mathbb{N}\to E^\mathbb{N}$.
\end{proposition}

The matrix $A$ is said to be {\it irreducible} \index{irreducible matrix}
\index{matrix! irreducible} if there
exists $\Phi \sbt E_A^*$ such that for all $i,j\in E$
there exists $\om\in \Phi$ for which $i\om j\in E_A^*$. If the set $\Phi$ is finite the matrix $A$ is called {\it finitely irreducible} \index{finitely irreducible matrix}
\index{matrix!finitely irreducible}.

Given a set $F\sbt E$ we  put
$$
F^{\mathbb N}
:=\{\om \in E^{\mathbb N}: \, \om_i\in F
\, \mbox{for all } \,  i \in \N\},
$$
and
$$
F_A^{\mathbb N}:=E^{\mathbb N}_A \cap F^\N
=\{\om \in F^{\mathbb N}:  \,\, A_{\om_i\om_{i+1}}=1  \,\, \mbox{for
  all}\,\,  1\leq i \leq
n-1\}.
$$
A sequence $(a_n)_{n=1}^\infty$ of real numbers is said to be
{\it subadditive} \index{subadditive sequence} \index{sequence!subadditive} if
$$
a_{n+m}\le a_n+a_n \qquad \mbox{for all $m,n\ge 1$.}
$$
We recall now the following standard lemma.

\begin{lemma}\lab{subadd}
If $\(a_n\)_{n=1}^\infty$ is subadditive, then $\lim_{n\to \infty}a_n$
exists and is equal to $\inf_{n\ge 1}(a_n/n)$.
\end{lemma}

The limit in Lemma \ref{subadd} could be $-\infty$, but if the
elements $a_n$ are uniformly bounded below, then this limit is
nonnegative.

Given $ F \sbt E$ and a function $f:F_A^\mathbb{N}\to
\mathbb{R}$ we define the {\it n-th partition function}
\index{nth partition function@$n$-th partition function}
$$
Z_n(F,f)=\sum_{\om \in F_A^n}\exp \lt(\sup_{\tau\in [\om]_F}\sum_{j=0}^{n-1}f(\sg^j(\tau))\rt),
$$
where $[\om]_F=\{\tau\in F_A^\mathbb{N}:\tau|_{|\om|}=\om\}$. If $F=E$,
we simply write $[\om]$ for $[\om]_F$.

The  following lemma is indispensable for the proper definition of topological pressure.
\begin{lemma}\lab{subadd2}
The sequence $(\log Z_n(F,f))_{n=1}^\infty$ is subadditive.
\end{lemma}

\begin{proof}
We need to show that the sequence $\N\ni n\mapsto
Z_n(F,f)$ is submultiplicative\index{submultiplicative
sequence}, i.e. that
$$
Z_{m+n}(F,f)\le Z_m(F,f)Z_n(F,f)
$$
for all $m,n\ge 1$. And indeed,
$$
\aligned Z_{m+n}(F,f)
&=\sum_{\om \in F_A^{m+n}}\exp\lt(\sup_{\tau\in
  [\om]}\sum_{j=0}^{mn-1}f(\sg^j(\tau))\rt) \\
&=\sum_{\om \in F_A^{m+n}}\exp\lt(\sup_{\tau\in
  [\om]}\lt\{\sum_{j=0}^{m-1}f(\sg^j(\tau))+
  \sum_{j=0}^{n-1}f(\sg^j(\sg^m(\tau)))\rt\}\rt) \\
&\le \sum_{\om\in F_A^{m+n}}\exp\lt(\sup_{\tau\in
  [\om]}\sum_{j=0}^{m-1}f(\sg^j(\tau))+
 \sup_{\tau\in [\om]_F} \sum_{j=0}^{n-1}f(\sg^j(\sg^m(\tau))\rt) \\
&\le \sum_{\om \in F_A^m}\sum_{\rho \in F_A^n}
\exp\lt(\sup_{\tau\in [\om]}\sum_{j=0}^{m-1}f(\sg^j(\tau))+
\sup_{\g\in [\rho]}\sum_{j=0}^{n-1}f(\sg^j(\g))\rt) \\
&=\sum_{\om \in F_A^m}\exp\lt(\sup_{\tau\in [\om]}
 \sum_{j=0}^{m-1}f(\sg^j(\tau))\rt)\cdot
 \sum_{\rho \in F_A^n}\exp\lt(\sup_{\g\in [\rho]}
 \sum_{j=0}^{m-1}f(\sg^j(\g))\rt) \\
&=Z_m(F,f)Z_n(F,f).
\endaligned
$$
\end{proof}

The topological pressure \index{topological pressure} \index{pressure} of $f$ with respect to the shift map
\index{shift map} $\sg:F_A^\mathbb{N}\to F_A^\mathbb{N}$ is defined to be
\begin{equation}\lab{2.1.1}
\P_F^\sg(f):=\lim_{n\to\infty}{1\over n}\log Z_n(F,f)
       = \inf\left\{{1\over n}\log Z_n(F,f)\right\}.
\end{equation}
If $F=E$ we suppress the subscript $F$ and write simply $\P^\sg(f)$  for $\P_E^\sg(f)$ and $Z_n(f)$ for $Z_n(E,f)$.

\begin{definition}\lab{d2.1.2}
A uniformly continuous function $f:E^{\mathbb N}\to \mathbb{R}$ is called {\it
  acceptable} provided \index{acceptable function}
$$
{\rm osc}(f):=\sup_{i \in E}\{\sup(f|_{[i]})-\inf(f|_{[i]})\}<\infty.
$$
\end{definition}

Note that if the alphabet $E$ is infinite, then acceptable functions need not be bounded and as a matter of fact, those most important for us, giving rise to Gibbs and equilibrium states \index{Gibbs state} \index{equilibrium state} will be unbounded below. We now examine in detail the connections between topological pressure, entropy and integral provided by various, though related, versions of the Variational Principle. \index{variational principle}

\begin{theorem}[1st Variational Principle]\lab{2.1.5}
If $f:E_A^\N\to\R$ is a continuous function and $\^\mu$ is a $\sg$-invariant Borel
probability measure on $E_A^\N$ such that $\int fd\^\mu>-\infty$, then
$$
\h_{\^\mu}(\sg)+\int fd\^\mu\le \P^\sg(f).
$$
In addition, if $\P^\sg(f)<+\infty$, then there exists an integer $q\ge 1$ such that
$\H_{\^\mu}(\a^q)<+\infty$.
\end{theorem}

\begin{proof} If $\P^\sg(f)=+\infty$, there is nothing to prove. So,
suppose that $\P^\sg(f)<+\infty$. Then there exists $q\ge 1$ such that
$Z_n(f)<+\infty$ for every $n\ge q$. Also, for every $n\ge 1$, we have
$$
\sum_{|\om|=n}\^\mu([\om])\sup(S_nf|{(\om)})\ge \int S_nfd\^\mu=n\int
fd\^\mu>-\infty.
$$
Therefore, using concavity of the function $h(x)=-x\log x$, we
obtain for every $n\ge q$,
$$
\aligned
\H_{\^\mu}(\a^n)+\int S_nf d\^\mu
&\le \sum_{|\om|=n}\^\mu([\om])(\sup S_nf|_{[\om]}-\log\^\mu([\om])) \\
&=Z_n(f)\sum_{|\om|=n}Z_n(f)^{-1}\ep^{\sup S_nf|_{[\om]}}h\(\^\mu([\om])
\ep^{-\sup S_nf|_{[\om]}}\) \\
&\le Z_n(f)h\lt(\sum_{|\om|=n}Z_n(f)^{-1}\ep^{\sup S_nf|_{[\om]}}\^\mu([\om])
\ep^{-\sup S_nf|_{[\om]}}\rt) \\
&= Z_n(f)h(Z_n(f)^{-1})\\
&=\log\(\sum_{|\om|=n}\exp\(\sup S_nf|_{[\om]}\)\)
=\log Z_n(f).
\endaligned
$$
Therefore, $\H_{\^\mu}(\a^n)\le \log Z_n(f)+n\int (-f)d\^\mu<\infty$ for every $n\ge
q$, and since, in addition $\a^q$ is a generator, we obtain
$$
\h_{\^\mu}(\sg)+\int fd\^\mu
\le\liminf_{n\to\infty}\lt({1\over n}\lt(\H_{\^\mu}(\a^n)+\int S_nf
d\^\mu)\rt)\rt)
\le\lim_{n\to\infty}{1\over n}\log Z_n(f)
=\P^\sg(f).
$$
The proof is complete.
\end{proof}

We will also need the following theorem, which was proved in \cite{MUGDMS}
as Theorem~2.1.5.

\begin{theorem}\lab{t2.1.3}
If $f:E^{\mathbb N}_A\to \mathbb{R}$ is acceptable and $A$ is finitely
irreducible, then
$$
\P^\sg(f)=\sup\,\P_F^\sg(f),
$$
where the supremum is taken over all finite subsets $F$ of $E$.
\end{theorem}

\begin{proof} The inequality $\P^\sg(f) \geq \sup\{\P^\sg_F(f)\}$ is
obvious. In order to prove the converse let $\Phi\sbt E_A^*$ be a set of words witnessing finite irreducibility of the matrix $A$. We assume first that $\P^\sg(f)<+\infty$. Put
$$
q:=\#\Phi, \ p:=\max\{|\om|:\om\in \Phi\}, \ \text{ and } \
T:=\min\left\{\inf\sum_{j=0}^{|\om|-1}f\circ \sg^j|_{[\om]}:\om\in
\Phi\right\}.
$$
Fix $\e>0$. By acceptability of $f:E_A^\N\to \mathbb{R}$, we have
$M:= {\rm osc}(f)<\infty$ and there exists $l\ge 1$ such that
$$
|f(\om)-f(\tau)|<\e
$$
whenever $\om|_l=\tau|_l$. Now, fix $k \geq l.$ By Lemma~\ref{subadd2},
${1\over k}\log Z_k(f) \ge \P^\sg(f)$. Therefore, there exists a finite
set $F\sbt E$ such that
\begin{equation}\lab{2.1.2}
{1\over k}\log Z_k(F,f) > \P^\sg(f)-\e.
\end{equation}
We may assume that $F$ contains $\Phi$. Put
$$
\ov f:=\sum_{j=0}^{k-1}f\circ\sg^j.
$$
Now, for every element $\tau=\tau_1, \tau_2,\ld,\tau_n\in F_A^k
\times \cdots \times F_A^k$ ($n$ factors) one can choose
elements $\a_1,\a_2,\ld,\a_{n-1}\in \Phi$ such that
$\ov{\tau}=\tau_1\a_1\tau_2\a_2\ld\tau_{n-1}\a_{n-1}\tau_n\in E_A^*$.
Notice that the so defined function $\tau\mapsto \ov{\tau}$ is at most
$q^{n-1}$-to-$1$ (in fact $u^{n-1}$-to-$1$, where $u$ is the number
of lengths of words composing $\Phi$). Then for every $n\ge 1$,
$$
\aligned q^{n-1}\sum_{i=kn}^{kn+p(n-1)}Z_i(F,f) &\ge \sum_{\tau\in
(F_A^k)^n}
\exp\lt(\sup_{[{\ov\tau}]}\sum_{j=0}^{|\ov\tau|}f\circ\sg^j\rt) \\
&\ge \sum_{\tau\in (F_A^k)^n}
     \exp\lt(\inf_{[\ov\tau]}\sum_{j=0}^{|\ov\tau|}f\circ\sg^j\rt) \\
&\ge \sum_{\tau\in (F_A^k)^n}\exp\lt(\sum_{i=1}^n\inf\ov
f|_{[\tau_i]} +T(n-1)\rt) \\
&=\exp(T(n-1))\sum_{\tau\in (F_A^k)^n}\exp\sum_{i=1}^n\inf\ov
f|_{[\tau_i]}\\
&\ge \exp(T(n-1))\sum_{\tau\in (F_A^k)^n}\exp\lt(
\sum_{i=1}^n(\sup\ov
f|_{[\tau_i]} -(k-l)\e-Ml)\rt) \\
&=\exp(T(n-1)-(k-l)\e n - Mln)\sum_{\tau\in (F_A^k)^n}\exp
\sum_{i=1}^n\sup\ov f|_{[\tau_i]}\\
&=\ep^{-T}\exp\(n(T-(k-l)\e-Ml)\)\lt(\sum_{\tau\in F_A^k
}\exp(\sup\ov f|_{[\tau]})\rt)^n.
\endaligned
$$
Hence, there exists $kn\le i_n\le (k+p)n$ such that
$$
Z_{i_n}(F,f)\ge {1\over pn}e^{-T}\exp\(n(T-(k-l)\e-Ml-\log
q)\)Z_k(F,f)^n
$$
and therefore, using (\ref{2.1.2}), we obtain
$$
\P^\sg_F(f) =\lim_{n\to\infty}{1\over i_n}\log Z_{i_n}(F,f) \ge
{-|T|\over k}-\e +{l\e\over k+p}-{Ml+\log p\over k}+\P^\sg(f)-2\e \ge
\P^\sg(f)-7\e
$$
provided that $k$ is large enough. Thus, letting $\e\downto 0$, the
theorem follows. The case $\P^\sg(f)=+\infty$ can be treated similarly.
\end{proof}

\begin{remark}\label{supoverfinitir}
In Theorem \ref{t2.1.3} the supremum can be taken over all finite and irreducible sets $F \subset E$. This follows because if $\Phi \subset E_A^\ast$ is the set witnessing finite irreducibility for the matrix $A$ and if
$$
G=\{e \in E: e=\om_i \mbox{ for some }\om \in \Phi, i=1,\dots,|\om|\},
$$
then $F \cup G$ is irreducible and $\P^\sg_{F \cup G} \geq \P^\sg_{F}$.
\end{remark}

We say a $\sg$-invariant Borel probability
measure $\^\mu$ on $E_A^\N$ is {\it finitely supported} provided
there exists a finite set $F\sbt E$ such that $\^\mu(E_F^\infty)=1$. The
well-known {\it variational principle} for finitely supported measures
(see \cite{Bow}, \cite{Ru}, \cite{Wa} and \cite{PU}) tells us
that for every finite set $F\sbt E$
$$
\P^\sg_F(f)=\sup\{\h_{\^\mu}(\sg)+\int fd\^\mu\},
$$
where the supremum is taken over all $\sg$-invariant ergodic Borel
probability measures $\^\mu$ with $\^\mu(F^\infty)=1.$ Applying
Theorem~\ref{t2.1.3}, we therefore obtain the following. \index{variational principle}

\begin{theorem}[2nd Variational Principle]\lab{t2.1.4}
If $A$ is finitely
irreducible and if $f:E_A^\N\to \R$ is acceptable, then
$$
\P^\sg(f)=\sup\{\h_{\^\mu}(\sg)+\int fd\^\mu\},
$$
where the supremum is taken over all $\sg$-invariant ergodic Borel
probability measures $\^\mu$ which are finitely supported.
\end{theorem}

As an immediate consequence of Theorem~\ref{t2.1.4} and
Theorem~\ref{2.1.5}, we get the following.

\begin{theorem}[3rd Variational Principle]\lab{t2.1.6}
Suppose the incidence matrix $A$ is
finitely irreducible. If $f:E_A^\N\to\R$ is  acceptable, then
$$
\P^\sg(f)=\sup\{\h_{\^\mu}(\sg)+\int fd\^\mu\},
$$
where the supremum is taken over all $\sg$-invariant ergodic Borel
probability measures $\^\mu$ on $E_A^\N$ such that $\int fd\^\mu >-\infty$.
\end{theorem}

We end this section with the following useful technical fact.

\begin{proposition}\lab{p2.1.7}
If the incidence matrix $A$ is finitely irreducible and the function $f:E_A^\N\to\R$ is acceptable, then $\P^\sg(f)<+\infty$ if and only if $Z_1(f)<+\infty$.
\end{proposition}

\begin{proof} Let $\Phi\sbt E_A^*$ be a set of words
which witness the finite irreducibility of the incidence matrix
$A$. Let $s = \max\{|\a|:\a\in \Phi\}$ and let
$$
M=\min\lt\{\inf_{[\a]}\lt\{\sum_{j=0}^{|\a|-1}f\circ \sg^j\rt\}:\a\in
\Phi\rt\}.
$$
For $n\ge 1$ and $\om\in E_A^n$, let
$\ov{\om}=\om_1\a_1\om_2\a_2\ld\om_{n-1}\a_{n-1}\om_n$, where all
$\a_1,\ld,\a_n$ are appropriately taken from $\Phi$. Thus, $n+(n-1)
\leq |\ov{\om}| \leq n+s(n-1).$ Since $f$ is
acceptable, we therefore get
$$
\aligned
\sum_{i=n+(n-1)}^{n+s(n-1)} Z_i(f)
&=\sum_{i=n+(n-1)}^{n+s(n-1)}\sum_{\om\in E^i}\exp\lt(\sup_{[\om]}\lt\{\sum_{j=0}^{i-1}f\circ\sg^j\rt\} \rt) \\
&\ge \sum_{\om\in E_A^n}\exp\lt(\sup_{[\ov{\om}]}\lt\{\sum_{j=0}^{|\ov{\om}|}f\circ
\sg^j\rt\} \rt) \\
&\ge\sum_{\om\in E_A^n}\exp\lt(\sum_{j=1}^n\inf\(f|_{[\om_j]}\) +M(n-1)\rt)\\
&\ge \ep^{M(n-1)}\sum_{\om\in E_A^n} \exp\lt(\sum_{j=1}^n\sup\(f|_{[\om_j]}\) -\osc(f)n\rt)\\
&=\exp\(-M+(M-{\rm osc}(f))n\)\lt(\sum_{e\in
I}\exp\(\sup\(f|_{[e]}\)\)\rt)^n \\
&=\exp\(-M+(M-{\rm osc}(f))n\)Z_1(f)^n.
\endaligned
$$
From this it follows that if $\P^\sg(f)<+\infty$, then
also $Z_1(f)<+\infty$. The opposite implication is obvious since
$Z_n(f)\le Z_1(f)^n$. The proof is complete.
\end{proof}

\section{Gibbs states, equilibrium states and potentials}

\fr If $f:E_A^\N\to\R$ is a continuous function, then, following \cite{MUGDMS} (see also the references therein), a Borel probability
measure $\^m$ on $E_A^\N$ is called a {\it Gibbs state} \index{Gibbs state} for $f$ if there exist constants $Q_g\ge 1$ and $\P_{\^m}\in\R$ such that for every $\om\in E_A^*$ and every
$\tau\in [\om]$
\begin{equation}\lab{2.2.1}
Q_g^{-1}\le {\^m([\om])\over \exp\(S_{|\om|}f(\tau)-\P_{\^m}|\om|\)}
\le Q_g.
\end{equation}
If additionally $\^m$ is shift-invariant, then $\^m$ is called an
{\it invariant Gibbs state}. \index{invariant Gibbs state}

\begin{remark}\lab{r2.2.1}
Notice that the sum $S_{|\om|}f(\tau)$ in
(\ref{2.2.1}) can be replaced by $\sup(S_{|\om|}f|_{[\om]})$ or by
$\inf(S_{|\om|}f|_{[\om]})$. Also, notice that if $\^m$ is a Gibbs
state and if $\^\mu$
and $\^m$ are boundedly equivalent, meaning there some $K\ge 1$ such that
$$
K^{-1}\le\^\mu([\om])/\^m([\om])\le K
$$
for all $\om\in E_A^*$, then $\^\mu$ is also a Gibbs state for the potential $f$. We will occasionally use these facts without explicit indication.
\end{remark}

We start with the following proposition.

\begin{proposition}\lab{p2.2.2}
If $f:E_A^\N\to\R$ is a continuous function, then the following hold:
\begin{itemize}
\item[(i)] For every Gibbs state $\^m$ for $f$, $\P_{\^m}=\P^\sg(f)$.
\item[(ii)] Any two Gibbs states for the function $f$ are equivalent
with Radon-Nikodym derivatives bounded away from zero and infinity.
\end{itemize}
\end{proposition}

\begin{proof} We shall first prove (i). Towards this end fix
$n\ge 1$ and, using Remark~\ref{r2.2.1}, sum up (\ref{2.2.1}) over all words
$\om\in E_A^n$. Since $\sum_{|\om|=n}\^m([\om])=1$, we therefore get
$$
Q_g^{-1}\ep^{-\P_{\^m}n}\sum_{|\om|=n}\exp\(\sup S_nf|_{[\om]}\)
\le 1
\le Q_g\ep^{-\P_{\^m}n}\sum_{|\om|=n}\exp\(\sup S_nf|_{[\om]}\).
$$
Applying logarithms to all three terms of this formula, dividing all
the terms by $n$ and taking the limit as $n\to\infty$, we obtain
$-\P_{\^m}+\P(f)\le 0\le -\P_{\^m}+\P^\sg(f)$, which means that
$\P_{\^m}=\P^\sg(f)$. The proof of item (i) is thus complete.

In order to prove part (ii) suppose that $m$ and $\nu$ are two
Gibbs states of the function $f$. Notice now that part (i) implies the
existence of a constant $T\ge 1$ such that
$$
T^{-1}\le {\nu([\om])\over m([\om])}\le T
$$
for all words $\om\in E_A^*$. Straightforward reasoning gives now that
$\nu$ and $m$ are equivalent and $T^{-1}\le{d\nu\over dm}\le T$. The
proof is complete.
\end{proof}

As an immediate consequence of (\ref{2.2.1}) and
Remark~\ref{r2.2.1} we get the following.

\begin{proposition}\lab{p2.2.3}
Any uniformly continuous function $f:E_A^\N\to\R$ that has a Gibbs state is acceptable. \index{acceptable function}
\end{proposition}

For $\om\in E_A^*$ and $n\ge 1$, let
$$
E_n^\om(A):=\{\tau\in E_A^n:A_{\tau_n\om_1}=1\} \  \text{ and } \
E_*^\om(A):=\{\tau\in E_A^*:A_{\tau_{|\tau|}\om_1}=1\} .
$$
We shall prove the following result concerning uniqueness and some
stochastic properties of Gibbs states.

\begin{theorem}\lab{t2.2.4} If an acceptable function $f:E_A^\N\to\R$ has a
Gibbs state and the incidence matrix $A$ is finitely irreducible, then
f has a unique invariant Gibbs state. The invariant Gibbs state is
ergodic. Moreover, if $A$ is finitely primitive, the invariant Gibbs state is
completely ergodic.
\end{theorem}

\begin{proof} Let $\^m$ be a Gibbs state for $f$. Fixing $\om\in
E_A^*$, using (\ref{2.2.1}), Remark~\ref{r2.2.1} and
Proposition~\ref{p2.2.2}(a) we get for every $n\ge 1$
\begin{equation}\lab{2.2.2}
\aligned
\^m(\sg^{-n}([\om]))
&=\sum_{\tau\in E_n^\om}\^m([\tau\om]) \\
&\le \sum_{\tau\in E_n^\om}Q_g\exp\(\sup(S_{n+|\om|}f|_{[\tau\om]})-
\P^\sg(f)(n+|\om|)\) \\
&\le \sum_{\tau\in E_n^\om}Q_g\exp\(\sup(S_nf|_{[\tau]})-\P^\sg(f)n\)
\exp\(\sup(S_{|\om|}f|_{[\om]})-\P^\sg(f)|\om|\) \\
&\le \sum_{\tau\in E_n^\om}Q_gQ_g\^m([\tau])Q\^m([\om])
\le Q_g^3\^m([\om]).
\endaligned
\end{equation}
Let the finite set of words $\Phi$ witness the
finite irreducibility of the incidence matrix $A$ and let $p$ be the
maximal length of a word in $\Phi$. Since $f$ is acceptable,
$$
T=\min\{\inf(S_qf|_{[\a]})-\P^\sg(f)|\a| :\a\in \Phi\}>-\infty.
$$
For each $\tau,\om \in E_A^*$, let $\a = \alpha (\tau,\om) \in \Phi$ be such
that $\tau\alpha\om \in E_A^*.$ Then, we have for all $\om\in E_A^*$
and all $n$
\begin{equation}\lab{2.2.3}
\aligned
 \sum_{i= n}^{n+p} &\^m(\sg^{-i}([\om]))
=\sum_{i= n}^{n+p}\sum_{\tau\in E_i^\om}\^m([\tau\om]) \\
&\ge \sum_{\tau\in E_n} \^m([\tau\alpha (\tau,\om)\om]) \\
&\ge \sum_{\tau\in E_n}
  Q_g^{-1}\exp\(\inf(S_{|\tau|+|\alpha (\tau,\om)|+|\om|}f|_{[\tau\a\om]})-\P^\sg(f)(|\tau|+|\alpha (\tau,\om)|+|\om|)\) \\
&\ge Q_g^{-1}\sum_{\tau\in E_n}
 \exp\(\inf(S_n f|_{[\tau]})-\P^\sg(f)(n) + \\
&\qquad+\inf(S_{|\alpha (\tau,\om)}|f|_{[\alpha (\tau,\om)]})-\P^\sg(f)|\alpha
 (\tau,\om)|) + \inf(S_{|\om|}f|_{[\om]})-\P^\sg(f)|\om|\) \\
&\ge Q_g^{-1}\ep^T\exp\(\inf(S_{|\om|}f|_{[\om]})\!\!-\!\!\P^\sg(f)|\om|\)
 \sum_{\tau\in E_n}
 \exp\(\inf(S_nf|_{[\tau]})\!\!-\!\!\P^\sg(f)(n)\) \\
&\ge Q_g^{-2}\ep^T\^m([\om])
\sum_{\tau\in E_n}\exp\(\inf(S_nf|_{[\tau]})-\P^\sg(f)(n)\)\\
&\ge Q_g^{-2}\ep^T\^m([\om])Q_g^{-1}\sum_{\tau\in E_n|}\^m([\tau]) \\
&= Q_g^{-3}\ep^T\^m([\om]).
\endaligned
\end{equation}
Let $L:\ell_\infty\to\ell_\infty$ be a Banach limit \index{Banach limit} defined on $\ell_\infty$, the Banach space of all bounded sequences of real numbers endowed with the supremum norm. It is not difficult to check that the formula $\^\mu(A)=L\((\^m(\sg^{-n}(A)))_{n\ge 0}\)$
defines a finite, non-zero, invariant, finitely additive measure on
Borel sets of $I^\infty$ satisfying
\begin{equation}\lab{2.2.4}
{\frac {Q_g^{-3}e^T}{ p}}\^m(A)
\le \^\mu(A)
\le Q_g^3\^m(A).
\end{equation}
Since $\^m$ is a countably additive measure, we deduce that $\^\mu$ is also countably additive.

Let us prove the ergodicity of $\^\mu$ or, equivalently, of
$\^m$. Let  $\om\in E_A^n$. For each $\tau\in E^*$, we find:
\begin{equation}\lab{2.2.5}
\aligned
 \sum_{i=n}^{n+p}\^m(\sg^{-i}([\tau])\cap [\om])
&\ge \^m([\om\alpha(\omega,\tau)\tau]) \\
&\ge Q_g^{-3}\ep^T\^m([\tau])\^m([\om]).
\endaligned
\end{equation}
Take now an arbitrary Borel set $A\sbt E_A^\N$. Fix $\e>0$. Since the
nested family of sets $\{[\tau]:\tau\in E_A^*\}$ generates the Borel
$\sg$-algebra on $E_A^\N$, for every $n\ge 0$ and every $\om\in E_A^n$
we can find a subfamily $Z$ of $E_A^*$ consisting of mutually incomparable
words such that $A\sbt \bu\{[\tau]:\tau\in Z\}$ and for $n\leq i \leq n+p$,
$$
\sum_{\tau\in Z}\^m(\sg^{-(i)}([\tau])\cap [\om]) \le \^m\([\om]\cap
\sg^{-(i)}(A)\) +\e/p.
$$
Then, using \eqref{2.2.5} we get
\begin{equation}
\aligned
\e+\sum_{i=n}^{n+p}\^m\([\om]\cap\sg^{-(i)}(A)\)+ \e \\
&\ge \sum_{i=n}^{n+p}\sum_{\tau\in Z} \^m\([\om]\cap\sg^{-(i)}(\tau)\)
&\ge \sum_{\tau\in Z}Q_g^{-3}\ep^{rT}\^m([\tau])\^m([\om]) \\
&\ge Q_g^{-3}\ep^{rT}\^m(A)\^m([\om]).
\endaligned
\end{equation}
Hence, letting $\e\downto 0$, we get
$$
\sum_{i=n}^{n+p}\^m\([\om]\cap\sg^{-(i)}(A)\) \ge Q_g^{-3}\ep^T\^m(A)\^m([\om]).
$$
From this inequality we find
$$
\aligned
\sum_{i=n}^{n+p}\^m\(\sg^{-i}(E_A^\N\sms B)\cap [\om]\)
& = \sum_{i=n}^{n+p}\^m\([\om]\sms \sg^{-i}(B)\cap [\om]\) \\
&= \sum_{i=n}^{n+p}\^m([\om])-\^m\(\sg^{-i}(B)\cap [\om]\)
\le (p-Q_g^{-3}\ep^T\^m(B))\^m([\om]).
\endaligned
$$
Thus, for every Borel set $B\sbt E_A^\N$, for every $n\ge 0$, and
for every $\om\in E_A^n$ we have
\begin{equation}\lab{2.2.6}
\sum_{i=n}^{n+p}\^m(\sg^{-i}(B)\cap[\om]\)
\le (p-Q_g^{-3}\ep^T(1-\^m(B)))\^m([\om]).
\end{equation}
In order to conclude the proof of the ergodicity of $\sg$, suppose
that $\sg^{-1}(B)=B$ with
$0<\^m(B)<1$. Put $\g=1-Q_g^{-3}\ep^T(1-\^m(B))/p$. Note that $0<\g<1$.
In view of (\ref{2.2.6}), for every $\om\in E_A^*$ we get
$\^m(B\cap[\om])= \^m(\sg^{-i}(B)\cap[\om]\) \le \g\^m([\om])$.
Take now $\eta>1$ so small that
$\g\eta<1$ and choose a subfamily $R$ of $E_A^*$ consisting of mutually
incomparable words and such that $B\sbt \bu\{[\om]:\om\in R\}$ and
$\^m\(\bu\{[\om]:\om\in R\}\) \le \eta\^m(B)$. Then
$\^m(B)\le \sum_{\om\in
R}\^m(B\cap[\om])\le \sum_{\om\in R}\g\^m([\om])=\g\^m
\(\bu\{[\om]:\om\in
R\}\) \le \g\eta\^m(B)<\^m(B)$. This contradiction finishes the
proof of the existence part.

The uniqueness of the invariant Gibbs state follows immediately from
ergodicity of any invariant Gibbs state and Proposition~\ref{p2.2.2}(b).

Finally, let us prove the complete ergodicity of $\^\mu$ or, equivalently, of
$\^m$ in case $A$ is
finitely primitive. Essentially, we repeat the argument just given.
Let $\Phi$ be a finite set of words all of length $q$
which witness the finite primitiveness of $A$. Fix $r\ge
1$. Let  $\om\in E_A^n$. For each $\tau\in E_A^*$, we find the following improvement
of (\ref{2.2.3}).
\begin{equation}\lab{2.2.7}
\aligned
 \^m(\sg^{-(n+qr)}([\tau])\cap [\om])
&\ge \sum_{\a\in\Phi^r\cap E^{qr}:A_{\om_n\a_1}=A_{\a_{qr}\tau_1}=1}
\^m([\om\a\tau]) \\
&\ge Q_g^{-3}\ep^{rT}\^m([\tau])\^m([\om]).
\endaligned
\end{equation}
Take now an arbitrary Borel set $A\sbt E_A^\N$. Fix $\e>0$. Since the
nested family of sets $\{[\tau]:\tau\in E_A^*\}$ generates the Borel
$\sg$-algebra on $E_A^\N$, for every $n\ge 0$ and every $\om\in E^n$
we can find a subfamily $Z$ of $E_A^*$ consisting of mutually incomparable
words such that $A\sbt \bu\{[\tau]:\tau\in Z\}$ and
$$
\sum_{\tau\in Z}\^m(\sg^{-(n+qr)}([\tau])\cap [\om]) \le \^m\([\om]\cap
\sg^{-(n+qr)}(A)\) +\e.
$$
Then, using (\ref{2.2.7}) we get
$$
\aligned
\e+\^m\([\om]\cap\sg^{-(n+qr)}(A)\)
&\ge \sum_{\tau\in Z}Q^{-3}\ep^{rT}\^m([\tau])\^m([\om]) \\
&\ge Q_g^{-3}\ep^{rT}\^m(A)\^m([\om]).
\endaligned
$$
Hence, letting $\e\downto 0$, we get
$$
\^m\([\om]\cap\sg^{-(n+qr)}(A)\)
\ge \^Q(r)\^m(A)\^m([\om]),
$$
where $\^Q(r):=Q_g^{-3}\exp(rT)$. Note that it follows from this
last inequality that $\^Q = \^Q(r) \leq 1.$  Also, from this inequality we find
$\^m\(\sg^{-(n+qr)}(I^\infty\sms B)\cap [\om]\) =\^m\([\om]\sms \sg^{-
n}(B)\cap [\om]\) =\^m([\om])-\^m\(\sg^{-(n+qr)}(B)\cap [\om]\)\le(1-\^Q
\^m(B))\^m([\om])$. Thus, for every Borel set $B\sbt E_A^\N$, for every $n\ge 0$, and for every $\om\in E_A^n$ we have
\begin{equation}\lab{2.2.8}
\^m(\sg^{-(n+qr)}(B)\cap[\om]\) \le \(1-\^Q(1-\^m(B))\)\^m([\om]).
\end{equation}
In order to conclude the proof of the complete ergodicity of $\sg$ suppose
that $\sg^{-r}(B)=B$ with
$0<\^m(B)<1$. Put $\g=1-\^Q(1-\^m(B))$. Note that $0<\g<1$.
In view of (\ref{2.2.8}), for every $\om\in (E^r)^*$ we get
$\^m(B\cap[\om])= \^m(\sg^{-(|\om|+qr)}(B)\cap[\om]\) \le \g\^m([\om])$.
Take now $\eta>1$ so small that
$\g\eta<1$ and choose a subfamily $R$ of $(E^r)^*$ consisting of mutually
incomparable words and such that $B\sbt \bu\{[\om]:\om\in R\}$ and
$\^m\(\bu\{[\om]:\om\in R\}\) \le \eta\^m(B)$. Then
$\^m(B)\le \sum_{\om\in
R}\^m(B\cap[\om])\le \sum_{\om\in R}\g\^m([\om])=\g\^m
\(\bu\{[\om]:\om\in
R\}\) \le \g\eta\^m(B)<\^m(B)$. This contradiction finishes the
proof of the complete ergodicity of  $\^m$.
The proof is complete.
\end{proof}

There is a sort of converse to part of the preceding theorem which claims that finite irreducibility of the incidence matrix $A$ is necessary for the existence of Gibbs states. This justifies well our restriction to irreducible matrices. We need
the following lemma first. The next two results are due to Sarig \cite{Sar}.

\begin{lemma}\lab{t1021302}
Suppose the incidence matrix $A:E\times E\to\{0,1\}$ is irreducible and $\^m$ is an invariant Gibbs state for some acceptable function $f:E_A^\N\to\R$. Then there is a positive constant $K$ such that for every $e\in E$,
$$
\min\lt\{\sum_{a\in E:A_{ea} = 1}\exp\(\sup\(f|_{[a]}\),
\sum_{a\in E:A_{ae} = 1}\exp\(\sup\(f|_{[a]}\)\)\rt\}\ge K
$$
\end{lemma}

\begin{proof} Let $\^m$ be a probabbility invariant Gibbs state for $f$. Fix
$a,b\in E$ arbitrary. Since $\^m$ is a Gibbs state for $f$, we have
$$
\begin{aligned}
Q_g^{-1}e^{-2\P^\sg(f)}&\exp\(\inf\(f|_{[a]}\)\exp\(\inf\(f|_{[b]}\)\) \le \\
&\le \^m([ab]) \le \\
\le Q_g e^{-2P(f)}&\exp\(\sup\(f|_{[a]}\exp\(\sup\(f|_{[b]}\)\).
\end{aligned}
$$
Therefore,
$$
\^m([a])
= \sum_{c\in E:A_{ac}=1}\^m([ac])
\leq Q_g e^{-2\P^\sg(f)}\exp\(\sup\(f|_{[a]}\)\)\sum_{c\in E:A_{ac}=1}\exp\(\sup(f|_{[c]}\).
$$
Since $\^m$ is a Gibbs state for $f$, we have that $\^m([a])\ge Q_g^{-1}e^{-\P^\sg(f)}\exp\(\sup\(f|_{[a]}\)$. Hence,
\beq\label{11_2016_01_20}
Q_g^{-2}e^{\P^\sg(f)} \leq \sum_{c\in E:A_{ac}=1}\exp\(\sup\(f|_{[c]}\)\).
\eeq
Since $\^m$ is shift-invariant, we have that
$$
\^m([a])
= \^m(\sigma^{-1}([a]))
= \sum_{c\in E:A_{ca}=1}\^m([ca]).
$$
But then
$$
\begin{aligned}
Q_g^{-1}e^{-\P^\sg(f)}&\exp\(\sup\(f|_{[a]}\)\) \le \\
&\le \^m(\sigma^{-1}([a])) \le \\
&\le Q_g e^{-2\P^\sg(f)}\!\!\!\sum_{c\in E:A_{ca}=1}
\!\!\!\exp\(\sup\(f|_{[c]}\)\)\exp(\sup\(f|_{[a]}\)\).
\end{aligned}
$$
Therefore,
$$
Q_g^{-2}e^{\P^\sg(f)} \leq \sum_{c\in E:A_{ca}=1}\exp\(\sup\(f|_{[c]}\)\).
$$
Along with \eqref{11_2016_01_20} this completes the proof.
\end{proof}

\begin{theorem}\lab{t2.2.6}
Assume that an incidence matrix $A:E\times E\to\{0,1\}$ is irreducible.
If an acceptable function $f:E_A^\N\to\R$ has an invariant
Gibbs state, then the incidence matrix $A$ is finitely irreducible.
\end{theorem}

\begin{proof}
Since the incidence matrix $A$ is irreducible, it suffices to show that there exists a finite set of letters $F\sbt E$ such that for every letter $e\in E$ there exist $a,b\in F$ such that
$$
A_{ae} =A_{ea} = 1.
$$
Without loss of generality $E=\N$. Since $f$ has a Gibbs state, we have that
$$
\sum_{n\in \N} \exp(\sup\(f|_{[n]}\)\) \leq Qe^{\P^\sg(f)}<+\infty.
$$
So, there exists some $q\in\N$ such that
\beq\label{12_2016_01_20}
\sum_{j > q} \exp\(\sup\(f|_{[j]}\)\) < K,
\eeq
where $K$ is the constant coming from Lemma~\ref{t1021302}. Let
$$
F:=\{1,2,\ld,q\}.
$$
It now follows from \eqref{12_2016_01_20} and from Lemma~\ref{t1021302} that
every letter is followed by some element of $F$ and every letter is
preceded by some element of $F$. The proof is complete.
\end{proof}

Recall for $\om,\tau\in E^\N$, we defined $\om\wedge\tau\in
E^\N\cup E^*$ to be the longest initial block common for both $\om$ and
$\tau$. Of course if both $\om,\tau\in E_A^\N$, then also $\om\wedge\tau\in
E_A^\N\cup E_A^*$

We say that a function $f:E_A^\N\to\R$ is {\it H\"older continuous
with an exponent $\a>0$} if
$$
V_\a(f):=\sup_{n\ge 1}\{V_{\a,n}(f)\}<\infty,
$$
where
$$
V_{\a,n}(f)=\sup\{|f(\om)-f(\tau)|e^{\a(n-1)}:\om,\tau\in E_A^\N
\text{ and } |\om\wedge \tau|\ge n\}.
$$
Note if $gf:E_A^\N\to\R$ is H\"older continuous of order $\a$ and
$\theta,\psi\in E_A^\N$, then
$$
V_\a(g)d_\a(\theta,\psi)
= V_\a(g)e^{-\a|\theta\wedge\psi|}
\ge e^{-\a}|g(\theta)-g(\psi)|.
$$
Also, note that each H\"older continuous function is acceptable.

A {\it potential} \index{potential} is just a continuous function $f:E_A^\N\to\R$.
We call a $\sg$-invariant probability measure $\^\mu$ on $E_A^\N$ an {\it equilibrium
state} of the potential $f$ \index{equilibrium state} if $\int -fd\mu<+\infty$ and
\begin{equation}\lab{2.2.9}
\h_{\^\mu}(\sg)+\int f d\^\mu=\P^\sg(f).
\end{equation}
We end this section with the following two results.

\begin{lemma} \lab{l2.2.6}
Suppose that the incidence matrix $A$ is
finitely irreducible and that a continuous function
$f:E_A^\N\to\R$ has a Gibbs state. Denote by $\^\mu_f$ its unique
invariant Gibbs state (see Theorem~\ref{t2.2.4}). Then the following three
conditions are equivalent:
\begin{itemize}
\item[(a)] $\int_{E_A^\N}-f d\^\mu_f<+\infty.$

\sp\item[(b)] $\sum_{e\in E}\inf(-f|_{[e]})\exp(\inf f|_{[e]})<+\infty.$

\sp \item[(c)] $\H_{\^\mu_f}(\a)<+\infty$, where $\a=\{[e]:e\in
E\}$ is the partition of $E_A^\N$ into initial cylinders of length $1$.
\end{itemize}
\end{lemma}

\begin{proof} $(a)\imp (b).$ Suppose that $\int-f
d\^\mu_f<\infty.$  This obviously
means that $\sum_{e\in E}\int_{[e]}-f d\^\mu_f<+\infty$ and
consequently
$$
\aligned
+\infty
&>\sum_{i\in E}\inf(-f|_{[i]})\^\mu_f([i])
\ge Q_g^{-1}\sum_{i\in E}\inf(-f|_{[i]})\exp\(\inf f|_{[i]} -\P^\sg(f)\) \\
&=Q_g^{-1}e^{-\P^\sg(f)}\sum_{i\in E}\inf(-f|_{[i]})\exp\(\inf f|_{[i]}\).
\endaligned
$$
$(b)\imp (c)$. Assume that $\sum_{i\in I}\inf(-f|_{[i]})
\exp\(\inf(f|_{[i]})\)<\infty$. We shall show that
$\H_{\^\mu_f}(\a)<+\infty$. By definition,
$$
\H_{\^\mu_f}(\a)
=\sum_{i\in E}-\^\mu_f([i])\log\^\mu_f([i])
\le \sum_{i\in E}-\^\mu_f([i])\(\inf(f|_{[i]})-\P^\sg(f)-\log Q_g).
$$
Since $\sum_{i\in E}\^\mu_f([i])(\P^\sg(f)+\log Q_g\)<\infty$, it suffices to
show that
$$
\sum_{i\in E}-\^\mu_f([i])\inf(f|_{[i]})<+\infty.
$$
And indeed,
$$
\sum_{i\in E}-\^\mu_f([i])\inf(f|_{[i]})
= \sum_{i\in E}\^\mu_f([i])\sup(-f|_{[i]})
\le \sum_{i\in E}\^\mu_f([i])\(\inf(-f|_{[i]})+{\rm osc}(f)\).
$$
Since $\sum_{i\in E}\^\mu_f([i]){\rm osc}(f)={\rm osc}(f)$, it is enough to show
that
$$
\sum_{i\in E}\^\mu_f([i])\inf(-f|_{[i]}) <+\infty.
$$
Since $\^\mu_f$ is a probability measure,
$\lim_{i\to\infty}\^\mu_f([i])=0$. Therefore, it follows from (\ref{2.2.1})
that $\lim_{i\to\infty}\(\sup(f|_{[i]})-\P^\sg(f)\)=-\infty$. Thus, for
all $i$ sufficiently large, say $i\ge k$, $\sup(f|_{[i]})<0$. Hence,
for all $i\ge k$, $\inf(-f|_{[i]})=-\sup(f|_{[i]})>0$. So, using
(\ref{2.2.1}) again, we get
$$
\aligned
\sum_{i\ge k}\^\mu_f([i])\inf(-f|_{[i]})
&\le \sum_{i\ge k}Q_g\exp\(\inf(f|_{[i]})-\P^\sg(f)\)\inf(-f|_{[i]}) \\
&=Q_g\ep^{-\P^\sg(f)}\sum_{i\ge k}\exp\(\inf(f|_{[i]})\)\inf(-f|_{[i]})
\endaligned
$$
which is finite due to our asssumption. Finally, we find
$\H_{\^\mu_f}(\a)<+\infty$.

\sp\fr $(c)\imp (a)$. Suppose that
$\H_{\^\mu_f}(\a)<+\infty$. We need to show that $\int-f\,
d\^\mu_f <+\infty.$ We have
$$
\infty
>\H_{\^\mu_f}(\a)
= \sum_{i\in E}-\^\mu_f([i])\log\(\^\mu_f([i])\)
\ge \sum_{i\in E}-\^\mu_f([i])\(\inf(f|_{[i]})-\P^\sg(f)+\log Q_g\).
$$
Hence, $\sum_{i\in E}-\^\mu_f([i])\inf(f|_{[i]})<+\infty$ and
therefore
$$
\int-f d\^\mu_f
= \sum_{i\in E}\int_{[i]} -f d\^\mu_f
\le \sum_{i\in E}\sup(-f|_{[i]})\^\mu_f([i])
= \sum_{i\in E}-\inf(f|_{[i]})\^\mu_f([i])
<+\infty.
$$
The proof is complete.
\end{proof}

The next theorem shows that the assumption $\int -f\,d\^\mu_f+<\infty$ is sufficient for the appropriate Gibbs state to be
a unique equilibrium state.

\begin{theorem}\lab{t2.2.7}
Suppose that the incidence matrix $A$ is
finitely irreducible. Suppose that $f:E_A^\N\to\R$,a H\"older
continuous function, has a Gibbs state and that $\int
-fd\^\mu_f<+\infty$, where $\^\mu_f$ is the unique invariant Gibbs
state for the potential $f$ (see Theorem~\ref{t2.2.4}). Then $\^\mu_f$ is the
unique equilibrium state for the potential $f$.
\end{theorem}
\begin{proof} In order to show that $\^\mu_f$ is an equilibrium
state of the potential $f$ consider $\a=\{[e]:e\in E\}$, the
partition of $E^\infty$ into initial cylinders of length one. By
Lemma~\ref{l2.2.6}, $\H_{\^\mu_f}(\a)<+\infty$. Applying the
Breiman--Shannon--McMillan Theorem \cite[Theorem 2.5.4]{PU}, \index{Breiman-Shannon-McMillan theorem} Birkhoff's Ergodic Theorem, \index{Birkhoff's Ergodic Theorem} and (\ref{2.2.1}), we get for $\^\mu_f$-a.e. $\om\in E^\infty$
$$
\aligned
\text{h}_{\^\mu_f}(\sg)
&\ge\text{h}_{\^\mu_f} (\sg,\a)
=\lim_{n\to\infty}-{1\over n} \log\^\mu_f([\om|_n]) \\
&\ge \lim_{n\to\infty}-{1\over n} \(\log Q_g+S_nf(\om)-\P^\sg(f)n\)
=\lim_{n\to\infty}{-1\over n}S_nf(\om)+\P^\sg(f) \\
&=\int -fd\^\mu_f+\P^\sg(f)
\endaligned
$$
which, in view of Theorem~\ref{t2.1.4}, implies that $\^\mu_f$ is an
equilibrium state for the potential $f$.

 In order to prove uniqueness of equilibrium states we follow
the reasoning taken from the proof of Theorem 1 in \cite{DKU}. So, suppose
that $\^\nu\ne \^\mu_f$ is an equilibrium state for the potential
$f:E_A^\N\to \R$. Applying the ergodic decomposition theorem, we may
assume that $\^\nu$ is ergodic. Then, using (\ref{2.2.1}), we have for
every $n\ge 1$:
$$
\aligned
0
&=n(\h_{\^\nu}(\sg)+\int(f-\P^\sg(f))d\^\nu)
\le \H_{\^\nu}(\a_n)+\int(S_nf-\P^\sg(f)n)d\^\nu \\
&=-\sum_{|\om|=n}\^\nu([\om])\lt(\log\^\nu([\om])-{1\over
\^\nu([\om])} \int_{[\om]}(S_nf-\P^\sg(f)n)d\^\nu\rt)\\
&\le -\sum_{|\om|=n}\^\nu([\om])\lt(\log\^\nu([\om])-
(S_nf(\tau_\om)-\P^\sg(f)n)\rt) \text{ for a suitable } \tau_\om\in [\om] \\
&=-\sum_{|\om|=n}\^\nu([\om])\lt(\log[\^\nu([\om])\exp\(\P^\sg(f)n-
  S_nf(\tau_\om)\)]\rt) \\
&\le-\sum_{|\om|=n}\^\nu([\om])\lt(\log\^\nu([\om])(\mu_f([\om])Q_g)^{-1}
\rt) \\
&=\log Q_g - \sum_{|\om|=n}\^\nu([\om])\log\lt({\^\nu([\om])\over
\^\mu_f([\om])}\rt).
\endaligned
$$
Therefore, in order to conclude the proof, it suffices to show that
$$
\lim_{n\to\infty}\lt(-\sum_{|\om|=n}\^\nu([\om])\log\lt({\^\nu([\om])\over
\^\mu_f([\om])}\rt)\rt)=-\infty.
$$
Since both measures $\^\nu$ and $\^\mu_f$ are ergodic and
$\^\nu\ne\^\mu_f$, the measures $\^\nu$ and $\^\mu_f$ must be mutually
singular. In particular,
$$
\lim_{n\to\infty}\^\nu\lt(\lt\{\om\in E_A^\N:{\^\nu([\om|_n])\over
\^\mu_f([\om|_n])}\le S\rt\}\rt) =0
$$
for every $S>0$. For every $j\in \Z$ and every $n\ge 1$, set
$$
F_{n,j}=\lt\{\om\in E_A^\N:e^{-j}\le {\^\nu([\om|_n])\over
\^\mu_f([\om|_n])}< e^{-j+1}\rt\}.
$$
Then
$$
\^\nu(F_{n,j})
=\int_{F_{n,j}}{\^\nu([\om|_n])\over \^\mu_f([\om|_n])}d\^\mu_f(\om)
\le e^{-j+1}\^\mu_f(F_{n,j})\le e^{-j+1}.
$$
Notice
$$
-\sum_{|\om|=n}\^\nu([\om])\log\lt({\^\nu([\om|_n])\over
\^\mu_f([\om|_n])}\rt)
=-\int \log\lt({\^\nu([\om|n])\over\^\mu_f([\om|n])}\rt) d\^\nu(\om)
\le \sum_{j\in\Z}j\^\nu(F_{n,j}).
$$
Now, for each $k=-1,-2,-3,\ld$ we have
$$
\aligned
 -\sum_{|\om|=n}\^\nu([\om])\log\lt({\^\nu([\om|_n])\over
\^\mu_f([\om|_n])}\rt)
& \le k\sum_{j\le k}\^\nu(F_{n,j})+\sum_{j\ge 1}je^{-j+1} \\
&=k\^\nu\lt(\lt\{\om\in E_A^\N:{\^\nu([\om|_n])\over
\^\mu_f([\om|_n])}\ge \ep^{-k}\rt\}\rt) +\sum_{j\ge 1}je^{-j+1}.\\
\endaligned
$$
Thus, we have for each negative integer $k$,
$$
\limsup_{n\to\infty}\lt(-\sum_{|\om|=n}\^\nu([\om])\log\lt({\^\nu([\om])\over
\^\mu_f([\om])}\rt)\rt) \le k+\sum_{j\ge 1}je^{-j+1}.
$$
Now, letting $k$ go to $-\infty$ completes the proof.
\end{proof}

\section{Perron-Frobenius Operator}\label{P-FO}

In this section we define the appropriate Perron-Frobenius operators and collect some basic properties of them. These operators are primarily applied (see Theorem~\ref{t2.7.3})
to prove the existence of Gibbs states. We start  with the following technical result usually refered to as a bounded distortion lemma.

\begin{lemma}\lab{l2.3.1}
If $g:E_A^\N\to\C$ and $V_\a(g)<\infty$, then for all $n\ge 1$, for all
$\om,\tau\in E_A^\N$ with $\om_1=\tau_1$, and all $\rho\in
E_A^n$ with $A_{\rho_n\om_1}=A_{\rho_n\tau_1}=1$ we have
$$
\big|S_ng(\rho\om)-S_ng(\rho\tau)\big|\le {V_\a(g)\over e^\a-1}d_\a(\om,\tau).
$$
\end{lemma}

\begin{proof}
We have
$$
\aligned
\big|S_ng(\rho\om)-S_ng(\rho\tau)\big|
&\le \sum_{i=0}^{n-1}|g(\sg^i(\rho\om))-g(\sg^i(\tau\om))| \\
&\le \sum_{i=0}^{n-1}e^{\a}V_\a(g)d_\a(\sg^i(\rho\om),\sg^i(\tau\om)) \\
&\le e^{\a}V_\a(g)\sum_{i=0}^{n-1}e^{-\a(n-i)}d_\a(\om,\tau) \\
&\le V_\a(g){e^{-2\a}\over 1-e^{-\a}}d_\a(\om,\tau) \\
&\le {V_\a(g)\over e^\a-1}d_\a(\om,\tau).
\endaligned
$$
The proof is complete.
\end{proof}

We set
$$
T(g)=\exp\lt({V_\a(g)\over e^\a-1}\rt).
$$
From now throughout this section $f:E_A^\N\to\R$ is assumed to be a
H\"older continuous function with an exponent $\b>0$, and it is assumed to satisfy the following requirement
\begin{equation}\lab{2.3.1}
\sum_{e\in E}\exp(\sup(f|_{[e]}))<\infty.
\end{equation}
Functions $f$ satisfying this condition will be called in the sequel
{\it summable}. We note that if $f$ has a Gibbs state, then $f$ is summable.
This requirement allows us to define the {\it Perron--Frobenius operator} \index{Perron--Frobenius operator}
$\pf_f:C_b(E_A^\N)\to C_b(E_A^\N)$, acting on the space of bounded
continuous functions $C_b(E_A^\N)$ endowed with $||\cdot||_\infty$, the
supremum norm, as follows:
$$
\pf_f(g)(\om)=\sum_{e\in E:A_{e\om_1}=1}\exp(f(e\om)g(e\om).
$$
Then $||\pf_f||_\infty\le \sum_{e\in E}\exp(\sup(f|_{[e]}))<+\infty$ and for every
$n\ge 1$
$$
\pf_f^n(g)(\om)=\sum_{\tau\in E_A^n:A_{\tau_n\om_1}=1}\exp\(S_nf(\tau\om)\)
g(\tau\om).
$$
The conjugate operator $\pf_f^*$ acting on the space
$C_b^*(E_A^\N)$ has the following form:
$$
\pf_f^*(\mu)(g)=\mu(\pf_f(g)) = \int\pf_f(g)\,d\mu.
$$
Our first goal now is to study eigenmeasures of the conjugate operator $\pf_f^*$, precisely to show that these eigenmeasures coincide with Gibbs states of $f$.  We next show the existence of eigenmeasures in the case of finite alphabet, and then, using the two above facts, to prove the the existence of eigenmeasures of
$\pf_f^*$ for any countable alphabet.

We now assume that $\^m$ is an eigenmeasure of the conjugate operator
$\pf_f^*:C_b^*(E_A^\N)\to C_b^*(E_A^\N)$. The corresponding eigenvalue is
denoted by $\l$. Since $\pf_f$ is a positive operator, we have that $\l\ge 0$. Obviously
$$
\pf_f^{*n}(\^m)=\l^n\^m
$$
for every integer $n\ge 0$.  The integral version of this equality takes on the
following form
\begin{equation}\lab{2.3.2}
\int_{E_A^\N}\sum_{\tau\in E^n:A_{\tau_n\om_1}=1}\exp\(S_nf(\tau\om)\)g(\tau\om)
\,d\^m(\om)=\l^n\int_{E_A^\N} g\,d\^m,
\end{equation}
for every function $g\in C_b(E_A^\N)$. In fact this equality extends to
the space of all bounded Borel functions on $E_A^\N$. In particular, taking
$\om\in E_A^*$, say $\om\in E_A^n$, a Borel set $A\sbt E_A^\N$ such that
$A_{\om_n\tau_1}=1$, for every $\tau\in A$, and $g=\1_{[\om A]}$, we obtain
from (\ref{2.3.2})
\begin{equation}\lab{2.3.3}
\aligned
\l^n\^m([\om A])
&= \int\sum_{\tau\in E^n:A_{\tau_n\rho_1}=1}\exp\(S_nf(\tau\rho)\)
   \1_{[\om A]}(\tau\rho)\, d\^m(\rho) \\
&=\int_{\{\rho\in A:A_{\om_n\rho_1}=1\}}\exp\(S_nf(\om\rho)\)\,d\^m(\rho) \\
&=\int_A\exp\(S_nf(\om\rho)\)\, d\^m(\rho)
\endaligned
\end{equation}

\begin{remark}\lab{r2.3.2}
Note that if (\ref{2.3.3}) holds, then by representing a
Borel set $B\sbt E_A^\N$ as a union $\bu_{\om\in E^n}[\om B_\om]$,
where $B_\om=\{\a\in E_A^\N:A_{\om_n\a_1}=1 \text{ and } \om\a\in
B\}$, a straightforward calculation based on (\ref{2.3.3}) demonstrates that
(\ref{2.3.2}) is satisfied for the characteristic function $\1_B$ of the set
$B$. Next, it follows from standard approximation arguments, that
(\ref{2.3.2}) is
satisfied for all $\^m$-integrable functions $g$. Finally, we note
that $\^m$ is an eigenmeasure of the conjugate
operator $\pf_f^*$ if and only if formula (\ref{2.3.3}) is satisfied.
\end{remark}

\begin{theorem}\lab{t2.3.3}
If the incidence matrix $A$ is finitely irreducible and $f:E_A^\N\to\R$ is a summable H\"older continuous function, then the eigenmeasure $\^m$ is a Gibbs state for $f$. In addition, its corresponding eigenvalue is equal to $e^{\P^\sg(f)}$.
\end{theorem}

\begin{proof} It immediately follows from (\ref{2.3.3}) and
Lemma~\ref{l2.3.1} that
for every $\om\in E_A^*$ and every $\tau\in [\om]$
\begin{equation}\lab{2.3.1227a}
\^m([\om])
\le \l^{-n}T(f)\exp\(S_nf(\tau)\)
=T(f)\exp\(S_nf(\tau)-n\log\l\),
\end{equation}
where $n=|\om|$. On the other hand, let $\Phi$ be a minimal set which witnesses the finite
irreducibility of $A$. For every $\a\in\Phi$, let
$$
E_a=\{\tau\in E_A^\N:\om\a\tau\in E_A^\N\}.
$$
By the definition of $\Phi$, $\bu_{\a\in\Phi}E_a=E_A^\N$. Hence, there exists
$\g\in \Phi$ such that $\^m(E_\g)\ge (\#\Phi)^{-1}$. Writing $p=|\g|$ we
therefore have
\begin{equation}\lab{2.3.1227b}
\aligned
\^m([\om])
&\ge \^m([\om\g])
=\l^{-(n+p)} \int_{\rho\in E_A^\N:A_{\g_p\rho_1}=1}
 \exp\(S_{n+p}f(\om\g\rho)\) d\^m(\rho) \\
&=\l^{-(n+p)} \int_{\rho\in E^\infty:A_{\g_p\rho_1}=1}\exp\(S_nf(\om\g\rho)\)
 \exp\(S_pf(\g\rho)\) d\^m(\rho) \\
&\ge \l^{-n}\exp\(\min\{\inf(S_{|\a|}f|_{[\a]}):\a\in \Phi\}-p\log\l\)
 \int_{\rho\in E_A^\N:A_{\g_p\rho_1}=1}\!\!\!\!\!\!\!\!\!\!\!\!\!\!\!\!\!\!\!\!
 \!\!\!\!\!\!\!\!\exp\(S_nf(\om\g\rho)\)d\^m(\rho) \\
 &=C\l^{-n}\int_{E_\g}\exp\(S_nf(\om\g\rho)\)d\^m(\rho)
\ge CT(f)^{-1}\l^{-n}m(E_\g)\exp\(S_nf(\tau)\) \\
&\ge CT(f)^{-1}(\#\Phi)^{-1}\exp\(S_nf(\tau)-n\log\l\),
\endaligned
\end{equation}
where $C=\exp\(\min\{\inf(S_{|\a|}f|_{[\a]}):\a\in
\Phi\}-p\log\l\)$. Thus $\^m$ is a
Gibbs state for $f$. The equality $\l=e^{\P^\sg(f)}$ follows now immediately
from Proposition~\ref{p2.2.2}. The proof is complete. \end{proof}

\begin{theorem}\lab{t2.3.6}
If the incidence matrix $A$ is finitely irreducible and $f:E_A^\N\to\R$ is a summable H\"older continuous function, then the conjugate operator $e^{-\P^\sg(f)}\pf_f^*$ fixes at most one Borel probability measure.
\end{theorem}

\begin{proof} Suppose that $\^m$ and $\^m_1$ are such two fixed points.
In view of Proposition~\ref{p2.2.2}(b) and Theorem~\ref{t2.3.3}, the
measures $\^m$ and
$\^m_1$ are equivalent. Consider the Radon-Nikodym derivitive $\rho={d\^m_1\over d\^m}$.
Temporarily fix $\om
\in E_A^*$, say $\om\in E_A^n$. It then follows from (\ref{2.3.3}) and
Theorem~\ref{t2.3.3} that
$$
\aligned
&\^m([\om])=\\
&=\int_{\tau\in E^\infty:A_{\om_n\tau_1}=1}\exp\(S_nf(\om\tau)
 -\P^\sg(f)n\)d\^m(\tau) \\
&=\int_{\tau\in E^\infty:A_{\om_n\tau_1}=1}\exp\(S_nf(\sg(\om\tau))\)
 -\P^\sg(f)(n-1)\)\exp\(f(\om\tau)-\P^\sg(f)\)d\^m(\tau)\\
&=\int_{\tau\in E^\infty:A_{(\sg(\om))_{n-1}\tau_1}=1}\!\!\!\!\!\!\!\!\!\!\!\!
\exp\(S_nf(\sg(\om\tau))\) -\P^\sg(f)(n-1)\)
\exp\(f(\om\tau)-\P^\sg(f)\)d\^m(\tau).
\endaligned
$$
Thus,
$$
\inf\(\exp\(f|_{[\om]}-\P^\sg(f)\)\)\^m([\sg\om])
\le
\^m([\om])
\le\sup\(\exp\(f|_{[\om]}-\P^\sg(f)\)\)\^m([\sg\om]).
$$
Since $f:E_A^\N\to\R$ is H\"older continuous, we therefore conclude that
for every $\om\in E_A^\N$
\begin{equation}\lab{2.3.5}
\lim_{n\to\infty}{\^m\([\om|_n]\)\over \^m\([\sg(\om)|_{n-1}]\)}=
\exp\(f(\om)-\P^\sg(f)\)
\end{equation}
and the same formula is true with $\^m$ replaced by $\^m_1$. Using
Theorem~\ref{t2.3.3} and Theorem~\ref{t2.2.4}, there exists a set of
points $\om\in
E_A^\N$ with $\^m$ measure 1 for which the Radon-Nikodym derivatives
$\rho(\om)$ and
$\rho(\sg(\om))$ are both defined. Let $\om\in E_A^\N$ be such a point.
Then from (\ref{2.3.5}) and its version for $\^m_1$ we obtain
$$
\aligned
\rho(\om)
&=\lim_{n\to\infty}\lt({\^m_1\([\om|_n]\)\over \^m\([\om|_n]\)}\rt)\\
 &=\lim_{n\to\infty}\lt({\^m_1\([\om|_n]\)\over \^m_1\([\sg(\om)|_{n-1}]\)}
\cdot{\^m_1\([\sg(\om)|_{n-1}]\)\over \^m\([\sg(\om)|_{n-1}]\)}\cdot
{\^m\([\sg(\om)|_{n-1}]\)\over \^m\([\om|_n]\)}\rt) \\
&=\exp\(f(\om)-\P^\sg(f)\)\rho(\sg(\om))\exp\(\P^\sg(f)-f(\om)\)
=\rho(\sg(\om)).
\endaligned
$$
But according to Theorem~\ref{t2.2.4}, $\sg:E_A^\N\to E_A^\N$
is ergodic
with respect to a $\sg$-invariant measure equivalent with $\^m$, we conclude
that $\rho$ is $\^m$-almost everywhere constant. Since $\^m_1$ and $\^m$ are
both probability measures, $\^m_1=\^m$. The proof is complete.
\end{proof}

We now, for a brief moment, deal with the case of finite alphabet. We shall prove the following.

\begin{lemma}\lab{l2.7.1}
If the alphabet $E$ is finite and the incidence matrix $A$
is irreducible, then there exists an eigenmeasure $\^m$ of the conjugate
operator $\pf_f^*$.
\end{lemma}

\begin{proof}
By our assumptions, primarily by irreducibility of the incidence matrix $A$, the operator $\pf_f$ is strictly positive
(in the sense that it maps strictly positive functions into strictly positive functions). In particular the following formula
$$
\nu\mapsto {\pf_f^*(\nu)\over \pf_f^*(\nu)(\1)}
$$
defines a continuous self-map of $M_A(\sg)$, the space of Borel probability measures on
$E_A^\N$ endowed with the topology of weak convergence. Since $E_A^\N$ is a compact metrizable space, $M_A(\sg)$ is a convex compact subset of $C^*(E_A^\N)$ which itself is a locally convex topological vector space when endowed with the weak topology. The \index{Schauder--Tychonoff Theorem} Schauder--Tychonoff Theorem \cite[V.10.5 Theorem 5]{DuS1} thus applies, and as its consequence, we conclude that the map defined above has a fixed point, say $\^m$. Then $\pf_f^*(\^m)=\l\^m$,
where $\l=\pf_f^*(\^m)(\1)$, and the proof is complete.
\end{proof}

For the proof of Theorem~\ref{t2.7.3}, which is actually the main result of this section, we will need a
simple fact about irreducible matrices. We will provide its short proof for
the sake of completeness and convenience of the reader. It is more natural and convenient to formulate it
in the language of directed graphs. Let us recall that a directed graph is
said to be strongly connected if and only if its incidence matrix is
irreducible. In other words, it means that every two vertices can be joined
by a path of admissible edges.

\begin{lemma}\lab{l2.7.2}
If $\Ga=<E,V>$ is a strongly connected directed
graph, then there exists a sequence of strongly connected
subgraphs $<E_n, V_n>$ of $\Ga$ such that all the vertices $V_n\sbt V$
and all the edges $E_n$ are finite,
$\{V_n\}_{n=1}^\infty$ is an increasing sequence of vertices,
$\{E_n\}_{n=1}^\infty$ is an increasing sequence of edges,
$\bu_{n=1}^\infty V_n=V$ and $\bu_{n=1}^\infty E_n=E$.
\end{lemma}

\begin{proof} Indeed, let $V=\{v_n:n\ge 1\}$ be a sequence of all
vertices of $\Ga$. and let $E=\{e_n:n\ge 1\}$ be a sequence of edges
of $\Ga$. We will proceed inductively to construct the
sequences $\{V_n\}_{n=1}^\infty$ and $\{E_n\}_{n=1}^\infty$. In order to
construct $<E_1,V_1>$ let $\a$ be a path joining $v_1$ and $v_2$
($i(\a)=v_1$, $t(\a)=v_2$) and let $\b$ be a path joining $v_2$ and $v_1$
($i(\b)=v_2$, $t(\b)=v_1$). These paths exist since $\Ga$ is strongly
connected. We define $V_1\sbt V$ to be the set of all vertices of
paths $\a$ and $\b$ and $E_1\sbt E$ to be the set of all edges from
$\a$ and $\b$ enlarged by $e_1$ if this edge is
among all the edges joining the vertices of $V_1$. Obviously
$<E_1,V_1>$ is strongly
connected and the first step of inductive procedure is
complete. Suppose now that a strongly connected graph $<E_n, V_n>$ has
been
constructed. If $v_{n+1}\in V_n$, we set $V_{n+1}=V_n$ and
$E_{n+1}$ is then defined to be the union of $E_n$ and all the edges
from $\{e_1,e_2,\ld, e_n,e_{n+1}\}$ that are among all the edges
joining the vertices of $V_n$. If $v_{n+1}\notin V_n$, let $\a_n$ be a
path joining $v_n$ and $v_{n+1}$ and let $\b_n$ be a path joining $v_{n+1}$
and $v_n$. We define $V_{n+1}$ to be the union of $V_n$ and the set of all
vertices of $\a_n$ and $\b_n$. $E_{n+1}$ is then defined to be the
union of $E_n$, all the edges building the paths $\a_n$ and $\b_n$ and all
the edges from $\{e_1,e_2,\ld, e_n,e_{n+1}\}$ that are among all the edges
joining the vertices of $V_{n+1}$. Since $<E_n,
V_n>$ was strongly connected, so is $<E_{n+1}, V_{n+1}>$. The
inductive procedure is complete. It immediately follows from the
construction that $V_n\sbt V_{n+1}$, $E_n\sbt
E_{n+1}$. $\bu_{n=1}^\infty V_n=V$ and $\bu_{n=1}^\infty E_n=E$. We are
done. \end{proof}

Our first main result is the following.

\begin{theorem}\lab{t2.7.3}
Suppose that $A:E\times E\to\{0,1\}$ is an irreducible incidence matrix and that $f:E_A^\N\to\R$ is a H\"older continuous function such that
$$
\sum_{e\in E}\exp(\sup(f|_{[e]}))<+\infty.
$$
Then there exists a Borel probability measure
$\^m$ on $E_A^\N$ which is an eigenvector of the conjugate operator $\pf_f^*:C_b^*(E_A^\N)\to C_b^*(E_A^\N)$.
\end{theorem}

\begin{proof} Without loss of generality we may assume that $E=\N$. Since
the incidence matrix $A$ is irreducible, it follows from
Lemma~\ref{l2.7.2} that we can
reorder the set $\N$ such that there exists an increasing to infinity sequence
$\(l_n\)_{n\ge 1}$ such that for every $n\ge 1$ the matrix $A|_{\{1,\ld,l_n\}\times
\{1,\ld,l_n\}}$ is irreducible. Then, in view of Lemma~\ref{l2.7.1},
there exists an eigenmeasure $\^m_n$ of the operator $\pf_n^*$, conjugate to Perron-Frobenius operator
$$
\pf_n:C(E_{l_n}^\N)\to C(E_{l_n}^\N)
$$
associated to the potential  $f|_{E_{l_n}^\N}$, where, for any $q\ge 1$,
$$
E_q^\N
:=E_A^\N\cap \{1,\ld,q\}^\N
=\{(e_k)_{k\ge 1}:1\le e_k\le q \text{ and } A_{e_ke_{k+1}}=1 \text{ for all }
k\ge 1\}.
$$
Our first aim is to show that the
sequence $\{\^m_n\}_{n\ge 1}$ is tight, where all $\^m_n$, $n\ge 1$, are treated
here as Borel probability measures on $E_A^\N$. Let $\P_n=
\P(\sg|_{E_{l_n}^\infty}, f|_{E_{l_n}^\infty})$. Obviously $\P_n\ge \P_1$ for
all $n\ge 1$. For every $k\ge 1$ let $\pi_k:E_A^\N\to \N$ be the projection
onto $k$-th coordinate, i.e. $\pi(\{(\tau_u)_{u\ge 1}\})=\tau_k$. By
Theorem~\ref{t2.3.3},
$e^{\P_n}$ is the eigenvalue of $\pf_n^*$ corresponding to the eigenmeasure
$\^m_n$. Therefore, applying (\ref{2.3.3}), we obtain for every $n\ge
1$, every $k\ge 1$, and every $s\in\N$ that
$$
\aligned
\^m_n(\pi_k^{-1}(s))
&=\sum_{\om\in E_{l_n}^k:\om_k=s}\^m_n([\om])
 \le \sum_{\om\in E_{l_n}^k:\om_k=s}\exp\(\sup(S_kf|_{[\om]})-\P_nk\) \\
&\le e^{-\P_nk}\sum_{\om\in E_{l_n}^k:\om_k=e}\exp\(\sup(S_{k-1}f|_{[\om]})
 +\sup(f|_{[s]})\) \\
&\le e^{-\P_1k}\lt(\sum_{i\in \N}e^{\sup(f|_{[i]})}\rt)^{k-1}e^{\sup(f|_{[s]})}
\endaligned
$$
Therefore
$$
\^m_n(\pi_k^{-1}([s+1,\infty)))
\le e^{-\P_1k}\lt(\sum_{i\in \N}e^{\sup(f|_{[i]})}\rt)^{k-1}
\sum_{j>s}se^{\sup(f|_{[j]})}.
$$
Fix now $\e>0$ and for every $k\ge 1$ choose an integer $n_k\ge 1$ so large that
$$
e^{-\P_1k}\lt(\sum_{i\in \N}e^{\sup(f|_{[i]})}\rt)^{k-1}
\sum_{j>n_k}e^{\sup(f|_{[j]})} \le {\e\over 2^k}.
$$
Then, for every $n\ge 1$ and every $k\ge 1$, $\^m_n(\pi_k^{-1}([n_k+1,\infty)))
\le \e/2^k$. Hence
$$
\^m_n\lt(E_A^\N\cap\prod_{k\ge 1}[1,n_k]\rt)
\ge 1-\sum_{k\ge 1}\^m_n(\pi_k^{-1}([n_k+1,\infty)))
\ge 1-\sum_{k\ge 1}{\e\over 2^k}
=1-\e.
$$
Since $E_A^\N\cap\prod_{k\ge 1}[1,n_k]$ is a compact subset of $E_A^\N$,
the tightness of the sequence $\{\^m_n\}_{n\ge 1}$ is therefore proved. Thus,
in view of Prohorov's Theorem, see \cite[Book II, Theorem 8.6.2]{bog:meas} there exists $\^m$, a Borel probability measure on $E_A^\N$ which is a weak-limit point of the sequence $\{\^m_n\}_{n\ge 1}$. Passing to a subsequence, we may assume that the sequence $\{\^m_n\}_{n\ge 1}$ itself converges weakly to the measure $\^m$. Let
$$
\pf_{0,n}=e^{-\P_n}\pf_n \  \  \  {\rm and } \   \   \   \pf_0=e^{-\P(f)}\pf
$$
be the corresponding normalized operators. Fix $g\in
C_b(E_A^\N)$ and $\e>0$. Let us now consider an integer $n\ge 1$ so large
that the following requirements are satisfied.
\beq\label{2.7.1}
\sum_{i>n}||g||_0\exp\(\sup(f|_{[i]})-\P^\sg(f)\) \le {\e\over 6},
\eeq
\beq\lab{2.7.2}
\sum_{i\le n}||g||_\infty\exp\(\sup(f|_{[i]})\)\big|e^{-\P^\sg(f)}- e^{-\P_n}\big|
\le {\e\over 6},
\eeq
\beq\label{2.7.3}
|\^m_n(g)-\^m(g)|\le {\e\over 3},
\eeq
and
\beq\label{2.7.4}
\left|\int\pf_0(g)d\^m -\int\pf_0(g)d\^m_n\right|\le  {\e\over 3}.
\eeq
It is possible to make condition (\ref{2.7.2}) satisfied since, due to
Theorem~\ref{t2.1.3}, $\lim_{n\to\infty}\P_n=\P^\sg(f)$. Let
$g_n:=g|_{E_{l_n}^\infty}$.
The first two observations are the following.
\beq\label{2.7.5}
\aligned
\pf_{0,n}^*\^m_n(g)
&=\int_{E_A^\N}\sum_{i\le n:A_{i\om_n}=1}g(i\om)\exp(f(i\om)
 -\P_n)d\^m_n(\om)\\
&=\int_{E_{l_n}^\infty}\sum_{i\le n:A_{i\om_n}=1}g(i\om)\exp(f(i\om)
 -\P_n)d\^m_n(\om)\\
&=\int_{E_{l_n}^\infty}\sum_{i\le n:A_{i\om_n}=1}g_n(i\om)\exp(f(i\om)
 -\P_n)d\^m_n(\om)\\
&=\pf_{0,n}^*\^m_n(g_n)
 =\^m_n(g_n),
\endaligned
\eeq
and
\begin{equation}\label{2.7.6}
\^m_n(g_n)-\^m_n(g)
=\int_{E_{l_n}^\infty}(g_n-g)d\^m_n
=\int_{E_{l_n}^\infty}0d\^m_n
=0.
\end{equation}
Using the triangle inequality we get the following.
\begin{equation}\label{2.7.7}
\aligned
\big|\pf_0^*\^m(g)-\^m(g)\big|
&\le |\pf_0^*\^m(g)-\pf_0^*\^m_n(g)| + |\pf_0^*\^m_n(g)-\pf_{0,n}^*\^m_n(g)|+ \\
&+|\pf_{0,n}^*\^m_n(g)-\^m_n(g_n)|+|\^m_n(g_n)-\^m_n(g)|
   +|\^m_n(g)-\^m(g)|
\endaligned
\end{equation}
Let us look first at the second summand. Applying (\ref{2.7.2}) and
(\ref{2.7.1}) we get
\begin{equation}\label{2.7.8}
\aligned
|\pf_0^*\^m_n(g) &-\pf_{0,n}^*\^m_n(g)|= \\
&=\bigg|\int_{E_A^\N}\sum_{i\le n:A_{i\om_n}=1}g(i\om)\(\exp(f(i\om)
 -\P^\sg(f))-\exp(f(i\om)-\P_n)\)d\^m_n(\om) \\
&+ \int_{E_A^\N}\sum_{i> n:A_{i\om_n}=1}g(i\om)\exp\(f(i\om)-\P^\sg(f)\)d\^m_n(\om)\bigg| \\
&\le \sum_{i\le n}||g||_\infty\exp\(\sup(f|_{[i]})\)
     \big|e^{-\P^\sg(f)}- e^{-\P_n}\big|+ \\
& \  \  \   \   \   \  \  \   \   \   \  \  \   \   \   \  \  \   \   \  +\sum_{i>n}||g||_\infty \exp\(\sup(f|_{[i]}-\P^\sg(f)\) \\
&\le {\e\over 6}+{\e\over 6}
={\e\over 3}.
\endaligned
\end{equation}
Combining now in turn (\ref{2.7.4}), (\ref{2.7.8}), (\ref{2.7.5}), (\ref{2.7.6}) and (\ref{2.7.3}) we get
from (\ref{2.7.7}) that
$$
|\pf_0^*\^m(g)-\^m(g)|\le {\e\over 3}+{\e\over 3}+{\e\over 3}=\e.
$$
Letting $\e\downto 0$ we therefore get $\pf_0^*\^m(g)=\^m(g)$ or $\pf_f^*\^m(g)
=e^{\P^\sg(f)}\^m(g)$. Hence $\pf_f^*\^m=e^{\P(f)}\^m$ and the proof is
complete.
\end{proof}

As an immediate consequence of Theorem~\ref{t2.7.3}, Theorem~\ref{t2.3.6},
Theorem~\ref{t2.3.3}, Theorem~\ref{t2.2.4}, and Theorem~\ref{t2.2.7}, we get the following result summarizing what we did about the thermodynamic formalism.

\begin{corollary}\lab{c2.7.5}
Suppose that $f:E_A^\N\to\R$ is a H\"older continuous function such that
$$
\sum_{e\in E}\exp(\sup(f|_{[e]}))<+\infty
$$
and the incidence matrix $A$ is finitely irreducible. Then

\sp\begin{itemize}
\item[(a)] There exists a unique eigenmeasure $\^m_f$ of the conjugate
Perron--Frobenius operator \index{Perron--Frobenius operator} $\pf_f^*$ and the corresponding eigenvalue is equal
to $e^{\P^\sg(f)}$.

\sp\item[(b)] The eigenmeasure $\^m_f$ is a Gibbs state \index{Gibbs state} for $f$.

\sp\item[(c)] The function $f:E_A^\N\to\R$ has a unique $\sg$-invariant Gibbs
state $\^\mu_f$.

\sp\item[(d)] The measure $\tilde\mu_f$ is ergodic, equivalent to $\tilde m_f$ and $\log(d\tilde\mu_f/d\tilde m_f)$ is uniformly bounded.

\sp\item[(e)] If $\int-f\,d\^\mu_f<+\infty$, then the $\sg$-invariant Gibbs
state $\^\mu_f$ is the unique equilibrium state for the potential $f$. \index{equilibrium state}

\sp\item[(f)] The Gibbs state $\^\mu_f$ is ergodic, and in case the incidence matrix $A$ is finitely primitive, it is completely ergodic.
\end{itemize}
\end{corollary}

\chapter{Conformal graph directed Markov systems}\label{chap:CGDMS}

In this chapter we introduce conformal graph directed Markov systems (GDMS) \index{graph directed Markov system} \index{conformal graph directed Markov system} in Carnot groups. In Section \ref{sec:GDMS} we define graph directed Markov systems in metric spaces and we also discuss one important subclass of such systems, namely maximal systems, which will also occur in several subsequent chapters. In Section \ref{sec:iwagdms} we define conformal GDMS in general Carnot groups  and we employ results from Chapter \ref{chap:conformal-metric-and-geometric-properties} to obtain fundamental distortion properties of conformal GDMS in our setting.

\section{Graph Directed  Markov Systems}\label{sec:GDMS}

A {\it graph directed Markov system} (GDMS) \index{GDMS}
$$
\cS= \big\{ V,E,A, t,i, \{X_v\}_{v \in V}, \{\f_e\}_{e\in E} \big\}
$$
consists of
\begin{itemize}
\item a directed multigraph $(E,V)$ with a countable set of edges $E$, frequently refere to also as alphabet, and a finite set of vertices $V$,

\sp\item an incidence matrix $A:E\times E\to\{0,1\}$,

\sp\item two functions $i,t:E\to V$ such that $t(a)=i(b)$ whenever $A_{ab}=1$,

\sp\item a family of non-empty compact metric spaces $\{X_v\}_{v\in V}$,

\sp\item a number $s$, $0<s<1$, and

\sp\item  a family of injective contractions $$ \{\phi_e:X_{t(e)}\to X_{i(e)}\}_{e\in E}$$ such that every $\phi_e,\, e\in E,$ has Lipschitz constant no larger than $s$.
\end{itemize}

\sp\fr We will always assume that $E$ contains at least two elements, because otherwise the system is trivial. For the sake of brevity we will frequently use the notation $\cS=\{\f_e\}_{e \in E}$ for a GDMS. We will also assume that for every $v \in V$ there exist $e,e' \in E$ such that $t(e)=v$ and $i(e')=v$. A GDMS is called {\it finite}\index{GDMS!finite} if $E$ is a finite set. A GDMS $\cS$ is said to be {\it finitely irreducible} \index{GDMS!finitely irreducible}
\index{finitely irreducible GDMS} if its associated incidence matrix $A$ is finitely irreducible. Notice that if $\cS$ is a finite irreducible GDMS then it is finitely irreducible. In the particular case when $V$ is a singleton and for every $e_1,e_2 \in E$, $A_{e_1e_2}=1$ if and only if $t(e_1)=i(e_2)$, the GDMS is called an \textit{iterated function system} (IFS).\index{iterated function system}

For $\om \in E^*_A$ we consider the map coded by $\om$:
\begin{equation}\label{phi-om}
\phi_\om=\phi_{\om_1}\circ\cdots\circ\phi_{\om_n}:X_{t(\om_n)}\to
X_{i(\om_1)} \qquad \mbox{if $\om\in E^n_A$.}
\end{equation}
For the sake of convenience we will write $t(\om) = t(\om_n)$ and $i(\om)=i(\om_1)$ for
$\om$ as in \eqref{phi-om}.

For $\om \in E^{\mathbb N}_A$, the sets
$\{\phi_{\om|_n}\(X_{t(\om_n)}\)\}_{n=1}^\infty$ form a descending
sequence of non-empty compact sets and therefore have nonempty intersection.
Since
$$
\diam(\phi_{\om|_n}\(X_{t(\om_n)}\)\) \le s^n\diam\(X_{t(\om_n)}\)\le s^n\max\{\diam(X_v):v\in V\}
$$
for every $n\in\N$, we conclude that the intersection
$$
\bigcap_{n\in  \N}\phi_{\om|_n}\(X_{t(\om_n)}\)
$$
is a singleton and we denote its only element by $\pi(\om)$.
In this way we define the coding map \index{coding map}
$$
\pi:E^{\mathbb N}_A\to \du_{v\in V}X_v,
$$
the latter being a disjoint union of the sets $X_v$, $v\in V$.
The set
$$
J=J_\cS:=\pi(E^{\mathbb N}_A)
$$
will be called the {\it limit set} \index{limit set} (or {\it attractor}) \index{attractor} of the GDMS $\cS$.

We record the following standard fact concerning the coding map.

\begin{proposition}\label{regularity-of-pi}
The coding map $\pi:E^{\mathbb N}_A\to \du_{v\in V}X_v$ is H\"older continuous, when $E^{\mathbb N}_A$ is
equipped with any of the metrics $d_\alpha$ as in \eqref{d-alpha} and
$\du_{v\in V}X_v$ is equipped with the direct sum metric $\delta$.
\end{proposition}

\fr To see why Proposition \ref{regularity-of-pi} is true, observe that if
$\om,\tau \in E^{\mathbb N}_A$ with $\rho=\om\wedge\tau \in E^n$, then
$$
d(\pi(\om),\pi(\tau)) \le \diam
\phi_{\rho}(X_{t(\rho_n)}) \le s^n \max \{ \diam X_v : \, v \in V \}.
$$
This also shows that $\pi$ is Lipschitz continuous from $(E^{\mathbb
  N}_A,d_\alpha)$ to $\du_{v\in V}X_v$, when $\alpha = \log (1/s)$ because
$$d_\alpha(\om,\tau)=e^{-\alpha n}=e^{-\log(1/s)^n}=s^n.$$

\begin{definition}
\label{maximalmatrixdef}
Given a GDMS $\cS$
with an incidence matrix $A$, we define the matrix $\hat A:E\times
E\to\{0,1\}$ by
$$
\hat A_{ab}=
\begin{cases} 1 \  \text{ if }  \  t(a)=i(b) \\
                   0 \  \text{ if }  \  t(a)\ne i(b).
\end{cases}
$$
The GDMS $\hat\cS$ is then defined by means of the incidence matrix $\hat A$.
\end{definition}

Of course,
$$
E_A^n\sbt E_{\hat A}^n,\  \  E_A^*\sbt E_{\hat A}^*, \  \
          E_A^\N\sbt E_{\hat A}^\N,
$$
and
$$
J_\cS\sbt J_{\hat\cS}.
$$

\begin{definition}
\label{maximaldef}
A GDMS $\cS$
with an incidence matrix $A$ is called {\it maximal} \index{GDMS!maximal} \index{maximal graph directed Markov system} if $\hat\cS=\cS$. This equivalently means that $A=\hat A$ or further equivalently that $A_{ab}=1$ if and only if $t(a)=i(b)$.
\end{definition}

The following proposition asserts that for the study of finite GDMS it is actually enough to restrict our attention to maximal systems.

\begin{proposition}\label{finitemaxreduction}
If $\cS$ is a finite GDMS, then there exists  a maximal, finite GDMS $\hat{\cS}$ such that $J_{\cS} =J_{\hat{\cS}}$. Moreover if $\cS$ is irreducible then $\hat{\cS}$ is also irreducible.
\end{proposition}

\begin{proof} Let $\cS=\{ V,E,A, t,i, \{X_v\}_{v \in V}, \{\f_e\}_{e\in E} \}$ be a GDMS. Let $$\hat{\cS}=\{ \hat{V},\hat{E},\hat{A},\hat{t},\hat{i}, \{ \hat{X}_v\}_{v \in \hat{V}}, \{\hat{\f}_e\}_{e\in \hat{E}} \}$$ where
\begin{enumerate}
\item $\hat{V}=E$,
\item $\hat{E}=\{(a,b) \in E^2:A_{ab}=1\}$,
\item if $\mathtt{e} \in \hat{E}$, $\mathtt{e}=(\mathtt{e}^1,\mathtt{e}^2) \in E^2$,
$$\hat{t}(\mathtt{e}):=\mathtt{e}^2 \mbox{ and }\hat{i}(\mathtt{e}):=\mathtt{e}^1,$$
\item the matrix $\hat A:\hat{E}\times
\hat{E}\to\{0,1\}$ is defined by
$$
\hat {A}_{\mathtt{e},\mathtt{f}}=
\begin{cases} 1 \  \text{ if }  \  \hat{t}(\mathtt{e})=\hat{i}(\mathtt{f}) \\
                   0 \  \text{ if }  \ \hat{t}(\mathtt{e})\ne \hat{i}(\mathtt{f}),
\end{cases}
$$
\item $\hat{X}_{e}=\f_e(X_{t(e)})$ for $e \in \hat{V}=E$,
\item if $\mathtt{e}=(\mathtt{e}^1,\mathtt{e}^2) \in \hat{E}$, then
$$\hat{\f}_{\mathtt{e}}=\f_{\mathtt{e}^1}: \f_{\mathtt{e}^2}(X_{t(\mathtt{e}^2)}) \ra \f_{\mathtt{e}^1}(X_{t(\mathtt{e}^1)}).$$
\end{enumerate}
Observe that if $\mathtt{e}\mathtt{f} \in \hat{E}_{A}^\ast$ then $\hat{t}(\mathtt{e})=\hat{i}(\mathtt{f})$, that is $\mathtt{e}^2=\mathtt{f}^1$.

The system $\hat{\cS}$ is maximal by definition. We will now show that $J_\cS=J_{\hat{\cS}}$. Trivially $J_{\hat{\cS}} \subset J_\cS$. For any $ \om \in E_A^\N$ define $\hat{\om} \in \hat{E}_{\hat{A}}^\N$  by
$\hat{\om}=(\hat{\om}_n)_{n\in \N}$ where $\hat{\om}_n=(\om_n, \om_{n+1}) \in E_A^2$. Notice that
$$\bigcap_{n\in  \N}\phi_{\om|_n}\(X_{t(\om_n)}\)=\bigcap_{n\in  \N}\hat{\phi}_{\hat{\om}|_n}\(\hat{X}_{\hat{t}(\hat{\om}_n)}\),$$
hence $J_\cS \subset J_{\hat{\cS}}$.

Now suppose that $\cS$ is irreducible. Recall that since $\cS$ is finite, it is finitely irreducible. Let $\Phi\sbt E_A^*$ be a finite set witnessing finite for $\cS$. Let $\mathtt{e}, \mathtt{f} \in \hat{E}$. Then there exists some $\tau=(\tau_1,\dots, \tau_{|\tau|}) \in \Phi$ such that $\mathtt{e}^2 \tau \mathtt{f}^1 \in E_A^\ast$. Hence
$$\mathtt{e} \tau^0 \tau^1 \dots \tau^{|\tau|}\mathtt{f} \in \hat{E}^\ast_{\hat{A}},$$
where $\tau^0=(\mathtt{e}^2,\tau_1)$, $\tau^m=(\tau_m, \tau_{m+1})$ for $m=1,\dots,|\tau|-1$ and $\tau^{|\tau|}=(\tau_{|\tau|}, \mathtt{f}^1)$. Therefore $\hat{\cS}$ is finitely irreducible and the proof is complete.
\end{proof}

We end this section with the following obvious observation.

\begin{remark}
A GDMS is an IFS if and only if it is maximal and the set of vertices is a singleton.
\end{remark}

\section{Carnot conformal graph directed Markov systems}\label{sec:iwagdms}

We now introduce the primary objects of study in this monograph.

\begin{definition}\label{Carnot-conformal-GDMS}
A graph directed Markov system is called {\it Carnot conformal} \index{GDMS!Carnot conformal}
\index{Carnot conformal GDMS} if the following conditions are satisfied.

\sp\begin{itemize}

\sp\item[(i)] For every vertex $v\in V$, $X_v$ is a compact connected
subset of a fixed Carnot group $(\G,d)$ and $X_v=\ov{\Int(X_v)}$.

\item[(ii)] ({\it Open set condition} or {\it OSC}). \index{open set condition} For all $a,b\in
E$, $a\ne b$,
$$
\phi_a(\Int(X_{t(a)})) \cap \phi_b(\Int(X_{t(b)}))= \emptyset.$$
\item[(iii)] For every vertex $v\in V$ there exists an open connected
set $W_v\spt X_v$ such that for every $e\in E$ with $t(e)=v$, the map
$\f_e$ extends to a conformal diffeomorphism of $W_v$ into $W_{i(e)}$.
\end{itemize}
A graph directed Markov system is called {\it weakly Carnot conformal} \index{GDMS! weakly Carnot conformal}
if only conditions (i) and (iii) from Definition \ref{Carnot-conformal-GDMS} are required to be satisfied; (ii) may hold or not.
\end{definition}

\begin{remark}\label{abusingGDMS}
As previously indicated in Remark \ref{Carnot-homogeneous-metrics-remark}, Definition \ref{Carnot-conformal-GDMS} applies to conformal mappings of Iwasawa groups with the gauge metric, or to affine similarities of Carnot groups with the Carnot--Carath\'eodory metric. Nevertheless, the subsequent theory also applies to graph directed Markov systems comprised of contractive metric similarities of Carnot groups equipped with any homogeneous metrics. It would be natural to term such systems {\it Carnot similarity GDMS}. Abusing terminology, we choose to use the term {\it Carnot conformal GDMS} to refer to all of these cases. Hence our theory applies when
\begin{itemize}
\sp\item $(\G,d)$ is an Iwasawa group, $d=d_H$ and the maps $\f_e$ are conformal.
\sp\item $(\G,d)$ is a Carnot group not of Iwasawa type, $d=\dcc$ and the maps $\f_e$ are conformal.
\sp\item $(\G,d)$ is a Carnot group, $d$ is {\bf any} homomgeneous metric and the maps $\f_e$ are metric similarities.
\end{itemize}
\end{remark}

%\sp\fr We stress once more that if $\cS=\{\f_e\}_{e \in E}$ consists of affine similarities, $\G$ can be any Carnot %group and $d$ can be any homogeneous metric on $\G$. If some of the conformal mappings $\f_e$ are not affine %similarities, then by Theorem \ref{th:cowling-ottazzi} $\G$ is an Iwasawa group and $d=d_H$.

For each $v \in V$, we select a compact set $S_v$ such that
$X_v \subset \Int(S_v) \subset S_v \subset W_v$. Moreover the sets $S_v, v \in V,$ are chosen to be pairwise disjoint. The assumption that the compact sets $(X_v)_{v \in V}$ are pairwise disjoint is not essential. One could modify the definition of a GDMS in order to avoid this, see Remark \ref{formalGDMS} for more details. We also set
\begin{equation}\label{X}
X := \bigcup_{v\in V} X_v \text{ and }S:= \bigcup_{v\in V} S_v.
\end{equation}

Since
$\max\{\diam(X_v):v\in V\}$ is finite and
$\min\{\dist(X_v,\G\setminus\Int(S_v)):v\in V\}$ is positive, the
following is an immediate consequence of
Lemma~\ref{harnack}.

\begin{lemma}[Bounded Distortion Property]\label{l12013_03_11} \index{bounded distortion property}
Let $S=\{\phi_e\}_{e\in E}$ be a weakly Carnot conformal GDMS on $(\G,d)$. There exists a constant $K$ so that \index{GDMS!bounded distortion property of}
$$
\biggl|\frac{||D\f_\om(p)||}{||D\f_\om(q)||}-1\biggr|\le Kd(p,q)
$$
and
$$
K^{-1}\le\frac{||D\f_\om(p)||}{||D\f_\om(q)||}\le K
$$
for every $\om\in E_A^*$ and every pair of points $p,q\in S_{t(\om)}$.
\end{lemma}

Recalling \eqref{dfnorm}, for $\om \in E^*_A$ we denote
$$||D \f_\om||_\infty := ||D \f_\om||_{S_{t(\om)}}.$$

From Lemma \ref{l12013_03_11} and \eqref{leibniz} we easily see that
\begin{equation}\label{quasi-multiplicativity}
K^{-1} ||D\f_{\om|_{(n-1)}}||_\infty \, ||D\f_{\om_n}||_\infty \le
||D\f_\om||_\infty \le ||D\f_{\om|_{(n-1)}}||_\infty \,
||D\f_{\om_n}||_\infty
\end{equation}
whenever $\om \in E^n_A$. More generally if $\om \in E_A^\ast$ and $\om=\tau \upsilon$ for some $\tau, \upsilon \in E_A^\ast$,
\begin{equation}\label{quasi-multiplicativity1}
K^{-1} ||D\f_{\tau}||_\infty \, ||D\f_{\upsilon}||_\infty \le
||D\f_\om||_\infty \le ||D\f_{\tau}||_\infty \,
||D\f_{\upsilon}||_\infty.
\end{equation}
We record the following consequence of Corollary \ref{koebe2} and
\eqref{harnack-equation-2}.

\begin{corollary}\label{l42013_03_12}
Let $S=\{\phi_e\}_{e\in E}$ be a weakly Carnot conformal GDMS on a Carnot group $(\G,d)$. For all finite words $\om\in E_A^*$, all $p \in X_{t(\om)}$ and all
$0<r<\ \dist(X_{t(\om)}, \bd S_{t(\om)})/3\,L$,
\begin{equation}\lab{4.1.8}
B(\f_\om(p),(K\,C)^{-1}\|D\f_\om\|_\infty r) \subset
\f_\om(B(p,r)) \subset
B(\phi_\om(p),C\|D\phi_\om\|_\infty r).
\end{equation}
\end{corollary}

We shall now prove the following Lipschitz estimate.

\begin{lemma}\label{l22013_03_12}
Let $S=\{\phi_e\}_{e\in E}$ be a weakly Carnot conformal GDMS on $(\G,d)$. There exists a constant $\La\ge 1$ such that
\begin{equation}\lab{4.1.9a}
d\(\f_\om(p),\f_\om(q)\)\le \La\,||D\phi_\om||_\infty d(p,q)
\end{equation}
for all finite words $\om\in E_A^*$ and all $p,q\in X_{t(\om)}$.
In particular,
\begin{equation}\lab{4.1.9}
\diam\(\f_\om(X_{t(\om)})\)\le \La M\|D\f_\om\|_\infty,
\end{equation}
where $M:=\diam X$.
\end{lemma}

\begin{remark}
Note that if the set $X_{t(\om)}$ were geodesically convex, the
previous result would follow immediately from Lemma
\ref{upper-gradient}. However, there are very few nontrivial
geodesically convex sets in nonabelian Carnot groups. Convexity in
sub-Riemannian Carnot groups remains a topic of intense focus. We
refer the interested reader to the foundational papers by
Danielli--Garofalo--Nhieu \cite{dgn:convexity} and
Lu--Manfredi--Stroffolini \cite{lms:convex}, which introduced the
nowadays established notion of {\it horizontal convexity
  (H-convexity)} in Carnot groups.
%, as well as to recent papers on the
%regularity of H-convex functions and sets \cite{br:convex},
%\cite{cpt:convexity}, \cite{acm:convex}.
\end{remark}

\begin{proof}[Proof of Lemma \ref{l22013_03_12}]
Fix
\begin{equation}
\label{rdef}
\eta_\cS:=\min\{\dist(X_v,\bd S_v):v\in V\}/3L>0
\end{equation}
where, as before, $L$ denotes a quasiconvexity constant for
$(\G,d)$. Fix also some $\om \in E_A^\ast$. If $d(p,q)<\eta_\cS$ then there exists some $\ve \in (0,1)$ such that $(1+\ve)d(p,q)<\eta_\cS$ and Corollary \ref{l42013_03_12} implies that
$$\f_\om(B(p,(1+\ve)d(p,q))) \subset B(\f_\om(p), C \|D \f_\om\|_\infty (1+\ve) d(p,q)).$$
Thus
\begin{equation}
\label{410first}
d(\f_\om(p), \f_\om(q))\leq 2 C \|D \f_\om\|_\infty  d(p,q)
\end{equation}
and \eqref{4.1.9a} follows in this case. Hence we can assume $d(p,q)\ge \eta_\cS$. Since each $X_v$ is compact and connected and the vertex set
$V$ is finite, there exists an integer $N\ge 1$ so that for each $v
\in V$, the space $X_v$ can be covered by finitely many balls $\cB_v
:= \{B(p_{v,1},\eta_\cS/2),\ldots,B(p_{v,N},\eta_\cS/2)\}$ with centers
$p_{v,1},\ldots,p_{v,N}$ in $X_v$ and with the property that any two
points of $X_v$ lie in a connected union of balls chosen from $\cB_v$.

Therefore, for every vertex $v\in V$ and all points $p,q\in X_v$ there
are $k\le N$ points $p=z_0,z_1,\ld,z_k=q$ in $S_v$ such that for all $i=0,\dots,k-1$ the consecutive points $z_i,z_{i+1}$ belong to some ball $B(p_{v,n_i}, \eta_\cS/2) \in \cB_v$. Now by \eqref{ball-inclusions} $z_i,z_{i+1} \in B_{cc}(p_{v,n_i}, L \eta_\cS/2)$. Hence if $\gamma_{z_i,z_{i+1}}$ is the geodesic horizontal curve joining the points $z_i$ and $z_{i+1}$, we deduce, for example by the \index{segment property} \textit{segment property} \cite[Corollary 5.15.6]{BLU}, that $\gamma_{z_i,z_{i+1}} \in B_{cc} (p_{v,n_i}, 3 L \eta_\cS/2)$. Then again by \eqref{ball-inclusions} and the choice of $\eta_\cS$ we deduce that $\gamma_{z_i,z_{i+1}} \subset \Int (S_v)$. Thus for $v=t(\om)$ an application of Lemma~\ref{upper-gradient} gives that for all $i=0,\dots,k-1$
\begin{equation}
\label{dccbound}
d_{cc}(\phi_\om(z_i),\phi_\om(z_{i+1})) \leq \|D \f\|_\infty d_{cc}(z_i,z_{i+1}).
\end{equation}
Moreover note that $d(z_i,z_{i+1})\le \eta_\cS \le d(p,q)$ for all $i=1,2,\ld, k$. Using \eqref{dccbound} and \eqref{quasiconvexity} we get
\begin{equation*}\begin{split}
d(\phi_\om(p),\phi_\om(q))
&\le \sum_{i=0}^{k-1}d(\phi_\om(z_i),\phi_\om(z_{i+1})) \\
&\le \sum_{i=0}^{k-1} L ||D\phi_\om||_\infty d(z_i,z_{i+1}) \\
&\le L k ||D\phi_\om||_\infty d(p,q) \le  L N ||D\phi_\om||_\infty d(p,q).
\end{split}\end{equation*}
Recalling also \eqref{410first} the proof is complete upon setting $\La:=\max\{2C, LN\}$.
\end{proof}

Let $R_\mathcal{S}>0$ be the radius of the largest open ball that can be
inscribed in any of the sets $X_v$, $v \in V$. Let $p_v\in \Int(X_v)$ be the
centers of balls of this radius inscribed in the sets $X_v$. As an immediate
consequence of Corollary~\ref{l42013_03_12} and \eqref{ball-diameter}
we get the following conclusions.

\begin{lemma}\label{l52013_03_12}
Let $S=\{\phi_e\}_{e\in E}$ be a weakly Carnot conformal GDMS on $(\G,d)$. Let $\tilde{R}_\cS= \min \{R_\cS, \eta_\cS \}$. For all finite words $\om\in E_A^*$ we have
\begin{equation}\lab{4.1.10a}
\f_\om(\Int\(X_{t(\om)}\)) \spt
B\(\phi_\om(p_{t(\om)}),(K\,C)^{-1}\|D\f_\om\|_\infty \tilde{R}_\cS\),
\end{equation}
and hence
\begin{equation}\lab{4.1.10}
\diam(\f_\om(X_{t(\om)}))\ge 2(K\,C)^{-1} \|D\f_\om\|_\infty \tilde{R}_\cS.
\end{equation}
\end{lemma}

\begin{lemma}
\label{newbilip} Let $S=\{\phi_e\}_{e\in E}$ be a weakly Carnot conformal GDMS. Then for every $\om \in E^\ast_A$ and every pair of points $p,q \in X_{t(\om)}$,
\begin{equation*}
d(\f_\om (p),\f_\om(q)) \geq (L^2 K)^{-1} \kappa_0 \|D \f_\om\|_\infty d(p,q),
\end{equation*}
where $$\kappa_0=\min \left\{ \left\{\frac{\dist(X_v, \partial S_v)}{\diam (X_v)}\right\}_{v\in V}, 1 \right\}.$$
\end{lemma}

\begin{proof}
Let $\om \in E_A^\ast$ and $p,q \in X_{t(\om)}$. We will consider two cases. We first assume that the the arc of the CC-geodesic curve joining $\f_\om (p)$ and $\f_\om(q)$ is contained in $\f_\om (S_{t(\om)})$. Then by Lemma \ref{upper-gradient} and \eqref{leibniz} there exists some $ \xi \in S_{t(\om)}$ such that,
\begin{equation*}
\begin{split}
d_{cc}(p,q)&=d_{cc}(\f_\om^{-1}(\f_\om(p)),\f_\om^{-1}(\f_\om(q)))\\
& \leq \|D \f_\omega^{-1}(\f_\om(\xi))\| \, d_{cc}(\f_\om (p), \f_\om(q)) \\
&=\|D \f_\om( \xi)\|^{-1} \dcc (\f_\om (p), \f_\om(q)).
\end{split}
\end{equation*}
Hence by Lemma \ref{l12013_03_11},
\begin{equation}
\begin{split}
L d(\f_\om(p), \f_\om(q)) &\geq \dcc (\f_\om(p), \f_\om(q))\\
& \geq \|D \f_\om (\xi)\|\, \dcc (\f_\om(p), \f_\om(q))\\
& \geq K^{-1} \|D \f_\om\|_\infty d(p,q).
\end{split}
\end{equation}
Therefore if the arc of the CC-geodesic connecting $\f_\om(p)$ and $\f_\om (q)$ lies inside $\f_\om (S_{t(\om)})$,
\begin{equation}
\label{bilip1}
d(\f_\om(p), \f_\om(q)) \geq (KL)^{-1} \|D \f_\om\|_\infty d (p,q).
\end{equation}

If the arc of the CC-geodesic $\gamma:[0,T] \ra \G$ connecting $\f_\om(p)$ and $\f_\om (q)$ is not contained in $\f_\om(S_{t(\om)})$, let
$$t_0 :=\min\{t \in (0,T):  \gamma(t) \in \partial \f_\om (S_{t(\om)})\}.$$
Hence if $z=\gamma(t_0) \in \partial \f_\om (S_{t(\om)})$ there exists some $\zeta \in \partial S_{t(\om)}$ such that $z=\f_\om(\zeta)$. Using \eqref{bilip1} and the \index{geodesic} \index{segment property} segment property \cite[Corollary 5.15.6]{BLU} of CC-geodesics, we have
\begin{equation*}
\begin{split}
L d(\f_\om(p), \f_\om(q)) &\geq \dcc (\f_\om(p), \f_\om(q)) \geq \dcc(\f_\om(p), \f_\om(\zeta))\\
& \geq (LK)^{-1} \|D \f_\om\|_\infty \dcc(p,\zeta)\\
& \geq (LK)^{-1} \|D \f_\om\|_\infty d(p,\zeta) \\
&\geq (LK)^{-1} \|D \f_\om\|_\infty \dist(X_{t(\om)}, \partial S_{t(\om)})\\
& \geq (LK)^{-1} \|D \f_\om\|_\infty d(p,q) \frac{\dist(X_{t(\om)}, \partial S_{t(\om)})}{\diam(X_{t(\om)})}. \\
\end{split}
\end{equation*}
Thus,
$$d(\f_\om(p), \f_\om(q)) \geq (L^2K)^{-1} \kappa_0 \|D \f_\om\|_\infty d(p,q),$$
and the proof follows.
\end{proof}

\begin{proposition}
\label{newbilip21} Let $S=\{\phi_e\}_{e\in E}$ be a weakly Carnot conformal GDMS such that $\sharp (J_\cS \cap X_v) >1$ for all $v \in V$. Then for every $\om \in E^\ast_A$,
\begin{equation}
\label{bilip2eq}
\diam( \f_\om( J_\cS \cap X_{t(\om)})) \geq (2 L^2 K)^{-1} \kappa_0 \mu_0 \|D \f_\om\|_\infty,
\end{equation}
where $\kappa_0$ is as in Lemma \ref{newbilip} and $\mu_0=\min \{\diam(J_\cS \cap X_v)\}$.
\end{proposition}
\begin{proof} Notice that for every $v \in V$ there exist points $p_v, q_v \in J_\cS \cap X_v$ such that $d(p_v,q_v)\geq \mu_0/2$. Hence if $\om \in E_A^\ast$, Lemma \ref{newbilip} implies that
\begin{equation*}
\begin{split}
\diam(\f_\om(J_\cS \cap X_{t(\om)})) &\geq d(\f_\om(p_{t(\om)}),\f_\om(q_{t(\om)})) \\
& \geq (L^2 K)^{-1} \kappa_0 \|D \f_\om\|_\infty d(p_{t(\om)},q_{t(\om)})\\
& \geq (2 L^2 K)^{-1} \kappa_0 \mu_0 \|D \f_\om\|_\infty,
\end{split}
\end{equation*}
and the proof is complete.
\end{proof}

Lemmas~\ref{l12013_03_11} and~\ref{l22013_03_12} imply that the
function $p\mapsto \log \| D\f_\omega(p) \|$ is locally Lipschitz. This fact is proved in the following lemma.

\begin{lemma}\lab{l2.033101}
If $S=\{\phi_e\}_{e\in E}$ is a weakly Carnot conformal GDMS, then
$$
\bigl|\log\|D\phi_\om(p)\|-\log\|D\phi_\om(q)\|\bigr|
\le \frac{\La K}{1-s} d(p,q)
$$
for all $\om\in E_A^*$ and all $p,q\in X_{t(\om)}$. Here $K$ denotes the
constant from Lemma \ref{l12013_03_11} while $\La$ denotes the
constant from Lemma \ref{l22013_03_12}.
\end{lemma}

\begin{proof}
For every $\om\in E_A^*$, say $\om\in E_A^n$, and every $z \in X_{t(\om)}$ put
$$
z_k=\f_{\om_{n-k+1}}\circ\f_{\om_{n-k+2}}\circ\cdots
\circ\f_{\om_n}(z)
$$
Put also $z_0=z$. In view of Lemma~\ref{l12013_03_11} and
Lemma~\ref{l22013_03_12}, for any points $p,q\in X_{t(\om)}$, we get
\begin{equation*}\begin{split}
\bigl|\log\|D\f_\om(p)\|-\log\|D\f_\om(q)\|\bigr|
&= \left|\sum_{j=1}^n\(\log\|D\f_{\om_j}(p_{n-j})\|-\log\|D\f_{\om_j}(q_{n-j})\|\)\right|\\
&\le\sum_{j=1}^n\big|\log\|D\f_{\om_j}(p_{n-j})\|-\log\|D\f_{\om_j}(q_{n-j})\|\big|\\
&\le\sum_{j=1}^n\frac{\big|\|D\f_{\om_j}(p_{n-j})\|-\|D\f_{\om_j}(q_{n-j})\|\big|}
         {\min\{\|D\f_{\om_j}(p_{n-j})\|,\|D\f_{\om_j}(q_{n-j})\|\}} \\
&\le\sum_{j=1}^n K\frac{\big|\|D\f_{\om_j}(p_{n-j})\|-\|D\f_{\om_j}(q_{n-j})\|\big|}
         {||D\f_{\om_j}||_\infty} \\
&\le \sum_{j=1}^n \Lambda K d(p_{n-j},q_{n-j})\\
&\le \La K \sum_{j=1}^n s^{n-j} d(p,q) \le \frac{\La K}{1-s}d(p,q).
\end{split}\end{equation*}
The proof is complete.
\end{proof}

In several instances we are going to need slightly stronger versions of Lemmas \ref{l22013_03_12} and \ref{l2.033101}. We gather them in the following remark.

\begin{remark}\label{nsets}
Let $S=\{\phi_e\}_{e\in E}$ be a weakly Carnot conformal GDMS on $(\G,d)$. Set
$$
N_v:=B(X_v, \dist(X_v, \partial S_v)/2), v  \in V.
$$
Arguing exactly as in the proof of Lemma \ref{l22013_03_12}, one can show that there exists some $\La_0$ such that for all $\om \in E_A^\ast$ and $p,q \in N_{t(\om)}$,
\begin{equation}\label{l45pr}
d(\f_\om(p), \f_\om(q)) \leq \La_0 \|D \f_\om\|_\infty d(p,q).
\end{equation}
Without loss of generality we can assume that $\La_0 \geq \La$. Using \eqref{l45pr} as in the proof of Lemma \ref{l2.033101} we also obtain that for all $\om\in E_A^*$ and all $p,q\in N_{t(\om)}$
\begin{equation}\label{nsetsmod}
\bigl|\log\|D\phi_\om(p)\|-\log\|D\phi_\om(q)\|\bigr|
\le \frac{\La_0 K}{1-s} d(p,q).
\end{equation}
\end{remark}

\begin{lemma}\label{limdiam}
If $S=\{\phi_e\}_{e\in E}$ is a Carnot conformal GDMS, then
$$\sum_{e \in E} \|D \phi_e\|^Q_\infty < \infty \text{ and }\lim_{e \in E} \diam \phi_e(X_{t(e)})=0.$$
\end{lemma}

\begin{proof} Let $m_0=\min\{|\Int(X_v)| :v \in V\}>0$. By the open set condition, Theorem \ref{analytic-Carnot-conformal} and Lemma \ref{l12013_03_11}
\begin{equation*}
\begin{split}
|\Int(X)| &\geq \sum_{e \in E} |\f_e(\Int(X_{t(e)}))| =  \sum_{e \in E}\int_{\Int(X_{t(e)})} \|D \f_e(p)\|^Q dp \\
&\geq \sum_{e \in E} K^{-Q} \,|\Int(X_{t(e)})|\,  \|D \f_e\|_\infty^Q \geq m_0 K^{-Q} \sum_{e \in E}\|D \f_e\|_\infty^Q.
\end{split}
\end{equation*}
Therefore
$$\sum_{e \in E}\|D \f_e\|_\infty^Q \leq K^Q m_0^{-1}\,|\Int(X)|< \infty.$$
Now by \eqref{4.1.9},
$$\sum_{e \in E}\diam(\f_e(X_{t(e)}))^Q\leq (\La M)^Q\sum_{e \in E}\|D \f_e\|_\infty^Q<\infty,$$
and in particular $\lim_{e \in E} \diam \phi_e(X_{t(e)})=0.$
\end{proof}

In this section, as well as in the some subsequent chapters, we are assuming that the sets $X_v, v \in V$, (and as a result the sets $S_v$  as well) are disjoint. Although this assumption simplifies the proofs of some of our results, it is not essential. We will now describe how a GDMS $\cS$ can be lifted to a new GDMS $\tilde{\cS}$ such that $\cS$ and $\cS'$ have essentially the same limit sets but the compact sets $\tilde{X}_v$, corresponding to  $\tilde{\cS}$, are disjoint. For the sake of clarity and in order not to overly complicate the exposition of the material in this monograph, we chose to use GDMS with disjoint corresponding compact sets $X_v$ instead of the more formal route presented in the following remark.

\begin{remark}\label{formalGDMS}
Let $\cS=\{ V,E,A, t,i, \{X_v\}_{v \in V}, \{\f_e\}_{e\in E} \}$ be a GDMS such that the sets $X_v$ are compact subsets of a metric space $(M,d)$. Then $\tilde{M}=M \times V$ is a metric space endowed with the  product metric $\tilde{d}=d+d_0$ where $d_0$ denotes the discrete metric on $V$. Let $$\tilde{X}_v= X_v \times v, v \in V,$$ and notice that the sets $\tilde{X}_v$ are compact subsets of $\tilde{M}$. For every $e \in E$ we define maps $\tilde{\f}_e :\tilde{X}_{t(e)} \ra \tilde{X}_{i(e)}$ by
$$\tilde{\f}_e (x, t(e))=(\f_e(x), i(e)).$$
Notice that the maps $\tilde{\f}_e$ are contractions with respect to the  metric $\tilde{d}$ and they have the same contraction ratios as the maps $\f_e$.
We also define a projection $\tilde{\pi} : E_A^N \ra \tilde{M}$ by
$$\tilde{\pi}(\om)=(\pi(\om), i(\om)),\quad \om \in E_A^\N.$$
Then $\tilde{\pi}(E_A^\N)= J_\cS \times V$. We will call $$\tilde{\cS}=\{ V,E,A, t,i, \{\tilde{X}_v\}_{v \in V}, \{\tilde{\f}_e\}_{e\in E} \}$$ the \textit{formal lift} of $\cS$.
\end{remark}

\chapter{Examples of GDMS in Carnot groups}\label{chap:examples}

This chapter contains a variety of examples of conformal GDMS in Iwasawa
groups and similarity GDMS in general Carnot groups. First, we
consider self-similar iterated function systems \index{iterated function system} \index{open set condition} satisfying the open
set condition. The main novelty here is that we include also the case
of self-similar IFS with infinite generating set. Finite self-similar
IFS in general Carnot groups have previously been studied in
\cite{bt:horizdim} and \cite{btw:dimcomp}. Next, we give a genuinely
conformal (i.e., not self-similar) example of a Carnot conformal GDMS whose invariant set is a Cantor set. Following that, we define a class of conformal IFS in Iwasawa groups which are of continued fraction type. Continued fractions \index{continued fractions} in the first Heisenberg group have been studied by Lukyanenko and Vandehey. The examples which we give here correspond to subsystems associated to continued fractions with restricted digits. Finally, we show how conformal GDMS arise in relation to complex hyperbolic Schottky groups. The \index{complex hyperbolic Schottky group} diversity of these examples justifies our aim of providing a unified framework for the study of conformal dynamical systems in Iwasawa and other Carnot groups.

\section{Infinite self-similar iterated function systems}\label{sec:iifs1}

Let $\G$ be a Carnot group (not necessarily of Iwasawa type) equipped with any homogeneous metric. Let $E$
be a countable indexing set (either finite or countably infinite) and,
for each $e \in E$, let $\phi_e:\G \to \G$ be a contractive metric similarity
with contraction ratio $r_e < 1$. For example, $\phi_e$ could be a
contractive homothety (composition of a left translation and a
contractive dilation). We always assume that
$$
\sup\{r_e:e \in E\} < 1;
$$
needless to say, this assumption is automatically satisfied if $E$ is finite. The collection $\{\phi_e:e \in E\}$ is a self-similar
iterated function system in $\G$. Assuming the open set condition, it
follows (see, e.g., Lemma \ref{limdiam}) that
\begin{equation}\label{eq:ifs-eq1}
\sum_{e \in E} r_e^Q < \infty.
\end{equation}

\begin{remark}
Self-similar IFS with finite index set $E$ were previously studied
in \cite{bt:horizdim} and \cite{btw:dimcomp}, where formulas for the
Hausdorff dimension of invariant sets were established. In Chapter
\ref{chap:examples-2} we will extend such results to the case of
countably infinite index set, as an application of the general
dimension theory developed in the following chapters.
\end{remark}

\section{Iwasawa conformal Cantor sets}\label{sec:iwacantor}

We now present a very general method for constructing non-trivial, i.e. not consisting of similarities only,
conformal iterated function systems in any Iwasawa group $\G$. Let $G
\subset \G$ be an open set such that $o \notin \overline{G}$ and
$\overline{G}$ is compact. Let $P:=(p_n)_{n=1}^\infty \subset G$ be a
discrete sequence and let $d_n=\inf_{m \neq n} d(p_n,p_m)$. We will
assume that $\lim_{n \ra \infty} d_n=0$ and that $\dist(P, \partial
G)> \sup_{n \in \N} d_n$.

We now construct the conformal iterated function system.  Note that
$$
o \in \ell_{\cJ(p_n)^{-1}} \circ \cJ (G)
$$
for every $n \in \N$. Let $s \in (0,1)$ and choose real numbers
$(r_n)_{n=1}^\infty$ such that $r_n<s$ for all $n \in \N$ and
$$
r_n \, \diam \cJ(G) = \diam(\delta_{r_n} \circ \ell_{\cJ(p_n)^{-1}} \circ \cJ (G))<
d_n/2.
$$
We consider the iterated function system
$$\cS=\{\f_n: \overline{G} \rightarrow \overline{G}\}_{n \in \N}$$
where
$$
\f_n=\ell_{p_n} \circ \delta_{r_n} \circ \ell_{\cJ(p_n)^{-1}} \circ \cJ, \qquad n \in \N.
$$
The functions $\f_n$ are non-affine conformal maps. Moreover all
$\f_n$'s are injective contractions with contraction ratios uniformly
bounded by $s <1$ . It follows easily that $\cS$ satisfies the open
set condition since $\f_n(G) \subset G$ for all $n \in \N$ and
$\f_n(G) \cap \f_l(G)=\emptyset$ for all $n,l \in \N,\, n \neq l$.
By Lemma \ref{limdiam},
\begin{equation}\label{eq:conf-cantor-eq1}
\sum_{n=1}^\infty ||D\f_n||^Q < \infty;
\end{equation}
equation \eqref{eq:conf-cantor-eq1} can also be easily derived directly:
$$
\sum_{n=1}^\infty ||D\f_n||^Q \lesssim \sum_{n=1}^\infty r_n^Q \lesssim \sum_{n=1}^\infty \left( \frac{d_n}2 \right)^Q \lesssim |G| < \infty.
$$

\section{Continued fractions in Iwasawa groups}\label{iwacf:sec}

We start this section by introducing a version of integer lattices in Carnot groups of step 2.

\subsection{Integer lattices in Carnot groups of step two }

Let $\G\cong \R^{m_1} \times \R^{m_2}$ be a Carnot group of step two,
equipped with a homogeneous metric $d$. It follows by \cite[Section
3.2]{BLU}---see also the discussion \index{Baker-Campbell-Hausdorff formula} associated to the Baker--Campbell--Hausdorff formula \eqref{BCHformula}, especially \eqref{step2grouplaw}---that the group law $\ast$ in $\G$ has the following form; if $p,q \in \G$ such that $p=(z,t), z \in \R^{m_1}, t \in \R^{m_2},$ and
$q=(w,s), w \in \R^{m_1}, s \in \R^{m_2},$ then
\begin{equation}
\label{cg2grouplaw}
p \ast q=(z+w,t+s+(B^i z\cdot w)_{i=1}^{m_2})
\end{equation}
where the {\it structure matrices} $B^i$'s are skew-symmetric
$m_1 \times m_1$ matrices with real coefficients and $\cdot$ denotes
the usual inner product in $\R^{m_1}$. We remark that if $p=(z,t),
q=(w,s) \in \G$ such that $z=0$ or $w=0$ then
\index{structure matrices} \index{Carnot group!structure matrices of}
\begin{equation}\label{bizero}
B^i z \cdot w=0 \text{ for all } i=1, \dots, m_2.
\end{equation}
Let
$$
\G(\Z)=\{p=(z,t) \in \G: z \in \Z^{m_1} \text{ and }t\in \Z^{m_2}\}.
$$
In the following we are going to show that if the matrices
$(B^i)_{i=1}^{m_2}$ associated with the group law of $\G$ have integer
coefficients, then $\G(\Z)$ shares several properties with the usual
Euclidean integer lattices.

\begin{theorem}\label{integerlattice}
Let $\G$ be a Carnot group of step two with group operation $\ast$ and
let $d$ be a homogeneous metric on $\G$. If the structure matrices
$B^i$ of \eqref{cg2grouplaw} lie in $\Z^{m_1 \times m_1}$ for all
$i=1, \dots, m_2$ then
\begin{enumerate}
\item[(i)] $\gamma_1 \ast \gamma_2 \in \G(\Z)$ for all $\gamma_1, \gamma_2 \in \G(\Z)$,
\item[(ii)] there exists an absolute positive constant $A_1=A_1(\G,d)$ such
  that for all $\gamma \in \G(\Z) \stm \{o\}$
\begin{equation*}
d(\gamma,o)\geq A_1,
\end{equation*}
\item[(iii)] there exists an absolute positive constant $A_2=A_2(\G,d)$ such
  that for all $p \in \G$ there exists some $\gamma_p \in \G(\Z)$ such
  that
$$d(p,\gamma_p)\leq A_2.$$
\end{enumerate}
\end{theorem}

\begin{proof}
The proof of (i) is immediate since  $B^i \in \Z^{m_1 \times m_1}$.

We now move to the proof of (ii). Note that the function
$\|\cdot\|_d: \G \ra [0,\infty)$ defined by $\|p\|_d=d(p,o)$ for $p
\in \G$ is a homogeneous norm; i.e. it is a continuous function with
respect to the Euclidean topology of $\R^{m_1+m_2}$, $\|\delta_r
(p)\|_d=r \|p\|_d$ for all $r>0$ and $p \in \G$, $\|p\|_d>0$ if and
only if $p \neq o$, and $\|p^{-1}\|_d=\|p\|_d$ for all $p \in  \G \stm
\{o\}$.
For $p=(z,t) \in \G$  let
$$\|p\|=(|z|^4+|t|^2)^{1/4}.$$
Then it follows easily that $\|\cdot \|$ is a homogeneous quasi-norm
in $\G$, \cite[Section 5.1]{BLU}. Since all homogeneous quasi-norms
are globally equivalent, see for example \cite[Proposition
5.1.4]{BLU}, we conclude that there exists an absolute positive
constant $A_1$ such that
\begin{equation}
\label{globequiv}
A_1 \|p\| \leq \|p\|_d \leq A_1^{-1} \|p\|,
\end{equation}
for all $p\in \G$. But if $\gamma \in \gz \setminus \{o\}$ it follows immediately that $\|\gamma\| \geq 1$. Hence (ii) follows by \eqref{globequiv}.

We will now prove (iii). Let $$K_0=\{p=(w,s) \in \G: w \in [-1/2,1/2]^{m_1} \text{ and }s \in [-1/2,1/2]^{m_2}\}.$$ We will first show that
\begin{equation}
\label{tiling}
\G \subset \bigcup_{\gamma \in \gz} \gamma  \ast K_0.
\end{equation}
Let $p=(z,t) \in \G$, then there exists some $\gamma_1=(\gamma_z,0) \in \gz$ such that
\begin{equation}
\label{tiling1}
(z,0) \in \gamma_1 \ast K_0=\{(\gamma_z+w,t+(B^i\gamma_z\cdot w)_{i=1}^{m_2}):(w,t)\in K_0\}.
\end{equation}
Now notice that if $\gamma_2 \in \gz$ such that $\gamma_2=(0,\gamma_t)$, then recalling \eqref{bizero}
\begin{equation}
\label{tiling2}
\gamma_2 \ast \gamma_1 \ast K_0=\{(\gamma_z+w,\gamma_t+s+(B^i\gamma_z \cdot w)_{i=1}^{m_2}):(w,t)\in K_0\}.
\end{equation}
Therefore we can now choose $\gamma_t:=((\gamma_t)_1,\dots,(\gamma_t)_{m_2})\in \Z^{m_2}$ such that
$$|t_i-B^i\gamma_z \cdot w-(\gamma_t)_i|\leq 1/2$$ for all $i=1,\dots,m_2$. Therefore there exists some $s \in [-1/2,1/2]^{m_2}$ such that
\begin{equation}
\label{tiling3}
t=\gamma_t+(B^i\gamma_z \cdot w)_{i=1}^{m_2}+s.
\end{equation}
If $\gamma=(\gamma_z,\gamma_t) \in \gz$ then by \eqref{tiling1}, \eqref{tiling2} and \eqref{tiling3} we conclude that
$p \in \gamma \ast K_0$ and \eqref{tiling} follows.

Set $A_2=\max\{d(q,o): q \in K_0\}$ and let $p \in \G$. By \eqref{tiling} there exists some $\gamma_p \in \gz$ and $q \in K_0$ such that $p=\gamma_p \ast q$. Hence
$d(\gamma_p,p)=d(\gamma_p,\gamma_p \ast q)=d(0,q)\leq A_2$ and (iii) follows.
\end{proof}

\begin{remark}\label{iwacase}
Carnot groups of Iwasawa type, equipped with the gauge metric, satisfy the assumption of Theorem
\ref{integerlattice} with $A_1 = 1$.
\end{remark}

\subsection{Continued fractions as conformal iterated function systems}\label{subsec:cfiwasa}

Continued fractions \index{continued fractions} in the first Heisenberg group $\Heis$ have been considered by Lukyanenko and Vandehey \cite{luvacf}, see
also subsequent papers of Vandehey developing detailed number-theoretic properties of such continued fraction representations
\cite{vand:diophantine}, \cite{vand:lagrange}. As in the Euclidean
case, see e.g. \cite{MU1} and \cite{MU2}, continued fractions can be
realized as limit sets of conformal iterated function systems. In this
section we describe a class of conformal iterated function systems in
general Iwasawa groups which generalize those arising in connection
with Heisenberg continued fractions.

Let $\G$ be an Iwasawa group and recall that $d$ denotes its
Kor\'anyi--Cygan metric. Let $\G(\Z)$ be the integer lattice of $\G$
and for $\ve \geq 0$ set
$$I_\ve=\G(\Z) \cap B(o,\Delta_\ve)^c$$
for
$$
\Delta_\ve = \frac52+\ve.
$$
By Theorem \ref{integerlattice}(ii) it follows that for $p \in
\overline{B}\left(o,\frac{1}2\right)$ and $\gamma \in \gz \setminus
\{o\}$,
\begin{equation}\label{cfr1}
\frac{1}{2}d(\gamma,o) \leq d(\gamma, p) \leq 2 d(\gamma,o).
\end{equation}
We now consider the conformal iterated function system
\begin{equation}
\label{iwacfeq}
\cS_\ve=\left\{\f_\gamma: \overline{B}\left(o,1 / 2\right) \rightarrow \overline{B}\left(o,1/2\right) \right\}_{\gamma \in I_\ve}
\end{equation}
where
$$
\f_\gamma=\cJ \circ \ell_\gamma.
$$

Recalling the notation of Section \ref{sec:iwagdms} and in particular
\eqref{X}, we note that $X=\overline{B}\left(o,\frac{1}2\right)$ and we can take
$S=\overline{B}\left(o,\frac{2}3\right)$. Note that for every
$\gamma \in I_\ve$ and every $p \in
\overline{B}\left(o,\frac{1}2\right)$,
\begin{equation}\label{fginside}
d(\f_\gamma(p),o) = \frac{1}{d(\gamma \ast p,o)} \leq
\frac{1}{d(\gamma,o)-d(o,p)} \leq \frac{1}{2 +\ve} <
\frac{1}{2}
\end{equation}
by \eqref{conformal-inversion-one} and \eqref{cfr1}, and hence
$$\f_\gamma \left(\overline{B}\left(o,1/2 \right)\right) \subset \overline{B}\left(o,1/2\right).$$
The functions $\f_\gamma$ are injective contractions and
\begin{equation}
\label{cf}
\|D \f_\gamma (p)\| \approx d(\gamma,o)^{-2}
\end{equation}
for every $p \in S$, in particular $\|D \f_\gamma\|_\infty \approx
d(\gamma,o)^{-2}$. To see this, first note that as in
\eqref{conformal-inversion-two} for $p \in S$ and $\gamma \in \G(\Z)
\stm \{o\}$, $d(\gamma \ast p, o) \approx d(\gamma,o)$. Therefore, for
every $p,q \in S$,
\begin{equation}\begin{split}\label{cfderiv1}
d(\f_\gamma(p), \f_\gamma(q))
&=d(\cJ (\gamma \ast p),  \cJ(\gamma \ast q)) \\
&=\frac{d(p,q)}{d(\gamma \ast p, o) \, d(\gamma \ast q,o)} \\
&\approx  d(\gamma,o)^{-2} d(p,q)
\end{split}\end{equation}
by \eqref{cfr1}. Finally $\cS_\ve$ satisfies the open set condition,
as one can easily check that
\begin{equation}\label{osccf}
\f_{\gamma_1}\left(B \left(o,1/2\right)\right) \cap
\f_{\gamma_2}\left(B \left(o,1/2\right)\right)= \emptyset
\end{equation}
for all distinct $\gamma_1, \gamma_2 \in I_\ve$. Indeed suppose that
\eqref{osccf} fails. Then there exist distinct $\gamma_1,\gamma_2 \in
I_\ve$ and $p_1,p_2 \in B(o,\tfrac12)$ such that
$\f_{\gamma_1}(p_1)=\f_{\gamma_2}(p_2)$ or equivalently
$$
\gamma_2^{-1}\ast\gamma_1\ast p_1=p_2.
$$
Therefore $d(\gamma_2^{-1}\ast\gamma_1\ast p_1,o)=d(p_2,o)<\tfrac12$. But by Theorem \ref{integerlattice} we have
$$
d(\gamma_2^{-1}\ast\gamma_1\ast p_1,o)\ge
d(\gamma_2^{-1}\ast\gamma_1,o)-d(p_1,o)>1-\tfrac12 = \tfrac12
$$
and we have reached a contradiction.

\section{Complex hyperbolic Kleinian groups of Schottky type}

In this section we briefly recall the concept of complex hyperbolic Kleinian groups and especially the subclass of Schottky groups. The main objective of this section is to associate with each finitely generated Schottky group an Iwasawa maximal conformal graph directed Markov system having the same limit set as the limit set of the group. Having done this we essentially reduce the problem of studying geometric features of complex hyperbolic Schottky groups to the task of dealing with Iwasawa conformal GDMS. We emphasize that this is the only example in this chapter where the general framework of graph directed Markov systems is needed.

Complex hyperbolic space $H^{n+1}_\C$ can be modeled as the collection of timelike vectors in complex projective space $P^{n+1}_\C$. More precisely, we equip $\C^{n+1}$ with an indefinite Hermitian form $\langle\langle\cdot,\cdot\rangle\rangle$ of signature $(n,1)$  and define
$$
H^{n+1}_\C := \{ [z] \in P^{n+1}_\C \, : \, \langle\langle z,z \rangle\rangle < 0 \}.
$$
The space $H^{n+1}_\C$ can be identified with the unit ball of $\C^{n+1}$ equipped with a metric of constant negative holomorphic curvature, the Bergman metric. As mentioned in Section \ref{subsec:Iwasawa-conformal-classification3}, the compactified complex Heisenberg group $\overline{\Heis^n}$ arises as the boundary at infinity of $H^{n+1}_\C$, and conformal self-maps of $\Heis^n$ are the boundary values of isometries of $H^{n+1}_\C$.

A {\it complex hyperbolic Kleinian group} \index{complex hyperbolic Kleinian group} \index{Kleinian group!complex hyperbolic} is a discrete subgroup of the isometry group $\Isom(H^{n+1}_\C) = \PSU(n,1)$ of complex hyperbolic space. The theory of complex hyperbolic Kleinian groups is a rich and active area of research; we refer to the monograph \cite{cns:cKg} for more information and additional references to the literature. Here we only wish to recall the notion of complex hyperbolic Schottky group.

\begin{definition}
A complex Kleinian group $\Ga$, which we view both as a group of isometries of $H^{n+1}_\C$ and as a group of conformal transformations of $\overline{\Heis^n}$, is called a {\it Schottky group} if the following conditions are satisfied: \index{complex hyperbolic Schottky group} \index{Schottky group!complex hyperbolic}
\begin{enumerate}
\item for some integer $q \ge 2$, there exist $2q$ pairwise disjoint closed sets $B_1$, $B_2$, \dots, $B_q$ and $B_{-1}, B_{-2},\ld,B_{-q}$ in $\overline{\Heis^n}$ such that each $B_i$ is the closure of its interior,
\item for every $i\in V_+:=\{1,2,\ldots,q\}$ there exists $\g_i\in \Ga$ so that
\begin{equation}\label{giBi}
\g_i(B_i)=\ov{\overline{\Heis^n}\sms B_{-i}},
\end{equation}
and
\item the elements $\g_1,\g_2,\ld, g_q$ generate the group $\Ga$.
\end{enumerate}
\end{definition}

We set $V_- = -V_+ = \{-q,\ldots,-2,-1\}$ and $V:=V_+\cup V_-$. Applying $\g_{-i}:=\g_i^{-1}$ to both sides of \eqref{giBi} above, we see that $\g_{-i}(B_{-i})=\ov{\overline{\Heis^n}\sms B_{i}}$. Thus \eqref{giBi} holds for all $i\in V$. It also follows from \eqref{giBi} that
\beq\label{f1_2015_08_28}
\g_i(\ov{\overline{\Heis^n}\sms B_{i}})=B_{-i}.
\eeq
Recall that $\La(\Ga)$, the limit set of $\Ga$, is the set of all accumulation points of the orbit
$$
\Ga(x)
$$
for some, equivalently for any, point $x\in H^{n+1}_\C$. The limit set $\La(\Ga)$ is a non-empty closed (topologically) perfect subset of $\overline{\Heis^n}$ invariant under the action of $\Ga$. See \cite{cns:cKg} for this and other properties of the limit set. Now we associate a canonical Iwasawa conformal GDMS $\cS_\Ga$ to $\Ga$. The alphabet for $\cS_\Ga$ is the above defined set $V$. The edge set is
$$
E:=(V\times V)\sms\De,
$$
where $\De:=\{(i,i)\in V\times V:i\in V\}$. The incidence matrix $A:E\times E\to\{0,1\}$ is defined by the following formula:
$$
A_{(i,j),(k,l)}=
\begin{cases} 1 &\text{ if } j=-k \\
                        0 &\text{ if } j\ne -k
\end{cases}.
$$
Furthermore
$$
t((j,k)):=k \  \  \text{ and }  \  \  i((j,k)):=-j.
$$
For every $i\in V$ we set
$$
X_i:=B_i,
$$
and for every $(i,j)\in E$,
$$
\g_{(i,j)}:=\g_i|_{B_j}:B_j\to B_{-i}.
$$
The system
$$
\cS_\Ga= \{ V,E,A, t,i, (B_i)_{i \in V}, \{\g_e\}_{e\in E} \}
$$
satisfies all of the requirements of an Iwasawa maximal conformal graph directed Markov system, except for the fact that the maps $\{\g_e\}_{e\in E}$ need not be uniform contractions. However, since the diameters of the sets $\g_\om(B_{t(\om)})$ converge to zero uniformly with respect to the length of the word $\om$, the bounded distortion property implies that the mappings in a sufficiently high iterate of the system $\cS_\Ga$ are uniformly contracting. And this is precisely what we
need. We shall prove the following result establishing a close geometric connection between the complex Schottky group $\Ga$ and the associated Iwasawa conformal GDMS $\cS_\Ga$.

\begin{theorem}\lab{t4.9.3}
If $\Ga$ is a complex Schottky group, then $\La(\Ga)=J_{\cS_\Ga}$.
\end{theorem}

\begin{proof}
The inclusion $J_{\cS_\Ga}\sbt \La(\Ga)$ is obvious. In order
to prove the opposite inclusion, fix a sequence $\{g_n\}_{n=1}^\infty$
of mutually distinct elements of $\Ga$ such that the limit $\lim_{n\to\infty}g_n(z)$
exists for some (equivalently, all) $z\in H^{n+1}_\C$. For an appropriate choice of indices $i_{n,k_j}\in V$ we may write
$$
g_n=\g_{i_{n,k_n}}\circ \g_{i_{n,k_n-1}}\circ\ld\circ \g_{i_{n,2}}\circ \g_{i_{n,1}}
$$
in unique irreducible form. In other words, $i_{n,j+1}\ne -i_{n,j}$ for all $1\le j\le k_n-1$. Passing to a subsequence we may assume without loss of generality that $i_{n,1}=i$ for all $n\ge 1$ and for some $i\in \{1,2,\ldots, q\}$. Fix $z\in B_{-i}$. Then
$$
g_n(z)=\g_{(i_{n,k_n},-i_{n,k_n-1})}\circ \g_{(i_{n,k_n-1},-i_{n,k_n-2})} \circ \cdots \circ \g_{(i_{n,2},-i_{n,1})}\circ \g_{(i_{n,1},-i_{n,1})}(z)
$$
and
$$
\om^{(n)}:=(i_{n,k_n},-i_{n,k_n-1}) \, (i_{n,k_n-1},-i_{n,k_n-2}) \, \cdots \, (i_{n,2},-i_{n,1})(i_{n,1},-i_{n,1})
$$
is an element of $E_A^{k_n}$. Hence $g_n(z)\in g_\om(B_{t(\om^{(n)})})$. Since each of the sets $g_{\om^{(n)}}(B_{t(\om^{(n)})})$ intersects the limit set $J_\cS$, and since the diameters of those sets converge to zero as $n \to \infty$, we conclude that $\lim_{n\to\infty}g_n(z)\in \ov J_{\cS_\Ga}=J_{\cS_\Ga}$, where we have written the equality sign since the system $\cS_\Ga$ is finite. The proof is complete.
\end{proof}

\chapter{Hausdorff dimension of limit sets}\label{chap:IGDMS-dimensions}

In this chapter we employ the thermodynamic formalism \index{thermodynamic formalism} from Chapter \ref{Chapter CASD:TFF} as well as the distortion theorems from  Chapter \ref{chap:CGDMS} to study dimensions of limit sets of conformal GDMS in general Carnot groups. In Section \ref{sec-top-press} we revisit topological pressure \index{topological pressure} in the setting of weakly conformal GDMS $\cS$ and we define the $\theta$-number of $\cS$, as well as Bowen's parameter. We use the pressure function to define regular, strongly regular, and co-finitely regular systems. In Section \ref{sec-HD-Bowen} we introduce conformal measures and we prove a dynamical formula for the Hausdorff dimension of the limit set of a finitely irreducible Carnot conformal GDMS. This formula \index{Bowen's formula} traces back to the fundamental work of Rufus Bowen \cite{Bowen_QC} and is in spirit closest to an analogous formula in \cite{MUGDMS} in the context of Euclidean spaces. Section \ref{strong-reg} contains a characterization of strongly regular systems. Finally in Section \ref{spec} we prove that if $\cS$ is a Carnot conformal IFS and $t \in (0, \theta)$, there exists a subsystem of $\cS$ with Hausdorff dimension $t$.

\section{Topological pressure, $ \theta$-number, and Bowen's parameter}\label{sec-top-press}

Let $\mathcal{S}=\{\f_e\}_{e\in E}$  be a finitely irreducible Carnot conformal GDMS.
For $t\ge 0$, $n \in \N$ and $F \subset E$ let
$$
Z_{n}(F,t) = \sum_{\om\in F^n_A} ||D\phi_\om||^t_\infty.
$$
When $F=E$, we just write $Z_n(t)$ instead of $Z_n(E,t)$. By \eqref{quasi-multiplicativity1} we easily  see that
\begin{equation}
\label{zmn}
Z_{m+n}(t)\le Z_m(t)Z_n(t),
\end{equation}
and consequently, the
sequence $(\log Z_n(t))_{n=1}^\infty$ is subadditive. Thus, the limit
$$
\lim_{n \to  \infty}  \frac{ \log Z_n(t)}{n}
$$
exists and equals $\inf_{n \in \N} (\log Z_n(t) / n)$. The value of
the limit is denoted by $\P(t)$ or, if we want to be more precise, by
$\P_E(t)$ or $\P_\mathcal{S}(t)$. It is called the \index{topological pressure} \index{pressure} {\it topological
  pressure} of the system $\mathcal{S}$ evaluated at the parameter $t$.

Let $\zeta: E^\mathbb{N}_A \to \mathbb{R}$ be defined by the formula
\beq\label{1MU_2014_09_10}
\zeta(\om)= \log\|D\phi_{\om_1}(\pi(\sg(\om))\|.
\eeq
Using Lemma~\ref{l2.033101} we get easily (see \cite[Proposition 3.1.4]{MUGDMS} for complete details) the following.

\begin{lemma}\label{l1j85}
For $t \geq 0$ the function $t \zeta:E^\mathbb{N}_A \to
\mathbb{R}$ is H\"older continuous and $\P^\sg(t\zeta)=\P(t)$.
\end{lemma}

\begin{definition}
\label{fins}
We say that a nonnegative real number $t$ belongs to $\Fin(\cS)$ if
\beq\label{finite_parameters}
\sum_{e\in E}||D\phi_e||_\infty^t<+\infty.
\eeq
\end{definition}

Let us record the following immediate observation.

\begin{remark}
A nonnegative real number $t$ belongs to $\Fin(\cS)$ if and only if the H\"older continuous potential $t\zeta:E_A^\N\to\R$ is summable.
\end{remark}

Fix $t\in \Fin(\cS)$. The above observation along with
Chapter~\ref{Chapter CASD:TFF}, especially Section~\ref{P-FO}, allow us to consider the bounded linear operator $\mathcal{L}_t:=\pf_{t\zeta}$ acting on
$C_b(E^\mathbb{N}_A)$, which is, we recall, the Banach space of all real-valued bounded continuous functions on $E_A^\N$ endowed with the supremum norm $||\cdot||_\infty$. Immediately from the definition of $\pf_{t\zeta}$ we get the following.
\begin{equation}\label{1j89}
\mathcal{L}_t g(\om)= \sum_{i:\, A_{i \om_1}=1}
g(i \om)\| D\phi_i(\pi (\om))\|^t, \qquad \mbox{for $\om
  \in E^\mathbb{N}_A$.}
\end{equation}
A straightforward inductive calculation gives
\begin{equation}\label{1j90}
\mathcal{L}_t^n g(\om)= \sum_{\tau  \in E^n_A: \tau \om \in
E^\mathbb{N}_A} g(\tau\om)\|D\phi_\tau(\pi(\om))\|^t
\end{equation}
for all $n \in \N$. Recall that $\mathcal{L}_t^*: C^*(E^\mathbb{N}_A)\to C^*(E^\mathbb{N}_A)$ is the dual operator for $\mathcal{L}_t$. Formulas \eqref{1j89} and \eqref{1j90} clearly extend to all Borel
bounded functions $g:E^\mathbb{N}_A\to \mathbb{R}$.
Corrolary~\ref{c2.7.5} applied to the functions $t\zeta$ gives the following.

\begin{theorem}\lab{thm-conformal-invariant}
Suppose that the system $\cS$ is finitely irreducible and $t\in\Fin(\cS)$. Then
\begin{itemize}
\item[(a)] There exists a unique eigenmeasure $\^m_t$ of the conjugate
Perron-Frobenius operator $\pf_t^*$ and the corresponding eigenvalue is equal
to $e^{\P(t)}$.

\sp\item[(b)] The eigenmeasure $\^m_t$ is a Gibbs state for $t\zeta$.

\sp\item[(c)] The function $t\zeta:E_A^\N\to\R$ has a unique $\sg$-invariant Gibbs state $\^\mu_t$.

\sp\item[(d)] The measure $\tilde\mu_t$ is ergodic, equivalent to $\tilde m_t$ and $\log(d\tilde\mu_t/d\tilde m_t)$ is uniformly bounded.

\sp\item[(e)] If $\int\zeta\,d\^\mu_t>-\infty$, then the $\sg$-invariant Gibbs
state $\^\mu_t$ is the unique equilibrium state for the potential $t\zeta$.

\sp\item[(f)] The Gibbs state $\^\mu_t$ is ergodic, and in case the system $\cS$ is finitely primitive, it is completely ergodic.
\end{itemize}
\end{theorem}

A direct straightforward calculation gives us the following.

\begin{remark}\label{o5_2016_01_20}
If $t\in\Int(\Fin(\cS))$, then $\int\zeta\,d\^\mu_t>-\infty$.
\end{remark}

Item (a) of Theorem~\ref{thm-conformal-invariant} means that
\begin{equation}\label{2j91}
\mathcal{L}_t^* \^m_t = e^{\P(t)}\^m_t.
\end{equation}
Based on this, standard approximation arguments show that
\beq\label{mtdual}
\mathcal{L}_t^{*n} \^m_t(g)=e^{\P(t)}\^m_t(g)
\eeq
for all Borel bounded functions $g:E^\mathbb{N}_A\to\mathbb{R}$.
Given $t\in\Fin(\cS)$ it immediately follows from \eqref{2.3.1227a} and \eqref{2.3.1227b} that
\begin{equation}\label{515pre}
c_t^{-1} e^{-\P(t)|\om|}||D\phi_\om||_\infty^t
\leq \^m_t([\om])
\leq c_te^{-\P(t)|\om|}||D\phi_\om||_\infty^t
\end{equation}
for all $\om\in E_A^*$, where $c_t\ge 1$ denotes some constant.
For any set $F \subset E$ let
%$$F(\mathcal{S})=\{  t \geq 0:   \,\, \P(t)< \infty\}$$ and
$$
\theta_F = \theta_{\mathcal{S}_F} := \inf\Fin(\cS)=\inf\{t \geq 0 : Z_1(t)< +\infty\},
$$
where $\mathcal{S}_F=\{\f_e\}_{e \in F}$.
When $F=E$ we simply denote
$$
\theta = \theta_{\mathcal{S}} := \inf\Fin(\cS).
$$
The following facts easily follow from the bounded distortion property \index{GDMS!Bounded Distortion Property of}
(Lemma~\ref{l12013_03_11}) and the uniform contractivity of all generators of the system $\cS$.

\begin{proposition}\label{p2j85}
Let $\cS$ be a finitely irreducible Carnot conformal GDMS. Then the
following conclusions hold.
\begin{enumerate}[(i)]

\sp\item  $\Fin(\cS)=\{t\ge 0:\P(t)< +\infty\}$.

\sp\item  $\theta_{\mathcal{S}}=\inf \{t\ge 0: \, \P(t)< + \infty\}$.

\sp\item The topological pressure $\P$ is
strictly  decreasing  on $[0,+\infty)$ with $\P(t) \to -\infty$ as
$t\to+\infty$. Moreover, the function
$\P$ is convex and  continuous on the closure of $\Fin(\cS)$.

\sp\item $\P(0)=+\infty$  if and only if $E$  is infinite.
\end{enumerate}
\end{proposition}

\sp The number
$$
h = h_{\mathcal{S}} := \inf\{t\geq 0:  \P(t)\leq 0\}
$$
is called {\it Bowen's parameter} \index{Bowen's parameter} of the  system $\mathcal{S}$. It will soon turn out to coincide with the Hausdorff dimension of the limit
set $J_\cS$. In view of Proposition~\ref{p2j85} (iii) we have the following.

\begin{remark} \label{phleq0}
If $\cS$ is a finitely irreducible Carnot conformal GDMS, then $h_\cS\in\Fin(\cS)$ and $$
\P(h)\le 0.
$$
\end{remark}

\begin{definition}
\label{reguraldef}
A finitely irreducible Carnot conformal GDMS $\cS$ is
\begin{itemize}
\item {\it regular} \index{GDMS!regular} if $\P(h)=0$,

\sp\item {\it strongly regular} \index{GDMS!strongly regular} if there exists $t \geq 0$ such that $0< \P(t) <+\infty$, and

\sp\item {\it co-finitely regular} \index{GDMS!co-finitely regular} if $\P(\theta)=+\infty$.
\end{itemize}
\end{definition}

As an immediate consequence of Proposition~\ref{p2j85} we get the following.

\begin{remark}
A finitely irreducible Carnot conformal GDMS $\cS$ is regular if and only if $\P(t)=0$ for some $t\ge 0$. Moreover, if $\P(t)=0$ for some $t\ge 0$, then $t=h_\cS$.
\end{remark}

\begin{remark}
Each finite irreducible Carnot conformal GDMS $\cS$ is regular.
\end{remark}

In fact more is true; if in addition $E$  contains at least two elements then $\cS$ is strongly regular. This is proven in the following proposition.

\begin{proposition}
\label{p(0)positive}
If $\cS= \{\f_e\}_{e \in E}$ is a finite and  irreducible weakly conformal GDMS  then $\P(0)>0$. In particular $\cS$ is strongly regular.
\end{proposition}

\begin{proof} Let $\Phi \subset E_A^\ast$ be the finite set witnessing irreducibility. We can also also assume that $\Phi$ consists of words with minimal length, in the sense that if $a,b \in E$, $\om \in \Phi,$ and $a \om b \in E_A^\ast$ then if $a \tau b \in E_A^\ast$ for some $\tau \in E_A^\ast$ then $|\om| \leq |\tau|$.

Recall that we always assume that $E$ contains at least two elements. Let $a,b\in E$, $a \neq b$. By irreducibility there exist $\om,\tau, \in \Phi$  such that
\begin{align*}
\rho^1&:=a \om a \in E_A^\ast, \\
\rho^2&:=a \tau b \in E_A^\ast.
\end{align*}
Observe that the words $\rho^i, i=1,2,$ are mutually incomparable.
Assume by way of contradiction that $\rho^1$ and $ \rho^2$ are comparable, and without loss of generality also assume that $|\rho^2|>|\rho^1|$. Then there exists some $\upsilon \in E_A^\ast$ such that
$$\rho^2=\rho^1 \upsilon b=a \om a \upsilon b \in E_A^\ast.$$
Hence $\tau=\om a \upsilon$, $ a \upsilon b \in E_A^\ast$ and $|\upsilon|< |\tau|$. But this violates the minimality of $\tau$ as $\tau \in \Phi$. %An identical argument shows that $\rho^3$ and $\rho^4$ are also incomparable.
By irreducibility there exist $\tilde{\rho}^i \in E_A^\N, i=1, 2,$ such that
$$\tilde{\rho}^i|_{|\rho^i|}= \rho^i, \mbox{ for }i=1,\dots,2.$$
For example if $\gamma \in \Phi$ such that $b \gamma a \in E_A^\ast$ one could take $\tilde{\rho}^1=a\om a\om a\dots$ and $\tilde{\rho}^2=a\tau b \gamma \tilde{\rho}^1$. Let
\begin{equation}
\label{qlength}
\mathtt{q}=\max\{|\om|:\om \in \Phi\}+2.
\end{equation}
Then the words $\tilde{\rho}^i|_q, i=1,2,$ are mutually incomparable.

Therefore we have shown that for every $e \in E$ there exist incomparable words $\om^1(e), \om^2(e) \in E_A^{\mathtt{q}}$ such that $e={\om^1(e)}_1={\om^2(e)}_1$. This implies that for all $n \in \N$
\begin{equation}
\label{eaqncard}
\sharp E_A^{n \mathtt{q}} \geq 2^n.
\end{equation}

We can now finish the proof of the proposition. Recall that for $n \in \N$,
$$Z_n(0)= \sum_{\om \in E_A^n} 1= \sharp E_A^n.$$
Hence by \eqref{eaqncard},
\begin{equation*}
\begin{split}
\P(0)&=\lim_{n \ra \infty} \frac{ \log Z_n(0)}{n} =\lim_{ n \ra \infty}\frac{\log Z_{\mathtt{q}n}(0)}{\mathtt{q}n}\\
& = \lim_{ n \ra \infty}\frac{\log \sharp E_A^{\mathtt{q}n}}{\mathtt{q}n}
\ge \limsup_{n \ra \infty} \frac{\log 2^{n}}{\mathtt{q}n}
=\frac{\log2}{\mathtt{q}}.
\end{split}
\end{equation*}
The proof is complete.
\end{proof}

Obviously, if the system $\cS$ is strongly regular, then $h_\cS\in\Int(\Fin(\cS))$. Remark~\ref{o5_2016_01_20} then entails the following.

\begin{remark}\label{o6_2016_01_20}
If the system $\cS$ is strongly regular, then $\int\zeta\,d\^\mu_t>-\infty$.
\end{remark}

It is easy to see that each co-finitely regular system is strongly
regular and each strongly regular one is regular. If $F\subset E$ is co-finite, that is the set $E \setminus F$ is finite, we will say that $\cS_F=\{\f_e\}_{e \in F}$ is a \textit{co-finite subsystem} \index{co-finite subsystem} of $\cS=\{\f_e\}_{e \in E}$. We also record the following lemma which will turn out to be useful in the following. It is an immediate corollary of  Proposition \ref{p2j85} (i).

\begin{lemma}
\label{433mu}
Let $\cS$ be a finitely irreducible Carnot conformal GDMS. The following conditions are equivalent.
\begin{enumerate}
\item [(i)] $Z_1(t)<\infty$.
\item [(ii)] There exists a co-finite subsystem $\cS_F$ of $\cS$ such that $Z_1(F,t)<\infty$.
\item [(iii)] For every co-finite subsystem $\cS_F$ of $\cS$ it holds that $Z_1(F,t)<\infty$.
\item [(iv)] $\P(t)<\infty$.
\item [(v)] There exists a co-finite subsystem $\cS_F$ of $\cS$ such that $\P_{F}(t)<\infty$.
\item [(vi)] For every co-finite subsystem $\cS_F$ of $\cS$ it holds that $\P_{F}(t)<\infty$.
\end{enumerate}
\end{lemma}

The following proposition provides a useful characterization of co-finitely regular systems.
\begin{proposition}\label{cof-reg}
A Carnot conformal GDMS $\cS$ is  co-finitely regular if and only if every co-finite subsystem $\cS_F$ is regular.
\end{proposition}

\begin{proof}
First note that by Lemma \ref{433mu} $\th_F=\th_\cS=\theta$ for every co-finite subset $F$ of $E$. Suppose now that $\cS$ is co-finitely regular. In view of
Proposition~\ref{p2j85}(i) this means that $Z_1(\theta)=+\infty$. But
then again by Lemma \ref{433mu} $Z_1(F,\th)=+\infty$ for every co-finite subset $F$ of $E$. A second application of Proposition~\ref{p2j85}(a) ensures that each
such system $\cS_F$ is co-finitely regular, thus regular.

For the converse suppose that the system $\cS$ is not co-finitely
regular. A third application of Proposition~\ref{p2j85}(a) yields that
$Z_1(\th)<+\infty$. There exists a co-finite subset $F$ of $E$ such
that $Z_1(F,\th)<1$. Hence, by the definition of topological pressure, \index{topological pressure}
$$\P_F(\th) \leq \limsup_{n \ra \infty} \frac{1}{n} \log Z_1(F, \theta)^n<0.$$ Hence $F$ is not regular and we have reached a contradiction.
\end{proof}

Let us also record the following obvious fact. Proposition \ref{f1j87} provides a simple mechanism for obtaining lower bounds on Bowen's parameter \index{Bowen's parameter} $h_\mathcal{S}$ (which, we will shortly see, coincides with the Hausdorff dimension of $J_\mathcal{S}$ in many cases).

\begin{proposition}\label{f1j87}
If $\mathcal{S}$ is a finitely irreducible Carnot conformal GDMS,
then $\theta_{\mathcal{S}}\leq h_{\mathcal{S}}$. If $\mathcal{S}$ is
strongly regular, in particular, if $\mathcal{S}$ is co-finitely
regular, then $\theta_{\mathcal{S}}< h_\mathcal{S}$.
\end{proposition}

We will now provide another characterization of the $\theta_S$ number.

\begin{theorem}
\label{infcofi}
Let $\cS=\{\f_e\}_{e \in E}$ be a finitely irreducible Carnot conformal GDMS. Then
$$\theta_\cS=\inf \{h_{E\setminus T}: T  \text{ finite subset of }E\}.$$
\end{theorem}
\begin{proof} By Lemma \ref{433mu} $\theta_\cS=\theta_F$ for every co-finite set $F \subset E$. Moreover $\theta_F \leq h_F$ for every $F \subset E$, therefore,
$$\theta_\cS\leq \inf \{h_{E\setminus T}: T  \text{ finite subset of }E\}.$$
For the other direction let $t> \theta_\cS$. Then $\sum_{e \in E} \|D \f_e\|^t<\infty$, therefore there exists some finite set $F \subset E$ such that
$$\sum_{e \in E\stm F} \|D \f_e\|^t<1.$$
Thus for every finite set $T$ such that $F \subset T \subset E$,
$$Z_1(E \setminus T, t) \leq Z_1(E \setminus F, t)<1.$$
Now as in the proof of Proposition \ref{cof-reg} we deduce that
$$
\P_{E \setminus T}(t) \leq \log Z_1(E \setminus T, t)<0.
$$
Hence $t \geq h_{E \setminus T}$, for such finite sets $T$ and in particular $$\inf \{h_{E\setminus T}: T  \text{ finite subset of }E\}\leq t.$$
Therefore $\inf \{h_{E\setminus T}: T  \text{ finite subset of }E\}\leq \theta_\cS$ and the proof is complete.
\end{proof}

\section{Hausdorff dimension and Bowen's formula}\label{sec-HD-Bowen}

In this section we exhibit a dynamical formula for the Hausdorff
dimension of the limit set of a finitely irreducible Carnot conformal GDMS. Due to its correspondence to Bowen's work \cite{Bowen_QC}, we refer
to it as {\it Bowen's formula}. \index{Bowen's formula}

We will denote by $\cH^s$, resp.\ $\cP^s$, the Hausdorff, resp.\
packing measure of dimension $s$ in $(\G,d)$. We will also denote by $\cS^s$  the $s$-dimensional spherical Hausdorff measure. The
Hausdorff and packing dimensions of $S \subset (\G,d)$ will be denoted by
$\dim_{\mathcal{H}}(S)$ and $\dim_{\mathcal{P}}(S)$ respectively. See \cite{mat:geometry} for the exact definitions.
\index{Hausdorff dimension}
\index{packing dimension}
\index{Hausdorff measure}
\index{spherical Hausdorff measure}

We begin with the following simple observation following from the open
set condition.

\begin{lemma}\label{l1j81}
Let $\mathcal{S}$ be an Carnot conformal GDMS. For all $0< \kappa_1 <
\kappa_2< \infty$, for all $r >0$, and for all $p \in \G$, the
cardinality of any collection of mutually incomparable words $\om \in
E^*_A$ that satisfy the conditions
$B(p,r) \cap \phi_\om(X_{t( \om)}) \neq \es$
and
$\kappa_1 r \leq \diam(\phi_\om( X_{t(\om)}))\le \kappa_2 r$
is bounded above by
\begin{equation}\label{l1j81-upper-bound}
\mathtt{m}_{\kappa_1,\kappa_2}:=\left(\frac{(1+\kappa_2)M\La K\,C}{\tilde{R}_\mathcal{S} \kappa_1}\right)^Q,
\end{equation}
where $Q$ is the homogeneous dimension of $\G$, $\tilde{R}_\mathcal{S}$ was defined in Lemma \ref{l52013_03_12}, %is the radius of the largest open ball that can be inscribed in any of the sets $X_v$, $v \in V$,
$C$ is the quasiconvexity constant from Corollary \ref{koebe2}, and  $K$ and $\La$ denote the constants from Lemmas
\ref{l22013_03_12} and \ref{l12013_03_11} respectively, and $M = \diam X$.
\end{lemma}

\begin{proof}
Recall that $|\cdot|$ denotes the Haar measure on $\G$, and  $c_0=|B(o,1)|$. For every $v\in V$ let $p_v$ be
the center of a ball with radius $\tilde{R}_\cS$ contained in $\Int(X_v)$. Let
$F$ be any collection of $A$-admissible words satisfying the
hypotheses of the lemma. Then
$$
\phi_\om(X_{t(\om)})
\sbt B(p, r +\diam(\phi_\om(X_{t(\om)})))
\sbt B(p, (1+ \kappa_2)r)
$$
for every $\om \in  F$. Since, by the open set condition, the sets $\{\phi_\om(\Int
X_{t(\om)}):\om \in F\}$ are mutually disjoint, applying
\eqref{measure-of-bcc} along with \eqref{4.1.9} and \eqref{4.1.10a} yields
\begin{equation*}\begin{split}
c_0(1+\kappa_2)^Qr^Q
&= | B(p,(1+\kappa_2) r )| \\
&\ge \lt|\bigcup_{\om \in F} \phi_\om(X_{t(\om)})\rt|\\
&=  \sum_{\om \in F} |\phi_\om(\Int(X_{t(\om)}))| \\
% \geq \sum_{\om \in F} |\phi_\om(B(p_{t(\om)}, \tilde{R}_\mathcal{S}))|\\
& \ge \sum_{\om \in F} |B(\phi_\om(p_{t(\om)}),
(KC)^{-1}\tilde{R}_\mathcal{S}|| D\phi_\om||_\infty)|\\
& \ge \sum_{\om \in F} | B(\phi_\om(p_{t(\om)}),
(M\La KC)^{-1}\tilde{R}_\mathcal{S} \diam(\phi_\om(X_{t(\om)}))|\\
& \ge \sum_{\om \in F} |B(\phi_\om(p_{t(\om)}),
(M\La KC)^{-1}\tilde{R}_\mathcal{S} \kappa_1 r )| \\
&= c_0(\sharp F)((M\La KC)^{-1} \tilde{R}_\mathcal{S}\kappa_1)^Q \, r^Q.\\
\end{split}\end{equation*}
Hence $\sharp F$ is bounded above by $\mathtt{m}_{\kappa_1,\kappa_2}$ and we are done.
\end{proof}

In the following proposition we prove that if a GDMS $\cS$ is conformal or if $J_\cS$ has positive $h$-Hausdorff measure, then  for all $v\in V$, $J_\cS \cap X_v$ is an infinite set. This proposition will be essential for Chapter \ref{chap:separation-equiv}, where we will use the fact that the limit set does not collapse to a point inside any $X_v, v \in V$.

\begin{proposition}
\label{diamgtr1}
Let $\cS=\{\phi_e\}_{e\in E}$ be a finitely irreducible weakly Carnot conformal GDMS.
\begin{enumerate}[(i)]
\item If $\cS$ is Carnot conformal then $\sharp(J_\cS \cap X_v) =\infty$ for all $v \in V$. In fact $\ov{J_\cS} \cap X_v$ is a non-empty compact perfect set and $J_\cS \cap X_v$ is of cardinality $\mathfrak{c}$.
\item If $\cH^h (J_\cS)>0$  then $\sharp(J_\cS \cap X_v)=\infty$ for all $v \in V$. In fact $\ov{J_\cS} \cap X_v$ is a non-empty compact perfect set and $J_\cS \cap X_v$ is of cardinality $\mathfrak{c}$.
\end{enumerate}
\end{proposition}
\begin{proof}

We will first prove (i). Observe that it is enough to prove (i) for irreducible finite systems. To verify that such assertion is indeed enough,  let $\cS$ be a finitely irreducible infinite system and let $\Phi \subset E_A^\ast$ be the set witnessing finite irreducibility for the matrix $A$. Let $$F=\{e \in E: e=\om_i \mbox{ for some }\om \in \Phi, i=1,\dots,|\om|\},$$
i.e.  $F$ consists of the letters from $E$ appearing in the words of $\Phi$. Also for any $v\in V$ choose one $e_v \in E$ such that $i(e_v)=v$. Let $\tilde{F}=F \cup \{e_v:v \in V \}$. Then $\cS_{\tilde{F}}$ is an irreducible finite subsystem of $\cS$, and $J_\cS \supset J_{\cS_{\tilde{F}}}$. Hence if (i) holds for $\cS_{\tilde{F}}$ then it will also hold for $\cS$. Therefore we can assume that $E$ is finite. Fix $v \in V$ and $\xi \in J_\cS \cap X_v$. For $r>0$ define,
$$\cF_\xi(r)= \{ \om \in E_A^\ast: \xi \in \f_\om(X_{t(\om)}), \|D \f_\om\|_\infty \leq r, \mbox{ and }\|D \f_{\om|_{|\om|-1}}\|_\infty >r\}.$$
If $\om \in \cF_\xi(r)$, \eqref{4.1.9}  implies that
\begin{equation}
\label{fr1}
\diam(\f_\om(X_{t(\om)})) \leq \La M r.
\end{equation}
We now define
\begin{equation}
\label{D0} \mathtt{D_0}:=\min\{\|D \f_\e\|_\infty: e \in E\}>0.
\end{equation}
Hence if $\om \in \cF_\xi(r)$ by \eqref{quasi-multiplicativity1} and \eqref{4.1.10},
\begin{equation}
\label{fr2}
\diam(\f_\om(X_{t(\om)})) \geq 2 (K^2 C)^{-1} \mathtt{D}_0 \tilde{R}_\cS r.
\end{equation}
Therefore by Lemma \ref{l1j81}, \eqref{fr1} and \eqref{fr2} we deduce that
\begin{equation}
\label{cardfr}
\sharp \cF_\xi(r) \leq \mathtt{m}_{\kappa_1, \kappa_2},
\end{equation}
where $\kappa_1=2 (K^2 C)^{-1} \mathtt{D}_0 \tilde{R}_\cS$ and $\kappa_2=\La M$. Observe also that $\cF_\xi(r)$ consists of mutually incomparable words.

We will show that
\begin{equation}
\label{claim2}
\sharp \{ \tau \in E_A^N: \pi(\tau)=\xi\} \leq \mathtt{m}_{\kappa_1, \kappa_2}.
\end{equation}
Let $I$ be an index set such that  $\{ \tau \in E_A^N: \pi(\tau)=\xi\}=(\tau^i)_{i \in I}$ and the words $\tau^i$ are distinct. Since $E$ is finite,
$$r:=\min\{\|D \f_{\tau^i_1}\|_\infty: i \in I\} \geq \mathtt{D}_0>0.$$
Therefore for all $i \in I$, there exists some $k(i) \in \N$ such that
$$\|D \f_{\tau^i |_{k(i)}}\|_\infty \leq r/2 \mbox{ and }\|D \f_{\tau^i |_{k(i)-1}}\|_\infty>r/2.$$
Hence $\{ \tau^i|_{k(i)}: i \in \N\} \subset \cF_\xi(r/2)$, and \eqref{claim2} follows by \eqref{cardfr}. Recalling, \eqref{eaqncard}, there exist infinitely many words $\om \in E_A^N$ such that $\om_1=e_v$. Equivalently there exist infinitely many words $\om \in E_A^N$, such that $\pi(\om) \in J_\cS \cap X_{v}$.  Hence (i) follows by \eqref{claim2}.

We now move to the proof of (ii). First, we record that Proposition \ref{p(0)positive} implies that $\cS$ is strongly regular and $h>0$. If $\cH^h(J_\cS)>0$ then there exists some $v_0$ such that $J_\cS \cap X_{v_0}$ has the cardinality of the continuum. Let
$$E_0=\{e \in E: i(e)=v_0\},$$
and for $e \in E_0$, let
$$W_e=\{ x \in J_\cS \cap X_{v_0}: x=\pi(\om) \mbox{ for some }\om \in E_A^\N \mbox{ such that }\om_1=e\}.$$
Then $J_\cS \cap X_{v_0} = \cup_{e \in E} W_e$. Since $J_\cS \cap X_{v_0}$ has the cardinality of the continuum there exists some $e_0 \in E_0$ such that $\sharp W_{e_0}=\infty$. Therefore there exists a sequence of distinct words $(\om^i)_{i \in \N}$ such that $\pi(\om^i) \neq \pi(\om^j) \mbox{ for }i\neq j,$ and $\om^i_1=e_0$ for all $i \in \N$.

Now let $v \in V$ and let $e_v \in E$ such that $i(e_v)=v$. By irreducibility of $E$ there exists some $\rho \in \Phi$ such that $e_v \rho e_0 \in E_A^\ast$.  We consider the sequence $(x_i)_{i \in \N}$ where $x_i=\pi(e_v\rho \om^i)$. Notice that $x_i \in J_\cS \cap X_v$ for all $i \in \N$. In order to finish the proof of (ii) it is enough to show that  $x_i \neq x_j$ for $i \neq j$. This follows because if $i \neq j$, $x_i=\f_{e_v\rho}(\pi(\om^i))$ and $x_j=\f_{e_v\rho}(\pi(\om^j))$ and $\pi(\om^i) \neq \pi(\om^j)$. The proof of (ii) is complete.
%Notice that for all $x\in J_\cS \cap X_{v_0}$, $e \rho \om(x) \n E_A^\N$ because
\end{proof}

As an immediate consequence of the definition of regularity of a conformal GDMS and of Theorem~\ref{thm-conformal-invariant}, we get the following.

\begin{proposition}\label{p2j93}
If $\cS$ is a finitely irreducible weakly Carnot conformal GDMS, then the following conditions are equivalent:

\begin{itemize}
\item[(a)] The system $\cS$ is regular

\sp \item[(b)] There exist $t\in\Fin(\cS)$ and a Borel probability measure
$\hat m$ on $E_A^{\N}$ such that $\cL_t^*\hat m=\hat m$. Then necessarily $t=h$ and $\hat m=\tilde m_{h}$.

\sp \item[(c)] $\cL_{h}^*\tilde m_{h}=\tilde m_{h}$.
\end{itemize}
\end{proposition}

For all $t \in \Fin(\cS)$ we will denote
\begin{equation}\label{mutdef}
m_t := \^m_t\circ \pi^{-1} \  \  {\rm and  }  \  \  \  \mu_t := \^\mu_t\circ \pi^{-1}.
\end{equation}
If the system $\cS$ is regular, then formula \eqref{515pre} for $t=h$ yields that for every $\om \in E_A^\ast$
\begin{equation}\label{2j93}
c_{h}^{-1}||D\phi_\om||_\infty^{h}
\leq \^m_{h}([\om])
\leq c_{h}||D\phi_\om||_\infty^{h},
\end{equation}
where $c_h\ge 1$ is  some finite constant. Then, the measure
$$
m_{h} = \^m_{h}\circ \pi^{-1},
$$
supported on $J_\cS$, will be called throughout the manuscript the \textit{$h_{\cS}$-conformal measure} for $\cS$.  \index{measure!conformal} \index{conformal measure}

In the following theorem we prove that in the case when $\mathcal{S}$ is finite, the conformal measure $m_h$ is Ahlfors $h$-regular, thus obtaining information about the  Hausdorff and packing measures of $J_\cS$.

\begin{theorem} \label{t1j93}
Let $\mathcal{S}=\{\phi_e\}_{e\in E}$ be a finite
irreducible Carnot conformal GDMS. Then the measure $m_h$ is an
Ahlfors $h$-regular measure \index{Ahlfors regular measure} \index{measure!Ahlfors regular} on $(J_\cS,d)$, i.e., there exists a
constant $c_\cS \geq 1$  such that
\begin{equation}\label{h-regularity}
c_\cS^{-1}r^h \le m_h(B(p,r)) \le c_\cS r^h
\end{equation}
for all $p\in J_\cS$ and all $ 0<r \leq 1$. In particular, both the
Hausdorff and the packing measures of $J_\cS$, $\cH^h(J_\cS)$ and
$\cP^h(J_\cS)$, are positive and finite, and $(J_\cS,d)$ has
Hausdorff dimension equal to $h$. \index{Hausdorff measure} \index{packing measure} \index{Hausdorff dimension}
\end{theorem}

\begin{proof}
Without loss of generality we can assume that $E$ contains at least two elements. Proposition \ref{p(0)positive} implies that $\cS$ is strongly regular and $h>0$. Recalling \eqref{D0},
$\mathtt{D}_0 := \min\{||D\phi_e||_\infty \, : \,  e \in  E\} > 0.$

Fix  $p\in J_\cS$ and $0< r < \frac{1}{2}\min  \{  \diam(X_v) : \, v
\in V\}$. Then $p = \pi(\tau)$  for some $\tau \in E^\mathbb{N}_A$.
Let $n=n(\tau) \geq 0$ be the least integer such that
$\phi_{\tau |_n}(X_{t(\tau_n)})\sbt B(p,r)$. By \eqref{2j93} and
\eqref{quasi-multiplicativity} we have
\begin{equation*}\begin{split}
m_h(B(p,r))
&\ge m_h\(\phi_{\tau |_n}(X_{t(\tau_n)})\) \\
&\ge \^m_h([\tau |_n])) \\
&\ge c_h ||D\phi_{\tau |_n}||_\infty^h \\
&\ge \mathtt{D}_0 \, c_h K^{-h}||D\phi_{\tau |_{(n-1)}}||_\infty^h.
\end{split}\end{equation*}
By the definition of $n$ and Lemma~\ref{l22013_03_12}, we have
$$
r \le \diam(\phi_{\tau_{|n-1}}(X_{t(\tau_{n-1})}))\le M \La ||D\phi_{\tau |_{(n-1)}}||_\infty
$$
and hence
\begin{equation}\label{1j95}
m_h(B(p,r))\ge \xi\,c_h(M\La K)^{-h} r^h.
\end{equation}
To prove the opposite inequality, let $Z$ be the family of all minimal
length words $\om \in E^*_A$ such that
\begin{equation}\label{2j95}
\phi_\om(X_{t(\om)})\cap B(p,r) \ne \es \qquad \mbox{and} \qquad
\phi_\om(X_{t(\om)})\sbt B(p,2r).
\end{equation}
Consider an arbitrary $\om \in Z$ with $|\om|=n$. Then
\begin{equation}\label{120130320}
\diam\(\phi_\om(X_{t(\om)})\)\le 4r
\end{equation}
and
\begin{equation}
\label{diamgeqr}
\diam(\phi_{\om |_{(n-1)}}(X_{t(\om |_{(n-1)})}))\ge r.
\end{equation}
To prove \eqref{diamgeqr} first notice that as $\om \in Z$, $\f_{\om|_{(n-1)}}(X_{t(\om |_{(n-1)})}) \cap B(p,r) \neq \emptyset$. Moreover because $\om|_{(n-1)} \notin Z$ we also have that $\f_{\om|_{(n-1)}}(X_{t(\om |_{(n-1)})}) \not\subset B(p,2r)$. Therefore \eqref{diamgeqr} follows. Making use
of \eqref{4.1.9}, \eqref{4.1.10}, \eqref{quasi-multiplicativity} and \eqref{diamgeqr} we get
\begin{equation}\label{220130320}
\begin{aligned}
\diam\(\phi_\om(X_{t(\om)})\)
&\ge 2(KC)^{-1} \tilde{R}_\cS ||D\phi_\om||_\infty \\
&\ge 2(K^2C)^{-1} \tilde{R}_\cS ||D\phi_{\om |_{(n-1)}}||_\infty \cdot ||D\phi_{\om_n}||_\infty \\
&\ge 2\mathtt{D}_0 (K^2C)^{-1} \tilde{R}_\cS (M\La)^{-1} \diam\(\phi_{\om |_{(n-1)}}(X_{t(\om |_{(n-1)})})\)\\
&\ge 2\mathtt{D}_0 \tilde{R}_\cS (M \La K^2 C)^{-1} r.
\end{aligned}
\end{equation}
Since the family $Z$ consists of mutually incomparable words,
Lemma~\ref{l1j81} along with \eqref{120130320} and \eqref{220130320}
imply that
\begin{equation}\label{3j95}
\sharp Z \leq \Ga:=
\left( \frac{5M^2\La^2K^3C^2}{2\mathtt{D}_0 \tilde{R}_\cS^2} \right)^Q.
\end{equation}
Since $\pi^{-1}(B(p,r)) \sbt  \bigcup_{\om \in Z} [\om]$, we get
from %(\ref{2j95}),
\eqref{2j93}, \eqref{4.1.10}, \eqref{120130320}, and \eqref{3j95} that
\begin{equation*}\begin{split}
m_h(B(p,r)) &= \^m_h \circ \pi^{-1}(B(p,r)) \\
&\leq \^m_h\lt(\bigcup_{\om \in Z} [\om]\rt) = \sum_{\om \in
  Z}\^m_h([\om]) \\
&\leq \sum_{\om \in Z}||D\phi_\om||_\infty^h \leq \sum_{\om \in
  Z}\biggl( KC(2\tilde{R}_\cS)^{-1} \diam(\phi_{\om}(X_{t(\om)}))
\biggr)^h \\
&\leq (2KC\tilde{R}_\cS^{-1})^h \sum_{\om \in Z} r^h =(2KC\tilde{R}_\cS^{-1})^h
(\sharp Z) r^h \le (2KC\tilde{R}_\cS^{-1})^h \Ga r^h.
\end{split}\end{equation*}
Along with \eqref{1j95} this completes the proof of
\eqref{h-regularity}. The remaining conclusions are easy consequences
of the $h$-regularity of $(J_{\cS},d)$, see for example
\cite[Theorem 5.7]{mat:geometry}.
\end{proof}

The following is the main theorem of this section. Note that we do not
assume in this theorem that the edge set $E$ is a finite set. Recall also that Bowen's parameter $h_\mathcal{S}$ is defined to be
$h_\mathcal{S} = \inf\{ t\geq 0: \,\, \P(t)\leq 0 \}$.

\begin{theorem}\label{t1j97}
If $\mathcal{S}$ is a finitely irreducible Carnot conformal GDMS, then
$$
h_\mathcal{S}
= \dim_{\mathcal{H}}(J_\mathcal{S})
= \sup \{\dim_\cH(J_F):  \, F \sbt E \, \mbox{finite} \, \}.
$$
\end{theorem}

\begin{proof}
Put $h_\infty=\sup \{\dim_\cH(J_F): \, F \sbt E \, \mbox{finite} \, \}$
and $H=\dim_\cH(J_\mathcal{S})$. Fix $t > h_\mathcal{S}$. Then $\P(t)< 0$ and for all $n\in \N$ large enough, we have
$$
Z_n(t)=\sum_{\om \in E^n_A}||D\phi_\om||_\infty^t \le \exp \lt(
\frac{1}{2}\P(t)n\rt).
$$
Hence by \eqref{4.1.9}
$$
\sum_{\om \in E^n_A} \left( \diam(\phi_\om(X_{t(\om)}) \right)^t
\le (\La\, M)^t \sum_{\om \in E^n_A}||D\phi_\om||_\infty^t
\le (\La\, M)^t \exp\lt( \frac{1}{2}\P(t)n\rt).
$$
Since the family $\{\phi_\om (X_{t(\om)})\}_{\om \in E^n_A}$
covers $J_\cS$, by \cite[Lemma 4.6]{mat:geometry}, we obtain that $\cH^t(J_\cS)=0$ upon letting $n \to
\infty$. This implies that $t \ge H$, and consequently, $h_\mathcal{S} \geq H$.
Since obviously $h_\infty \leq H$, we thus have
$$
h_\infty \leq H \leq h_\mathcal{S}.
$$
We are left to show that $h_\mathcal{S} \leq h_\infty$. If $F$  is a finite and irreducible
subset of $E$, then in virtue  of
Theorem~\ref{t1j93}, $h_F\leq h_\infty$, and in particular $\P_F(h_\infty)\leq 0$. So, by
Theorem~\ref{t2.1.3}, Remark \ref{supoverfinitir} and Lemma~\ref{l1j85}, we have
$$
\P(h_\infty)=\sup\{P_F(h_\infty): F\sbt E \, \mbox{finite and irreducible} \, \} \le
0.
$$
Hence $h_\infty \geq h_\mathcal{S}$ and the proof is complete.
\end{proof}

In the particular case where $\cS$ consists of metric similarities we get the following conclusion. Recall that the open set condition is a standing assumption in our definition of conformal GDMS.

\begin{corollary}
Let $(\G,d)$ be an arbitrary Carnot group equipped with a homogeneous metric $d$. Let $\cS=\{\f_e\}_{e \in E}$ be a Carnot iterated function system consisting of metric similarities, i.e. the contractions $\f_e$ satisfy the equation $$d(\f_e(p), \f_e(q))= r_{\f_e} d(p,q)$$ for all $p,q \in W_{t(e)}$, where $r_{\f_e}=\|D \f_e\|_\infty$ is the scaling factor of $\f_e$.
Then
$$
h = \dim_{\mathcal{H}}(J_\mathcal{S})= \inf\left\{ t\geq 0: \,\, \sum_{e \in E} \|D \f_e\|^t_\infty<1 \right\}.
$$
\end{corollary}

\section{A characterization of strongly regular GDMS}\label{strong-reg}

We will now prove the following characterization of strongly regular GDMS via their subsystems. This characterization will be employed in the study of dimension of Iwasawa continued fractions in Section \ref{CF:II}.

\begin{theorem}
\label{4310mu}
Let $\mathcal{S}=\{\f_e\}_{e \in E}$ be a finitely irreducible  Carnot conformal GDMS. Then the following conditions are equivalent.
\begin{itemize}
\item [(i)] $\cS$ is strongly regular.
\item [(ii)] $h_{\cS}>\theta_{\cS}$.
\item [(iii)] There exists a proper co-finite subsystem $\cS' \subset \cS$ such that $h_{S'}<h_{S}$.
\item [(iv)] For every proper subsystem $\cS' \subset \cS$ it holds that $h_{S'}<h_{S}$.
\end{itemize}
\end{theorem}
\begin{proof} The implications (iv)$\Rightarrow$(iii) and (ii)$\Rightarrow$(i) are immediate. In order to prove the implication (iii)$\Rightarrow$(ii) suppose by way of contradiction that $h_\cS=\theta_\cS$. Let $\cS'$ be a co-finite subsystem of $\cS$. By Lemma \ref{infcofi} we deduce that $h_{S'} \geq \theta_S$, hence by our assumption $$h_{S'} \geq \theta_S=h_\cS.$$
Therefore $h_{\cS'}=h_{\cS}$ for every co-finite subsystem $\cS' \subset \cS$, which contradicts (iii).

For the remaining implication (i)$\Rightarrow$(iv) let $E' \subset E$ and consider the corresponding proper subsystem of $\cS$,  $\cS'=\{\f_e\}_{e \in E'}$. If $\cS'$ is not regular then by Remark \ref{phleq0} $P_{\cS'}(h_{\cS'})<0$ and by Proposition \ref{p2j85} we deduce that also $P_{\cS'}(\theta_{\cS'})<0$. Since $\cS$ is strongly regular, Proposition \ref{f1j87} implies that there exists $\alpha \in (\theta_\cS, h_\cS)$. Therefore since $\alpha > \theta_S \geq \theta_{\cS'}$ and the pressure function is strictly decreasing %on $\{t\geq 0: P(t)<\infty\}$
we deduce that $P_{\cS'}(\alpha)<0$. Thus by the definition of the parameter $h_{\cS'}$, we get that $h_{\cS'} \leq \alpha <h_\cS$ and we are done in the case when $\cS'$ is  not regular.

Now by way of contradiction assume that $\cS'$ is regular and
$$h_\cS=h_{\cS'}:=h.$$
By Theorem \ref{thm-conformal-invariant} there exist unique measures $\tilde\mu_h$ on $E_A^\N$ and $\tilde\mu'_h$ on ${E'_A}^\N$,  which are  ergodic and shift invariant with respect to $\sigma:E_A^N \ra E_A^N$ and $\sigma':{E'_A}^\N\ra {E'_A}^\N$  respectively. Moreover, again by Theorem \ref{thm-conformal-invariant} we have that
\begin{equation}
m_h \ll \tilde\mu_h \ll m_h \text{ and }m'_h \ll \tilde\mu'_h \ll m'_h,
\end{equation}
where $m'_h$ stands for the $h$-conformal measure corresponding to $\cS'$. Notice also that if $\om \in {E'_A}^\ast$, then by \eqref{2j93}
$$m'_h([\om]) \approx \|D \f_\om\|^h_\infty \approx m_h([\om]),$$
therefore

\begin{equation}\label{equicyl}
\tilde\mu_h([\om])\approx\tilde\mu'_h([\om]).
\end{equation}
Now in the obvious way we can extend $\tilde\mu_h$ to a Borel measure in $E_A^\N$, defined by
$$\tilde\nu_h(B):=\tilde\mu'_h(B \cap {E'_A}^\N)$$
for $B \subset E_A^\N$. By \eqref{equicyl} we deduce that $\tilde\nu_h$ is absolutely continuous with respect to $\tilde\mu_h$.

We will now show that $\tilde\nu_h$ is shift invariant with respect to $\sigma:E_A^N \ra E_A^N$. Let $A$ be a Borel  subset of $E_A^N$. First notice that
\begin{equation*}
\begin{split}
\tilde\nu_h(\sigma^{-1}({E'_A}^\N))&=\tilde\nu_h\left(\bigcup_{j \in E}\{j\om: \om \in {E'_A}^\N\}\right)\\
&=\tilde\mu'_h \left(\bigcup_{j \in E}\{j\om: \om \in {E'_A}^\N\} \cap {E'_A}^\N\right)\\
&=\tilde\mu'_h({E'_A}^\N)=\tilde\nu_h({E'_A}^\N).
\end{split}
\end{equation*}
Therefore,
\begin{equation}
\label{shifinv1}
\begin{split}
\tilde\nu_h(\sigma^{-1}(A))&=\tilde\nu_h(\sigma^{-1}(A)\cap {E'_A}^\N)=\tilde\nu_h(\sigma^{-1}(A) \cap \sigma^{-1}({E'_A}^\N))\\
&=\tilde\mu'_h(\sigma^{-1}(A \cap {E'_A}^\N)\cap {E'_A}^\N).
\end{split}
\end{equation}
We will now show that
\begin{equation}
\label{shifinv2}
\sigma^{-1}(A \cap {E'_A}^\N)\cap {E'_A}^\N=\sigma'^{-1}(A \cap {E'_A}^\N).
\end{equation}
Recall that $\sigma'$ stands for the shift map in ${E'_A}^\N$. We have,
\begin{equation*}
\begin{split}
\sigma^{-1}(A \cap {E'_A}^\N)\cap {E'_A}^\N&=\bigcup_{j \in E}\{j\om: \om \in A \cap {E'_A}^\N\}\cap {E'_A}^\N\\
&=\bigcup_{j \in E'}\{j\om: \om \in A \cap {E'_A}^\N\}\cap {E'_A}^\N\\
&=\sigma'^{-1}(A \cap {E'_A}^\N),
\end{split}
\end{equation*}
and \eqref{shifinv2} follows. Now using \eqref{shifinv1}, \eqref{shifinv2} and the $\sigma'$-invariance of $\tilde \mu'_h$ we get,
\begin{equation*}
\begin{split}
\tilde\nu_h(\sigma^{-1}(A))&=\tilde\mu'_h(\sigma'^{-1}(A \cap {E'_A}^\N))\\
&=\mu'_h(A \cap {E'_A}^\N)\\
&=\tilde\nu_h(A).
\end{split}
\end{equation*}

We will now show that $\tilde\nu_h$ is ergodic with respect to $\sigma$. To do so by way of contradiction suppose that there exists some Borel subset $F$ of $E_A^\N$ such that $\sigma^{-1}(F)=F$ and $0<\tilde \nu_h (F)<1$. Let
$$F_1=F \cap {E'_A}^\N \text{ and }F_2=F \setminus F_1.$$
Since $\sigma'^{-1}(F_1) \subset \sigma^{-1}(F_1) \subset F_1 \cup F_2$ and $\sigma'^{-1}(F_1) \cap F_2=\emptyset$ we deduce that
\begin{equation}
\label{shifergo}
\sigma'^{-1}(F_1) \subset F_1.
\end{equation}
Moreover,
\begin{equation*}
\begin{split}
F_1 \subset \sigma^{-1}(F) \subset \bigcup_{j \in E'} \{jf: f \in F\} \cup \bigcup_{j \in E \stm E'} \{jf: f \in F\}.
\end{split}
\end{equation*}
Therefore
$$F_1 \subset \bigcup_{j \in E'} \{jf: f \in F_1\} \cap {E'_A}^\N= \sigma'^{-1}(F_1),$$
which combined with \eqref{shifergo} implies that
\begin{equation}
\label{shifergo1}
F_1=\sigma'^{-1}(F_1).
\end{equation}
But since $\tilde\mu'_h$ is ergodic with respect to $\sigma'$ we deduce that either $\tilde\mu'_h(F_1)=0$ or $\tilde\mu'_h(F_1)=1.$ Therefore, since $\tilde\nu_h(F)=\tilde\mu_h(F_1)$,
$$\tilde\nu_h(F)=0 \text{ or }\tilde\nu_h(F)=1,$$
and we have reached a contradiction. Thus $\tilde\nu_h$ is ergodic with respect to $\sigma$.

Hence we have shown that there exist two probability Borel measures on $E_A^n$, $\tilde \mu_h$ and $\tilde\nu_h$, which are shift invariant and ergodic with respect to $\sigma$ and they are absolutely continuous with respect to $m_h$. Now Theorem \ref{thm-conformal-invariant} implies that
\begin{equation}
\label{equivnumu}
\tilde \mu_h\equiv\tilde\nu_h.
\end{equation}
If $j \in E \stm E'$, then $\tilde\nu_h([j])=0$. On the other hand, because $\tilde\mu_h$ is equivalent to $m_h$, by \eqref{2j93}  $\tilde\mu_h([j])>0$. Therefore \eqref{equivnumu} cannot hold and we have reached a contradiction. The proof of the theorem is complete.
\end{proof}

\section{Dimension spectrum for subsystems of Carnot conformal IFS}\label{spec}

In this section we show that when $\cS$ is a Carnot conformal IFS, the spectrum of the Hausdorff dimensions of its subsystems is at least $(0, \theta_\cS)$. This theorem will be applied when we will revisit continued fractions on Iwasawa groups (see Section \ref{CF:II}).

We start with a lemma.

\begin{lemma}
\label{spectrumlemma} Let $\cS=\{\f_e\}_{e \in E}$ be a Carnot conformal IFS and let $F$ be a subset of $E$ such that $E \setminus F$ is infinite. Then for every $\e>0$ there exists some $e \in E \setminus F$ such that
$$\dim_{\cH}(J_{F \cup \{e\}}) \leq \dim_{\cH}(J_F)+ \e.$$
\end{lemma}

\begin{proof} 
Without loss of generality assume that $E=\N$. Let $\e>0$ and set $h=\dim_{\cH}(J_F)$. We record now that $P_F(h+\e)<0$. We will now show that there exists some $\alpha \in (0,1)$ ans some $j_0 \in \N$ such that
\begin{equation}
\label{spele1}
Z_{j}(F,h+\e)<\alpha^j \text{ for all } j\geq j_0.
\end{equation}
If \eqref{spele1} does not hold then for every $\alpha \in (0,1)$ there exists a sequence of natural numbers $(j_m)_{m \in \N}$ such that $Z_{j_m}(F,h+\e)\geq\alpha^{j_m}$ for all $m \in \N$. But this implies that for every $\alpha \in (0,1)$,
\begin{equation*}
\begin{split}
P_F(h + \e)=\lim_{n \ra \infty} \frac{1}{n}  \log Z_{n}(F,h+\e) \geq \limsup_{m \ra \infty} \frac{1}{j_m} \log Z_{j_m}(F,h+\e) \geq \log \alpha.
\end{split}
\end{equation*}
Hence $P_F(h + \e) \geq 0$ and we have reached a contradiction, thus \eqref{spele1} holds.

We now wish to estimate $Z_{n}(F \cup \{e\},h+\e)$ for $e \in \N \setminus F$. We have that
\begin{equation}
\label{spele2}
Z_{n}(F \cup \{e\},h+\e)= \sum_{\om \in (F \cup \{e\})^n} \|D \f_\om\|_\infty^{h+\e}=\sum_{j=0}^n \sum_{\om \in F_j} \|D \f_\om\|_\infty^{h+\e},
\end{equation}
where
$$F_j=\{\om \in (F \cup \{e\})^n: \om_i=e \text{ for } n-j\,\, i\text{'s}\}.$$
Now by \eqref{quasi-multiplicativity1} and \eqref{spele1} we get that for $n>j_0$
\begin{equation*}
\begin{split}
\sum_{j=0}^n\sum_{\om \in F_j} \|D \f_\om\|_\infty^{h+\e} &\leq \|D \f_e\|_\infty^{n(h+\e)}+\sum_{j=1}^n {n \choose j} (K\|D \f_e\|_\infty)^{(n-j)(h+\e)}  \sum_{\om \in F^j} \|D \f_\om\|_\infty^{h+\e}\\
&=\|D \f_e\|_\infty^{n(h+\e)}+\sum_{j=1}^n {n \choose j} (K\|D \f_e\|_\infty)^{(n-j)(h+\e)} Z_j(F, h+\e)\\
&\leq \|D \f_e\|_\infty^{n(h+\e)}+\sum_{j=1}^{j_0} {n \choose j} (K\|D \f_e\|_\infty)^{(n-j)(h+\e)} Z_j(F, h+\e)\\
&\quad\quad+\sum_{j={j_0+1}}^n {n \choose j} (K\|D \f_e\|_\infty)^{(n-j)(h+\e)} \alpha^j \\
&=\|D \f_e\|_\infty^{n(h+\e)}+I_1+I_2,
\end{split}
\end{equation*}
where the last identity serves also as the definition of $I_1$ and $I_2$. Therefore by \eqref{spele2} we get that for $n> j_0$,
\begin{equation}
\label{spele3a}
Z_{n}(F \cup \{e\},h+\e) \leq \|D \f_e\|_\infty^{n(h+\e)}+I_1+I_2.
\end{equation}
For $I_2$ we have that,
\begin{equation}
\label{spele3}
I_2 \leq (\alpha + (K \|D \f_e\|)^{h+\e})^n.
\end{equation}
For $I_1$ we estimate,
$$I_1 \leq \|D \f_e\|_\infty^{(n-j_0)(h+\e)} n^{j_0} K^{n(h+\ve)} \sum_{j=0}^{j_0}Z_j(F, h+\e).$$

Since $P_F(h+\e)<\infty$, by Proposition \ref{p2j85} and \eqref{zmn} we deduce that $Z_j(F, h+\e)$ is finite for all $j \in \N$. Therefore,
\begin{equation}
\label{spele4}
I_1 \leq \|D \f_e\|_\infty^{(n-j_0)(h+\e)}\, n^{j_0} \,K^{n(h+\ve)}\, j_0 \max_{1 \leq j\leq j_0} Z_j(F, h+\e).
\end{equation}
Now notice that if $e$ and $n$ are chosen big enough, by Lemma \ref{limdiam}, \eqref{spele3a}, \eqref{spele3} and \eqref{spele4} we get that $Z_{n}(F \cup \{e\},h+\e)<1$. Therefore $P_{F \cup \{e\}}(h+\e) \leq 0$, and consequently by Theorem \ref{t1j97} $$\dim_\cH(J_{F \cup \{e\}}) \leq h + \e.$$ The proof of the lemma is complete.
\end{proof}

\begin{theorem}
\label{spectrum}
Let $\cS=\{\f_e\}_{e \in E}$ be a Carnot conformal IFS. Then for every $t \in (0, \theta_\cS)$ there exists a proper subsystem $\cS_t$ of $\cS$ such that $\dim_\cH(J_{\cS_t})=t.$
\end{theorem}
\begin{proof} Let $t \in (0, \theta_\cS)$. Again without loss of generality assume that $E=\N$. Let $E_1=\{1\}$, then trivially $\dim_\cH(J_{E_1})<t$. Now using Lemma \ref{spectrumlemma}  we can inductively construct a sequence of sets $\{E_n\}_{n \in \N}$ such that $\dim_\cH(J_{E_n})<t$ for all $n \in \N$. Suppose that $E_n$ has been constructed then by Lemma \ref{spectrumlemma}  choose the minimal $k_n \in \N$ such that $k_n > \max \{e: e  \in E_n\}$ and $\dim_\cH(J_{E_n \cup \{k_n\}})<t$. Then we choose $E_{n+1}=E_n \cup \{k_n\}$. Now let,
$$E_t = \bigcup_{n=1}^\infty E_n.$$
Notice that $E_t$ is infinite and by Theorem \ref{t1j97},
\begin{equation}
\label{spethe1}
\dim_{\cH}(J_{E_t})= \sup_{n \in \N} \{\dim_{\cH}(J_{E_n})\} \leq t.
\end{equation}
Now notice that $\N \setminus E_t$ is infinite. Because  if not, Theorem \ref{infcofi} and Theorem \ref{t1j97} imply that
$$\dim_\cH (J_{E_t}) \geq \theta_\cS>t,$$
and this contradicts \eqref{spethe1}. Now if $\dim_{\cH}(J_{E_t})=t$ we are done. If not, since $\N \setminus E_t$ is infinite, we can apply Lemma \ref{spectrumlemma} once more in order to find some $q \in \N \cap (k_n, k_{n+1})$ such that
$$\dim_\cH (J_{E_t \cup \{q\}})<t.$$
But in this case we also have that $\dim_\cH (J_{E_{n+1} \cup \{q\}})<t$ and this contradicts the minimality of $k_{n+1}$. Therefore we have reached a contradiction and the proof of the theorem follows.
\end{proof}

\chapter{Conformal measures and regularity of domains}\label{chap:geometric-properties}

In this chapter, under suitable additional hypotheses on the underlying
GDMS, we establish fundamental estimates for conformal measures which will play a crucial role in the subsequent chapters. Using these estimates, we show that under some mild assumptions the Hausdorff dimension of a conformal  GDMS in a Carnot group $(\G,d)$ is strictly less than the homogeneous dimension of $\G$.

\section{Bounding coding type and null boundary}\label{sec:conformal-measures}

As before, let $\mathcal{S} = \{ \phi_e \}_{e \in E}$ be a Carnot conformal  GDMS. Recall that any measure supported on $J_\cS$ which satisfies
$$
m_h := \^m_h\circ \pi^{-1}
$$
is called $h$-conformal. \index{conformal measure} Moreover recalling Definition \ref{fins} we say that $t \in \Fin(\cS)$ if $t \geq 0$ and $\sum_{e\in E}||D\phi_e||_\infty^t<+\infty.$

\begin{definition}\label{null_boundary}
Let $\cS=\{\phi_e\}_{e\in E}$ be a finitely irreducible conformal GDMS. If
$t\in\Fin(\cS)$, then the measure $m_t=\tilde{m}_t\circ\pi^{-1}$ is said to be of \textit{null boundary}\index{null boundary}
if
\begin{equation}\label{1_null_boundary}
m_t\(\phi_\om(X_{t(\om)})\cap \phi_\tau(X_{t(\tau)})\)=0
\end{equation}
whenever $\om$ and $\tau$ are two different $A$-admissible words of
the same length.
\end{definition}

Let us record the following immediate consequence of being of null
boundary. We record that any  $\om \in E_A^\N$  will be also called a \textit{code}\index{code}. If $x=\pi(\om)$, then we will say that $\om$ is a code of $x$.

\begin{remark}\label{o220130614}
If a finitely irreducible Carnot conformal  GDMS is of null boundary, then for every
$t\in\Fin(\cS)$, $m_t$ almost every point in $J_\cS$ has a unique code.
\end{remark}

\begin{remark}\label{null_boundary_rem}
If $m_t$ is of null boundary, then \eqref{1_null_boundary} holds for
all (not necessarily of the same length) incomparable $A$-admissible
words $\om$ and $\tau$.
\end{remark}

For the following definition recall that the matrix $\hat A$ was introduced in Definition \ref{maximalmatrixdef}. Any infinite word $\om \in E_{\hat A}^\N$  will be called a \textit{pseudocode}\index{pseudocode} (from the vantage point of the original system $\cS$); frequently also all elements of $E_{\hat A}^*$ will be called pseudocodes. In particular each element of $E_A^*\cup E_A^\N$ is a pseudocode.

\begin{definition}
Let $\cS$ be a graph directed (not necessarily conformal) Markov system.
Given $q\ge 1$, we say that two different words $\rho,\tau\in E_{\hat A}^*$, i.e two different pseudocodes,
of the same length, say $n\ge q$, form a {\it pair of $q$-pseudocodes}
at a point $x\in X$ if
$$
x\in \phi_\rho(X_{t(\rho)})\cap \phi_\tau(X_{t(\tau)})
$$
and
$$
\rho|_{n-q}=\tau|_{n-q}.
$$
The graph directed Markov systems $\cS$ is said to be of
{\it bounded coding type} \index{bounded coding type} \index{GDMS!of bounded coding type} if for every $q\ge 1$ there is no point in
$X$ (or equivalently in $J_\cS$) with arbitrarily long pairs of
$q$-pseudocodes.
\end{definition}

It will turn out that each conformal system with a mild
boundary regularity condition is of bounded coding type.

\begin{theorem}\lab{t3.1.7}
Suppose that a finitely irreducible Carnot conformal GDMS $\cS=\{\phi_e\}_{e\in E}$ is
of bounded coding type. If $t\in\Fin(\cS)$, then the measure $m_t$ is of null
boundary. \index{null boundary}
\end{theorem}

\begin{proof} 
Suppose on the contrary that
$$
m_t(\f_\rho(X_{t(\rho)})\cap \f_\tau(X_{t(\tau)}))>0
$$
for two different words $\rho,\tau\in E_A^*$ of the same length, say
$q\ge 1$, for which $i(\rho)=i(\tau)=v$ with some $v\in V$. This
equivalently means that
$$
\mu_t(\f_\rho(X_{t(\rho)})\cap \f_\tau(X_{t(\tau)}))>0,
$$
where $\mu_t$ was defined in \eqref{mutdef}. Let $E:=\f_\rho(X_{t(\rho)})\cap \f_\tau(X_{t(\tau)})$, and, for
every $n\in \N$, let $E_n:=\sg^{-n}(\pi^{-1}(E))$.
Note that each element of $\pi(E_n)$ admits at least two different
$q$-pseudocodes of length $n+q$. To see this let $x=\pi(\om) \in \pi(E_n)$, then $\sigma^n(\om) \in \pi^{-1}(E)$ hence
$$x \in \f_{\om |_n\rho}(X_{t(\rho)})\cap \f_{\om |_n\tau}(X_{t(\tau)}).$$
Since $|\om |_n\rho|=|\om |_n\tau|=n+q$, $x$ admits at least two $q$-pseudocodes of length $n+q$. Therefore since $\cS$ is of bounded coding type we conclude that
$$
\bigcap_{k=1}^\infty\bigcup_{n=k}^\infty \pi(E_n)=\es.
$$
Hence
$$
E_\infty:=\bigcap_{k=1}^\infty\bigcup_{n=k}^\infty E_n
\sbt \pi^{-1}\lt(\bigcap_{k=1}^\infty\bigcup_{n=k}^\infty \pi(E_n)\rt)
=\es
$$
and
$$
\tilde\mu_t(E_\infty)=0.
$$
On the other hand, by Theorem \ref{thm-conformal-invariant} $\tilde\mu_t$ is $\sg$-invariant, thus we have
$\tilde\mu_t(E_n)=\tilde\mu_t(\pi^{-1}(E))=\mu_t(E)$ and
$$
\tilde\mu_t(E_\infty)\ge \mu_t(E)>0.
$$
This contradiction finishes the proof.
\end{proof}

As an immediate consequence of Theorem \ref{t3.1.7} and
Remark~\ref{o220130614}, we get the following.

\begin{corollary}\label{c320130614}
If $\cS=\{\phi_e \}_{e\in E}$ is a finitely irreducible Carnot conformal GDMS
of bounded coding type, then for every $t\in\Fin(\cS)$, $m_t$ almost
every point in $J_\cS$ has a unique code.
\end{corollary}

\section{Regularity properties of domains in Carnot groups}\label{sec:regularity-properties}

In this section, we indicate geometric conditions on domains in
Iwasawa groups which suffice for the application of the results in
prior sections.

\begin{definition}\label{WCC}
An open subset $\Om$ of a Carnot group \index{weak corkscrew condition}
$\G$ satisfies the {\it weak corkscrew condition (WCC)} if there
exists $\a\in (0,1]$ such that for every $p\in \bd \Om$ and every $r>0$
sufficiently small (smallness possibly depending on $p$), there exists $q\in\G$ such that
$$
B(q,\a r)\sbt \Om\cap B(p,r).
$$
We say that $\Om$ satisfies the \textit{corkscrew condition (CC)} \index{corkscrew condition} if a common
$\a>0$ can be taken for all sufficiently small $r>0$ and all $p\in\bd
\Om$ (smallness of $r$ independent of $p$). We also say that $\ov \Om$, the closure of $\Om$, satisfies the weak
corkscrew condition or the corkscrew condition respectively if $\Om$ does.
\end{definition}

\begin{example}\label{wccballs}
Every $C^{1,1}$ domain in an Iwasawa group $\G$ satisfies the
corkscrew condition. This is proved in \cite[Theorem 14]{cg:nta} for
the Carnot--Carath\'eodory metric on general step two Carnot groups; its validity for
the gauge metric follows from the comparability of these two metrics.
\end{example}

Any number $\a\in (0,1]$ satisfying the property in Definition
\ref{WCC} is called a \index{corkscrew constant}
{\it corkscrew constant} of $\Om$. We denote by $\a_\Omega$ the supremum of
all such constants. As a consequence of Corollary~\ref{koebe2}
we obtain the following.

\begin{proposition}\label{p1_WCC_Conformal_inv}
Let $U$ be an open connected subset in a Carnot group $\G$. Let $\Om$
be a bounded open subset of $U$ satisfying the weak corkscrew
condition. Assume that $\dist(\Om, \bd U)>0$ and let $S$ be a compact set such that $\Om \subset S \subset U$. Let $K\ge 1$ be the
distortion constant associated to $S$ as in Lemma~\ref{harnack}. If $\phi:U\to\ \Om$ is a
conformal homeomorphism, then $\phi(\Om)$ also satisfies the weak
corkscrew condition, moreover,
$$
\a_{\phi(\Om)}\ge K^{-1}C^{-2}\alpha_\Om.
$$
\end{proposition}
\begin{proof} Let $p'=\f(p) \in \Om$. Set $R_0=\sup \{ R>0: B(p',R) \subset \f(\Om)\}$. Since $\Om$ satisfies the WCC, there exists some $r_0(p)$ such that for all $r<r_0(p)$, there exists some $q_{r,p}$ such that $B(q_{r,p}, \a_\Om r) \subset B(p,r) \cap \Om$. Now set
$$r_1(p)= \min \left\{ \frac{\dist(\Om, \partial U)}{ 4L}, \frac{r_0(p)}{2}, \frac{R_0}{C \|D \f\|_\infty}\right\},$$
and
$$R_1(p)=C  \|D \f\|_\infty r_1(p).$$
Using the WCC for $p$ and $r\leq r_1(p)$, and applying Corollary \ref{koebe2} twice we get,
\begin{equation}
\label{wccpropo1}
\begin{split}
B(\f(q), \a_\Om K^{-1} C^{-1} \|Df\|_\infty r)& \subset \f(B(q, \a_\Om r)) \\ &\subset  \f(B(p, r)) \\
&\subset B(\f(p), C \|D \f_\om\|_\infty r) \\
&\subset B(\f(p),R_0) \subset \f(\Om).
\end{split}
\end{equation}
Let $R \leq R_1(p)$. Then there exists some $r \leq r_1(p)$ such that $R= C \|D \f\|_\infty r$. Hence by \eqref{wccpropo1},
\begin{equation*}
B(\f(q), \a_\Om K^{-1}C^{-2} R)=B(\f(q), \a_\Om K^{-1}C^{-1} \|Df\|_\infty r) \subset B(\f(q), R) \subset \f(\Om).
\end{equation*}
Therefore $\f(\Om)$ satisfies the WCC and $\a_{\f(\Om)} \geq \a_\Om K^{-1}C^{-2}$. The proof of the proposition is finished.
\end{proof}

\begin{definition}\label{GDMS_WCC}
A Carnot conformal GDMS $\cS=\{\phi_e\}_{e\in E}$ satisfies the {\it weak corkscrew condition
(WCC)} if every set $X_v$, $v\in V$, satisfies such condition. We
then put
$$
\a_\cS:=\min\{\a_{X_v}:v\in V\}.
$$
\end{definition}

\begin{corollary}\lab{c3101804}
If $\cS$ is a Carnot conformal GDMS satisfying the
weak corkscrew condition, then for each point $x\in X$ there are at
most $K^{Q}C^{2Q}a_\cS^{-Q}$ pseudocodes of $x$.
\end{corollary}

\begin{proof}
Suppose $p\in X$ admits $N>1$ pseudocodes. Then there exist mutually incomparable words $\tau_1,\tau_2,\ld,\tau_N \in E_{\hat A}^\ast$. Then for each $1\le j\le
N$, $x\in\phi_{\tau_j}\(X_{t(\tau_j)}\)$. Since $N>1$, we have $p\in
\bd X$, and therefore for each $1\le j\le N$ there exists a point
$p_j\in \bd X_{t(\tau_j)}$ such that
$$
p=\phi_{\tau_j}(p_j).
$$
By our hypotheses and by Proposition~\ref{p1_WCC_Conformal_inv},
each set $\phi_{\tau_j}\(X_{t(\tau_j)}\)$ satisfies the weak corkscrew
condition and
$$
\a_{\phi_{\tau_j}(\Int(X_{t(\tau_j)}))}\ge K^{-1}C^{-2}\a_\cS.
$$
Hence for all $1\le j \le N$ and all sufficiently small $r>0$,
each set $\phi_{\tau_j}(\Int(X_{t(\tau_j)})) \cap B(p,r)$ contains a ball
$B(q_j,K^{-1}C^{-2}\a_\cS r)$. Due to the mutual incomparability of the words
$\tau_1,\tau_2,\ld,\tau_N$, all of these balls are pairwise disjoint.
Thus
$$
c_0r^Q = |B(p,r)|
\ge \sum_{j=1}^N |B(q_j,K^{-1}C^{-2}\a_\cS r)|
\ge Nc_0(K^{-1}C^{-2}\a_\cS r)^Q
$$
and the proof is complete.
\end{proof}

Arguing exactly as in Proposition \ref{diamgtr1} we can show that if a conformal GDMS is finite and irreducible then the conclusion of Corollary \ref{c3101804} holds without the weak corkscrew condition.

\begin{remark}
\label{finitepseudocodes}
If $\cS$ is a finite and irreducible Carnot conformal GDMS, then each point $x\in X$ has at most $\mathtt{m}_{\kappa_1,\kappa_2}$ pseudocodes. See Lemma \ref{l1j81} and the proof of Proposition \ref{diamgtr1}, especially \eqref{cardfr}, for the definition of $\mathtt{m}_{\kappa_1,\kappa_2}$.
 \end{remark}

 The following proposition is the main result of this section.

\begin{proposition}\lab{l1101804}
Every Carnot conformal GDMS ${\cS}$ satisfying the
weak corkscrew condition is of bounded coding type.
\end{proposition}

\begin{proof}
Suppose to the contrary that there exists a point $p\in X$ having, for
some $q\ge 1$, arbitrarily long pairs of $q$-pseudocodes. This means
that for each $k\in \N$ there exist finite words $\om^{(k)}\in E_{
A}^*$ and $\tau^{(k)}, \rho^{(k)}\in E_{ A}^q$, such that
$\tau^{(k)}$ and $\rho^{(k)}$ are different,
\begin{equation}\lab{1111204}
\lim_{k\to\infty}|\om^{(k)}|= \infty,
\end{equation}
and
$$
p \in \phi_{\om^{(k)}}\circ\phi_{\tau^{(k)}}\(X_{t(\tau^{(k)})}\)\cap
       \phi_{\om^{(k)}}\circ\phi_{\rho^{(k)}}\(X_{t(\rho^{(k)})}\)
$$
for all $k \in \N$. We now construct by induction for each $n \in \N$
a set $C_n$ which contains at least $n+1$ mutually incomparable
pseudocodes of $p$. The existence of such a set for large $n$ will
contradict the statement of Corollary~\ref{c3101804} and hence will
complete the proof.

Set
$$
C_1:=\{\om^{(1)}\tau^{(1)}, \om^{(1)}\rho^{(1)}\}
$$
and suppose that the set $C_n$ has been defined for some $n \in \N$.
In view of
(\ref{1111204}), there exists $k_n \in \N$ such that
\begin{equation}\lab{2111204}
|\om^{(k_n)}|>\max\{|\xi|:\xi\in C_n\}.
\end{equation}
If $\om^{(k_n)}\rho^{(k_n)}$ does not extend any word from $C_n$,
then by (\ref{2111204}) the word $\om^{(k_n)}\rho^{(k_n)}$
is not comparable
with any element of $C_n$. We obtain $C_{n+1}$
from $C_{n}$ by adding the word $\om^{(k_n)}\rho^{(k_n)}$ to $C_n$.
Similarly, if $\om^{(k_n)}\tau^{(k_n)}$ does not extend any
word from $C_n$,
then $C_{n+1}$ is formed by adding $\om^{(k_n)}\tau^{(k_n)}$ to $C_n$.
On the other hand, if $\om^{(k_n)}\rho^{(k_n)}$ extends an element
$\a\in C_n$ and $\om^{(k_n)}\tau^{(k_n)}$ extends an element $\b
\in C_n$, then we obtain from (\ref{2111204}) that
$\a=\om^{(k_n)}|_{|\a|}$ and $\b=\om^{(k_n)}|_{|\b|}$. Since $C_n$
consists of mutually incomparable words, this implies that $\a=\b$.
Now, form $C_{n+1}$ by first removing $\a=\b$ from $C_n$ and
then adding both
$\om^{(k_n)}\rho^{(k_n)}$ and $\om^{(k_n)}\tau^{(k_n)}$. Note that
no element $\g\in C_n\sms\{\a\}$ is comparable with
  $\om^{(k_n)}\rho^{(k_n)}$ or $\om^{(k_n)}\tau^{(k_n)}$, since
otherwise $\g=\om^{(k_n)}|_{|\g|}$, and consequently, $\g$ would be
comparable with $\a$. Since also $\om^{(k_n)}\rho^{(k_n)}$ and
$\om^{(k_n)}\tau^{(k_n)}$ are incomparable, it follows that
$C_{n+1}$ consists of mutually incomparable pseudocodes of $p$. This
completes the inductive construction, and hence finishes the proof.
\end{proof}

As an immediate consequence of this proposition and respectively
Theorem~\ref{t3.1.7} and Corollary~\ref{c320130614}, we get the
following results.

\begin{theorem}\lab{t420130613}
If $\cS=\{\phi_e\}_{e\in E}$ is a finitely irreducible Carnot conformal GDMS
satisfying the weak corkscrew condition, then for every
$t\in\Fin(\cS)$, the measure $m_t$ is of null
boundary.
\end{theorem}

\begin{corollary}\label{c520130614}
If $\cS=\{\phi_e\}_{e\in E}$ is a finitely irreducible Carnot conformal GDMS
satisfying the weak corkscrew condition, then for every
$t\in\Fin(\cS)$, $m_t$ almost every point in $J_\cS$ has a unique
code.
\end{corollary}

\section{Conformal measure estimates}\label{sec:regularity-properties}

In this section we prove several estimates for the conformal measure of a finitely irreducible Carnot conformal GDMS $\cS$ when it satisfies some of the conditions discussed in the previous sections.

\begin{lemma}
\label{820lem}
Let $\cS=\{\phi_e:e\in E\}$ be a finitely irreducible Carnot conformal GDMS. Let also let $t\in\Fin(\cS)$ and $\om\in E_A^*$. Then for every Borel set $F\sbt \Int( X_{t(\om)})$
\begin{equation}
\label{seteq}
\pi^{-1}( \f_\om(F))=\{\tau \in E_A^\N: \tau|_{|\om|}=\om \mbox{ and }\pi(\sg^{|\om|}(\tau)) \in F \}.
\end{equation}
If moreover $m_t$ is of null boundary, then for every Borel set $F\sbt  X_{t(\om)}$
\begin{equation}
\label{measeq}
\^m_t(\pi^{-1}( \f_\om(F)))=\^m_t(\{\tau \in E_A^\N: \tau|_{|\om|}=\om \mbox{ and }\pi(\sg^{|\om|}(\tau)) \in F \}).
\end{equation}
\end{lemma}

\begin{proof} For ease of notation let $$A_\om:=\pi^{-1}( \f_\om(F))=\{ \tau \in E_A^\N: \pi(\tau) \in \f_\om(F)\}$$ and $$B_\om=\{\tau \in E_A^\N: \tau|_{|\om|}=\om \mbox{ and }\pi(\sg^{|\om|}(\tau)) \in F\}.$$ First notice that since for every $\tau \in B_\om$, $\pi(\tau)= \f_\om (\pi(\sg^{|\om|}(\tau)))$ and $\pi(\sg^{|\om|}(\tau)) \in F$ we have that $B_\om \subset A_\om$. Moreover $$\{ \tau \in E_A^\N: \pi(\tau) \in \f_\om(F), \tau|_{|\om|}=\om \mbox{ and }\pi(\sg^{|\om|}(\tau))\notin F\}=\es,$$ hence
\begin{equation}
\label{aminb}
A_\om \stm B_\om \subset \{\tau \in E_A^\N:\pi(\tau) \in \f_\om(F) \mbox{ and }\tau|_{|\om|}\neq\om\}:=C_\om.
\end{equation}
Now if $\tau \in C_\om$ then $\pi(\sg^{|\om|}(\tau)) \in X_{t(\tau_{|\om|})}$, so $\pi(\tau)\in \f_{\tau|_{|\om|}}(X_{t(\tau_{|\om|})})$ and $\tau|_{|\om|}\neq\om$. Therefore
\begin{equation}
\label{star18b}
\pi(C_\om) \subset \bigcup_{\{\tau \in E_A^{|\om|}:\tau \neq \om\}} \f_{\tau}(X_{t(\tau)}) \cap \f_\om(F),
\end{equation}
and if $m_t$ is of null boundary we deduce that $\^m_t(C_\om)=0$ and \eqref{measeq} follows.

Let $F \subset \Int(X_{t(\om)})$. %For every $n \leq |\om|$ , $\f_\om (F) \subset \f_{\om|_n}(\Int(X_{t(\om_n)}))$. %Moreover If $\tau \in C_\om$ and $n \leq |\om|$, $\pi(\tau) \in \f_{\tau|_n}(X_{t(\tau_n)}).$
Then as in \eqref{star18b},
%\begin{equation}
%\label{doubleunion}
%\pi(C_\om) \subset \bigcup_{n \leq |\om|} \bigcup_{\substack{\tau \in E_A^n \\ \tau|_{n-1}=\om|_{n-1},\tau_n \neq \om_n}} \f_\tau(X_{t(\tau)}) \cap \f_{\om|_n}(\Int(X_{t(\om_n)})).
%\end{equation}
\begin{equation}
\label{doubleunion}
\pi(C_\om) \subset \bigcup_{\{\tau \in E_A^{|\om|}:\tau \neq \om\}} \f_{\tau}(X_{t(\tau)}) \cap \f_\om(\Int(X_{t(\om)})).
\end{equation}

By the open set condition, and recalling that the maps $\f_e$ are homeomorphisms, we get that for all $\tau \in E_A^{|\om|},\, \tau \neq \om$,
\begin{equation}
\label{oschomeo}
\f_\tau(X_{t(\tau)}) \cap \f_{\om}(\Int(X_{t(\om)}))=
\overline{ \f_{\tau}(\Int(X_{\tau(\tau)}))} \cap \f_{\om}(\Int(X_{t(\om)})) = \emptyset.
\end{equation}
Combining \eqref{doubleunion} and \eqref{oschomeo} we get that $C_\om = \emptyset$, which implies \eqref{seteq}.
\end{proof}

\begin{proposition}\label{p320130613}
Let $\cS=\{\phi_e\}_{e\in
E}$ be a finitely irreducible Carnot conformal GDMS. If $t\in\Fin(\cS)$ and $m_t$ is of null boundary then
for every $\om\in E_A^*$ and every Borel set $F\sbt X_{t(\om)}$ we
have
\begin{equation}\label{m-t-upper-bound}
m_t(\phi_\om(F))\le e^{-\P(t)|\om|}||D\phi_\om||_\infty^tm_t(F).
\end{equation}
\end{proposition}

\begin{proof}
Fix $\om\in E_A^*$ and a Borel set $F\sbt X_{t(\om)}$. Denote
$k:=|\om|$. Then, by Theorem~\ref{t420130613} and Lemma \ref{820lem}, in particular \eqref{measeq}, we have
\begin{equation}\label{1201306181}
\aligned
m_t(\phi_\om(F))
&=\tilde m_t\circ\pi^{-1}(\phi_\om(F))\\
&=\^m_t(\{ \tau \in E_A^\N: \pi(\tau) \in \f_\om(F)\})\\
&=\^m_t(\{\tau \in E_A^\N: \tau|_{|\om|}=\om \mbox{ and }\pi(\sg^{|\om|}(\tau)) \in F \})\\
 &=\tilde m_t\([\om]\cap\sg^{-k}(\pi^{-1}(F))\).
\endaligned
\end{equation}
Hence by \eqref{mtdual} and \eqref{2j91}
\beq\label{120130618}
\aligned
m_t(\phi_\om(F))
&=e^{-\P(t)k}\mathcal{L}_t^{*k}\tilde m_t(\1_{[\om]\cap\sg^{-k}(\pi^{-1}(F))})\\
& =e^{-\P(t)k}\int_{E^\mathbb{N}_A}\sum_{\tau \in E^k_A: \,
A_{\tau_k\rho_1=1}}||D\phi_\tau(\pi (\rho))||^t
 \1_{[\om] \cap\sg^{-k}(\pi^{-1}(F)}( \tau \rho) \, d\^m_t(\rho)\\
& =e^{-\P(t)k}\int_{\{\rho\in \pi^{-1}(F): A_{\om_k \rho_1}=1\}}
   ||D\phi_\om(\pi (\rho))||^t \, d\^m_t(\rho) \\
&\le e^{-\P(t)k}||D\phi_\om||_\infty^t
      \tilde m_t\(\{\rho\in \pi^{-1}(F): A_{\om_k \rho_1}=1\}\)\\
&\le e^{-\P(t)k}||D\phi_\om||_\infty^t\tilde m_t\(\pi^{-1}(F)\)
=   e^{-\P(t)k}||D\phi_\om||_\infty^tm_t(F).
\endaligned
\eeq
The proof is complete.
\end{proof}

Note that in the previous proof the fact that $m_t$ has null boundary was only used in the third equality of \eqref{1201306181}. Therefore if $F\sbt \Int( X_{t(\om)})$ using \eqref{seteq} we can prove \eqref{m-t-upper-bound}  without assuming the weak corkscrew condition. We state this observation in the following remark.

\begin{remark}
\label{820rem}
Let $\cS=\{\phi_e\}_{e\in
E}$ be a finitely irreducible Carnot conformal GDMS. If $t\in\Fin(\cS)$, then
for every $\om\in E_A^*$, \eqref{m-t-upper-bound} holds for every Borel set $F\sbt \Int( X_{t(\om)})$.
\end{remark}

A lower bound corresponding to \eqref{m-t-upper-bound} is given in
Proposition \ref{p120130617}. First, we need the following lemma and some new definitions.

\begin{lemma}\label{l220130617}
Let $\cS=\{\phi_e\}_{e\in
E}$ be a finitely irreducible Carnot conformal GDMS
of bounded coding type (for example satisfying the weak corkscrew
condition). If $t\in\Fin(\cS)$, then
$$
\tilde m_t(\pi^{-1}(A)\cap B)=m_t\(A\cap\pi(B)\)
$$
for every Borel set $A\sbt X$ and every Borel set $B\sbt
E_A^\N$. In particular (taking $A=J_\cS$),
$$
\tilde m_t(B)= m_t(\pi(B)).
$$
\end{lemma}

\begin{proof}
Assume that $\om\in \pi^{-1}(\pi(B))\sms B$. Then there must exist
$\tau\in B$ such that $\pi(\tau)=\pi(\om)$. Hence $\tau\ne\om$, and let
$k\in \N$ be the least integer such that $\tau_k\ne\om_k$. So,
$$
\om\in\pi^{-1}\(\phi_{\om|_k}(X_{t(\om_k)})\cap\phi_{\tau|_k}(X_{t(\tau_k)})\).
$$
In conclusion,
$$
\pi^{-1}(\pi(B))\sms B\sbt
\pi^{-1}\lt(\bu_{k=1}^\infty\bu_{|\om|=|\tau|=k\atop\om\ne\tau}
\phi_\om(X_{t(\om_k)})\cap\phi_{\tau}(X_{t(\tau_k)})\rt).
$$
Since $m_t$ is of null boundary (see Theorem~\ref{t3.1.7}), we get
$$
\tilde m_t\(\pi^{-1}(\pi(B))\sms B\)=0.
$$
Since also $B\sbt \pi^{-1}(\pi(B))$, we therefore get
$$
m_t\(A\cap\pi(B)\)
=\tilde m_t\(\pi^{-1}(A\cap\pi(B))\)
=\tilde m_t\(\pi^{-1}(A)\cap\pi^{-1}(\pi(B))\)
=\tilde m_t\(\pi^{-1}(A)\cap B\)
$$
and the proof is complete.
\end{proof}

For each $a\in E$, let $E_a^-:=\{b\in E: A_{ab}=1\}$ and
$E_a^\infty:=\{\om\in E_A^\N: \om_1\in E_a^-\}$.

\begin{definition}
\label{sscdef} A Carnot conformal GDMS $\cS=\{\f_e:e\in E\}$ satisfies the {\it strong separation condition}\index{strong separation condition} (SSC) if $\overline{\pi(E^\infty_{e})} \cap \overline{\pi(E^\infty_{i})}=\emptyset$ for all $e \neq i \in E$.
\end{definition}

\begin{remark}
Each system satisfying the strong separation condition is of bounded
coding type.
\end{remark}

\begin{proposition}\label{p120130617}
Suppose that $\cS=\{\phi_e\}_{e\in E}$ is a finitely irreducible Carnot conformal GDMS of bounded coding type (for example satisfying the weak corkscrew
condition). If $t\in\Fin(\cS)$, then for every
$\om\in E_A^*$ and
every Borel set $F\sbt X_{t(\om)}$, we have
$$
m_t(\phi_\om(F))\ge
K^{-t}e^{-\P(t)|\om|}||D\phi_\om||_\infty^tm_t\(F\cap\pi\(E_{\om_{|\om|}}^\infty\)\).
$$
\end{proposition}

\begin{proof}
Put $k:=|\om|$. Starting with the third line of \eqref{120130618} and
using Lemma~\ref{l220130617} we continue as follows:
$$
\aligned
m_t(\phi_\om(F))
& =e^{-\P(t)k}\int_{\{\rho\in \pi^{-1}(F): A_{\om_k \rho_1}=1\}}
   ||D\phi_\om(\pi (\rho))||^t \, d\^m_t \\
&\ge K^{-t}e^{-\P(t)k}||D\phi_\om||_\infty^t
      \tilde m_t\(\{\rho\in \pi^{-1}(F): A_{\om_k \rho_1}=1\}\) \\
&= K^{-t}e^{-\P(t)k}||D\phi_\om||_\infty^t\tilde m_t\(\pi^{-1}(F)\cap E_{\om_k}^\infty\)\\
&= K^{-t}e^{-\P(t)k}||D\phi_\om||_\infty^tm_t\(F\cap\pi(E_{\om_k}^\infty)\).
\endaligned
$$
The proof is complete.
\end{proof}

\begin{proposition}\label{p220130618}
Suppose that $\cS=\{\phi_e\}_{e\in E}$ is a finitely irreducible Carnot conformal GDMS of bounded coding type (for example satisfying the weak corkscrew
condition). If $t\in\Fin(\cS)$, then
$M_t:=\inf\lt\{m_t(\pi(E_e^\infty)):e\in E\rt\}>0$.
\end{proposition}

\begin{proof}
Let $\Phi\sbt E_A^*$ be a finite set witnessing finite irreducibility
of the incidence matrix $A$. Let
$$
\Phi_1:=\{\om_1:\om\in F\}.
$$
By \eqref{2j93}, we have
$$
\g=\min\{\tilde m_t([e]):e\in \Phi_1\}>0.
$$
By the definition of $\Phi$, for every $e\in E$ there exists
$b\in \Phi_1$ such that $A_{eb}=1$, which means that $b\in
E_e^{-}$. Hence, $E_e^{\infty}\spt [b]$. Thus, using
Lemma~\ref{l220130617}, we get
$$
m_t\(\pi(E_e^\infty)\)=\tilde m_t(E_e^\infty)\ge \g>0.
$$
The proof is complete.
\end{proof}

\begin{corollary}\label{c120130621}
Suppose that $\cS=\{\phi_e\}_{e\in E}$ is a finitely irreducible
maximal Carnot conformal GDMS of bounded coding type (for example satisfying
the weak corkscrew condition). If $t\in\Fin(\cS)$, then for every
$\om\in E_A^*$ and every Borel set $F\sbt J_\cS \cap X_{t(\om)}$, we have
$$
m_t(\phi_\om(F))\ge
K^{-t}e^{-\P(t)|\om|}||D\phi_\om||_\infty^tm_t(F).
$$
\end{corollary}
\begin{proof} In view Proposition~\ref{p120130617} it suffices to show that for every $\om\in E_A^*$ and every Borel set $F\sbt J_\cS \cap X_{t(\om)}$
\begin{equation}
\label{maxf}
F \cap \pi(J^-_{\om_{|\om|}})=F.
\end{equation}
By way of contradiction suppose that there exists some $x \in F \setminus \pi(J^-_{\om |_{|\om|}})$.
By maximality,
\begin{equation}
\label{max2}
\begin{split}
J^-_{\om |_{|\om|}}&=\{\tau \in E_A^\N :A_{\om_{|\om|}\tau_1 }=1 \}=\{\tau \in E_A^\N :t(\om_{|\om|}) =i(\tau_1)\}\\
&=\{\tau \in E_A^\N :\f_{\tau_1}(X_{t(\tau_1)}) \subset X_{t(\om)}\}.
\end{split}
\end{equation}
If $x= \pi(\tau)$ for some $\tau \in E_A^\N$, by \eqref{max2} $$\f_{\tau_1}(X_{t(\tau_1)}) \cap X_{t(\om)}=\emptyset,$$
and we have reached a contradiction  because $x \in F \subset X_{t(\om)}$ and trivially $x \in \f_{\tau_1}(X_{t(\tau_1)})$. The proof of the corollary is complete.
\end{proof}
\begin{remark}
\label{cor827rem}
Notice that for $\cS$ and $t$ as above, if $\om \in E^\ast_A$, $x \in X_{t(\om)}$ and $r< \eta_\cS$ then
\begin{equation*}
\begin{split}
m_t(\f_\om(B(x,r))) &\geq m_t(\f_\om(B(x,r) \cap J_\cS \cap X_{t(\om)})) \\
&\geq K^{-t}e^{-\P(t)|\om|}||D\phi_\om||_\infty^t m_t(B(x,r) \cap J_\cS \cap X_{t(\om)})\\
&=K^{-t}e^{-\P(t)|\om|}||D\phi_\om||_\infty^t m_t(B(x,r)),
\end{split}
\end{equation*}
where in the last equality we also used that the sets $S_v, v \in V,$ are disjoint.
\end{remark}

Before proving the next theorem which gives some very general bounds on the size of the limit set we need to introduce some notation. If $\cS=\{\phi_e\}_{e\in
E}$ is a GDMS, for $n \in \N$ let
$$X^n=\bigcup_{\om \in E_A^n} \f_\om(X_t(\om)).$$
%Let also $$\tilde{V}:=\{v \in V: \text{there exists some }e \in E \text{ such that }i(e)=v\},$$
%and set
%$$\tilde{X}= \bigcup_{v \in \tilde{V}}X_v.$$
\begin{theorem}
\label{dimbound} If $\cS=\{\phi_e \}_{e\in
E}$ is a finitely irreducible Carnot conformal GDMS such that $|\Int X \stm X^1|>0$ then
\begin{equation}
|J_\cS|=0.
\end{equation}
If moreover $\cS$ is regular and of bounded coding type then
\begin{equation}
\dim_\cH J_\cS < Q.
\end{equation}
\end{theorem}
\begin{proof}
For $n \in \N$ and $v \in V$, let
$$X^n_v:=X^n \cap X_v =\bigcup_{\om \in E_A^n:i(\om)=v} \f_\om(X_t(\om)).$$
Since $|\Int X \stm X^1|>0$ there exists some $v_0 \in V$  such that
\begin{equation}
\label{xvo}
|\Int X_{v_0} \stm X_{v_0}^1|>0.
\end{equation}
We will first show that there exists some $n_0 \in \N$ such that
\begin{equation}
\label{xv}
|\Int X_{v} \stm X_{v}^{n_0}|>0,
\end{equation}
for all $v \in V$.

Let $v \in V$.% If $v \in V \stm \tilde{V}$ then $X_{v}^{1}= \emptyset$, hence \eqref{xv} holds trivially for $n_0=1$.
and let $e, e_0 \in E$ such that $i(e)=v$ and $t(e_0)=v_0$. Since $\cS$ is finitely irreducible there exists some $\om \in E_A^\ast$, with $|\om| \leq m_0$ for some $m_0 \in \N$ depending only on $\cS$, such that $\om':=e\om e_0 \in E_A^\ast$. Therefore,
\begin{equation}
\label{poswprime}
|\f_{\om'}(\Int X_{v_0}) \stm \f_{\om'}(X^1_{v_0})|>0.
\end{equation}
Now if $|\om'|=k$ notice that
\begin{equation}
\label{posdif}
X^{k+1}_v \cap \f_{\om'}(\Int X_{v_0}) \subset \f_{\om'}(X^1_{v_0}).
\end{equation}
First observe that \eqref{posdif} makes sense because $v \in \tilde{V}$, hence  $X^1_{v_0} \neq \emptyset$. Moreover,
\begin{equation*}
\begin{split}
X_{v}^{k+1} &\subset \f_{\om'} \left( \bigcup_{e \in E:t(e)=v} \f_e (X_{t(e)})\right) \cup \bigcup_{\tau \in I_v} \f_\tau (X_{t(\tau)})\\
&=\f_{\om'}(X^1_{v_0})\cup \bigcup_{\tau \in I_v} \f_\tau (X_{t(\tau)}),
\end{split}
\end{equation*}
where $I_v=\{ \tau \in E_A^{k+1}:i(\tau)=v \text{ and }\tau|_k \neq \om'\}$. But for every $\tau \in I_v$ by the open set condition
$$\f_\tau (X_{t(\tau)}) \cap \f_{\om'}(\Int X_{v_0}) \subset \f_{\tau|_k} (X_{t(\tau_k)}) \cap \f_{\om'}(\Int X_{v_0})= \emptyset,$$
and \eqref{posdif} follows. Now using \eqref{poswprime} and \eqref{posdif} one can show that
\begin{equation}
\label{xvpre}|\Int X_v \setminus X_v^{k+1}| \geq |\f_{\om'}(\Int X_{v_0}) \stm \f_{\om'}(X^1_{v_0})|>0.
\end{equation}
To see \eqref{xvpre} notice that
$$\f_{\om'}(\Int(X_{t(\om')}))=\f_{\om'}(\Int(X_{t(e_0)}))\subset \Int X_{i(e)}=\Int X_v.$$
Hence by \eqref{posdif}
\begin{equation*}
\begin{split}
\Int X_v \setminus X_v^{k+1} &\supset \f_{\om'} (\Int X_{v_0}) \setminus X_v^{k+1} \\
&\supset \f_{\om'} (\Int X_{v_0}) \setminus \f_{\om'}(X^1_{v_0})= \f_{\om'}(\Int X_{v_0} \setminus X^1_{v_0}),
\end{split}
\end{equation*}
and \eqref{xvpre} follows by \eqref{poswprime}. Now  \eqref{xvpre} implies \eqref{xv} because $k=|\om'|\leq m_0+2$.

For $n_0$ as in \eqref{xv} set $G_v= \Int X_v \setminus X_v^{n_0}$ for $v \in V$. Then by Theorem \ref{analytic-Carnot-conformal} for every $\om \in E_A^\ast$,
\begin{equation}
\label{jacogv}
\frac{|\f_\om(G_{t(\om)})|}{|\f_\om(X_{t(\om)})|} \geq K^{-Q} \frac{|G_{t(\om)}|}{|X_{t(\om)}|} \geq K^{-Q}\gamma,
\end{equation}
where $\gamma:= \min_{v \in V}\frac{|G_{v}|}{|X_{v}|}>0$. because by \eqref{xv},  $|G_v|>0$ for all $v \in V$.

Let $n \in \N$, then
\begin{equation*}
\begin{split}
X^{n+n_0} &\subset \bigcup_{\om \in E_A^n} \f_\om( X^{n_0}_{t(\om)}) \subset \bigcup_{\om \in E_A^n}\f_\om(X_{t(\om)} \stm G_{t(\om)})\\
&=\bigcup_{\om \in E_A^n}\f_\om(X_{t(\om)}) \stm \bigcup_{\om \in E_A^n}\f_\om(G_{t(\om)})\\
&=X^n \stm \bigcup_{\om \in E_A^n}\f_\om(G_{t(\om)}).
\end{split}
\end{equation*}
Therefore by the open set condition and \eqref{jacogv} we have for every $n \in \N$
\begin{equation}
\begin{split}
\label{decrmeas}
|X^{n+n_0}| &\leq |X^n|-\left| \bigcup_{\om \in E_A^n} \f_\om(G_{t(\om)})\right|  \\
&=|X^n|- \sum _{\om \in E_A^n} |\f_\om(G_{t(\om)})| \\
&=|X^n|-K^{-Q}\gamma \sum _{\om \in E_A^n} |\f_\om(X_{t(\om)})|\\
&\leq |X^n|-K^{-Q}\gamma |X^n|=(1-K^{-Q}\gamma)|X^n|.
\end{split}
\end{equation}
Since $X^n$ is a decreasing sequence of sets, \eqref{decrmeas} implies that
\begin{equation}
\label{limxnzero}
\lim_{n\ra \infty } |X^n|=0,
\end{equation}
hence we also deduce that $|J_\cS|=0$, because $J_\cS \subset X^n$ for all $n \in \N$. Thus the first part of the theorem is proven.

We will now prove the second part of the theorem by way of contradiction. To this end suppose that $\dim_\cH  (J_\cS)=Q$. Since $\cS$ is regular Theorem \ref{t1j97} implies that $\P(Q)=0$. Hence by Proposition \ref{p2j85} $Q \in \Fin (\cS)$. Therefore for every $\om \in E_A^\ast$ and every Borel set $A \subset X_{t(\om)}$ such that $|A|>0$, using Proposition \ref{p320130613}, Theorem \ref{t3.1.7} and Theorem \ref{analytic-Carnot-conformal} we have,
\begin{equation}
\label{mqest}
\begin{split}
m_Q( \f_\om (A)) &\leq \|D \f_\om\|^Q m_Q (A) =K^Q K^{-Q} \|D \f_\om\|^Q |A| \frac{ m_Q (A)}{|A|} \\
&\leq K^Q |\f_\om (A)| \frac{ m_Q (A)}{|A|}.
\end{split}
\end{equation}

For $\om \in E_A^n$ we denote
$$Y_\om =\f_\om(X_{t(\om)}) \cap \bigcup_{\tau \in E_A^n \stm \{\om\}} \f_\tau(X_{t(\tau)}).$$
Since $\cS$ is of bounded coding type, by Theorem \ref{t3.1.7} we deduce that $m_Q$ is of null boundary hence,
\begin{equation}
\label{zeroyom}
m_Q(Y_\om)=0.
\end{equation}
Notice also that by the open set condition it is not difficult to see that
\begin{equation}
\label{subsetboundary}
\f_\om^{-1} (Y_\om) \subset \partial X.
\end{equation}
Using \eqref{zeroyom} we get
\begin{equation*}
\begin{split}
m_Q(X^n) &\leq \sum_{\om \in E_A^n} m_Q(\f_\om(X_{t(\om)}))= \sum_{\om \in E_A^n} m_Q(\f_\om (X_{t(\om)}) \stm Y_\om)
\\&= \sum_{\om \in E_A^n} m_Q(\f_\om (X_{t(\om)} \stm \f_\om^{-1}(Y_\om) )).
\end{split}
\end{equation*}
Notice that \eqref{subsetboundary} implies that $|X_{t(\om)} \stm \f_\om^{-1} (Y_\om)|>0$. Therefore by \eqref{mqest} and \eqref{subsetboundary}
\begin{equation*}
\begin{split}
m_Q(X^n) &\leq  \sum_{\om \in E_A^n} m_Q(\f_\om (X_{t(\om)} \stm \f_\om^{-1}(Y_\om) )) \\
&\leq K^Q \sum_{\om \in E_A^n} \frac{m_Q(X_{t(\om)} \stm \f_\om^{-1} (Y_\om))}{|X_{t(\om)} \stm \f_\om^{-1} (Y_\om)|} |\f_\om(X_{t(\om)} \stm \f_\om^{-1} (Y_\om))| \\
&\leq K^Q \frac{m_Q(X)}{\min \{|\Int X_v|:v \in V\}} \sum_{\om \in E_A^n} |\f_\om(X_{t(\om)} \stm \f_\om^{-1} (Y_\om))|.
\end{split}
\end{equation*}
Observe that
$$\f_\om(X_{t(\om)} \stm  \f_\om^{-1}(Y_\om)) \cap \f_\tau(X_{t(\tau)} \stm  \f_\tau^{-1}(Y_\tau))= \emptyset$$
for $\om, \tau \in E_A^n$ with $\om \neq \tau$. Therefore, setting $\delta=K^Q \frac{m_Q(X)}{\min \{|\Int X_v|:v \in V\}}$, we have
\begin{equation}
\label{mqxn}
\begin{split}
m_Q(X^n) &\leq \delta \sum_{\om \in E_A^n} |\f_\om(X_{t(\om)} \stm \f_\om^{-1} (Y_\om))|= \delta \left| \bigcup_{\om \in E^n_A} \f_\om(X_{t(\om)} \stm \f_\om^{-1} (Y_\om))\right| \\
& \leq  \delta \left|\bigcup_{\om \in E^n_A} \f_\om(X_{t(\om)}) \right|= \delta |X^n|,
\end{split}
\end{equation}
for all $n \in \N$. Using \eqref{mqxn} and \eqref{limxnzero} we deduce that $m_Q(J_\cS)=0$. Hence we have reached a contradiction, because $\cS$ is regular and for example by \eqref{2j93}, $m_Q(J_\cS)>0$. The proof of the theorem is complete.
\end{proof}

\chapter{Examples revisited}\label{chap:examples-2}

In this chapter we return to the examples described in Chapter \ref{chap:examples}. We illustrate the results of the previous chapters by discussing their implications for the invariant sets of the iterated function systems and graph directed Markov systems of that chapter.
We provide computations of and estimates for two different Hausdorff dimensions of such invariant sets: the dimension with respect to the sub-Riemannian metric and the dimension with respect to the underlying Euclidean metric. Note that these invariant sets are defined by iteration of mappings which, while they are conformal in the sub-Riemannian sense, are no longer conformal in the Euclidean sense. From the Euclidean perspective the maps in question are quite general nonlinear $C^1$ mappings.

Recall from section \ref{sec:DCP} that the Dimension Comparison Problem in sub-Rieman\-nian Carnot groups asks for sharp comparison estimates relating the Hausdorff dimensions of sets with respect to the aforementioned two bi-Lipschitz inequivalent metrics. In the following sections we point out the computations and estimates for dimensions of invariant sets arising from the results of the previous chapters together with the Dimension Comparison Theorem.

\section{Infinite self-similar iterated function systems}

Let $\G$ be an arbitrary Carnot group and let $\cS = \{\phi_e:\G\to\G\}_{e \in E}$ be a self-similar IFS with countably infinite alphabet $E$ as in section \ref{sec:iifs1}. Assume that the open set condition is satisfied. Then
$$
\dim_\mathcal{H} J_\cS = \dim_\mathcal{P} J_\cS = h,
$$
where $h$ denotes the similarity dimension for $\cS$, i.e.,
$$
h = \inf \left\{ t \ge 0 \, : \, \sum_{e \in E} r_e^t < 1 \right\}.
$$
Here $r_e$ denotes the contraction ratio for the similarity $\phi_e$. In view of the Dimension Comparison Theorem, we further obtain the estimates
$$
(\beta_+)^{-1}(h) \le \dim_{\mathcal{H},E} J_\cS \le (\beta_-)^{-1}(h)
$$
for the Euclidean Hausdorff dimension of $J_\mathcal{S}$. Here $\beta_+$ and $\beta_-$ denote the dimension comparison functions for the Carnot group $\G$. These results were obtained in \cite[Section 4]{btw:dimcomp} for finite alphabet self-similar IFS in Carnot groups; our primary contribution here is to extend these formulas and estimates to the case of countably infinite alphabets.

\section{Continued fractions in groups of Iwasawa type}\label{CF:II}

In this section we apply the results we obtained in Chapters \ref{chap:IGDMS-dimensions} and \ref{chap:geometric-properties} to the continued fractions systems that we introduced in Section \ref{iwacf:sec}. We remind the reader that we are studying continued fractions as limit sets of dynamical  systems in Iwasawa groups. In particular for $\ve \geq 0$ we consider the conformal iterated function systems
\begin{equation*}
\cS_\ve=\left\{\f_\gamma: \overline{B}\left(o,1/2\right) \rightarrow \overline{B}\left(o,1/2\right) \right\}_{\gamma \in I_\ve}
\end{equation*}
where $1$ the constant appearing  in Theorem \ref{integerlattice} and
\begin{itemize}
\item $I_\ve=\G(\Z) \cap B(o,\Delta_\ve)^c,$
\item $\Delta_\ve=\frac{5}{2}+\ve,$ and
\item $\f_\gamma=\cJ \circ \ell_\gamma.$
\end{itemize}
In our first theorem we calculate the $\theta$-number for such systems.

\begin{theorem}
\label{cftheta}
Let $\G$ be a Carnot group of Iwasawa type and for $\ve \geq 0$ let $\cS_\ve=\{\f_\gamma \}_{\gamma \in I_\ve}$ be the continued fraction conformal dynamical system. Then for all $\ve \geq 0$,
$$\theta_{\cS_\ve}=\frac{Q}{2}$$
and $\cS_\ve$ is co-finitely regular.
\end{theorem}
\begin{proof} Fix $\ve\geq 0$, and for simplifying notation set $I=I_\ve$ and $\Delta=\Delta_\ve$. Moreover without loss of generality we can assume that  $A_2$ from Theorem \ref{integerlattice} satisfies $A_2 \geq 1$. Let
$$I_0=\{\gamma \in I:\Delta \leq d(\gamma,o)<\Delta+D\}$$ and for $k \in \N,$ let
$$I_k=\{\gamma \in I:\Delta+D^k \leq d(\gamma,o)<\Delta+D^{k+1} \}$$
where $D=\max\{2 \Delta, 10 A_2\}$. Using \eqref{cf} for $t \ge 0$,
\begin{equation}
\label{z1cf}
Z_1(t)= \sum_{\gamma \in I} \|D \f_\gamma\|_\infty^t \approx \sum_{k \ge 0} \sum_{\gamma  \in I_k} d(\gamma,0)^{-2t}\approx \sum_{k \ge 0} \sum_{\gamma  \in I_k} D^{-2kt}.
\end{equation}

We will now fix some $k \in \N, $ and we will estimate the cardinality of the sets $I_k$. First note that
$$\sharp I_k \leq \sharp \{\G(\Z) \cap B(o, \Delta+D^{k+1})\}:=\sharp J_k.$$
By Theorem \ref{integerlattice}(ii) the balls $\{B(\gamma, 1/2)\}_{\gamma \in \G(Z)}$ are pairwise disjoint, hence
\begin{equation*}
\begin{split}
\sharp I_k \,{2} ^{-Q}
& \leq \sharp J_k \, {2}^{-Q} \\
&= \frac{1}{c_0}\sum_{\gamma \in J_k} |B(\gamma, 1/2)| \\
&= \frac{1}{c_0}\left| \bigcup_{\gamma \in J_k} B(\gamma, 1/2) \right| \\
&\leq \frac{1}{c_0}|B(o, \Delta+D^{k+1}+1/2)| \approx D^{kQ},
\end{split}
\end{equation*}
where we remind the reader that $c_0=|B(o,1)|$. Therefore
\begin{equation}
\label{upbouik}
\sharp I_k \lesssim D^{kQ}.
\end{equation}

We now turn our attention to the lower bound of $\sharp I_k$. Set
$$S_k=B(o,\Delta+D^{k+1}-A_2) \setminus \overline{B}(o, \Delta+D^{k}+A_2).$$
Recalling Theorem \ref{integerlattice} it follows easily that
\begin{equation*}
S_k \subset \bigcup_{\gamma \in I_k} B(\gamma, A_2).
\end{equation*}
Therefore $|S_k| \leq c_0 \sharp I_k \, (A_2)^Q$ and since $|S_k|\approx D^{kQ}$ we conclude that
\begin{equation}
\label{lobouik}
\sharp I_k \gtrsim D^{kQ}.
\end{equation}
Hence by \eqref{z1cf}, \eqref{upbouik} and \eqref{lobouik},
$$Z_1(t) \approx \sum_{k=0}^\infty (D^{Q-2t})^k,$$
and in view of Proposition \ref{p2j85}(i), $\theta=Q/2$. Moreover $Z_1(Q/2)=\infty$ and recalling Definition \ref{reguraldef} we conclude that the system is co-finitely regular. The proof is complete.
\end{proof}

We can now provide size estimates, on the level of Hausdorff dimension, for the limit sets of continued fraction conformal dynamical systems.

\begin{theorem}\label{cfhaus}
Let $\G$ be a Carnot group of Iwasawa type and for $\ve \geq 0$ let $\cS_\ve=\{\f_\gamma \}_{\gamma \in I_\ve}$ be the continued fraction conformal dynamical system. Then
\begin{enumerate}
\item for all $\ve \geq 0$, $\dim_\cH J_{\cS_\ve} > Q/2$,
\item for all $\ve >0$, $\dim_\cH J_{\cS_\ve} <Q$.
\end{enumerate}
\end{theorem}

\begin{proof}
Theorem \ref{cftheta} implies that for all $\ve \geq 0$ the system $\cS_\ve$ is co-finitely regular. Therefore by Proposition \ref{f1j87}, Theorem \ref{t1j97} and Theorem \ref{cftheta} we conclude that
$$\dim_\cH J_{\cS_\ve}=h_{\cS_\ve}> \theta_{\cS_\ve}=\frac{Q}{2}.$$

For the proof of $(2)$ it suffices to show that the systems $\cS_\ve$ satisfy the assumptions of Theorem \ref{dimbound} for $\ve>0$. Notice that for all $\ve \geq 0$ by Example \ref{wccballs} the systems $\cS_\ve$ satisfy the weak corkscrew condition. Hence by Proposition \ref{l1101804} the systems $\cS_\ve$ are of bounding coding type. Moreover by Theorem \ref{cftheta} the systems $\cS_\ve$ are co-finitely regular, hence regular. Therefore in order to prove $(2)$ it suffices to show that for $\ve>0$
\begin{equation}
\label{intxx1}
\left| \overline{B}\left(o,1/2 \right) \setminus \bigcup_{\gamma \in I_\ve} \f_\gamma\left(\overline{B}\left(o,1/2\right)\right) \right|>0.
\end{equation}
For $\gamma \in I_\ve$ arguing as in \eqref{fginside} we have that
$$\f_\gamma\left(\overline{B}\left(o,1/2\right)\right) \subset \overline{B}\left(o,\frac{1}{2 +\ve}\right).$$
Hence we deduce that
$$\bigcup_{\gamma \in I_\ve} \f_\gamma\left(\overline{B}\left(o,1/2\right)\right) \subset \overline{B}\left(o,\frac{1}{2 +\ve }\right),$$
and \eqref{intxx1} follows. The proof of the theorem is complete.
 \end{proof}

The following theorem concerns the spectrum of Hausdorff dimensions of the limit sets of subsystems of continued fractions conformal dynamical systems.

\begin{theorem}\label{cfsubsystem}
Let $\G$ be a Carnot group of Iwasawa type and for $\ve \geq 0$ let $\cS_\ve=\{\f_\gamma \}_{\gamma \in I_\ve}$ be the continued fraction conformal dynamical system. For every $t \in (0,Q/2)$ there exists  a proper subsystem $\cS_{\ve,t}$ of $\cS_\ve$ such that
$$\dim_\cH J_{\cS_{\ve,t}}=t.$$
\end{theorem}
\begin{proof}
The proof is a direct application of Theorem \ref{spectrum}.
\end{proof}

We can also prove that in any Iwasawa group $G$ there exist continued fractions conformal dynamical systems whose limit set has Hausdorff dimension arbitrarily close to $Q/2$.

\begin{theorem}\label{convtotheta}
Let $\G$ be a Carnot group of Iwasawa type. Then there exists an increasing sequence $(R_n)_{n=1}^\infty$ with $R_n \rightarrow \infty$ such that
$$\lim_{n \rightarrow \infty}\dim_\cH J_{\cS_{R_n}}= \frac{Q}{2}.$$
Here $\cS_{R_n}=\{\f_\gamma \}_{\gamma \in I_{R_n}}$ corresponds to the continued fractions conformal system as in \eqref{iwacfeq}.
\end{theorem}

\begin{proof}
To simplify notation let $I_0=I$ and if $T \subset I$ let $\cS_{I\setminus T}=\{\f_\gamma\}_{\gamma \in I\setminus T}$. By Theorem \ref{infcofi}, Theorem \ref{t1j97} and Theorem \ref{cftheta} we have that
\begin{equation}
\label{q2615}
\frac{Q}{2}=\inf \{\dim_{\cH} J_{\cS_{I \setminus T}}: T \subset I \text{ finite} \}.
\end{equation}

Therefore there exists an increasing sequence $(T_n)_{n=1}^\infty$ of finite subsets of $I$ so that
\begin{equation}
\label{convdim1}
\dim_\cH J_{\cS_{I \setminus T_n}} \rightarrow Q/2.
\end{equation}
Moreover for every $n \in \N$ there exists some $R_n$ such that $T_n \subset B(o, R_n)$, recalling Subsection \ref{subsec:cfiwasa} this implies that $T_n \subset B(o, \Delta_{R_n})$. Since $I_{R_n}=I \setminus T$ for some finite set $T \subset I$, \eqref{q2615} implies
$$\dim_{\cH} J_{\cS_{I \setminus T_n}} \geq \dim_{\cH} J_{\cS_{R_n}} \geq \frac{Q}{2}.$$
By \eqref{convdim1} we deduce that $\dim_\cH J_{\cS_{R_n}} \rightarrow Q/2$
and the proof is complete.
\end{proof}

In view of the Dimension Comparison Theorem \ref{th:DCP}, Theorems \ref{cfhaus}, \ref{cfsubsystem} and \ref{convtotheta} have obvious Euclidean consequences. For instance, the following corollary is obtained by applying the estimates in Theorem \ref{th:DCP} in connection with Theorem \ref{cfhaus}.

\begin{corollary}\label{cfhauscorollary}
Let $\G$ be a Carnot group of Iwasawa type and for $\ve \geq 0$ let $\cS_\ve=\{\f_\gamma \}_{\gamma \in I_\ve}$ be the continued fraction conformal dynamical system. Identify $\G$ with the Euclidean space $\R^N$ equipped with the Euclidean metric $d_E$. Then
\begin{enumerate}
\item for all $\ve \geq 0$, $\dim_{\cH,E} J_{\cS_\ve} > (\beta_+)^{-1}(Q/2)$,
\item for all $\ve >0$, $\dim_{\cH,E} J_{\cS_\ve} < N$.
\end{enumerate}
\end{corollary}

Note that the value of $(\beta_+)^{-1}(Q/2)$ depends on the particular Iwasawa group $\G$. In fact, the choice of which expression in \eqref{eq:betaplus} is relevant is different depending on the group $\G$. In the complex Heisenberg groups $\Heis^n = \Heis^n_\C$ we always have
$$
(\beta_+)^{-1}(\tfrac{Q}2) = \tfrac{Q}2-1 = n,
$$
while in the first octonionic Heisenberg group $\Heis_\Oct^1$ we have
$$
(\beta_+)^{-1}(\tfrac{Q}2) = \tfrac{Q}4 = \tfrac{11}2.
$$
In the quaternionic Heisenberg groups $\Heis_\Qua^n$ the answer depends on the value of $n$:
$$
(\beta_+)^{-1}(\tfrac{Q}2) = \begin{cases} \tfrac{Q}2-1 = 4 & \mbox{in $\Heis_\Qua^1$,} \\ \tfrac{Q}4 = n+\tfrac32 & \mbox{in $\Heis_\Qua^n$, $n\ge 2$.} \end{cases}
$$

\begin{remark}
Estimates for the Euclidean dimensions of Carnot--Carath\'eodory self-similar sets were previously obtained in \cite{bt:horizdim} and \cite{btw:dimcomp}.
\end{remark}

\section{Iwasawa conformal Cantor sets}

In this section we return to the conformal Cantor sets that where introduced in Section \ref{sec:iwacantor}. Recall that conformal Cantor sets are limit sets of conformal iterated function systems
$$\cS=\{\f_n: \overline{G} \rightarrow \overline{G}\}_{n \in \N}$$
where
$$\f_n=\ell_{p_n} \circ \delta_{r_n} \circ \ell_{\cJ(p_n)^{-1}} \circ \cJ,$$
and
\begin{itemize}
\item $G$ is an open set, $\overline{G}$ is compact and $o \notin G$,
\item $P=(p_n)_{n=1}^\infty$ is a discrete sequence of points in $G$,
\item $d_n=\inf_{m \neq n} d(p_n,p_m)$ and $\lim_{n \ra \infty} d_n=0$,
\item $\dist(P, \partial
G)> \sup_{n \in \N} d_n$,
\item $\rm{d}_0:=\diam \cJ(G)$, and $r_n< \min \{s,\frac{d_n}{2\rm{d}_0}\}$ for some $s<1$.
\end{itemize}
For all $n \in \N$ and $p \in G$, by \eqref{leibniz} and \eqref{formula-for-norm-of-Df}
\begin{equation*}
\label{derivcantor}
\|D \f_n(p)\|=r_n \|D \cJ (p)\|=\frac{r_n}{d(p,o)^2}.
\end{equation*}
Hence
$\|D \f_n\|_\infty \approx r_n$
where the constant depends only on the open set $G$. Therefore if $\cS$ is a conformal iterated function system as above, by Proposition \ref{p2j85}
\begin{equation}
\label{thetaconf}
\theta_\cS= \inf \left \{t \geq 0: \sum_{n \in \N} r_n^t< \infty \right \}.
\end{equation}

We now describe a specific type of conformal Cantor sets in Iwasawa groups and provide estimates for their Hausdorff dimensions.

Let $\ve>1$ and for $n \in \N$ set $$\Sigma_n= \partial B\left(o,\, \sum_{j=1}^n \frac{1}{j^\ve}\right).$$
Denote by $\Pi_n$ the maximal collection of points in $\Sigma_n$ with mutual distances at least $\left(\frac1{n+2}\right)^\ve$. It is well known, see e.g. \cite{cg:ahlfors} and \cite{dgn:memoirs}, that the $(Q-1)$-dimensional spherical Hausdorff measure $\cS^{Q-1}$ restricted to the boundaries of Kor\'anyi balls is Ahlfors $(Q-1)$-regular, recall \eqref{h-regularity} for the definition of Ahlfors regularity. Since for all $n  \in \N$ the restriction of $\cS^{Q-1}$ on $\Sigma_n$  is Ahlfors $(Q-1)$-regular one can show that
\begin{equation}
\label{pincard}
\sharp \Pi_n \approx \left(\frac1{n+2}\right)^{\ve(1-Q)}.
\end{equation}
Let $\Pi= \bigcup_{n \in \N} \Pi_n$. We will now show that if $p \in \Pi_n$ then
\begin{equation}
\label{infpn}
\left( \frac{1}{n+2} \right)^\ve\leq \inf_{q \in \Pi \setminus \{p\}} d(p,q) \leq 4 \left( \frac{1}{n+2} \right)^\ve.
\end{equation}
For the right hand side inequality  first notice that
$$\Sigma_n \subset \bigcup_{q \in \Pi_n} B\left(q, 2 \left( \frac{1}{n+2} \right)^\ve\right).$$
Now fix some $p \in \Pi_n$. Since $\Sigma_n$ is a connected metric space with the relative topology of $d$ it follows that $B\left(p, 2 \left( \frac{1}{n+2} \right)^\ve\right)$ intersects some set from the family
$\left\{\Sigma_n \cap B\left(q, 2 \left( \frac{1}{n+2} \right)^\ve\right) \right\}_{q \in \Pi_n \setminus \{p\}}$ of (relative) open sets. Therefore there exists $q \in \Pi_n$ such that
$$\Sigma_n \cap B\left(p, 2 \left( \frac{1}{n+2} \right)^\ve \right) \cap B\left(q, 2 \left( \frac{1}{n+2} \right)^\ve \right)\neq \emptyset.$$
In particular
$$d(p,q) \leq 4 \left( \frac{1}{n+2} \right)^\ve $$
and the desired inequality follows.

For the remaining inequality we will first show that for all $n \in \N$
\begin{equation}
\label{snsn1}
\dist(\Sigma_n,\Sigma_{n+1}) \geq \left( \frac{1}{n+1} \right)^\ve.
\end{equation}
To prove \eqref{snsn1} suppose by way of contradiction that there exist $x_n \in \Sigma_n$ and $x_{n+1} \in \Sigma_{n+1}$ such that $d(x_n,x_{n+1})<\left( \frac{1}{n+1} \right)^\ve$. Then
$$d(x_{n+1},o) \leq d(x_n,x_{n+1})+d(x_n,o)< \sum_{j=1}^n \frac{1}{j^\ve} +\left( \frac{1}{n+1} \right)^\ve= \sum_{j=1}^{n+1}\frac{1}{j^\ve} $$
and we have reached a contradiction since $x_{n+1} \in \Sigma_{n+1}$. Therefore by \eqref{snsn1}
$$\inf_{q \in \Pi \setminus \{p\}} d(p,q)= \inf_{q \in \Pi_n \setminus \{p\}} d(p,q) \geq \left( \frac{1}{n+2} \right)^\ve$$
and \eqref{infpn} follows.

We remark that if we write $\Pi=(p_k)_{k=1}^\infty$ and set
$$d_k=\inf_{l \neq k} d(p_k,p_l)$$
then \eqref{infpn} implies that
$$\lim_{k \ra \infty} d_k=0.$$

We can now define the CIFS as in Section \ref{sec:iwacantor} . Let
$$G_\ve=\left\{ B\left(o, \sum_{n\in \N} \frac{1}{n^\ve} +1\right)\setminus \overline{B}\left(o, \frac{1}2\right)\right\}.$$
For $k \in \N$ we define the conformal maps $\f_k : \overline{G_\ve} \ra \overline{G_\ve}$ by
\begin{equation}
\label{fkconformal}
\f_k= \ell_{p_k}\circ \delta_{d_k/(10\rm{d}_0)} \circ \ell_{\cJ(p_k)^ {-1}} \circ \cJ,
\end{equation}
where as before $\rm{d}_0=\diam \cJ(G)$, and we set $\cS_\ve= \{\f_k:\overline{G_\ve} \ra \overline{G_\ve}\}_{k\in \N}$.

We will now verify that $\cS_\ve$ satisfies the open set condition. Let $p_k, p_l \in \Pi,\,p_k \neq p_l,$ and let $n_l, n_k \in \N$ such that $p_k \in \Pi_{n_k}$ and $p_l \in \Pi_{n_l}$. If $n_l \neq n_k$, assume without loss of generality that $n_l >n_k$ and notice that
$$d(p_k,p_l) \geq \dist(\Sigma_{n_k},\Sigma_{n_l})\geq \dist(\Sigma_{n_k},\Sigma_{n_k+1})\geq \left( \frac1{n_k+1}\right)^\ve .$$
But by \eqref{infpn}
\begin{equation}
\label{diamfkg}
\diam (\f_k(G_\ve))=\frac{d_k}{10 \rm{d}_0} \diam (\cJ(G_\ve))\leq \frac{4}{10} \left( \frac{1}{n_k+2} \right)^\ve.
\end{equation}
Moreover $\diam (\f_l (\overline{G_\ve}))<\diam (\f_k (\overline{G_\ve}))$ because $n_l>n_k$. Hence by \eqref{diamfkg} we deduce that $\f_k (G_\ve) \cap \f_l (G_\ve)= \emptyset$. If $n_l=n_k:=n$ then
$$d(p_l,p_k) \geq \left(\frac{1}{n+2}\right)^\ve.$$
As in \eqref{diamfkg} we have that $$\diam(\f_k(G_\ve))+\diam(\f_l(G_\ve)) \leq \frac{8}{10} \left( \frac{1}{n+2} \right)^\ve.$$
Hence for all $k,l \in \N, k \neq l,$ $$\f_k(G_\ve) \cap \f_l(G_\ve)=\emptyset.$$

To finish the proof of the open set condition it remains to show that for all $k\in \N$, $\f_k (G_\ve) \subset G_\ve.$ First notice that by \eqref{diamfkg} for all $k\in \N$, $\diam (\f_k(G_\ve))<1/2$ and there exist $n_k \in \N$ such that $p_k \in \Sigma_{n_k}$, hence for all $k \in \N$,
$$1 \leq d(p_k,o)< \sum_{n\in \N} \frac{1}{n^\ve}.$$
Therefore if $x \in  \f_k(G_\ve)$, then
$$d(x,o)\leq d(x, p_k)+d(p_k,o)< \sum_{n\in \N} \frac{1}{n^\ve}+1/2,$$
and in the same way
$$d(x,o) \geq d(p_k,o)-d(x, p_k)> 1/2.$$ Therefore the conformal iterated function system $\cS_ \ve$ satisfies the open set condition.

The following theorem gathers information about the Hausdorff dimension of the Cantor sets $J_{\cS_\ve}$. Note that by part (iv) of the following theorem, it follows that every value $t \in (0,Q)$ arises as the Hausdorff dimension of the invariant set of some subsystem of a GDMS associated to a conformal Cantor set.

\begin{theorem}\label{fkconftheor}
Let $\G$ be a Carnot group of Iwasawa type. Let $\ve>1$ and consider the conformal iterated function system $\cS_\ve$ defined by the conformal maps $\f_k$ as in \eqref{fkconformal}. Then
\begin{enumerate}
\item[(i)] $\theta_{\cS_\ve}=Q-\frac{\ve-1}{\ve},$
\item[(ii)] the system $\cS_\ve$ is co-finitely regular,
\item[(iii)] $\dim_\cH J_{\cS_\ve}>Q-\frac{\ve-1}{\ve}$,
\item[(iv)] for all $t \in (0, Q-\frac{\ve-1}{\ve})$ there exists a  proper subsystem $\cS_{\ve,t}$ of $\cS_\ve$ such that
$$\dim_\cH J_{\cS_{\ve,t}}=t.$$
\end{enumerate}
\end{theorem}

\begin{proof}
First observe that once we have shown (i) and (ii) then the remaining statements of the theorem follow easily. Indeed, (iii) follows by (i), (ii), Proposition \ref{f1j87}, Theorem \ref{t1j97} and Theorem \ref{cftheta}, while (iv) follows from (i) and Theorem \ref{spectrum}.

It remains to show (i) and (ii). Recalling \eqref{thetaconf} and the way that the maps $\f_k$ were defined, we have
$$
\theta_{\cS_\ve}=\inf \left\{t \geq 0: \sum_{k \in \N} \left( \frac{d_k}{10 \rm{d}_0}\right)^t<\infty\right\}.
$$
By \eqref{pincard} and \eqref{infpn}
\begin{equation}\label{confosumtheta}
\begin{split}
\sum_{k \in \N} \left( \frac{d_k}{10 \rm{d}_0}\right)^t &\approx \sum_{n \in \N} \sum_{k \in \Pi_n} d_k^t \approx \sum_{n \in \N} \sum_{k \in \Pi_n} \left(\frac{1}{n+2} \right)^{ \ve t}\\
&\approx \sum_{n \in \N} \left(\frac{1}{n+2} \right)^{ \ve (1-Q)}  \left(\frac{1}{n+2} \right)^{ \ve t}\\
&=\sum_{n \in \N}\left(\frac{1}{n+2} \right)^{ \ve(1- Q+t)}.
\end{split}
\end{equation}
Hence in view of Proposition \ref{p2j85}(i), $\theta_{\cS_\ve}=Q-\frac{\ve-1}{\ve}$. Moreover by \eqref{confosumtheta}, we easily see that $$Z_1 \left(Q-\frac{\ve-1}{\ve}\right)=\infty.$$ Recalling Definition \ref{reguraldef} we conclude that the system $\cS_\ve$ is co-finitely regular. The proof is complete.
\end{proof}

\begin{remark}
Notice that the Hausdorff dimension of the conformal Cantor sets $J_{\cS_\ve}$ can be arbitrarily close to $Q$, since
$$
\lim_{\ve \rightarrow 1} \dim_{\cH} J_{\cS_\ve}=Q
$$
by Theorem \ref{fkconftheor}.
\end{remark}

The following corollary of Theorem \ref{fkconftheor} can be established using the Dimension Comparison Theorem \ref{th:DCP}.

\begin{corollary}\label{fkconfcor}
Let $\G$ be a Carnot group of Iwasawa type. Let $\ve>1$ and consider the conformal iterated function system $\cS_\ve$ defined by the conformal maps $\f_k$ as in \eqref{fkconformal}. Then the Euclidean Hausdorff dimension of the invariant set of $\cS_\ve$ satisfies
$$
\dim_{\cH,E} J_{\cS_\ve}>N-\frac{\ve-1}{\ve}.
$$
\end{corollary}

\chapter{Properties of Hausdorff and packing measures of limit sets}\label{chap:measure-properties}

In this chapter we investigate fine properties of limit sets of Carnot conformal GDMS. In particular we are concerned with the positivity of Hausdorff and packing measures of limit sets. Under some mild assumptions on the GDMS we can prove that the $h$-Hausdorff measure of the limit set is finite and under some standard separation condition we can show that the $h$-packing measure is positive. This is performed in Section \ref{sec:finihau}. Deciding whether the $h$-Hausdorff measure is positive or if the $h$-packing measure is finite are subtler questions. In Section \ref{sec:posihau} we provide necessary and sufficient conditions for the Hausdorff measure of the limit set to be positive and for its packing measure to be finite.

\section{Finiteness of Hausdorff measure and positivity of packing measure}\label{sec:finihau}

The next theorem shows that very generally the $h$-Hausdorff measure of the limit set of a Carnot conformal GDMS is finite.

\begin{theorem}\label{t4.4.1}
Let $\cS$ be a finitely irreducible weakly Carnot conformal GDMS.
\begin{enumerate}
\item If the system is regular, then

\sp \begin{itemize}
\item[(1a)] the restriction of the Hausdorff
measure to $J_\cS$, i.e. $\cH^h\lfloor J_\cS$, is absolutely
continuous with respect to the conformal measure $m_h$, and

\sp\item[(1b)]
$\|d(\cH^h\lfloor J_\cS) /dm_h \|_\infty<\infty$.

\sp\item[(1c)] In particular,
$\cH^h(J_\cS)<\infty$. \index{Hausdorff measure!finiteness of for limit sets of GDMS}
\end{itemize}

\sp \item If the system is not regular then $\cH^h(J_{\cS})=0.$
\end{enumerate}
\end{theorem}

\begin{proof}
Let $A$ be an arbitrary closed subset of $J_\cS$. For every $n\in \N$
put
$$
A_n:=\{\om\in E_A^n:\f_\om(J_\cS)\cap A\ne\es\}.
$$
Then the sequence of sets
$$
\lt(\bu_{\om\in A_n}\f_\om(X_{t(\om)})\rt)_{n=1}^\infty
$$
is descending and
$$
\bigcap_{n\ge 1} \( \bu_{\om\in A_n}\f_\om(X_{t(\om)}) \) = A.
$$

Notice also that since for every $\om \in E_A^\ast$ $\pi([\om]) \subset \f_\om(X_{t(\om)})$,
\begin{equation}\begin{split}
\label{uniosum}
\sum_{\om \in A_n} \^m_h([\om]) &=\^m_h\( \bu_{\om \in A_n} [\om]\)= m_h\(\bigcup_{\om \in A_n} \pi^{-1}([\om])\) \\
&\leq m_h\(\bu_{\om \in A_n}\f_\om(X_{t(\om)})\).
\end{split}\end{equation}
Using \eqref{4.1.9}, \eqref{515pre} and \eqref{uniosum}  we obtain
\begin{equation}\begin{split}
\label{hausest}
\cH^h(A)
&\le \liminf_{n\to\infty}\sum_{\om\in A_n}\(\diam(\f_\om(X_{t(\om)}))\)^h\\
&\le (M\Lambda)^h\liminf_{n\to\infty}\sum_{\om\in A_n}||D\f_\om||_\infty^h\\
&\leq c_h (M\Lambda)^h  \liminf_{n\to\infty} e^{n\P(h)}\sum_{\om\in A_n}\tilde m_h([\om])\\
&\leq c_h (M\Lambda)^h\liminf_{n\to\infty}
          e^{n\P(h)} \lt(m_h\(\bu_{\om\in A_n}\f_\om( X_{t(\om)})\)\rt).
\end{split}
\end{equation}
If $\cS$ is not regular then $\P(h)<0$ and by \eqref{hausest} we get that $\cH^h(J_\cS)=0$. On the other hand if $\cS$ is regular we have that $\P(h)=0$ hence \eqref{hausest} gives that $$\cH^h(A) \leq (M\Lambda)^h c_h m_h(A).$$
Since the metric space $J_\cS$ is separable, the measure $m_h$ is
regular and consequently the inequality $\cH^h(A)\le (M\Lambda)^h
c_h^{-1}m_h(A)$ extends to all Borel subsets $A$ of $J_\cS$. Thus the proof
is finished.
\end{proof}

%\begin{remark}\label{p420130508}
%If $\cS$ is a finitely irreducible system Iwasawa  conformal GDMS satisfying the weak corkscrew condition then $\cH^h(J_\cS)<+\infty$.
%\end{remark}
\begin{definition}
\label{SOSCdef} We say that the Carnot conformal  GDMS $\cS$ satisfies the {\it strong open set condition} (SOSC) if
\index{strong open set condition} \index{open set condition!strong}
\beq\label{520130510}
J_\cS\cap \Int(X)\ne\es.
\eeq
Recall that the set $X$ was defined in \eqref{X}.
\end{definition}

Assuming the strong open set condition we can show in a straightforward manner that the limit set of a finitely irreducible regular conformal GDMS has positive $h$-packing measure.

\begin{proposition}\label{p4201305081}
If $\cS$ is a finitely irreducible, regular Carnot conformal GDMS  satisfying the strong open set condition, then $\cP^h(J_\cS)>0$.
 \end{proposition}
\begin{proof} Let $v \in V$ such that $J_{\cS} \cap \Int(X_v) \neq \es$. Then there exists some $\tau \in E_A^q$ for some $q\in \N$ such that $i(\tau)=v$ and $\f_\tau (X_{t(\tau)}) \subset \Int (X_v)$. Set
$$\gamma=\min\{\dist(\f_\tau (X_{t(\tau)}), \partial X), \eta_\cS\},$$
where $\eta_\cS$ was defined in $\eqref{rdef}$. Let
$$\cW=\{ \om \in E_A^\infty:\om|_{[n+1,n+q]}=\tau\mbox{ for infinitely many }n
{'s}\}$$
and let $\cW_0$ be the subset of $E_A^\N$ whose elements do not contain  $\tau$ as a subword. Since $[\tau] \cap \cW_0= \es$ we conclude that $\^\mu_h (\cW_0)<1$.  Recall that by Theorem \ref{thm-conformal-invariant}  $\^\mu_h$ is ergodic with respect to $\sg$. Ergodicity implies that for any  Borel set $S \subset E_A^\N$ such that $\sg^{-1}(S) \subset S$, $\^\mu_h(S) \in \{0,1\}$. To see this, let $S$ be a Borel subset of $E_A^\N$ such that $\sg^{-1}(S) \subset S$. Since $\sg$ is $\^\mu_h$-measure preserving, $\^\mu_h(\sg^{-1}(S))=\^\mu_h(S)$.  Hence
$$\^\mu_h(\sg^{-1}(S) \triangle S)=\^\mu_h(S)-\^\mu_h(T^{-1}(S))=0,$$ and by the ergodicity of $\^\mu_h$ we conclude that  $\^\mu_h(S) \in \{0,1\}$. Now notice that $\sg^{-1}(E^\N_A \setminus \cW_0) \subset E^\N_A \setminus \cW_0$. Hence by the previous observation and the fact that $\^\mu_h( E^\N_A \setminus \cW_0)>0$  we deduce that $\^\mu_h(\cW_0)=0$. It is also easy to see that $E^\N_A \setminus \cW=\cup_{n \in \N} \sg^{-n} (\cW_0)$. Therefore,  since $\^\mu_h$ is shift invariant, we deduce that $\^\mu_h(E^\N_A \setminus \cW)=0$ and consequently
$$\mu_h(J_\cS \stm \pi(\cW))=\^\mu_h \circ \pi^{-1}(J_{\cS} \stm \pi(\cW)) \leq \^\mu_h(E^\N_A \setminus \cW)=0.$$
Now for any $\om \in \cW$ and $n \in \N$ such that $\om|_{[n+1,n+q]}=\tau$ by \eqref{4.1.8} we have that
$$B(\pi(\om),(KC)^{-1} \|D \f_{\om|_n}\|_\infty \, \gamma) \subset \f_{\om|_n}(B(\pi(\sg^n(\om)), \gamma)).$$
Moreover by the choice of $\gamma,$ $$B(\pi(\sg^n(\om)), \gamma) \subset B(\f_t(X_{t(\tau)}), \gamma) \subset \Int(X_{i(\tau)})=\Int(X_{t(\om|_n)}).$$
Hence by Remark \ref{820rem};
\begin{equation*}
\begin{split}
m_h(B(\pi(\om),(KC)^{-1} \|D \f_{\om|_n}\|_\infty \, \gamma)) &\leq \|D \f_{\om|_n}\|^h_\infty  m_h(B(\pi(\sg^n(\om)), \gamma))\\
&\leq (KC\gamma^{-1})^h \, ((KC)^{-1}  \|D \f_{\om|_n}\|_\infty\gamma)^h.
\end{split}
\end{equation*}
Since by Theorem \ref{thm-conformal-invariant} $m_h$ and $\mu_h$ are equivalent, the proof is concluded by invoking Theorem \cite[A2.0.13(1)]{MUGDMS}.
\end{proof}

\section{Positivity of Hausdorff measure and finiteness of packing measure}\label{sec:posihau}

Let $(X,\rho)$ be a metric space. Let $\nu$ be a finite Borel measure
on $X$. Fix $s>0$. We say that the measure $\nu$ is {\it upper geometric}
with exponent $s$ if \index{measure!upper geometric}
\beq\label{520130513}
\nu(B(x,r))\le c_\nu \, r^s
\eeq
for all $x\in X$, all radii $r\ge 0$ and some constant $c_\nu \in
[0,+\infty)$. Likewise, we say that the
measure $\nu$ is {\it lower geometric} with exponent $s$ if \index{measure!lower geometric}
\beq\label{5201305131}
\nu(B(x,r))\ge c_\nu\, r^s
\eeq
for all $x\in X$, all radii $r\in [0,1]$ and some constant $c_\nu \in
(0,+\infty]$. If $\nu$ is both upper geometric and lower geometric
with the same exponent $s>0$, it is called {\it geometric}\index{measure!geometric}\index{geometric measure} with exponent
$s$. Geometric measures with exponent $s$ are also frequently referred to as
Ahlfors $s$-regular measures\index{Ahlfors regular measure}\index{measure!Ahlfors regular}. If we do not care at a moment about the
particular value of the exponent $s$, we simply refer to the aforementioned measures as upper geometric, lower geometric, geometric, or Ahlfors regular measures.

\begin{definition}\label{brc}
A set $X\sbt \G$ is said to satisfy the \textit{boundary regularity
condition} \index{boundary regularity condition} if there exists a constant $\g_X\in (0,1]$ such that
$$
|\Int(X)\cap B(x,r)|\ge \ga_X |B(x,r)|
$$
for all $x\in X$ and all radii $r\in (0,\diam(X))$.
\end{definition}

Alternatively, for all $t>0$ there exists $\g_{X,t}$ such that
$$
|\Int(X)\cap B(x,r)|\ge \ga_{X,t} |B(x,r)|
$$
for all $x\in X$ and all radii $r\in (0,t)$.

\begin{definition}\label{GDMSbrc}
A Carnot conformal GDMS $\cS=\{\phi_e\}_{e \in E}$ is said to be {\it boundary regular} if each set $X_v$,
$v\in V$, satisfies the boundary regularity condition. We put
$\g:=\min\{\g_{X_v}:v\in V\}$ and $\g_t:=\min\{\g_{X_v,t}:v\in V\}$
for all $t>0$.
\end{definition}

\begin{remark}\label{o120130613}
Every Carnot conformal GDMS satisfying the corkscrew condition
is boundary regular.
\end{remark}

For boundary regular Carnot conformal GDMS we
shall prove now the following improvement of Lemma~\ref{l1j81}.

\begin{lemma}\label{l1j81A}
If $\mathcal{S}$ is a boundary regular Carnot conformal GDMS, then
for all $\kappa>0$, for all $r>0$ and for all $x\in X$, the
cardinality  of any  collection of mutually incomparable words $\om
\in E^*_A$ that satisfy the conditions
\beq\label{520130514}
B(x,r) \cap \phi_\om(X_{t( \om)})\neq \es
\eeq
and
\beq\label{120130529}
\diam(\phi_\om( X_{t(\om)}))\ge \kappa r,
\eeq
is bounded  above by $(3 \Lambda K)^Q\g_{\La M\ka^{-1}}^{-1}$.
\end{lemma}

\begin{proof}
Let $F$ be any collection of $A$-admissible words satisfying
the hypotheses of our lemma. By \eqref{120130529} and
\eqref{4.1.9} we get  that
$$
||D\phi_\om||_\infty^{-1}r \le \Lambda M\ka^{-1}
$$
for every $\om\in F$. Again by hypotheses, for every $\om\in F$ there exists $p_\om\in \Int(X_{t(
  \om)})$ such that $\phi_\om(p_\om)\in B(x,r)$. By the open set condition all the sets
$$
\phi_\om\(B(p_\om,||D\phi_\om||_\infty^{-1}r)\cap\Int(X_{t( \om)})\), \
\om\in F,
$$
are mutually disjoint. Also by \eqref{4.1.9a}
$$
\diam\(\phi_\om\(B(p_\om,||D\phi_\om||_\infty^{-1}r)\cap\Int(X_{t( \om)})\)\)
\le 2 \Lambda r.
$$
Therefore, since $\f_\om(B(p_\om))\in B(x,r)$,
$$
\phi_\om\(B(p_\om,||D\phi_\om||_\infty^{-1}r)\cap\Int(X_{t( \om)})\)
\sbt B(x,3\Lambda r).
$$
Therefore, by Theorem \ref{analytic-Carnot-conformal}
\begin{equation*}\begin{split}
c_0(3 \Lambda r)^Q
&= |B(x,3 \Lambda r)|
 \ge \lt|\bigcup_{\om \in F}
     \phi_\om\(B(p_\om,||D\phi_\om||_\infty^{-1}r)\cap\Int(X_{t(\om)})\)\rt|\\
&=  \sum_{\om \in F}
    |\phi_\om\(B(p_\om),||D\phi_\om||_\infty^{-1}r)\cap\Int(X_{t(\om)})|
    \\
& \geq \sum_{\om \in F} K^{-Q}||D\phi_\om||_\infty^Q
    \,|B(p_\om,||D\phi_\om||_\infty^{-1}r)\cap\Int(X_{t(\om)})|\\
& \geq \sum_{\om \in F}K^{-Q}\g_{\La M\ka^{-1}}
    ||D\phi_\om||_\infty^Q \, |B(p_\om,||D\phi_\om||_\infty^{-1}r)|\\
&=c_0 \,K^{-Q}\g_{\La M \ka^{-1}}\sum_{\om \in F}
    ||D\phi_\om||_\infty^Q(||D\phi_\om||_\infty^{-1}r)^Q \\
%&=cK^{-Q}\g_{D\ka^{-1}}\sum_{\om \in F}r^Q
&=  c_0 K^{-Q}\g_{\La M\ka^{-1}}\# Fr^Q
\end{split}\end{equation*}
Therefore $\#F\le (3 \Lambda K)^Q\g_{\La M\ka^{-1}}^{-1}$ and the proof is complete.
\end{proof}

The first main result of this section gives necessary and
sufficient conditions for the Hausdorff measure of $J_\cS$ to be
positive. In particular it shows that $\cH^h (J_\cS)$ is positive
if and only if the conformal measure $m_h$ is upper geometric with
exponent $h$.

\begin{theorem}\lab{t4.4.3}
If $\cS$ is a maximal regular Carnot conformal GDMS satisfying the corkscrew
condition, then the following conditions are equivalent.
\begin{itemize}
\item[(a)] $\cH_h(J_\cS)>0$.
\item[(b)] There exists $H>0$ such that
$$
m_h(B(y,r))\le Hr^h
$$
for every $e\in E$, every radius $r\ge \diam(\f_e(X_{t(e)}))$, and
every $y\in \f_e(X_{t(e)})$.
\item[(c)] There exist $H>0$ and $\g\ge 1$ such that
for every $e\in E$ and every radius $r\ge \g\diam(\f_e(X_{t(e)}))$
there exists $y\in \f_e(X_{t(e)})$ such that
$$\
m_h(B(y,r))\le Hr^h.
$$
\item[(d)] The measure $m_h$  is upper geometric with exponent $h$.
More precisely, there exists $c_{\cS}>0$ such that
$$
m_h(B(x,r))\le c_{\cS}\,r^h
$$
for every $x \in J_\cS$ and all $r\in[0,+\infty)$.
\end{itemize}
\end{theorem}

\begin{proof} (a)$\imp$(b). In order to prove this implication
suppose that (b) fails. Then for every
$H> \eta_\cS^{-h}$, where $\eta_\cS$ is as in \eqref{rdef}, there exists $j\in E$ such that
$$
m_h(B(x,r))> Hr^h
$$
for some $x\in \phi_j(X_{t(j)})$ and some $r\ge\diam(\phi_j(X_{t(j)}))$.
Let $E_A^\infty(\infty)$ be the set of words in $E_A^\N$
that contain each element of $E$ infinitely often. Let
$$
J_1:=\pi(E_A^\infty(\infty)).
$$
Fix $\om\in E_A^\infty(\N)$ and put $z:=\pi(\om)$. Then $\om_{n+1}=j$
for some $n\in \N$. Set $z_n=\pi(\sg^n(\om))$. So,
$z=\phi_{\om|_n}(z_n)$ and . Therefore $z_n,x \in \f_j(X_{t(j)}) \subset X_{t(\om_n)}$ and $z_n\in \ov B(x,r)$, hence by \eqref{4.1.9a}
$$
d(\phi_{\om|_n}(z_n),\phi_{\om|_n}(x))\le \La ||D\phi_{\om|_n}||_\infty \,r.
$$
Moreover notice that $r\le 1/H^{1/h}\le
\eta_\cS$, hence by Corollary~\ref{l42013_03_12} we get
$$
B(\phi_{\om|_n}(x), C\,||D\phi_{\om|_n}||_\infty\,r)\spt \phi_{\om|_n}(B(x,r)).
$$
Thus,
\begin{equation}
\label{inclu97}
B(z,(C+\La)||D\phi_{\om|_n}||_\infty \,r\)\spt \phi_{\om|_n}(B(x,r)).
\end{equation}
Since $r \leq \eta_\cS$ and $\cS$ is maximal by \eqref{inclu97} and Remark \ref{cor827rem} we  get
\begin{equation*}\begin{split}
m_h\(B(z,(C+\La)||D\phi_{\om|_n}||_\infty \,r\)\)
&\ge K^{-h}||D \phi_{\om|_n}||_\infty ^h m_h(B(x,r)) \\
 &> K^{-h}H||D\phi_{\om|_n}||_\infty^hr^h \\
&={H\over ((C+\La)K)^h}\, \((C+\La)||D\phi_{\om|_n}||_\infty \,r\)^h.
\end{split}\end{equation*}
Hence, by \cite[Theorem A2.0.12]{MUGDMS},
$$
\cH_h(J_1) \lesssim H^{-1},
$$
with constants independent of $H$. Now, letting $H\to\infty$ we conclude
that $\cH_h(J_1)=0$. By Theorem~\ref{thm-conformal-invariant},
Birkhoff's Ergodic Theorem, \index{Birkhoff's Ergodic Theorem} and \eqref{2j93}, using a standard argument it follows that
$m_h(J\sms J_1)=0$. This in turn, in
view of Theorem~\ref{t4.4.1}, shows that $\cH_h(J\sms J_1)=0$. Thus,
$\cH_h(J)=0$ and therefore the proof of the implication (a)$\imp$(b)
is finished.

\sp\fr The implication (b)$\imp$(c) is obvious.

\sp\fr (c)$\imp$(d) Increasing $ \La$ or $M$ if
necessary, we may assume that $\La M \tilde{R}_\cS^{-1}\ge 1$, where $ \tilde{R}_\cS$ was defined in Lemma  \ref{l52013_03_12}. Take an arbitrary
point $x\in J_\cS$ and radius $r>0$. Set
$$
\^r=2KC\La M\,\eta_\cS^{-1}r,
$$
where $\eta_\cS$ is as in \eqref{rdef}. Notice that increasing again $ \La$ or $M$ we can also assume that $\tilde{r} \geq r$.

For every $z\in B(x,r)\cap J_\cS$ consider a shortest word
$\om=\om(z)$ such that
$z\in\pi([\om])$ and $\f_\om(X_{t(\om)})\sbt B(z,\^r)$. Then
\beq\label{105302013}
\diam(\f_{\om|_{|\om|-1}}(X_{t(\om_{|\om|-1})})\ge \^r.
\eeq
Let
$$
\cW:=\{\om(z)|_{|\om(z)|-1}:z\in J_\cS\cap B(x,r)\}.
$$
Since $\lim_{e\in E}\diam(\f_e(X_{t(e)}))=0$ and
$\lim_{n\to\infty}\sup\{\diam(\f_\om(X_{t(\om)})):\om\in
E_A^n\}=0$ the set $\cW$ is finite. In particular this implies that there exists a finite set  $\{z_1,z_2,\ld,z_k\}\sbt J_\cS\cap B(x,r)$ such that all the words $\om(z_j)|_{|\om(z_j)|-1}$, $j=1,2,\ld,k$, are
mutually incomparable and the collection
$$\cW^\ast=\{\pi([\om(z_j)|_{|\om(z_j)|-1}]):j=1,2,\ld,k\}$$ covers the set $J_\cS\cap B(x,r)$. Notice that since $$\diam(\f_{\om(z_j)|_{|\om(z_j)|-1}}(\Int(X_{t({\om(z_j)|_{|\om(z_j)|-1}}})))\ge
\^r$$ for all $j=1,\dots,k$ it follows from Lemma~\ref{l1j81A} and
Remark~\ref{o120130613} that
$
k\le(3 \Lambda K)^Q\g_{\La M \kappa^{-1}}^{-1}.
$
where $\kappa=2KC\La M\eta_\cS^{-1}$. Therefore
\begin{equation}
\label{cardw}
k \leq (3 \Lambda K)^Q\g_{\eta_\cS (2KC)^{-1}}.
\end{equation}

Now temporarily fix an element $z\in
\{z_1,z_2,\ld,z_k\}$, set $\om=\om(z)$, $q:=|\om|$, and $\psi:=\f_{\om|_{q-
1}}$.
By the choice of $\om$, \eqref{quasi-multiplicativity} and (\ref{4.1.10a}) imply
\begin{equation*}
2 \tilde{r} \geq \diam(\f_\om(X_{t(\om)})) \geq 2(KC)^{-1} \|D \f_\om\|_\infty \tilde{R}_\cS \geq 2 K^{-2}C^{-1}\|D \f_{\om_q}\|_\infty  \|D \psi\|_\infty \tilde{R}_\cS.
\end{equation*}
Therefore using (\ref{4.1.9}) we deduce that
\begin{equation}
\label{note5}
 \diam(\f_{\om_q}(X_{t(\om_q)})) \leq \La M \|D \f_{\om_q}\|_\infty \leq \La M K^{2} C \tilde{R}_{\cS}^{-1}\,\|D \psi\|^{-1}_\infty \tilde{r}.
\end{equation}
By the assumption (c) there exists some $y\in \f_{\om_q}(X_{t(\om_q)})$,
corresponding to the radius $\g 	3\La MK^2C \tilde{R}_\cS^{-1}||D\psi||_\infty^{-1}\^r\ge
\g\diam(\f_{\om_q}(X_{t(\om_q)}))$, such that
\begin{equation}
\label{geommh}
m_h(B(y,3\g \La MK^2C \tilde{R}_\cS^{-1}||D\psi||_\infty^{-1}\^r)) \leq H (3\g\La MK^2C \tilde{R}_\cS^{-1}||D\psi||_\infty^{-1}\^r)^h.
\end{equation}

Now notice that by \eqref{105302013} and \eqref{4.1.9}
\begin{equation}
\label{3pr}
\tilde{r} \leq \La M \|D \psi\|_\infty,
\end{equation}
therefore
\begin{equation}
\label{etas}
2KC\|D \psi\|^{-1}_\infty \,r \leq 2KC\La M \tilde{r}^{-1}r= \eta_S.
\end{equation}
Since $z \in \pi([\om])$, we have that
$\psi^{-1}(z)\in \phi_{\om_q}(X_{t(\om_q)})$. Hence  by Corollary \ref{l42013_03_12}  and \eqref{etas},
\begin{equation}
\label{39app}
\psi\(B(\psi^{-1}(z),2rKC|| D\psi ||_\infty^{-1})\) \supset B(z,	2r) \supset B(x,r).
\end{equation}
Noting that $y,  \psi^{-1}(z) \in \f_{\om_q}(X_{t(\om_q)})$ and using \eqref{note5} we have that
\begin{equation*}
\begin{split}
B(\psi^{-1}(z),2KC \|D \psi\|^{-1}_\infty r) &\subset B(y,2KC \|D \psi\|^{-1}_\infty r+\diam(\f_{\om_q}(X_{t(\om_q)})))\\
&\subset B(y,(2KC+\La MK^2C\tilde{R}_\cS^{-1})||D\psi||_\infty^{-1}\^r ).
\end{split}
\end{equation*}
Therefore using \eqref{39app} and recalling that $\tilde{r} \geq r$ we obtain
\begin{equation}
\label{note7}
B(x,r) \subset \psi\(B(y,3\La MK^2C\tilde{R}_\cS^{-1} ||D\psi ||_\infty^{-1}\^r\),
\end{equation}
where we also used the fact that $\La M \tilde{R}_\cS^{-1} \geq 1$. So, employing Proposition~\ref{p320130613}, \eqref{note7} , \eqref{geommh} and recalling that $\gamma \geq 1$ we get
\begin{equation}
\begin{split}
\label{note8}
m_h\(B(x,r)\cap\psi(X_{t(\om_{q-1})})\)
&\le m_h\(\psi(X_{t(\om_{q-1})})\cap\psi(B(y,3\La MK^2C\tilde{R}_\cS^{-1}||D\psi||_\infty^{-1}\^r))\) \\
&= m_h\(\psi\(X_{t(\om_{q-1})}\cap B(y,3\La MK^2C\tilde{R}_\cS^{-1}||D\psi||_\infty^{-1}\^r)\)\) \\
&\le ||D\psi||_\infty^h m_h\(X_{t(\om_{q-1})}\cap B(y,3\La MK^2C\tilde{R}_\cS^{-1}||D\psi||_\infty^{-1}\^r)\) \\
&\le ||D\psi||_\infty^hm_h\(B(y,3\g\La MK^2C\tilde{R}_\cS^{-1}||D\psi||_\infty^{-1}\^r)\) \\
&\le ||D\psi||_\infty^hH\(3\g\La MK^2C\tilde{R}_\cS^{-1}||D\psi||_\infty^{-1}\^r\)^h \\
&=H\left(\frac{6\g\La^2 M^2K^3C^2}{\tilde{R}_\cS \eta_\cS}\right)^h r^h.
\end{split}
\end{equation}
By the definition of $\cW^\ast$ we see that
$$
\{\f_{\om(z_j)|_{|\om(z_j)-1|}}(X_{t(\om(z_j)|_{|\om(z_j)-1|})})\}_{j=1}^k
$$
covers the set $J_\cS \cap B(x,r)$. Finally by \eqref{note8}, \eqref{cardw}, and Remark \ref{o120130613}, since the words $\om(z_j)|_{|\om(z_j)-1|}$ are mutually incomparable, we get
\begin{equation*}
\begin{split}
m_h(B(x,r))
&\le \sum_{j=1}^k m_h(B(x,r) \cap \f_{\om(z_j)|_{|\om(z_j)-1|}}(X_{t(\om(z_j)|_{|\om(z_j)-1|})})) \\
&\leq \#\cW^\ast \, H\left(\frac{6\g\La^2 M^2K^3C^2}{R_\cS \eta_\cS}\right)^h r^h \\
&\leq H (3 \Lambda K)^Q\g_{\eta_\cS (2KC)^{-1}} \left(\frac{6\g\La^2 M^2K^3C^2}{R_\cS \eta_\cS}\right)^h r^h\\
&:= c_\cS r^h,
\end{split}
\end{equation*}
and (d) is proved.

The implication (d)$\imp$(a) is an immediate consequence of
Frostman's Lemma (see for example \cite{MSU} or \cite{MUGDMS}). \index{Frostman's Lemma}
Thus the whole theorem has been proved.
\end{proof}

\begin{remark}\label{r4.4.4}
It is obvious that it suffices for the above proof that conditions
(b) and (c) of Theorem~\ref{t4.4.3} be satisfied for a cofinite
subset of $I$.
\end{remark}

\begin{remark}\label{4220130613}
Notice that in the proof of Theorem~\ref{t4.4.3} the corkscrew
condition was only needed to establish the implication (c)$\imp$(d).
\end{remark}

\begin{remark}\label{r1_2015_08_26}
If condition (d) of Theorem~\ref{t4.4.3} holds, then the stated inequality holds in fact for all $x\in \overline J_\cS$.
\end{remark}

We now move to the second main result of this section, which provides necessary and
sufficient conditions for the packing measure of $J_\cS$ to be finite. In particular we prove that under certain conditions $\cP^h(J_\cS)$ is finite
if and only if the conformal measure $m_h$ is lower geometric with
exponent $h$. Before proving the theorem we make the following observation.

\begin{remark}\label{sosc2}
If $\cS=\{\f_e:e\in E\}$ is a maximal, finitely irreducible Carnot conformal  GDMS satisfying the strong open set condition, then
$$J_\cS \cap \Int(\f_e(X_{t(e)})) \neq \emptyset$$
for all $e \in E$.
\end{remark}

\begin{proof} Let $e \in E$. By the strong open set condition there exist $v \in V$ and $e_0 \in E$ such that $X_v=X_{t(e_0)}$ and $\Int(X_{t(e_0)}) \cap J_\cS \neq \emptyset$. Since $\cS$ is finitely irreducible there exists $\om  \in E^\ast_A$ such that $e\om e_0 \in E^\ast_A$. Let $x_0=\pi(v)  \in J_\cS \cap \Int(X_{t(e_0)})$. Then $x_0 \in X_{t(e_0)} \cap X_{i(v_1)}$, hence $t(e_0)=i(v_1)$ and by the maximality of $\cS$ we deduce that $A_{e_0v_1}=1$. In particular $\f_{e_0}(x_0) \in J_\cS$, and
\begin{equation}
\label{eomeo}
\f_{e \om e_0}(x_0) \in J_\cS.
\end{equation}
Moreover,
\begin{equation}
\label{eomeo2}
\f_{e \om e_0}(x_0) \in  \f_e(\Int( \f_{\om e_0}(X_t(e_0)))) \subset \Int(\f_e(X_{i(\om_1)}))=\Int(\f_e(X_{t(e)})).
\end{equation}
The proof follows by \eqref{eomeo} and \eqref{eomeo2}.
\end{proof}

\begin{definition}
\label{thinboundary}
We say that a Carnot conformal  GDMS $\cS=\{\f_e\}_{e\in E}$ has \textit{thin boundary} \index{thin boundary} if there exists some $\varepsilon>0$ such that
$$
\sharp \{e \in E: \f_e(X_{t(e)}) \cap B(\partial X, \varepsilon) \neq \emptyset\}<\infty.
$$
\end{definition}
Of course a Carnot conformal GDMS $\cS=\{\f_e\}_{e\in E}$ has thin boundary if it holds that the itersection $\f_e(X_{t(e)}) \cap B(\partial X, \varepsilon)$ is empty for some $\varepsilon>0$.

\begin{theorem}\lab{t4.4.5}
Let $\cS=\{\f_e\}_{e\in E}$ be a maximal regular Carnot conformal  GDMS with thin boundary which
satisfies the strong open set condition and the weak corkscrew condition. Then the following conditions
are equivalent:
\begin{itemize}
\item[(a)] $\cP_h(J_\cS)<+\infty$.

\item[(b)] There are three constants $H>0$, $\xi>0$, and $\g> 1$ such
that for every
$e\in E$ and every $r \in [\g\diam(\f_e(X_{t(e)})),\xi]$ there exists $y\in
\f_e(X_{t(e)})$ such that $m_h(B(y,r))\ge Hr^h$.

\item[(c)] The measure $m_h$  is lower-geometric with exponent
$h$. This precisely means that there exists a constant $c_\cS\in
(0,+\infty)$ such that $m_h(B(x,r))\ge c_\cS r^h$ for every $x\in J_\cS$
and all $r\in[0,\diam(X)]$.
\end{itemize}
\end{theorem}

\begin{proof}
(a)$\imp$(b). By Lemma \ref{sosc2} for every $e \in E$ there exists some $x_e \in J_\cS \cap \Int(\f_e(X_{t(e)}))$. For $e \in E$ set $d_e=d(x_e, \partial X)$. Since $\cS$ has thin boundary there exists some $\varepsilon>0$ such that the set $E_\e=\{e \in E: \f_e(X_{t(e)}) \cap B(\partial X, \varepsilon) \neq \emptyset\}$ is finite. Therefore if $\delta_\varepsilon:=\frac{1}{2}\min\{d_e:e \in E_e\}$ we get that $\min_{e \in E} d_\e > \delta_\varepsilon.$

By way of contradiction assume that (b) fails. Fix $H>0$, $\xi\in (0,\min\{\eta_S, \delta_\ve\})$ and $\gamma>1$. Then there exist  $e\in E$ and a radius $r$ with
$\diam(\phi_e(X_{t(e)}))< r\le\xi$ such that for every $ y \in
\phi_e(X_{t(e)})$, we have
$$
m_h(B(y,r))\le H\, r^h.
$$
Notice that   $$B(x_e, r/2) \subset \Int X_{i(e)} \setminus B(\partial X, \delta_\e)$$ because  $d(x_e, \partial X) > \delta_\e$ and $r \leq \xi <\delta_\ve$. Hence if $x_e=\pi( \om^e)$ for some $\om^e \in E_A^\N$ then there exists some $n_0 \in \N$ such that $\pi([\om^e]) \subset B(x_e, r/2)$. By \eqref{2j93} and Theorem \ref{thm-conformal-invariant} we deduce that $\tilde\mu_h (\pi^{-1}(B(x_e, r/2)))>0$. Therefore Birkhoff's Ergodic Theorem implies that if \index{Birkhoff's Ergodic Theorem}
$$A=\{ \om \in E_A^\N: \sigma^n(\om) \in \pi^{-1}(B(x_e, r/2)) \text{ for infinitely many }n\}$$
then $\tilde \mu_h(A)=1$. Hence by Theorem \ref{thm-conformal-invariant} $\tilde m_h (A)=1$, and if $B=\pi(A) \subset J_\cS$ then
$m_h(B)=1$  and for every $z =\pi(\om)\in B, \om \in A,$  $$\pi(\sigma^n(\omega)) \in B(x_e,r/2)$$ for infinitely many
$n$'s.  Notice that for such a point $z$ and such an integer $n\ge 1$, $\pi (\sg^n(\om)) \in X_{t(\om_n)}=X_{i(e)}$ and $d(\pi (\sg^n(\om)), \partial X_{i(e)})> \delta_\e$ therefore
\begin{equation}
\label{inside}
B(\pi (\sg^n(\om)), \delta_\e)\subset \Int (X_{i(e)}).
\end{equation}
Hence by Remark~\ref{820rem}, recalling that $r \leq \xi < \delta_\ve$,
$$
\aligned
m_h(\phi_{\omega|_n}(B(\pi(\sigma^n(\omega)) ,r/2)))
  & \le  \|D\phi_{\omega|n}\|_\infty^{h}m_h(B(\pi(\sigma^n(\omega)),r/2) \\
&\le \|D\phi_{\omega|_n}\|_\infty^{h}m_h(B(x_e,r)) \le \|D\phi_{\omega|_n}\|_\infty^{h}H\,r^h.
\endaligned
$$
By (\ref{4.1.8}), since $z=\f_{\om|_k}(\pi(\sigma^n(\om))),$
$$
\phi_{\omega|_n}\(B(\pi(\sigma^n(\omega)),r/2)\)
\supset B(z,\|D\phi_{\omega|_n}\|_ \infty (KC)^{-1}r/2).
$$
So,
\begin{equation*}
\begin{split}
m_h&\(B(z,(KC)^{-1}\|D\phi_{\omega|_n}\|_\infty r/2)\) \leq H \|D\phi_{\omega|_n}\|_\infty^h r^h.
\end{split}
\end{equation*}
We can now apply Frostman's Lemma \index{Frostman's Lemma} (as in \cite{MSU}, or as in \cite[Theorem A.2.013]{MUGDMS}) and obtain
$$
\cP^h(J_\cS)\gtrsim (2KC)^{-h} \,H^{-1}.
$$
So letting $H\downto 0$, we get $\cP^h(J_\cS)=\infty$. This
finishes the contrapositive proof of the implication (a)$\imp$(b).

\sp\fr (b)$\imp$(c).
\sp\fr (b)$\imp$(c).
First notice that by \eqref{4.1.8} for all $\om \in E_A^\ast$ and $p\in X_{t(\om)},$
\begin{equation}
\label{l46pr}
\f_\om(N_{t(\om)}) \supset \f_\om(B(p, 4^{-1}\eta_\cS)) \supset B(\f_\om(p), (4KC)^{-1} \|D \f_\om\|_\infty \eta_\cS),
\end{equation}
where the sets $N_v=B(X_v, \dist(X_v, \partial S_v)/2), v  \in V$ where defined in Remark \ref{nsets}.

First notice that decreasing $H$ if necessary, the
assumption of the lemma continues to be fulfilled if the number $\xi$ is
replaced by any other positive number, for example by $\eta_\cS/4$. Fix $0<r<\xi$, $x=\pi(\om)\in J_\cS$,
and take maximal $k\in \N$ such that
\begin{equation}\lab{4.4.1}
\f_{\om|_k}(N_{t(\om_k)})\spt B(x,(4KC \La_0 M_0)^{-1}\eta_\cS r),
\end{equation}
where $M_0=\diam (\cup_{v \in V} N_{v})$. By \eqref{l46pr}, since $x= \f_{\om|_{k+1}}(\pi(\sg^{k+1}(\om)))$ and $\pi(\sg^{k+1}(\om)) \in X_{t(\om_{k+1})},$
$$\f_{\om|_{k+1}}(N_{t(\om_{k+1})}) \supset B(x,(4KC)^{-1} \|D \f_{\om|_{k+1}}\|_\infty \eta_\cS).$$
By the maximality of $k$, $\f_{\om|_{k+1}}(N_{t(\om_{k+1})})$ does not
contain $B(x, (4KC\La_0 M_0)^{-1}\eta_\cS r)$, hence $(4KC\La_0 M_0)^{-1}\eta_\cS r> (4KC)^{-1} \|D \f_{\om|_{k+1}}\|_\infty \eta_\cS,$ or equivalently
\begin{equation}
\label{462}
r > \La_0 M_0 \|D \f_{\om|_{k+1}}\|_\infty.
\end{equation}
Hence, using (\ref{4.1.9}), we get
$$
B(x,r)\spt B(x,\La_0 M_0 \|D \f_{\om|_{k+1}}\|_\infty)\spt\f_{\om|_{k+1}}(X_{t(\om_{k+1})}).
$$
Hence by Propositions~\ref{p120130617} and \ref{p220130618}, we get
\begin{equation}
\label{464}
\begin{split}
m_h(B(x,r)) &\ge K^{-h}||D \f_{\om|_{k+1}}||_\infty^h\,m_h\(X_{t(\om_{k+1})} \cap \pi(J^{-}_{\om|_{k+1}})\)
\\ &\geq M_h K^{-h}||D \f_{\om|_{k+1}}||_\infty^h,
\end{split}
\end{equation}
where as in Proposition \ref{p220130618}, $M_h= \inf \{m_h(\pi(J_e^-)):e \in E\}$. Now notice that by \eqref{l45pr}
\begin{equation}
\label{465}
\f_{\om|_k}(N_{t(\om_k)}) \subset B(x, \La_0 M_0 \|D \f_{\om|_k}\|_\infty),
\end{equation}
therefore, using \eqref{4.4.1}, $M_0 \La_0\|D \f_{\om|_k}\|_\infty> (4KC \La_0 M_0)^{-1} \eta_\cS r$ or equivalently,
\begin{equation}
\label{466}
\frac{r}{\|D \f_{\om|_k}\|_\infty}< \frac{4(\La_0 M_0)^2 KC}{\eta_{\cS}}.
\end{equation}
Put
$$
\a:=\min\left\{\frac{\eta_\cS^2}{8(\La_0 M_0)^3 K^2 C},\frac{1}{2 \La M C K}\right\}.
$$
Notice that by the choice of $\a$ and \eqref{466}, we have that
\begin{equation}
\label{467}
2 \La_0 M_0 K \a \|D \f_{\om|_k}\|_\infty^{-1}r< \eta_\cS.
\end{equation}
We now consider two cases. First, if $\g||D \f_{\om|_{k+1}}||_\infty \ge \a r$, then by \eqref{464}
$$
m_h(B(x,r))\ge (\a (\g K)^{-1})^{h} M_h r^{h},
$$
and we are done. Otherwise, i.e. the second case,
\begin{equation}\lab{4.4.2}
\g||D \f_{\om|_{k+1}}||_\infty <\a r.
\end{equation}
By \eqref{quasi-multiplicativity} and \eqref{4.1.9} we have that
$$\diam( \f_{\om_{k+1}}(X_{t(\om_{k+1})})) \leq \La M K \|D \f_{\om|_{k+1}}\|_\infty \|D \f_{\om|_k}\|^{-1}_\infty.$$
Hence, using also \eqref{4.4.2}, if $y \in \f_{\om_{k+1}}(X_{t(\om_{k+1})})$,
\begin{equation}
\label{969}
B(y, \La M K \alpha \|D \f_{\om|_k}\|^{-1}r) \subset B(\pi(\sg^k(\om)), 2 \La M K \alpha \|D \f_{\om|_k}\|_\infty^{-1}r).
\end{equation}
Now by \eqref{4.1.8} and \eqref{467}
\begin{equation}
\label{9610}
\f_{\om|_k}(B(\pi(\sg^k(\om)), 2 \La M K \alpha \|D \f_{\om|_k}\|_\infty^{-1}r)) \subset B(x,2 C\La M K \alpha r) \subset B(x,r).
\end{equation}
Now notice that by \eqref{4.4.2}, \eqref{quasi-multiplicativity} and \eqref{4.1.9}
\begin{equation}
\label{9611}
\La M K \alpha \|D \f_{\om|_k}\|_\infty^{-1}r \geq \La M \gamma \|D \f_{\om|_{k+1}}\|_\infty \geq \gamma \diam(\f_{\om_{k+1}}(X_{t(\om_{k+1})})).
\end{equation}
By \eqref{969}, \eqref{467} and Remark \ref{cor827rem},
\begin{equation*}
\begin{split}
m_h(\f_{\om|_k}(B(\pi(\sg^k(\om)), 2 \La M K \alpha \|D \f_{\om|_k}\|_\infty^{-1}r))) \geq \|D \f_{\om|_k}\|^h_\infty m_h(B(y, \La M K \alpha \|D \f_{\om|_k}\|^{-1}r)).
\end{split}
\end{equation*}
Also by \eqref{9611} and the assumption (b)
$$m_h(B(y, \La M K \a \|D \f_{\om|_k}\|_\infty^{-1}r)) \geq H \, (\La M K \a \|D \f_{\om|_k}\|_\infty^{-1}r)^h.$$
Hence
$$m_h(\f_{\om|_k}(B(\pi(\sg^k(\om)), 2 \La M K \alpha \|D \f_{\om|_k}\|_\infty^{-1}r))) \geq H (\La M K \a)^h \, r^h$$
and by \eqref{9610}
$$m_h(B(x,r)) \geq H (\La M K \a)^h \, r^h.$$

The implication (c)$\imp$(a) is an immediate consequence of
Frostman's Lemma \index{Frostman's Lemma} (see for example \cite{MSU} or \cite{MUGDMS}).
Thus the whole theorem has been proved.
\end{proof}

We close this section with a few remarks and observations concerning Theorem \ref{t4.4.5}.

\begin{remark}\label{r4.4.6}
It is obvious that in Theorem~\ref{t4.4.5} conditions
(b) and (c) suffices to be satisfied for a cofinite
subset of $E$.
\end{remark}

\begin{remark}\label{mnotesB}
Notice that we do not need to assume that $\cS$ has thin boundary and satisfies the SOSC in order to prove the implications (b)$\imp$(c)$\imp$(a).
\end{remark}

Moreover following the reasoning of the proof of implication (a) $\imp$ (b) we can prove the following weaker statement where we do not need to assume the thin boundary and weak corkscrew conditions.

\begin{proposition}\label{mnotesA}
Let $\cS=\{\f_e\}_{e\in E}$ be a maximal regular Carnot conformal  GDMS which satisfies the SOSC. If $\cP_h(J_\cS)<+\infty$ then there are two constants $H>0$ and $\g > 1$ such
that for every $e\in E$, every $r \in [\g\diam(\f_e(X_{t(e)})),\eta_\cS]$ and every $x \in \f_e(X_{t(e)})$ such that $B(x,r) \subset X_{i(e)}$, $m_h(B(x,r))\ge Hr^h$.
\end{proposition}

\section[Hausdorf and packing measures for continued fraction systems]{Hausdorff and packing measures for continued fraction systems in groups of Iwasawa type}

We apply the general theorems about Hausdorff and packing measures, proved in the previous section, to the class of continued fraction systems introduced in Section~\ref{iwacf:sec} and further explored in Section~\ref{CF:II}. Our main theorem in this section reads as follows.

\begin{theorem}
Let $\G$ be a Carnot group of Iwasawa type and let $\ve > 0$. Let $\cS_\ve=\{\f_\gamma \}_{\gamma \in I_\ve}$ be the corresponding continued fraction conformal iterated function system. Then
$$
\cH_{h_\ve}(J_{\cS_\ve})=0 \  \  \text{ and } \  \  0<\cP_{h_\ve}(J_{\cS_\ve})<+\infty,
$$
where, we recall, $h_\ve:=\dim_\cH J_{\cS_\ve}$ is the Hausdorff dimension of the limit set $J_{\cS_\ve}$.
\end{theorem}

\begin{proof} We want to apply Theorem~\ref{t4.4.3} and Theorem~\ref{t4.4.5} respectively to prove the first and the second assertion of our current theorem. The hypotheses of thin boundary, the strong open set condition, and the weak corkscrew condition, required to apply these two theorems, are immediate from the definition of the system $\cS_\ve$.

\sp We know from Theorem~\ref{cftheta} that the system $\cS_\ve$ is co-finitely regular. Let then $m_\ve$ be the corresponding $h_\ve$-conformal measure. As always, put
$$
X:=\overline{B}(o,1/2).
$$
Formula \eqref{cf} yields
\begin{equation}\label{cf_B}
\|D \f_\gamma\|_\infty \approx d(\gamma,o)^{-2}
\end{equation}
for all $\gamma\in I_\ve$. It also immediately follows from \eqref{cf}, in fact the calculation of \eqref{cfderiv1} does it, that
\begin{equation}\label{f1_2015_08_26}
\diam(\f_\gamma(X))\approx d(\gamma,o)^{-2}
\end{equation}
for all $\gamma\in I_\ve$. It furthermore follows from \eqref{fginside} that
\begin{equation}\label{f2_2015_08_26}
\f_\gamma(X)\subset \overline{B}\left(o,K_\ve d(\gamma,o)^{-1}\right)
\end{equation}
with some $K_\ve>0$ for all $\gamma\in I_\ve$. Now for every $r>0$ let
$$
I(r):=\{\gamma\in I_\ve:r/2<K_\ve d(\gamma,o)^{-1}<r\}
     =\{\gamma\in I_\ve:K_\ve r^{-1}< d(\gamma,o)<2K_\ve r^{-1}\}.
$$
Improving in a straightforward way the arguments leading to \eqref{upbouik} and \eqref{lobouik}, we get that there exists some $\beta_\ve>1$ such that
\begin{equation}\label{f3_2015_08_26}
\beta^{-1}_\ve r^{-Q}\le \sharp I(r)\le \beta_\ve r^{-Q}
\end{equation}
for all $\gamma\in I_\ve$ and every $r\in(0,1)$ small enough. Therefore, using Theorem \ref{t420130613}, Proposition \ref{p320130613} and \eqref{cf_B}-\eqref{f3_2015_08_26}, we get
\begin{equation}\label{f1_2015_08_28}
\begin{aligned}
m_\ve (B(o,r))
&\ge \sum_{\gamma\in I(r)}m_\ve(\phi_\gamma(X))
\approx \sum_{\gamma\in I(r)} \|D \f_\gamma\|_\infty^{h_\ve}
\approx \sum_{\gamma\in I(r)} d(\gamma,o)^{-2h_\ve} \\
&\approx \sharp I(r)r^{2h_\ve}
\approx r^{2h_\ve-Q}.
\end{aligned}
\end{equation}
Therefore, by virtue of Theorem~\ref{cfhaus}, we get that
$$
\limsup_{r\to 0}\frac{m_\ve (B(o,r))}{r^{h_\ve}}
\gtrsim \limsup_{r\to 0}r^{h_\ve-Q} =+\infty.
$$
This in turn, in conjunction with Theorem~\ref{t4.4.3} and Remark~\ref{r1_2015_08_26} entails that $\cH_{h_\ve}(J_{\cS_\ve})=0$, and the first part of our theorem is thus proved.

Passing to the second part, note that $\cP_{h_\ve}(J_\cS)>0$ follows from Proposition~\ref{p4201305081}. In order to prove that $\cP_{h_\ve}(J_\cS)<+\infty$ we will check that condition (b) of Theorem~\ref{t4.4.5} holds. So, let $\gamma\in I_\ve$ be arbitrary and let $x\in\phi_\gamma(X)$. Fix a radius
\begin{equation}\label{5_2015_08_27}
r\in \( \rho \diam(\f_\gamma(X)),1)
\end{equation}
with a sufficiently large constant $\rho>0$, independent of $\gamma$, $x$, and $r$, to be determined in the course of the proof. We will consider three cases. Assume first that
\begin{equation}\label{f1_2015_08_27}
r\le d(\cJ(\gamma),0).
\end{equation}
It then follows from \eqref{conformal-inversion-one} and \eqref{conformal-inversion-two} that
$$
\begin{aligned}
\inf\{d(\cJ(x),\gamma): &x\in \partial B(\cJ(\gamma),r)\}= \\
&=\inf\{d(\cJ(x),\cJ(\cJ(\gamma))):x\in \partial B(\cJ(\gamma),r)\} \\
&=\inf\left\{\frac{d(x,\cJ(\gamma))}{d(x,o)d(\cJ(\gamma),o)}:x\in \partial B(\cJ(\gamma),r)\right\}  \\
&=rd(\gamma,o)\inf\left\{\frac{1}{d(x,o)}:x\in \partial B(\cJ(\gamma),r)\right\}.
\end{aligned}
$$
But
$$
d(x,o)\le d(x,\cJ(\gamma))+d(\cJ(\gamma),o)=r+\frac1{d(\gamma,o)}\le \frac2{d(\gamma,o)}.
$$
Therefore,
$$
\inf\{d(\cJ(x),\gamma):x\in \partial B(\cJ(\gamma),r)\}\ge \frac12{d(\gamma,o)^2r}.
$$
Hence,
$$
\cJ(B(\cJ(\gamma),r))\supset B\left(\gamma,\frac12{d(\gamma,o)^2r}\right).
$$
Now let
$$
I_\gamma(r):=\{g\in I_\ve:\ov B(g,1/2)\sbt B\left(\gamma,\frac12{d(\gamma,o)^2r}\right).
$$
With considerations analogous to those leading to \eqref{upbouik} and \eqref{lobouik} (see also \eqref{f3_2015_08_26}), we get
\begin{equation}\label{1cth1}
\sharp I_\gamma(r)
\approx\sharp\left\{g\in \G(\Z): \ov B(g,1/2)\subset
       B\left(o,\frac12{d(\gamma,o)^2r}\right)\right\}
\approx d(\gamma,o)^{2Q}r^Q.
\end{equation}
Observe that for every $g\in I_\gamma(r)$ we have
\begin{equation}
\label{bjysubset}
\begin{aligned}
\phi_g(X)
&=\cJ\circ \ell_g(X)(\ov B(g,1/2))
\subset \cJ\left(B\left(\gamma,\frac12{d(\gamma,o)^2r}\right)\right) \\
&\subset \cJ(\cJ(B(\cJ(\gamma),r))) \\
&=B(\cJ(\gamma),r).
\end{aligned}
\end{equation}
Therefore by \eqref{bjysubset}, Theorem \ref{t420130613}, Corollary \ref{c120130621} and \eqref{cf_B}, we obtain
$$
\begin{aligned}
m_\ve (B(\phi_\gamma(o),r))
&=m_\ve((B(\cJ(\gamma),r))
\ge \sum_{g\in I_\gamma(r)}m_\ve(\phi_g(X)) \\
&\approx  \sum_{g\in I_\gamma(r)}||D\phi_g||_\infty^{h_\ve} \\
&\approx  \sum_{g\in I_\gamma(r)}d(g,o)^{-2h_\ve}.
\end{aligned}
$$
But, for every $g\in I_\gamma(r)$,
$$
\begin{aligned}
d(g,o)
&\le g(g,\gamma)+d(\gamma,o)
 \le \frac12{d(\gamma,o)^2}r+d(\gamma,o) \\
& \le \frac12d(\gamma,o)+d(\gamma,o) \\
&\le 2d(\gamma,o).
\end{aligned}
$$
Notice also that by \eqref{f1_2015_08_26} and the fact that $\diam(\f_\gamma(X)) \geq \rho^{-1} r$ we deduce that
\begin{equation}
\label{833exp}
d(\gamma,0)^2 \gtrsim r^{-1}.
\end{equation}
Therefore, using  \eqref{1cth1} along with \eqref{5_2015_08_27} and \eqref{833exp}, we get that
$$
\begin{aligned}
m_\ve (B(\phi_\gamma(o),r))
&\gtrsim \sharp I_\gamma(r)d^{-2h_\ve}(\gamma,o)
 \approx d(\gamma,o)^{2(Q-h_\ve)}r^Q \\
&\gtrsim r^{h_\ve-Q}r^Q \\
&=r^{h_\ve},
\end{aligned}
$$
and we are done in this case. As for the second case, suppose that $d(J(\gamma),o)<r\le 2d(J(\gamma),o)$. Then $d(J(\gamma),o)\ge r/2$, and so, using what we have obtained in the first case,we get
$$
m_\ve (B(\phi_\gamma(o),r))
\ge m_\ve (B(\phi_\gamma(o),r/2))
\gtrsim (r/2)^{h_\ve}
=2^{-h_\ve}r^{h_\ve},
$$
and we are done in this case too. Finally, if $r\ge 2d(J(\gamma),o)$, then $B(\phi_\gamma(o),r)\spt B(o,r/2)$ and, in view of \eqref{f1_2015_08_28}, along with the already noted fact that $h_\ve<Q$, we get
$$
m_\ve (B(\phi_\gamma(o),r))
\ge m_\ve (B(o,r/2))
\gtrsim (r/2)^{2h_\ve-Q}
=2^{Q-2h_\ve}r^{h_\ve-Q}r^{h_\ve}
\gtrsim r^{h_\ve},
$$
and we are done in this case as well. The proof is complete.
\end{proof}

\chapter{Equivalent separation conditions for finite GDMS} \label{chap:separation-equiv}

In this chapter we consider the problem of finding equivalent separation conditions for finite graph directed Markov systems (GDMS). We record that the topic of equivalent separation conditions for iterated function systems (IFS) has attracted considerable attention and has been investigated from various viewpoints, see \cite{Schief:separation}, \cite{Schief:complete}, \cite{PRSS}, \cite{lau:sep}, \cite{KaVI},  \cite{BR}, \cite{RV}. In the following we prove that for a finite irreducible weakly Carnot conformal GDMS $\cS$, the open set condition, the strong open set condition and the positivity of the $h$-dimensional Hausdorff measure of the limit set $J_\cS$, are equivalent conditions. 

For self-similar Euclidean iterated function systems this equivalence is a celebrated result of Schief \cite{Schief:separation}. Peres, Rams, Simon and Solomyak \cite{PRSS} provided a beautiful proof of the aforementioned equivalence for conformal Euclidean iterated function systems.  Our main result in this Chapter, namely Theorem \ref{prss}, extends the result of Peres, Rams, Simon and Solomyak in a twofold manner. Our theorem is valid for the broader class of conformal graph directed Markov systems (IFS are examples of GDMS) and it also holds on general Carnot groups.  We stress that the equivalences proved in Theorem \ref{prss}  involve GDMS and they are new even in Euclidean spaces. Although our proof follows the scheme of Peres, Rams, Simon and Solomyak from \cite{PRSS}, many nontrivial modifications are needed, partly because of the sub-Riemannian structure of $\G$ and partly because we work with GDMS.

If $\cS=\{\f_e\}_{e \in E}$ is a weakly Carnot conformal GDMS we will use the notation $$J_\om:=J_{\cS, \om}:=\f_\om (J_\cS \cap X_{t(\om)}),$$
for $\om \in E_A^\N$.

\begin{definition}
\label{erclose}
Let $\cS=\{\f_e\}_{e \in E}$ be a weakly Carnot conformal GDMS on $(\G, d)$ and let $\ve>0$. The finite words $\om, \tau \in E_A^\ast$ are \textit{$\ve$-relatively close} if $t(\om)=t(\tau)$ and for every $p  \in J_\cS \cap X_{t(\om)}$
$$d(\f_\om(p),\f_\tau(p))< \ve \min\{\diam(J_\om), \diam(J_\tau)\}.$$
\end{definition}

\begin{remark}
\label{keyercloserem}
Notice that if $\om, \tau$ are $\ve$-relatively close and $$\ve< \frac{ \min\{ \dist (X_{v_1},X_{v_2}):v_1,v_2 \in V\}}{\max \{ \diam (X_v):v \in V\}}:=\lambda_\cS$$ then $i(\om)=i(\tau)$.
\end{remark}

The following condition was introduced by Bandt and Graf in \cite{BG}.

\begin{definition} 
A weakly Carnot conformal GDMS $\cS=\{\f_e\}_{e \in E}$ on $(\G, d)$ satisfies the \index{Bandt-Graf condition} \textit{Bandt--Graf condition} if $\sharp(J_\cS \cap X_v)>1$ for all $v \in V$ and there exists some $\ve>0$ such that for every $\om, \tau \in E_A^\ast$ with $t(\om)=t(\tau)$ the maps $\f_\om$ and $\f_\tau$ are not $\ve$-relatively close.
\end{definition}
\begin{definition}
\label{equivgdms}
\index{GDMS!equivalent}
Two weakly Carnot conformal GDMS $\cS$ and $\cS'$ in $(\G,d)$ are called \textit{equivalent} if:
\begin{itemize}
\item[(i)] they share the same associated directed multigraph $(E,V)$,
\item[(ii)] they have the same incidence matrix $A$ and the same functions $i,t:E \ra V$,
\item[(iii)]  they are defined by the same set of conformal mappinps $\{\f_e : W_{t(e)} \ra W_{i(e)}\}$, where $W_v$ are open connected sets, and for every $v \in V$, $X_v \cup X'_v \subset W_v$.
\end{itemize}
\end{definition}
We state the following straightforward observation as a separate remark.
\begin{remark}
\label{equivlimset}
If $\cS$ and $\cS'$ are two equivalent weakly Carnot conformal GDMS then $J_\cS = J_{\cS'}$.
\end{remark}

Our main result in this chapter reads as follows.

\begin{theorem}
\label{prss}
Let $\cS$ be a finite, irreducible and maximal weakly Carnot conformal GDMS on $(\G,d)$. Then the following conditions are equivalent.
\begin{itemize}

\sp\item[(i)] There exists a Carnot conformal GDMS $\cS'$ which is equivalent to $\cS$.

\sp\item[(ii)] $\cH^h(J_{\cS})>0$ where $h$ is Bowen's parameter. \index{Bowen's parameter}

\sp\item[(iii)] There exists a weakly Carnot conformal $\cS'$ equivalent to $\cS$ which satisfies the Bandt--Graf condition. \index{Bandt-Graf condition}

\sp\item[(iv)] There exists a Carnot conformal GDMS $\cS'$ equivalent to $\cS$ which satisfies the strong open set condition. \index{strong open set condition}
\end{itemize}
\end{theorem}

Recall that the strong open set condition (SOSC) was introduced in Definition \ref{SOSCdef}.

\begin{remark} If $\cS$ is any finite and irreducible GDMS then by Proposition \ref{finitemaxreduction} there exists another finite and irreducible GDMS $\hat{\cS}$ which is maximal and $J_\cS= J_{\hat{\cS}}$. In several instances in the proof of Theorem \ref{prss} we will use the fact that the sets $X_v$ are disjoint, while recalling the proof of Proposition \ref{finitemaxreduction}, the sets $\hat{X}_v, v\in \hat{V},$ are not necessarily disjoint. Nevertheless this can be rectified using formal lifts of GDMS as it was described in Remark \ref{formalGDMS}. Therefore  the maximality assumption in Theorem \ref{prss} is not essential. \end{remark}

\begin{proof}The implication (i)$\imp$(ii) was proved in Theorem \ref{t1j93}, and the implication (iv)$\imp$(i) is obvious, hence we only need to show that (ii)$\imp$(iii)$\imp$(iv). Before giving the proof of (ii)$\imp$(iii) we need several auxiliary propositions. The first one is an immediate corollary of Propositions \ref{newbilip21} and \ref{diamgtr1}.
\begin{corollary}
\label{newbilip2} Let $S=\{\phi_e\}_{e\in E}$ be a weakly Carnot conformal GDMS. If $\cS$ is conformal or if $\cH^h(J_\cS)>0$, then for every $\om \in E^\ast_A$,
\begin{equation*}
\diam( \f_\om( J_\cS \cap X_{t(\om)})) \geq (2 L^2 K)^{-1} \kappa_0 \mu_0 \|D \f_\om\|_\infty,
\end{equation*}
where $\kappa_0$ is as in Lemma \ref{newbilip} and $\mu_0=\min \{\diam(J_\cS \cap X_v):v \in V\}$.
\end{corollary}
We will also need the following two propositions involving properties of $\ve$-relatively close words.

\begin{proposition}
\label{erclose1} Let $\cS=\{\f_e\}_{e \in E}$ be a maximal and finitely irreducible weakly Carnot conformal GDMS on $(\G, d)$ such that  $\sharp(J_\cS \cap X_v)>1$.
\begin{itemize}
\item[(a)] If $\tau$ and  $\tau'$ are $\ve$-relatively close and $\om \in E_A^\ast$ is such that both $\om \tau, \om \tau' \in E_A^\ast$ then $\om \tau$ and $\om \tau'$ are $c_1\ve$-relatively close, where
$$c_1=\frac{2(\Lambda K L)^2 M}{\kappa_0 \mu_0}.$$
\item[(b)] If $\tau$ and  $\tau'$ are $\ve$-relatively close and $\om \in E_A^\ast$ is such that both $\tau \om , \tau'\om  \in E_A^\ast$ then $\tau \om $ and $ \tau'\om$ are $c_2\|D \f_\om\|_\infty^{-1} \ve$-relatively close, where
$$c_2=\frac{2(K L)^2 \Lambda M}{\kappa_0 \mu_0}.$$
\item[(c)] If $\om$ and $\tau$ are $\ve$-relatively close then
$$\diam(J_\om) \leq (1+2\ve) \diam(J_\tau).$$
\item[(d)] If the words $\om, \tau$  are $\ve_1$-relatively close and the words $\tau, \rho,$ are  $\ve_2$-relatively close words, then $\om$ and $\rho$ are  $(\ve_1+\ve_2+4\ve_1 \ve_2)$-relatively close.
%\item[(e)] If $\om^1$ and $\om^2$ are $\ve$-relatively close and $\tau, \rho, \rho \in E^\ast_A$ satisfy $\om^1 \tau \rho, \om^2 \rho \rho \in E^\ast_A$, then $\om^1 \tau \rho$ and  $\om^2 \rho \rho$ are $$c_2\|D \f_\rho\|_\infty^{-1} (c_1 \ve+c_2 \|D \f_\rho\|_\infty^{-1} \ve+4c_1c_2\|D \f_\rho\|_\infty^{-1} \ve^2)\text{-relatively close.}$$
\end{itemize}
\end{proposition}
\begin{proof} (a) Let $x\in J_\cS \cap X_{t(\tau)}$. Then by \eqref{4.1.9a},
\begin{equation}
\label{claim1e1}
\begin{split}
d(\f_{\om \tau}(x), \f_{\om \tau'}(x)) &\leq \Lambda \|D \f_\om\|_{\infty} d(\f_\tau(x), \f_{\tau'}(x)) \\
&\leq \Lambda \|D \f_\om\|_{\infty} \ve \min\{\diam(J_{\tau}), \diam(J_{\tau'})\}.
\end{split}
\end{equation}
Moreover by \eqref{quasi-multiplicativity1}, \eqref{4.1.9} and Proposition \ref{newbilip21} we obtain
\begin{equation}
\label{claim1e2}
\begin{split}
\|D \f_\om\|_\infty\diam(J_{\tau}) &\leq \Lambda M \|D \f_\om\|_\infty \|D \f_\tau\|_\infty \\
&\leq \Lambda M K \|D \f_{\om \tau}\|_\infty \\
&\leq \frac{2 \La M L^2 K^2}{\kappa_0 \mu_0} \diam (J_{\om\tau}),
\end{split}
\end{equation}
and in an identical manner
\begin{equation}
\label{claim1e3}
\|D \f_\om\|_\infty\diam(J_{\tau'}) \leq \frac{2 \La M L^2 K^2}{\kappa_0 \mu_0} \diam(J_{\om\tau'}).
\end{equation}
Hence (a) follows after combining \eqref{claim1e1}, \eqref{claim1e2} and \eqref{claim1e3}.

(b) Let $x \in J_\cS \cap X_{t(\om)}$, then $$\f_\om(x) \in J_\cS \cap \f_\om (X_{t(\om)}) \subset J_\cS \cap X_{i(\om)}=J_\cS \cap X_{t(\tau)}=J_\cS \cap X_{t(\tau')}.$$
Hence, since $\tau, \tau'$ are $\ve$-relatively close,
\begin{equation}
\label{claim2e1}
d(\f_{\tau \om}(x), \f_{\tau' \om}(x)) \leq \ve \min\{\diam(J_{\tau}), \diam(J_{\tau'})\}.
\end{equation}
Notice also that by \eqref{quasi-multiplicativity1}, \eqref{4.1.9} and Proposition \ref{newbilip21}
\begin{equation*}
\begin{split}
\diam(J_{\tau})& \leq \La M \|D \f_\tau\|_\infty \leq \La M K \|D \f_\om\|_\infty^{-1} \|D \f_{\tau \om}\|_\infty \\
&\leq \frac{2 \La M (L K)^2}{\kappa_0 \mu_0}  \|D \f_\om\|_\infty^{-1}  \diam(J_{\tau\om}),
\end{split}
\end{equation*}
and an identical inequality holds if we replace $\tau$ by $\tau'$. These, combined with \eqref{claim2e1}, imply (b).

(c) Since $\om$ and $\tau$ are $\ve$-relatively close, $X_{t(\om)}=X_{t(\tau)}$. Then if $x,y \in J_\cS \cap X_{t(\om)}$
\begin{equation*}
\begin{split}
d(\f_\om(x), \f_\om(y)) &\leq d(\f_\om(x), \f_\tau(x))+d(\f_\tau(x), \f_\tau(y))+d(\f_\om(y), \f_\tau(y)) \\
&\leq (1+2\ve) \diam( J_{\tau}).�\end{split}
\end{equation*}

(d) First notice that $X_{t(\om)}=X_{t(\tau)}=X_{t(\rho)}$.  Then if $x \in J_\cS \cap X_{t(\om)}$,
\begin{equation}
\label{claim4e1}
\begin{split}
d(\f_\om(x), \f_\rho(x)) &\leq d(\f_\om(x), \f_\tau(x))+d(\f_\rho(x), \f_\tau(x)) \\
&\leq \ve_1 \min\{\diam( J_{\om}), \diam(J_{\tau}) \} \\
&\quad\quad+\ve_2  \min\{\diam( J_{\rho}), \diam( J_{\tau}) \}.
\end{split}
\end{equation}
If $\diam( J_{\om}) \leq \diam( J_{\rho})$ then trivially,
\begin{equation*}
\begin{split}
\min&\{\diam( J_{\om}), \diam(J_{\tau}) \} \leq \diam( J_{\om}) =\min \{\diam( J_{\om}), \diam(J_{\rho}) \}.
\end{split}
\end{equation*}
On the other hand if $\diam( J_{\om}) \geq \diam( J_{\rho})$ then by (c)
\begin{equation*}
\begin{split}
\diam(J_{\tau})  &\leq (1+2 \ve_2)  \diam( J_{\rho})\\
&= (1+2 \ve_2)  \min \{\diam( J_{\om})  ,\diam( J_{\rho})\}.
\end{split}
\end{equation*}
Thus,
\begin{equation}
\label{diamstar}
\begin{split}
\min&\{\diam( J_{\om}), \diam(J_{\tau}) \} \leq  (1+2 \ve_2)  \min \{ \diam( J_{\om})  ,\diam( J_{\rho})\}.
\end{split}
\end{equation}
In the same manner we also obtain,
\begin{equation}
\label{diam2star}
\begin{split}
\min&\{\diam( J_{\rho}), \diam(J_{\tau}) \} \leq  (1+2 \ve_1)  \min \{ \diam( J_{\om})  ,\diam( J_{\rho})\}.
\end{split}
\end{equation}
Hence by \eqref{claim4e1}, \eqref{diamstar}, and \eqref{diam2star} for every $x \in J_\cS \cap X_{t(\om)}$
\begin{equation*}
\begin{split}
d(\f_\om(x), \f_\rho(x)) &\leq \ve_1(1+2\ve_2) \min \{ \diam( J_{\om})  ,\diam( J_{\rho})\} \\
&\quad +\ve_2(1+2\ve_1) \min \{ \diam( J_{\om})  ,\diam( J_{\rho})\}\\
&=(\ve_1+\ve_2+4\ve_1 \ve_2)\min \{ \diam( J_{\om})  ,\diam( J_{\rho})\}
\end{split}
\end{equation*}
and the words $\om$ and $\rho$ are $(\ve_1+\ve_2+4\ve_1 \ve_2)$-relatively close.
%(e) By (a), $\om^1 \tau$ and $\om^1 \rho$ are $c_1 \ve$-relatively close. By (b), $\om^1 \rho$ and $\om^2 \rho$ are $c_2  \|D \f_\rho\|_\infty^{-1}\ve$-relatively close. Therefore by (d), $\om^1 \tau$ and $\om^2 \rho$ are
%$$(c_1 \ve+c_2 \|D \f_\rho\|_\infty^{-1} \ve+4c_1c_2\|D \f_\rho\|_\infty^{-1} \ve^2)\text{-relatively close.}$$
%Applying (b) once more, we deduce that $\om^1 \tau \rho$ and $\om^2 \rho \rho$ are$$c_2\|D \f_\rho\|_\infty^{-1} (c_1 \ve+c_2 \|D \f_\rho\|_\infty^{-1} \ve+4c_1c_2\|D \f_\rho\|_\infty^{-1} \ve^2)\text{-relatively close.}$$
\end{proof}

\begin{proposition}
\label{bgpre}
Let $\cS=\{\f_e\}_{e \in E}$ be a maximal and finitely irreducible weakly Carnot conformal GDMS on $(\G, d)$ such that  $\sharp(J_\cS \cap X_v)>1$. If for every $\ve \in (0, \lambda_\cS)$ there exist $\om^1, \om^2 \in E_A^\ast$ which are $\ve$-relatively close, then for every $N \in \N, N \geq 2,$ and for every $\ve \in (0, \lambda_\cS)$ there exist distinct words $\rho^1,\dots,\rho^N \in E_A^\ast$ which are pairwise $\ve$-relatively close.
\end{proposition}
\begin{proof} It is enough to show that if  for every $\ve>0$  there exist $N \geq 2$ distinct words in $E_A^\ast$  which are pairwise $\ve$-relatively close then for every $\ve$ we can find $2N$ distinct words in $E_A^\ast$ which are $\ve$-relatively close.

For the rest of the proof $c_1$ and $c_2$ will be as in the proof of Proposition \ref{erclose1}. Let $\ve \in (0,\lambda_\cS)$ and let $\rho^1,\dots,\rho^N \in E_A^\ast$ be distinct words which are pairwise $\ve_1$-relatively close, where
$$\ve_1=\min \left\{1,\lambda_\cS,\frac{\ve}{10 c_1}\right\}.$$
Let $\Phi \subset E_A^\ast$ be the set of words witnessing finite irreducibility for $\cS$. We set
\begin{equation}
\label{barm}
\bar{m}=\min\{\|D \f_{\rho^i}\|_\infty:i=1,\dots,N\} \cdot \min \{\|D \f_{\tau}\|_\infty: \tau \in \Phi\}.
\end{equation}
%Moreover if $\Phi \subset E_A^N$ is the set witnessing finite irreducibility, recalling ,$$\xi_\Phi=\min\{\|D \f_\om\|^{-1}_\infty:\om \in \Phi\}.$$
Let also $\om^1,\om^2 \in E_A^\ast$ which are $\ve_2$-relatively close and
$$\ve_2=\min \left\{\lambda_\cS, \frac{\ve  \,\bar{m}}{10 \, K \, c_2  }\right\}.$$

Notice that by the definition of relatively close words, $t(\om_1)=t(\om_2)$ and $i(\rho^i)=i(\rho^j)$ for all $i,j=1,\dots,N$. Since $\cS$ is maximal and finitely irreducible there exists $\tau \in E_A^\ast$ such
\begin{equation*}
\om^1 \tau \rho^i \in E_A^\ast \text{ and } \om^2\tau \rho^i \in E_A^\ast,
\end{equation*}
for all $i=1,\dots, N.$ Let
$$\cW=\{\om \in  E_A^\ast : \om=\om^1 \tau \rho^i  \text{ or }\om=\om^2 \tau \rho^i \text{ for some }i=1,\dots,N\}.$$
Since $\sharp \cW=2N$, in order to finish the proof of the proposition it suffices to show that if $\om, \om' \in \cW$ then they are $\ve$-relatively close. We are going to prove this assertion by considering several cases.

By Proposition \ref{erclose1} (a) the words $\om^1 \tau \rho^i$ and $\om^1 \tau \rho^j, i,j=1,\dots,N,$ are $$c_1 \ve_1\mbox{-relatively close}.$$  For the same reason the words $\om^2\tau \rho^i$ and $\om^2 \tau \rho^j, i,j=1,\dots,N,$ are also
$$c_1 \ve_1\mbox{-relatively close.}$$
Moreover by Proposition \ref{erclose1} (b) we deduce that  the words  $\om^1 \tau \rho^i$ and $\om^2 \tau \rho^i$ are
$$c_2 \|D \phi_{\tau \rho}\|^{-1}_{\infty} \ve_2\mbox{-relatively close.}$$
By \eqref{quasi-multiplicativity1} and recalling \eqref{barm}
$$\|D \phi_{\tau \rho}\|_{\infty} \geq K^{-1} \|D \f_\tau\|_{\infty} \|D \f_{\rho_i}\|_{\infty} \geq K^{-1} \bar{m},$$
hence the words $\om^1 \tau \rho^i$ and $\om^2 \tau \rho^i$ are
$$c_2 K \bar{m}^{-1} \ve_2\mbox{-relatively close.}$$
Finally if $i,j =1,\dots,N$ by Proposition \ref{erclose1}(d) the words $\om^1 \tau \rho^i$ and $\om^2 \tau \rho^j$ are
$$(c_1 \ve_1+K\,c_2 \bar{m}^{-1} \ve_2+4 K \,c_1 c_2 \bar{m}^{-1} \ve_1 \ve_2)\mbox{-relatively close.}$$ By the choice of $\ve_1$ and $\ve_2$ we deduce that any two words from $\cW$ are $\ve$-relatively close. The proof is complete.
\end{proof}

(ii)$\imp$(iii):
First of all Proposition \ref{diamgtr1} implies that $\sharp(J_\cS \cap X_v)>1$ for all $v \in V$. We fix some $N \in \N$. By Proposition \ref{bgpre} there exist $\rho^1, \dots, \rho^N$ which are $\lambda_\cS/2$-relatively close. Recalling Theorem \ref{thm-conformal-invariant}, let $\tilde{\mu}_h$ be the unique shift-invariant probability measure on $E_A^\N$. By \eqref{2j93} $\tilde{\mu}_h([\rho^1])>0$, hence by Birkhoff's Ergodic Theorem \index{Birkhoff's Ergodic Theorem} if
$$A=\{\om \in E_A^\N: \sigma^n(\om) \in [\rho^1] \text{ for infinitely many }n \in \N\},$$
then $\tilde{\mu}_h(A)=1$. So by Theorem \ref{thm-conformal-invariant} $\tilde{m}_h(A)=\tilde{\mu}_h(A)=1$.

Let $\om \in A$ and let $n \in \N$ such that $\sigma^n(\om) \in [\rho^1]$. Then $t(\om_n)=i(\rho^1_1)$ and since $\cS$ is maximal  and the words $\rho^1, \dots, \rho^N$ are $\lambda_\cS/2$-relatively close we deduce by Remark \ref{keyercloserem} that $t(\om_n)=i(\rho^i_1)$ for all $i=1, \dots N$ hence
$$\om |_n \rho^i \in E_A^\ast,$$
for all $i=1, \dots, N.$ Hence Proposition \ref{erclose1} (a) implies that the words $\om |_n \rho^1, \dots, \om |_n \rho^N$ are $c_1 \lambda_\cS/2$-relatively close. Now notice that if $p,q \in J_\cS \cap X_{t(\rho^i)}$ and $i,j=1,\dots,N$ then
\begin{equation*}
\begin{split}
d(\f_{\om |_n \rho^i}(p),\f_{\om |_n \rho^i}(q))&\leq d(\f_{\om |_n \rho^i}(p),\f_{\om |_n \rho^j}(p))+d(\f_{\om |_n \rho^j}(p), \f_{\om |_n \rho^j}(q))\\
&\quad+d( \f_{\om |_n \rho^j}(q),\f_{\om |_n \rho^i}(q))\\
&\leq (1+ c_1\lambda_\cS) \diam (J_{\om |_n \rho^j}).
\end{split}
\end{equation*}
Hence we deduce that
\begin{equation}
\label{maxminprss}
\max_{i=1,\dots,N} \diam (J_{\om |_n \rho^i}) \leq (1+ c_1\lambda_\cS) \min_{i=1,\dots,N} \diam (J_{\om |_n \rho^i}).
\end{equation}

We will now show that if $x=\pi(\om), x \in A$ then
\begin{equation}
\label{sc2imp3a}
\bigcup_{i=1}^N J_{\om |_n \rho^i} \subset B(x,r),
\end{equation}
where
\begin{equation}
\label{radiiprss}
r=\left(1+ \frac{c_1\lambda_\cS
}2\right) \max_{i=1, \dots, N} \{ \diam(J_{\om |_n \rho^i})\}.
\end{equation}
First notice that since the words $\rho^i$ are $\lambda_\cS/2$-relatively close, Proposition \ref{erclose1}(a) implies that for any $i=1,\dots, N$ and any $p \in J_\cS \cap X_{t(\rho^1)}$ (recall that $t(\rho^1)=t(\rho^i)$ for all $i=1, \dots ,N$) we have
\begin{equation}
\label{ercloseimp3}
\begin{split}
d(\f_{\om|_n \rho^1}(p), &\f_{\om|_n \rho^i }(p)) \leq \frac{c_1\lambda_\cS}{2} \min\{\diam(J_{\om |_n \rho^1}),\diam(J_{\om |_n \rho^i})\}.
\end{split}
\end{equation}

Since $x=\pi(\om)$ and $\sigma^n(\om) \in [\rho^1]$, $x$ can be written as
$$x=\f_{\om|_n \rho^1}(\pi (\sigma^{n+|\rho^1|}(\om))):=\f_{\om|_n \rho^1}(x_0),$$ where $x_0 \in J_\cS \cap X_{t(\rho^1)}$. Therefore, for any $i=2, \dots, \N,$ and for any $q=\f_{\om |_n \rho^i}(p), p \in J_\cS \cap X_{t(\rho^i)}$,
\begin{equation*}
\begin{split}
d(x,q)&=d(\f_{\om|_n \rho^1}(x_0),\f_{\om |_n \rho^i}(p)) \\
&\leq d(\f_{\om|_n \rho^1}(x_0),\f_{\om |_n \rho^1}(p))+d(\f_{\om|_n \rho^1}(p),\f_{\om |_n \rho^i}(p)) \\
&\leq \diam(J_{\om |_n \rho^1})+ \frac{c_1\lambda_\cS
}2 \max_{i=1, \dots, N} \{ \diam(J_{\om |_n \rho^i})\} \\
&\leq \left(1+ \frac{c_1\lambda_\cS
}2\right) \max_{i=1, \dots, N} \diam ( J_{\om |_n \rho^i})\},
\end{split}
\end{equation*}
and \eqref{sc2imp3a} follows.

Now notice that $[\om|_n \rho^i] \subset \pi^{-1}(J_{\om |_n \rho^i})$. Hence by \eqref{sc2imp3a}, \eqref{2j93} and \eqref{4.1.9} we deduce that if $x=\pi(\om)$ and $\om \in A$,
\begin{equation*}
\begin{split}
m_h (B(x,r)) &=\tilde m_h (\pi^{-1}(B(x,r))) \\
&\geq m_h \left( \bigcup_{i=1}^N \pi^{-1}(J_{\om |_n \rho^i})\right)\\
& \geq \tilde{m}_h \left(\bigcup_{i=1}^N [\om|_n \rho^i]\right) = \sum_{i=1}^N \tilde{m}_h([\om|_n \rho^i]) \\
& \geq c_h \sum_{i=1}^N \|D \f_{\om|_n \rho^i}\|^h_\infty\\
&\geq c_h (\La M)^{-1} \sum_{i=1}^N \diam(J_{\om |_n \rho^i})^h\\
&\geq c_h (\La M)^{-1} N \min_{i=1,\dots,N} \{\diam(J_{\om |_n \rho^i})^h\}.
\end{split}
\end{equation*}
Therefore by \eqref{maxminprss} and \eqref{radiiprss}
\begin{equation}
\label{frost1prss}
\frac{m_h(B(x,r))}{r^h} \geq \frac{c_h}{\La M (1+c_1 \lambda_\cS)^{2h}} N.
\end{equation}
Since $r \ra 0$ as $n \ra \infty$, \eqref{frost1prss} implies that
$$\limsup_{s \ra 0} \frac{m_h(B(x,s))}{s^h} \geq  \frac{c_h}{\La M (1+c_1 \lambda_\cS)^{2h}} N,$$
for $x \in \pi (A)$. Hence by \cite[Theorem A2.0.12] {MUGDMS} we deduce that $$\cH^h(\pi(A)) \lesssim N^{-1} m_h(\pi(A))$$ where the constant only depends on the system $\cS$.  Moreover,
$$m_h(\pi(A))=\tilde m_h(\pi^{-1}(\pi(A))) \geq \tilde{m}_h (A)=1,$$
hence $m_h (J_\cS \setminus \pi(A))=0$ and by Theorem \ref{t4.4.1} we also get that $\cH^h(J_\cS \setminus \pi(A))=0$. Therefore $$\cH^h(J_\cS) \lesssim N^{-1} m_h( J_\cS).$$ Since $N$ can be taken arbitrarily large we deduce that $\cH^h(J_\cS)=0$ and we have reached a contradiction. The proof of the implication (ii)$\imp$(iii) is complete.

For the proof of the implication (iii)$\imp$(iv) we will use several lemmas. Following \cite{PRSS} for any $\om \in E_A^\ast$ and any $\a>0$ and $T \geq 1$ we define
\begin{equation*}
\begin{split}
W_{\a, T}(\om)=\Big\{\rho &\in E_A^\ast:\, i(\om)=i(\rho),\quad T^{-1} \leq \frac{\diam(J_{\rho})}{\diam(J_\om)} \leq T, \\
&\quad \quad \quad \quad \quad \quad \quad  \text{ and }\dist(J_{\rho}, J_\om) \leq \a \diam (J_\om) \Big\}.
\end{split}
\end{equation*}
In the following lemma we show that the Bandt--Graf condition implies that the cardinality of  $W_{\a,T}(\om)$ is bounded  for all $\om \in E_A^\ast$ and the bounds only depend on $\a$ and $T$.
\begin{lemma}
\label{watbound}
Let $\cS$ be a finite, irreducible, and maximal weakly Carnot conformal GDMS on $(\G,d)$. If $\cS$  satisfies the Bandt--Graf condition, \index{Bandt-Graf condition} then for all $\a>0$ and $T \geq 1,$ there exist positive constants $C(\a,T)$ such that for all $\om \in E_A^\ast$,
$$\sharp W_{\a,T}(\om) \leq C(\a,T).$$
\end{lemma}
\begin{proof}
Fix $\om \in E_A^\ast$ and $\a>0, T \geq 1$.  For $v\in V$ let
$$W^v_{\a,T}(\om)=\{\rho \in W_{\a,T}: t(\rho)=v\}.$$
Notice that it is enough to show that for all $v \in V$
there exist positive constants $C^v(\a,T)$ such that
\begin{equation}
\label{watv}
\sharp W^v_{\a,T}(\om) \leq C^v(\a,T).
\end{equation}
To this end, fix some $v \in V$ and let $$r_\ve=\min \left \{\eta_\cS, \frac{\kappa_0 \mu_0}{8 \Lambda^3 K T^2} \ve \right\} $$ where all the appearing constants are as in Section \ref{sec:iwagdms}, %$\eta_\cS$ was defined in $\eqref{rdef}$
and $\ve$ comes from the Bandt-Graf condition.  Since $\cS$ is finite, the limit set $J_\cS$ is compact. Hence there exists some $n_0 \in \N$ and $\{x_1,\dots,x_{n_0}\} \in J_\cS \cap X_v$ such that
$$J_\cS \cap X_v \subset \bigcup_{i=1}^{n_0} B(x_i, r_\ve).$$
By the Bandt-Graf condition there exists $x_v \in J_\cS \cap X_v$ such that for all $\tau, \rho \in W^v_{\a,T}$,
\begin{equation}
\label{bgrawat}
d(\f_\tau(x_v),\f_\rho(x_v)) \geq \ve \min \{ \diam(J_\tau) ,\diam( J_{\rho})\}.
\end{equation}
If $\tilde{x} \in \G$ and $d(x_v,\tilde{x}) \leq r_\ve$ then $\tilde{x} \in S_v$. Hence by the choice of $r_e$, \eqref{bgrawat}, \eqref{l22013_03_12} and Corollary \ref{bilip1} imply that
\begin{equation*}
\begin{split}
d (\f_\tau(\tilde{x}),\f_\rho(\tilde{x})) &\geq d(\f_\tau(x_v), \f_\rho(x_v))-d(\f_\rho(\tilde{x}), \f_\rho(x_v))-d(\f_\tau(\tilde{x}), \f_\tau(x_v)) \\
& \geq \ve \min \{ \diam(J_\tau), \diam( J_{\rho})\}-\Lambda d(x_v, \tilde{x}) (\|D \f_\rho\|_\infty+\|D \f_\tau\|_\infty )\\
&\geq \ve \min \{ \diam(J_\tau), \diam( J_{\rho})\}-\Lambda r_\ve (\|D \f_\rho\|_\infty+\|D \f_\tau\|_\infty )\\
&\geq \ve \min \{ \diam(J_\tau) ,\diam( J_{\rho})\}-\frac{2 \Lambda^3 K}{\kappa_0 \mu_0} r_\ve(\diam(J_\tau) +\diam( J_{\rho}) ).
\end{split}
\end{equation*}
Since $\rho, \tau \in W_{\a,T}( \om)$,
\begin{equation*}
\begin{split}
d (\f_\tau(\tilde{x}),\f_\rho(\tilde{x})) \geq \left( \ve T^{-1}- \frac{4 \Lambda^3 K}{\kappa_0 \mu_0} T\,r_\ve \right) \diam( J_\om),
\end{split}
\end{equation*}
hence by the choice of $r_\ve$,
\begin{equation}
\label{lbwat}
d (\f_\tau(\tilde{x}),\f_\rho(\tilde{x})) \geq \frac{\ve}{2T} \diam( J_\om).
\end{equation}

Using the standard method for constructing product measures, see e.g.  \cite[Section 2.5]{fol},  there exists a Borel measure $\nu$ on $\G^{n_0}$, which is identified with $(\R^N)^{n_0}$, such that if $\{A_i\}_{i=1}^{n_0}$ are Borel subsets of $\G$ then
\begin{equation}
\label{prodmeasure}
\nu \left( \prod_{i=1}^{n_0} A_i\right)=|A_i|^{n_0}.
\end{equation}
For $\Xi =(\xi_1, \dots, x_{n_0}) \in \G^{n_0}$  and $r>0$ let
$$\B( \Xi,r):=\{Y=(y_1,\dots,y_{n_0}) \in \G^{n_0}: d(y_i, \xi_i)<r\}.$$
For any $\rho \in W^v_{\a,T}(\om)$, let
$$\Xi_\rho= (\f_\rho(x_1), \dots ,\f_\rho(x_{n_0})) \in \G^{n_0}.$$
Notice that $\Xi_\rho$ is well defined because for all $i=1,\dots, n_0,$ $x_i \in X_v$, $t(\rho)=v$ and and $\cS$ is maximal, hence $\f_\rho(x_i)$ makes sense for $i=1,\dots, n_0$.
We also define,
$$\B_\rho:= \B \left(\Xi_{\rho}, \frac{\diam( J_\om) \ve}{4T} \right).$$
We are now going to prove that
\begin{equation}
\label{prodballsdisjoint}
\B_\rho \cap \B_\tau = \emptyset, \text{ for } \tau \neq \rho, \tau, \rho \in  W^v_{\a,T}(\om).
\end{equation}
By way of contradiction assume that \eqref{prodballsdisjoint} fails. Then there exists some $Y=(y_1,\dots,y_{n_0}) \in \B_\rho \cap \B_\tau$ such that
$$d(y_i, \f_\rho(x_i))< \frac{\diam( J_\om) \ve}{4T}\text{ and }d(y_i, \f_\tau(x_i))< \frac{\diam( J_\om) \ve}{4T},$$
for all $i=1,\dots,n_0$. Hence
\begin{equation}
\label{proball1}
d(\f_\rho(x_i),\f_{\tau}(x_i))< \frac{\diam( J_\om) \ve}{2T}.
\end{equation}
Since $J_\cS \cap X_v \subset \cup_{i=1}^{n_0} B(x_i, r_\ve)$ there exists some $i_0=1,\dots, n_0,$ such that $d(x_v, x_{i_0})<r_\ve$. Hence by \eqref{lbwat}
$$d(\f_\rho(x_{i_0}),\f_{\tau}(x_{i_0})) \geq \frac{\diam( J_\om) \ve}{2T},$$
which contradicts \eqref{proball1}. Therefore \eqref{prodballsdisjoint} follows.

For every $x \in J_\cS \cap X_v$ and for every $\rho,\tau \in W^v_{\a,T}(\om)$,
\begin{equation}
\label{forproballinc}
\begin{split}
d(\f_\rho(x),\f_\tau(x)) &\leq \diam( J_\tau)+\dist(J_\tau,J_\om)+ \diam( J_\om)\\
&\quad\quad\quad\quad+\dist(J_\om,J_{\rho})+\diam( J_{\rho})\\
&\leq (2(\a+T)+1)\diam(J_\om).
\end{split}
\end{equation}

Fix some $\rho \in W^v_{\a,T}(\om)$ and let
\begin{equation*}
\B_0=\B(\Xi_\rho,(1+2(\a+T)+\ve T^{-1} )\diam(J_\om)\).
\end{equation*}
We are now going to show that
\begin{equation}
\label{prodballinc}
\B_\tau \subset \B_0
\end{equation}
for every $\tau \in W^v_{\a,T}(\om)$. Let $\tau \in  W^v_{\a,T}(\om)$ and let $Y=(y_1,\dots,y_{n_0}) \in \B_\tau$. Then by \eqref{forproballinc} for all $i=1,\dots,n_0$
\begin{equation*}
\begin{split}
d(y_i, \f_\rho(x_i)) &\leq d(y_i, \f_{\tau}(x_i))+d(\f_\tau(x_i),\f_\rho(x_i)) \\
&\leq \left(\frac{1}{4T}\ve+2(\a+T)+1\right) \diam(J_\om),
\end{split}
\end{equation*}
and \eqref{prodballinc} follows.

Notice that by \eqref{measure-of-bcc}
\begin{equation*}
\begin{split}
\nu (\B_0)&=\nu\left( \prod_{i=1}^{n_0} B(\f_\rho(x_i),(1+2(\a+T)+\ve T^{-1} )\diam(J_\om))\right)\\
&=\prod_{i=1}^{n_0} \left|  B(\f_\rho(x_i),(1+2(\a+T)+\ve T^{-1} )\diam(J_\om)) \right| \\
&=c_0 \left( (1+2(\a+T)+\ve T^{-1} )\diam(J_\om \right))^{n_0 Q }.
\end{split}
\end{equation*}
Moreover by \eqref{prodballsdisjoint} and \eqref{prodballinc}
\begin{equation*}
\begin{split}
\nu(\B_0) \geq \sum_{\tau \in W^v_{\a,T}(\om)} \nu(\B_\tau)=\sharp W^v_{\a,T}(\om) \left(\frac{\diam(J_\om)}{4T} \ve\right)^{n_0Q}.
\end{split}
\end{equation*}
Hence,
\begin{equation*}
\sharp W^v_{\a,T}(\om) \leq \left( \frac{  4T (1+2(\a+T)+\ve T^{-1} ) }{\ve} \right)^{n_0 Q},
\end{equation*}
and the proof is complete.
\end{proof}

\begin{remark}
\label{watcc}
For any $\om \in E_A^\ast$ and any $\a>0, T \geq 1$ let
\begin{equation*}
\begin{split}
W^{cc}_{\a, T}(\om)=\Big\{\rho &\in E_A^\ast:\, i(\om)=i(\rho),\quad T^{-1} \leq \frac{\diamcc(J_{\rho})}{\diamcc(J_\om)} \leq T, \\
&\quad \quad \quad \quad \quad \quad \text{ and }\distcc(J_{\rho}, J_\om) \leq \a \diamcc (J_\om) \Big\}.
\end{split}
\end{equation*}
It follows immediately by Lemma \ref{watbound} and \eqref{quasiconvexity} that there exist positive constants $C^{cc}(\a,T)$ such that
\begin{equation}
\label{wccboundcc}
\sharp W^{cc}_{\a, T}(\om) \leq C^{cc}(\a,T).
\end{equation}
In fact $C^{cc}(\a,T)=C(L\a, L T)$.
\end{remark}

\begin{lemma}
\label{dccdistortlemma}
Let $\cS$ be a  weakly Carnot conformal GDMS on $(\G,d)$.  If $\tau \in E^\ast_A$ satisfies $\diam_{cc}(J_{\tau})<\frac{\eta_\cS}{LKC}$ then for all $\om \in E_A^\ast$ such that $\om \tau \in E_A^\ast$ and for all $y \in \G$ such that $\dist_{cc}(y, J_{\tau}) \leq \diam_{cc}(J_{\tau})$,
\begin{equation}
\begin{split}
\label{expest}
\exp (-\tilde{C} \diam_{cc}(J_{\tau}) )& \leq \frac{ \diam_{cc}(J_{\om\tau})}{ \|D \f_\om(y)\|\diam_{cc}(J_{\tau}) } \leq \exp(\tilde{C} \diam_{cc}(J_{\tau})),
\end{split}
\end{equation}
where $\tilde{C}=\frac{\Lambda_0 K (2+LKC)}{1-s}$.
\end{lemma}
\begin{proof} We will first establish the right hand  inequality. Notice that
\begin{equation}
\label{welldefom}
\begin{split}
\distcc(y, X_{i(\tau)}) &\leq \distcc(y, \f_\tau(X_t(\tau))) \leq \distcc(y, J_{\tau})\\
&\leq \diamcc(J_{\tau}) <  \eta_\cS.
\end{split}
\end{equation}
Since $\om \tau \in E_A^\ast$, we know that $i(\tau)=t(\om)$ hence by \eqref{welldefom}  $y \in S_{t(\om)}$. Therefore $\f_\om (y)$ and $\|D \f_\om (y)\|$ are well defined.

For $p,q \in \G$, if $\gamma_{p,q}:[0,T] \ra \G$ is the horizontal geodesic curve joining $p$ and $q$, we will denote its arc by
$$[\gamma_{p,q}]=\{\gamma_{p,q}(t): t \in [0,T]\}.$$
Let $p,q \in J_\cS \cap X_{t(\tau)}$. Then by the segment property \cite[Corollary 5.15.6]{BLU},
\begin{equation}
\label{arcinclu}
\begin{split}
[\gamma_{\f_\tau(p), \f_\tau(q)}] &\subset B_{cc}(\f_\tau(p), d_{cc}(\f_\tau(p), \f_\tau(q)))\\
& \subset B_{cc}(\f_\tau(p),\diamcc(J_\tau))\\ &\subset B_{cc}(\f_\tau(p),\eta_\cS) \subset N_{t(\om)}.
\end{split}
\end{equation}
By Lemma \ref{upper-gradient} and \eqref{arcinclu}  there exists some $z \in \G$ such that,
\begin{equation}
\label{cccurve1}
\distcc(z,J_\tau)<\eta_\cS,
\end{equation}
\begin{equation}
\label{cccurve2}
\|D \f_\om(z)\|=\max \{\|D \f_\om(\zeta)\|:\zeta \in [\gamma_{\f_\tau(p),\f_{\tau}(q)}]\},
\end{equation}
and
\begin{equation}
\label{subrmvt}
d_{cc}(\f_{\om\tau}(p),\f_{\om\tau}(q)) \leq \|D \f_\om(z)\| d_{cc}(\f_\tau(p),\f_{\tau}(q)).
\end{equation}

If $y \in \G$ satisfies $\distcc (y, J_\tau)<\diamcc(J_\tau)$, then $y \in N_{t(\om)}$ and by Remark \ref{nsets}, since also $z \in N_{t(\om)}$ by \eqref{arcinclu},
\begin{equation}
\label{expo0}
\frac{\|D \f_\om(z)\|}{\|D \f_\om(y)\|} \leq \exp\left( \frac{\Lambda_0 K}{1-s} d_{cc}(y,z)\right).
\end{equation}
Moreover by \eqref{cccurve1},
\begin{equation*}
\begin{split}
d_{cc}(y,z) &\leq \distcc(y, J_\tau)+\distcc(z, J_\tau)+\diamcc(J_\tau)\\
&\leq 3 \diamcc(J_\tau),
\end{split}
\end{equation*}
hence by \eqref{expo0}
\begin{equation}
\label{expo1}
\frac{\|D \f_\om(z)\|}{\|D \f_\om(y)\|} \leq \exp\left( \frac{3 \Lambda_0 K}{1-s}\diamcc(J_\tau) \right).
\end{equation}
By \eqref{subrmvt} and \eqref{expo1} we deduce that
\begin{equation*}
\begin{split}
&\diamcc( J_{\om\tau}) \leq \|D \f_\om(y)\| \exp\left( \frac{3 \Lambda_0 K}{1-s}\diamcc(J_\tau) \right) \diamcc( J_{\tau}),
\end{split}
\end{equation*}
and the right hand inequality has been proven.

We now move to the left hand inequality. First notice that by \eqref{subrmvt} and Corollary \ref{l42013_03_12} if $p \in J_\cS \cap X_{t(\tau)}$ then
\begin{equation*}
\begin{split}
B_{cc}(\f_{\om\tau}(p), \diamcc(J_{\om\tau})) &\subset B_{cc}(\f_{\om\tau}(p), \|D \f_{\om}\|_\infty \diamcc (J_\tau))\\
&\subset \f_{\om} (B_{cc}(\f_\tau(p), LKC \diamcc (J_\tau))).
\end{split}
\end{equation*}
Hence if $\xi \in B_{cc}(\f_{\om\tau}(p), \diamcc(J_{\om\tau}))$ then
\begin{equation}
\label{invinclu}
\f_{\om}^{-1}(\xi) \in B_{cc}(\f_\tau(p),L KC \diamcc (J_\tau)) \subset  B_{cc}(\f_\tau(p), \eta_\cS).
\end{equation}
In particular,
\begin{equation}
\label{starcclemma}
\distcc (\f^{-1}(\xi), J_\tau) \leq LKC \diamcc(J_\tau).
\end{equation}
If $p,q \in J_\cS \cap X_{t(\tau)}$ then by \eqref{subrmvt},
$$d_{cc}(\f_{\om\tau}(p),\f_{\om\tau}(q)) \leq \|D \f_\om\|_\infty \diamcc(J_\tau).$$
Hence if $\gamma_{\f_{\om\tau}(p),\f_{\om\tau}(q)}$ is the horizontal geodesic joining $\f_{\om\tau}(p)$ and $\f_{\om\tau}(q)$,
$$[\gamma_{\f_{\om\tau}(p),\f_{\om\tau}(q)}]\subset B_{cc}(\f_{\om\tau}(p), d_{cc}(\f_{\om\tau}(p),\f_{\om\tau}(q))) \subset N_{i(\om)}.$$
Therefore by Lemma \ref{upper-gradient} there exists some $z \in B_{cc}(\f_{\om\tau}(p), d_{cc}(\f_{\om\tau}(p),\f_{\om\tau}(q)))$ such that
\begin{equation}
\begin{split}
\label{inversubmv}
d_{cc}(\f_{\tau}(p),\f_{\tau}(q))&=d_{cc}(\f_{\om}^{-1}(\f_{\om\tau}(p)),\f_{\om}^{-1}(\f_{\om\tau}(p))) \\&\leq \|D \f_\om^{-1}(z)\|\,d_{cc}(\f_{\om\tau}(p),\f_{\om\tau}(q))\\
&=\|D \f_\om(\f^{-1}_\om(z))\|^{-1}d_{cc}(\f_{\om\tau}(p),\f_{\om\tau}(q)).
\end{split}
\end{equation}
If $y$ satisfies $\dist_{cc}(y, J_{\tau}) \leq \diamcc (J_{\tau})$, then by \eqref{starcclemma}
\begin{equation*}
d_{cc}(y,\f_\om^{-1}(z) ) \leq (2+LKC) \diamcc (J_{\tau}).
\end{equation*}
Now by Remark \ref{nsets}, since by \eqref{invinclu} $\f_\om^{-1}(z) \subset B_{cc}(\f_\tau(p) ,\eta_\cS) \subset N_{t(\om)}$,
\begin{equation}
\label{expo2}
\frac{\|D \f_\om(y)\|}{\|D \f_\om(\f_\om^{-1}(z))\|} \leq \exp\left( \frac{(2+LKC) \Lambda_0 K}{1-s}\diamcc(J_\tau) \right).
\end{equation}
The left hand side inequality follows after combining \eqref{inversubmv} and \eqref{expo2}. The proof of the lemma is complete.
\end{proof}

\begin{lemma}
\label{33prss}
Let $\cS$ be a finite, maximal, weakly Carnot conformal GDMS on $(\G,d)$. Let $T_0 \geq 1$ and $\ve>0$. Then there exists some  $\mathtt{d}_{T_0,\ve}$  such that for every $\tau \in E_A^\ast$ such that $\diamcc (J_\tau) \leq \mathtt{d}_{T_0,\ve}$,  every $\a \in [0,1]$, every $T\in [T_0, 2T_0] $ and every pair of words, $\om, \rho \in E_A^\ast$ such that  $\rho \in W^{cc}_{\a,T}(\tau)$ and $\om \rho \in E_A^\ast$,
\begin{equation}
\label{joinwordsw}
\om\rho \in W^{cc}_{\a(1+\ve),T(1+\ve)}(\om\tau).
\end{equation}
\end{lemma}
\begin{proof} Recalling the definition of $W^{cc}_{\a,T}(\tau)$, see Remark \ref{watcc}, we deduce that $i(\rho)=i(\tau)$. Since $\om\rho \in E_A^\ast$, $i(\rho)=t(\om)$; hence $t(\om)=i(\tau)$ and by the maximality of $\cS$, $\om\tau \in E_A^\ast$. We will assume that $\mathtt{d}_{T_0,\ve} <\frac{\eta_\cS}{2T_0LKC}$. Since trivially $i(\om \tau)=i(\om\rho)$, in order to establish \eqref{joinwordsw} it suffices to show that
\begin{equation}
\label{jw1}
\frac1{T(1+\ve)} \leq \frac{\diamcc(J_{\om\tau})}{\diamcc(J_{\om\rho})}\leq T(1+\ve),
\end{equation}
and
\begin{equation}
\label{jw2}
\distcc(J_{\om\tau},J_{\om\rho}) \leq \a(1+\ve) \diamcc(J_{\om\tau}).
\end{equation}

We will first check \eqref{jw1}. Since $\rho \in W^{cc}_{\a,T}(\tau)$,
\begin{equation}
\label{disttaurho}
\distcc(J_{\tau},J_{\rho}) \leq \a \diamcc(J_{\tau}).
\end{equation}
Thus, there exists some $z \in J_{\rho}$ such that
\begin{equation*}
d_{cc}(z, J_{\tau}) \leq \a \diamcc(J_{\tau}) \leq \mathtt{d}_{T_0,\ve} < \frac{\eta_\cS}{LKC}.
\end{equation*}
Therefore we can apply Lemma \ref{dccdistortlemma} and obtain,
\begin{equation}
\label{expest11}
\begin{split}
&\diamcc(J_{\om\tau}) \leq \|D \f_\om(z)\| \exp(\tilde{C}\diamcc(J_{\tau})) \diamcc(J_{\tau}).
\end{split}
\end{equation}
Notice that since $\rho \in W^{cc}_{\a,T}(\tau)$,
\begin{equation*}
\begin{split}
\diamcc(J_{\rho})& \leq T \diamcc(J_{\tau})  \leq 2 T_0 \mathtt{d}_{T_0,\ve}<\frac{\eta_\cS}{LKC}.
\end{split}
\end{equation*}
Therefore, since also $z \in J_{\rho}$, we can apply Lemma \ref{dccdistortlemma} once more and get
\begin{equation}
\label{expest12}
\begin{split}
&\diamcc(J_{\om\rho}) \geq \|D \f_\om(z)\| \exp(-\tilde{C}\diamcc(J_{\rho})) \diamcc(J_{\rho}).
\end{split}
\end{equation}
Combining \eqref{expest11} and \eqref{expest12} we obtain
\begin{equation*}
\begin{split}
\frac{\diamcc(J_{\om\tau}) }{\diamcc(J_{\om\rho})}& \leq \frac{\diamcc(J_{\tau})}{\diamcc(J_{\rho}))}\exp({\tilde{C}\mathtt{d}_{T_0,\ve}(1+2T_0)})\\
&\leq T \exp({\tilde{C}\mathtt{d}_{T_0,\ve}(1+2T_0)}).
\end{split}
\end{equation*}
Choosing $\mathtt{d}_{T_0,\ve}$ small enough, we obtain the right hand inequality in \eqref{jw1}. The proof of the remaining inequality in \eqref{jw1} is very similar and we omit it.

We will now establish \eqref{jw2}. By \eqref{disttaurho} there exist $p \in J_\cS \cap X_{t(\tau)} $ and $q \in J_\cS \cap X_{t(\rho)}$ such that
\begin{equation}
\label{pqtaurho}
\begin{split}
\distcc(J_{\tau},J_{\rho})&=d_{cc}(\f_\tau(p), \f_\rho(q))\leq \a \diamcc(J_{\tau}).
\end{split}
\end{equation}
Hence, if $\gamma_{\f_\tau(p), \f_\rho(q)}$ is the horizontal geodesic curve joining $\f_\tau(p)$ and $\f_\rho(q)$, arguing as in \eqref{arcinclu} and using \eqref{pqtaurho} we deduce that
\begin{equation*}
\begin{split}[\gamma_{\f_\tau(p), \f_\rho(q)}]&\subset B_{cc}(\f_\tau(p),\dcc(\f_\tau(p), \f_\tau(q))\\
&\subset B_{cc}(\f_\tau(p),\a \diamcc(J_{\tau})) \\
&\subset B_{cc}(\f_\tau(p), \mathtt{d}_{T_0,\ve}) \subset N_{i(\om)}.
\end{split}
\end{equation*}
Therefore by Lemma \ref{upper-gradient} there exists some $\zeta \in N_{i(\om)}$ such that $$\distcc(\zeta,J_{\tau} )\leq  \diamcc(J_{\tau}) \leq \mathtt{d}_{T_0,\ve}$$
and
$$d_{cc}(\f_{\om\tau}(p),\f_{\om\rho}(q)) \leq \| D \f_\om(\zeta)\| d_{cc}(\f_\tau(p), \f_\rho(q)).$$
Thus,
\begin{equation}
\label{adiamest}
\begin{split}
\distcc(J_{\om\tau},J_{\om\rho})\leq \a \| D \f_\om(\zeta)\|\ \diamcc(J_{\tau}).
\end{split}
\end{equation}
By Lemma \ref{dccdistortlemma},
\begin{equation}
\label{derivexpest}
\|D \f_\om(\zeta)\| \leq \frac{\diamcc(J_{\om\tau})}{ \diamcc(J_{\tau})} \exp (\tilde{C} \diamcc(J_{\tau})).
\end{equation}
Combining \eqref{adiamest} and \eqref{derivexpest},
\begin{equation*}
\begin{split}
\distcc(J_{\om\tau},J_{\om\rho})& \leq \a\,\diamcc(J_{\om\tau})   \exp (\tilde{C} \mathtt{d}_{T_0,\ve} ))\\
& \leq \a (1+\ve) \diamcc(J_{\om\tau}),
\end{split}
\end{equation*}
assuming that $\mathtt{d}_{T_0,\ve}$ is taken small enough. Therefore \eqref{jw2} has been proven and the proof of the lemma is complete.
\end{proof}

\begin{proposition}
\label{lastfstep}
Let $\cS$ be a finite, irreducible and maximal weakly Carnot conformal GDMS on $(\G,d)$. If $\cS$  satisfies the Bandt-Graf condition, then there exist open sets $O_v, v \in V,$ such that,
\begin{enumerate}
\item $J_\cS \cap O_v \neq \emptyset$ for all $v\in V$,
\item $O_v \subset \Int (N_v) \subset W_v,$ for all $v \in V,$
\item if $O:=\cup_{v\in V} O_v$ then for all $\tau, \rho \in E_A^\ast, \tau \neq \rho,$
$$\f_\tau(O) \subset O \mbox{ and }\f_\tau(O) \cap \f_\rho (O)=\emptyset.$$
\end{enumerate}
\end{proposition}
\begin{proof}
Recall that the sets $N_v$ where defined in Remark \ref{nsets}. Since $E$ is finite, there exists $T_0$ large enough such that for all $e \in E$ and for all $\tau \in E_A^\ast$ such that $\tau e \in E_A^\ast$,
\begin{equation}
\label{lastprss1}
\diamcc(J_\tau) \leq T_0^2 \diamcc(J_{\tau e})
\end{equation}
and
\begin{equation}
\label{lastprss2}
T_0\diamcc(J_e) \geq 1.
\end{equation}
This is possible because by \eqref{quasiconvexity}, \eqref{4.1.9}, \eqref{quasi-multiplicativity1}, and Corollary \ref{newbilip21},
\begin{equation*}
\begin{split}
\diamcc(J_\tau) &\leq L \diam (\f_\tau(J_\cS \cap X_{t(\tau)})) \\
&\leq L M \|D \f_\tau\|_\infty  \|D \f_e\|_\infty  \frac{1}{\|D \f_e\|_\infty } \\
&\leq \frac{LMK}{\min_{e \in E} \{\|D \f_e\|_\infty\}} \|D \f_{\tau e}\|_\infty \\
&\leq \frac{2L^3MK^2}{\kappa_0 \mu_0\min_{e \in E}\{\|D \f_e\|_\infty \}} \diamcc(J_{\tau e}).
\end{split}
\end{equation*}
We will now show that if $r \in (0,1]$, then for any $\om \in E_A^\ast$ such that $\diamcc(J_\om)<r$ there exists some $k=1,\dots,|\om|,$ such that
\begin{equation}
\label{lastprsss3}
T_0^{-1} \leq \frac{\diamcc(J_{\om|_k})}{r} \leq T_0.
\end{equation}
By \eqref{lastprss2}, there exists some $k=1,\dots,|\om|,$ such that
\begin{equation}
\label{lastprss3b}
\diamcc(J_{\om|_k}) \geq \frac{r}{T_0}.
\end{equation}
Let $k_0$ be the minimal $k\in \N$ satisfying \eqref{lastprss3b}. Since $\diamcc(J_\om)<r$, $k_0 < |\om|$. By the minimality of $k_0$,
\begin{equation}
\label{lastprss4}
\diamcc (J_{\om|_{k_0+1}}) <\frac{r}{T_0}.
\end{equation}
By \eqref{lastprss1} and \eqref{lastprss4},
$$\diamcc(J_{\om|_{k_0}}) \leq T_0^2 \diamcc (J_{\om|_{k_0+1}})  \leq T_0 r.$$
Hence \eqref{lastprss2} follows.

We now introduce some notation following \cite{PRSS}. Recalling Remark \ref{watcc}, for $\a \in [0,1]$ and $\om \in E_A^\ast$ let
$$W_\a(\om)=W^{cc}_{\a, (1+\a)T_0}(\om) \mbox{ and }M_\a(\om)= \sharp W_\a^{cc}(\om).$$
Then by \eqref{wccboundcc},
\begin{equation}
\label{lastprss5}
M_\a (\om) \leq C^{cc}(1,2T_0).
\end{equation}
Let $\overline{C}=\left \lceil{ C^{cc}(1,2T_0)}\right \rceil +1$, where  $\left \lceil{x }\right \rceil =\min\{k \in \N: k \geq x\}$, for $x \geq 0$.

For fixed $\om \in E_A^\ast$ consider the function $f_\om :[0,1] \ra \N$ defined by $f_\om (\a)=M_\a (\om)$. By the definition of the sets $W_\a(\om)$ we deduce that the function $f_\om$ is increasing. For $r\geq 0$, let
\begin{equation}
\label{tildemar}
\widetilde{M_\a}(r)=\sup\{M_\a(\om): \om \in E_A^\ast, \diamcc(J_\om)\leq r\}.
\end{equation}
Notice that the function $\widetilde{f_r}: [0,1] \ra \N$ defined by $\widetilde{f_r}(\a)=\widetilde{M_\a}(r)$ has the following properties,
\begin{enumerate}
\item it is increasing,
\item it is bounded by  $C^{cc}(1,2T_0)$,
\item there exist $\alpha_1, \alpha_2, \in [0,1]$ such that $\a_2-\a_1 \geq \frac{1}{\overline{C}}$ and $\widetilde{f_r}$ is constant on $[a_1,a_2]$.
\end{enumerate}
The first property follows by the definition of $W_\a^{cc}(\om)$ and the second property follows by \eqref{lastprss5}. To show that (iii) holds, suppose on the contrary that it fails. Subdivide $[0,1]$ in $\overline{C}$ subintervals of length $1/\overline{C}$. Since $\widetilde{f_r}$ is interger valued and increasing, if $[a,b]$ is any of the aforementioned subintervals, $\widetilde{f_r}(b)-\widetilde{f_r}(a)>1$. Hence $\widetilde{f_r}(1) \geq \overline{C}>C^{cc}(1,2T_0) $, which contradicts \eqref{lastprss5}.

Observe that the maximum is attained in \eqref{tildemar}. That is there exists some $\overline{\om} \in E_A^\ast$ such that ,  $$\diamcc(J_{\bom})\leq r$$
and $$\widetilde{M_{\a_1}}(r)=M_{\a_1}(\overline{\om}).$$
Since $f_{\bom}$ is increasing,
\begin{equation}
\label{lastprss6}
M_{\a_2}(\bom) \geq M_{\a_1} (\bom)=\widetilde{M_{\a_1}}(r)=\widetilde{M_{\a_2}}(r) \geq M_{\a_2}(\bom),
\end{equation}
in particular
\begin{equation}
\label{lastprss61}
M_{\a_1} (\bom)=M_{\a_2}(\bom).
\end{equation}
Observe that since $a_1 \leq a_2$, $W_{\a_1}(\bom) \subset W_{\a_2}(\bom)$. Thus by \eqref{lastprss61}, $W_{\a_1}(\bom) = W_{\a_2}(\bom)$.

Now let $\ve=\frac{1}{2 \overline{C}}$ and let $r=\min\{1, \mathtt{d}_{T_0,\ve}\}$, where $\mathtt{d}_{T_0,\ve}$ is as in Lemma \ref{33prss}. We will first show that if $\tau \in W_{\a_1}(\bom)$ and $\rho \in E_A^\ast$ such that $\rho\bom \in E_A^\ast$ then
\begin{equation}
\label{claim1}
\rho\tau \in W_{\a_2}(\rho\bom).
\end{equation}
Notice that  since $\rho \bom \in E_A^\ast$, $t(\rho)=i(\bom)$. By the definition of the sets $W_{\a_1}(\bom)$, if $\tau \in W_{\a_1}(\bom)$ then $i(\tau)= i(\bom)$. Hence $i(\tau)=t(\rho)$, and by the maximality of $\cS$ we deduce that $\rho \tau \in E_A^\ast$. Observe that since $\tau \in W_{\a_1}(\bom)=W^{cc}_{\a_1,(1+\a_1)T_0}, \rho \tau \in E_A^\ast$ and $\diamcc(J_{\bom})< \mathtt{d}_{T_0,\ve},$  Lemma \ref{33prss} implies that,
$$\rho \tau \in W^{cc}_{\a_1(1+\ve), (1+\ve)(1+\a_1)T_0}(\rho \bom).$$
Hence \eqref{claim1} will follow if we show
\begin{equation}
\label{lastprss7}
\a_1(1+\ve) \leq \a_2,
\end{equation}
and
\begin{equation}
\label{lastprss8}
(1+\a_1)(1+\ve) \leq(1+ \a_2).
\end{equation}
Since $\a_2-\a_1 \geq \frac{1}{\overline{C}}$, and $\a_1 \leq 1$,
$$\a_2 \geq \a_1+ \frac{\a_1}{\overline{C}}> \a_1\left(1+\frac{1}{2 \overline{C}}\right)=\a_1(1+\ve),$$
hence \eqref{lastprss7} follows. Moreover
\begin{equation*}
\begin{split}
(1+\a_1)(1+\ve)& \leq (1+\a_1)\left(1+\frac{1}{2 \overline{C}}\right)=1+\a_1+\frac{1}{2 \overline{C}}+\frac{\a_1}{2 \overline{C}} \\
&\leq 1+\a_1+\frac{1}{ \overline{C}} \leq 1+\a_2,
\end{split}
\end{equation*}
hence \eqref{lastprss8} follows. Therefore \eqref{claim1} follows as well.

We will now show that if $\rho\bom \in E_A^\ast$ then
\begin{equation}
\label{lastprss12}
M_{\a_2}(\rho \bom) = M_{\a_2}(\bom).
\end{equation}
We will first prove that  if $\rho\bom \in E_A^\ast$ then
\begin{equation}
\label{lastprss9}
M_{\a_2}(\rho \bom) \geq M_{\a_1}(\bom).
\end{equation}
As we have remarked already in the proof of \eqref{claim1}, if $\tau \in W_{\a_1}(\bom)$ then $\rho\tau \in E_A^\ast$. Hence by \eqref{claim1}, $\rho \tau \in W_{\a_2}(\rho\bom)$. This implies that
$$M_{\a_2}(\rho \bom)= \sharp W_{\a_2}(\rho \bom) \geq \sharp W_{\a_1}( \bom)=M_{\a_1(\bom)},$$
and thus \eqref{lastprss9} holds.

Recall that $r \leq \mathtt{d}_{T_0,\ve}<\eta_\cS$, hence since $\diamcc (J_{\bom}) \leq r$ using Lemma \ref{upper-gradient} as in Lemma \ref{dccdistortlemma} we deduce that
$$ \diamcc( J_{\rho \bom}) \leq\| D \f_{\rho}\|_\infty \diamcc(J_{\bom})< \diamcc (J_{\bom})<r.$$
Hence $M_{\a_2}(\rho \bom) \leq \widetilde{M_{\a_2}}(r)$ and by \eqref{lastprss6} we deduce that
\begin{equation}
\label{lastprss10}
M_{\a_2}(\rho \bom) \leq M_{\a_1}(\bom).
\end{equation}
By \eqref{lastprss9} and \eqref{lastprss10}, we deduce that
\begin{equation}
\label{prelastprss12}
M_{\a_2}(\rho \bom) = M_{\a_1}(\bom).
\end{equation}
Now \eqref{lastprss12} follows by \eqref{prelastprss12} and \eqref{lastprss61}.  Thus \eqref{claim1} and \eqref{lastprss12} imply that
\begin{equation}
\label{lastprss13}
W_{\a_2}(\rho \bom)=\{\rho \tau: \tau \in W_{\a_2}(\bom)\}.
\end{equation}

Let
$$O=\bigcup_{ \tau \in E_A^\ast: \tau \bom \in E_A^\ast} \f_\tau (B_{cc}(J_{\bom}, \texttt{d}_1)),$$
where
$$\texttt{d}_1=\min \left\{\frac{\eta_\cS}{10},  \frac{\min\{\lambda_\cS^{cc}, \a_2\} \, \kappa_0 \mu_0 \|D \f_{\bom}\|_\infty}{5 K^2 \Lambda_0 }\right\},$$
and recalling Remark \ref{keyercloserem}  $\lambda^{cc}_\cS$ is  just $\lambda_\cS$ with respect to the $d_{cc}$ metric. Observe that if we set $O_v=W_v \cap O$ for $v\in V$ then $O=\cup_{v\in V} O_v$, $$O_v \cap J_\cS \neq \emptyset \text{ and }O_v \subset \Int( N_v).$$  Therefore in order to finish the proof of the proposition it is enough to show that for all $\rho, \tau \in E_A^\ast, \rho \neq \tau$, $\f_\tau (O) \subset O$ and
\begin{equation}
\label{lastprssclaimd}
\f_\rho(O) \cap \f_\tau(O)=\emptyset.
\end{equation}
First notice that $\f_\tau(O)$ is well defined for all $\tau \in E^\ast_A$, because by the irreducibility of $\cS$ for all $\tau \in E$ there exists some $\upsilon \in E_A^\ast$ such that $\tau \upsilon \bom \in E_A^\ast$. Moreover  by the definition of the set $O$ it readily follows that $\f_\tau (O) \subset O$. Hence in order to prove \eqref{lastprssclaimd} it is enough to show that for all $e,e' \in E$ and for all $\tau, \tau' \in E_A^\ast$ such that $e\tau \bom \in E_A^\ast$ and $e'\tau' \bom \in E_A^\ast$,
\begin{equation}
\label{lastprss14}
\f_{e\tau}(B_{cc}(J_{\bom}, \texttt{d}_1))\cap\f_{e'\tau'}(B_{cc}(J_{\bom},\texttt{d}_1))=\emptyset.
\end{equation}
For $p \in B_{cc}(J_{\bom}, \texttt{d}_1)$ and $q \in J_{\bom}$ such that $d_{cc}(p,q)< \texttt{d}_1$, using \eqref{quasiconvexity}, \eqref{l45pr}, \eqref{quasi-multiplicativity1} and Proposition \ref{newbilip21} we obtain
\begin{equation*}
\begin{split}
\dist_{cc} (\f_{e\tau} (p), J_{e \tau \bom}) &\leq \dcc (\f_{e\tau} (p),\f_{e\tau} (q)) \\
&\leq L\Lambda_0 \|D \f_{e\tau}\|_\infty \dcc(p,q) \leq L \Lambda_0 \|D \f_{e\tau}\|_\infty  \texttt{d}_1\\
& \leq L \Lambda_0 \|D \f_{e\tau \bom}\|_\infty \frac{\|D \f_{e\tau}\|_\infty}{\|D \f_{e\tau \bom}\|_\infty}\texttt{d}_1\\
& \leq K  L\Lambda_0 \|D \f_{e\tau \bom}\|_\infty \frac{\|D \f_{e\tau}\|_\infty}{\|D \f_{e\tau}\|_\infty\|D \f_{\bom}\|_\infty}\texttt{d}_1\\
&= \frac{K L \Lambda_0 \texttt{d}_1}{\|D \f_{\bom}\|_\infty} \|D \f_{e\tau \bom}\|_\infty \\
&\leq  \frac{2 K^2 L \Lambda_0   \texttt{d}_1}{\kappa_0 \mu_0 \|D \f_{\bom}\|_\infty} \diamcc (J_{e\tau \bom}).
\end{split}
\end{equation*}
Hence,
\begin{equation}
\label{d2a}
\f_{e\tau}(B_{cc}(J_{\bom}, \texttt{d}_1)) \subset B_{cc}(J_{e\tau\bom}, \texttt{d}_2 \diamcc (J_{e\tau \bom})),
\end{equation}
where
$$ \texttt{d}_2=\frac{2 K^2 L \Lambda_0   }{\kappa_0 \mu_0 \|D \f_{\bom}\|_\infty}\texttt{d}_1.$$
In the same manner
\begin{equation}
\label{d2b}
\f_{e'\tau'}(B_{cc}(J_{\bom}, \texttt{d}_1)) \subset B_{cc}(J_{e'\tau'\bom}, \texttt{d}_2 \diamcc (J_{e'\tau' \bom})).
\end{equation}
Notice that by by \eqref{d2a} and \eqref{d2b},  in order to prove \eqref{lastprss14} it is enough to show that
\begin{equation}
\label{fff1}
 B_{cc}(J_{e\tau\bom}, \texttt{d}_2 \diamcc (J_{e\tau \bom})) \cap B_{cc}(J_{e'\tau'\bom}, \texttt{d}_2 \diamcc (J_{e'\tau' \bom}))=\emptyset.
 \end{equation}

If $i(e) \neq i(e')$, then $X_{i(e)} \cap X_{i(e')}= \emptyset$.  We will prove \eqref{fff1} by contradiction. If \eqref{fff1} fails, then using that $ \texttt{d}_2 < \frac{\lambda^{cc}_\cS}{2}$,
 \begin{equation*}
 \begin{split}
 \distcc(X_{i(e)}, X_{i(e')}) &\leq \distcc(J_{e\tau \bom},J_{e'\tau' \bom})\\
 &\leq 2 \texttt{d}_2 \max \{ \diamcc (X_v), v \in V\} \\
 &< \lambda^{cc}_\cS  \max \{ \diamcc (X_v), v \in V\}\\
 &= \min\{ \distcc (X_{v_1},X_{v_2}):v_1,v_2 \in V\},
 \end{split}
 \end{equation*}
 which is impossible. Therefore  in the case where $i(e) \neq i(e')$, \eqref{fff1}  holds, and hence \eqref{lastprss14} follows.

 Now suppose that $i(e) = i(e')$. Without loss of generality we can assume that $$\diamcc (J_{e\tau \bom}) \geq \diamcc (J_{e'\tau' \bom}).$$
 Let $\bom'=\tau' \bom$. By \eqref{lastprsss3} there exists some $m=1, \dots, |\tau'|+|\bom|,$ such that
 \begin{equation}
 \label{lastprss17}
 T_0^{-1} \leq \frac{\diamcc (J_{e' \bom'|_{m}})}{\diamcc (J_{e\tau \bom})} \leq T_0.
 \end{equation}
 We are going to prove that
 \begin{equation}
 \label{lastprss18}
 \distcc (J_{e\tau \bom},J_{e' \bom'|_{m}})> \alpha_2 \diamcc (J_{e\tau \bom}).
 \end{equation}
 Recall that
 \begin{equation*}
\begin{split}
W^{cc}_{\a_2, T_0}(e\tau \bom)=\Big\{\rho &\in E_A^\ast:\, i(e\tau \bom)=i(\rho),\\
&\quad \quad T_0^{-1} \leq \frac{\diamcc(J_{\rho})}{\diamcc(J_{e\tau \bom})} \leq T_0, \\
&\quad \quad \quad \text{ and }\distcc(J_{\rho}, J_{e\tau \bom}) \leq \a_2 \diamcc (J_{e\tau \bom}) \Big\}.
\end{split}
\end{equation*}
Suppose that \eqref{lastprss18} fails. Hence, by \eqref{lastprss17} and the fact that $i(e)=i(e')$ we deduce that
\begin{equation}
\label{lfeeprime}
e' \bom'|_{m} \in W^{cc}_{\a_2, T_0}(e\tau \bom) \subset W^{cc}_{\a_2, (1+\a_2)T_0}(e\tau \bom)=W_{\a_2}(e\tau \bom).
\end{equation}
But this is not possible, because by \eqref{lastprss13}
$$W_{\a_2}(e\tau \bom)=\{e\tau \upsilon: \upsilon \in W_{\a_2}(\bom)  \},$$
and this contradicts \eqref{lfeeprime} because $e' \bom'|_{m} \notin W_{\a_2}(e\tau \bom).$
Hence  \eqref{lastprss18} holds and in particular it implies that
\begin{equation}
\label{lastprss19}
 \distcc(J_{e\tau \bom},J_{e' \tau' \bom}) > \alpha_2 \diamcc (J_{e\tau \bom}).
\end{equation}
Since $\texttt{d}_2 \leq \a_2/2$, \eqref{lastprss19} implies \eqref{fff1}. Therefore \eqref{lastprss14} follows.
\end{proof}
We can now complete the proof of Theorem \ref{prss} by proving the remaining implication (iii)$\imp$ (iv). For $v \in V$ let $X'_v= \overline{O_v}$, where the sets $O_v$ are as in Proposition \ref{lastfstep}. By the same proposition  the GDMS  $\cS'=\{\f_\e: X'_{t(e)} \ra X'_{i(e)}\}$ satisfies the strong open set condtition and it is equivalent to $\cS$. The proof is complete.
 \end{proof}

\begin{remark}
\label{rajvil}
In the very nice paper \cite{RV}, Rajala and Vilppolainen studied finite weakly controlled Moran contractions in metric spaces.  We remark that due to our bounded distortion theorems proved earlier in Chapters \ref{chap:conformal-metric-and-geometric-properties} and \ref{chap:CGDMS}, it follows immediately that finite conformal IFS on a Carnot group $(\G,d)$ are properly semiconformal, see \cite{RV} for the exact definition. In particular if $\cS$ if a finite conformal IFS the equivalence (i)$\Leftrightarrow$(ii)$\Leftrightarrow$(iv) of Theorem \ref{prss} follows also from the results of Rajala and Vilppolainen, see \cite[Theorem 4.9]{RV}.
\end{remark}

\chapter{Hausdorff dimensions of invariant measures}\label{chap:invariant-measures}

In this short concluding chapter we establish a formula for the Hausdorff dimension of the projection of a shift-invariant measure onto the limit set of a Carnot IFS. We assume throughout this chapter that $\cS$ is a finitely irreducible Carnot conformal GDMS with countable alphabet. Let us observe first that the same argument as at the beginning of the proof of Theorem~\ref{t3.1.7} gives
the following fact which can be called a measure theoretic open set condition.

\begin{theorem}\label{t4.3.1}
Let $\cS$ be a Carnot conformal GDMS of bounded coding type. If $\mu$ is a Borel probability shift-invariant ergodic measure on $E_A^\infty$, then
\begin{equation}\lab{4.3.1}
\mu\circ\pi^{-1}\(\phi_\om(X_{t(\om)})\cap\phi_\tau\(X_{t(\tau)})\)=0
\end{equation}
for all incomparable words $\om,\tau\in E_A^*$.
\end{theorem}

Recall that if $\nu$ is a finite Borel \index{Hausdorff dimension!of a measure} \index{measure!Hausdorff dimension of}
measure on a metric space $X$, then $\dim_{\cH}(\nu)$, the Hausdorff dimension of $\nu$, is the
minimum of the Hausdorff dimensions of sets of full $\nu$ measure.
If $\cA$ and $\cB$ are two partitions of $E_A^\infty$ their \textit{join} \index{join!of two partitions} \index{partition!join of two} is defined as $\cA \vee \cB=\{A \cap B: A\in \cA, B \in \cB\}$. We also set $\cA^n:=\cA \vee \sigma^{-1}(\cA) \vee \dots  \vee \sigma^{-n}(\cA)$, where $\sg:E_A^\infty\to E_A^\infty$ is the shift map.  By
$\a=\{[e]:e\in E\}$ we denote the partition of
$E_A^\infty$ into initial cylinders of length $1$. Let $\mu$ be a Borel
shift-invariant ergodic probability measure on $E_A^\infty$ and denote by
$$
\H_\mu(\a)=-\sum_{e\in E}\mu([e])\log(\mu([e]))
$$
the \textit{entropy} \index{entropy!of a partition} \index{partition!entropy of} of the partition $\a$ with respect to the measure $\mu$. Since $\a$ is a generating partition, when $\H_\mu(\a)<\infty$ the \textit{Kolmogorov--Sinai entropy} \index{Kolmogorov-Sinai entropy} with respect to the shift map can be defined as
$$
\hmu(\sigma)=-\lim_{n \rightarrow \infty} \frac{1}{n} \sum_{A \in \a^n} \mu(A) \log (\mu(A)).
$$
See, for example, \cite[Theorem 2.8.7(b)]{PU}. Moreover $h_\mu(\sigma) \leq \H_\mu (\a)$. Notice also that $\a^n$ is the partition consisting of the cylinders of length $n$. Finally the \textit{characteristic Lyapunov exponent} with \index{Lyapunov exponent} respect to $\mu$ and $\sigma$ is defined as
$$
\chi_\mu(\sg)=-\int_{E_A^\infty}\zeta d\mu > 0
$$
where $\zeta:E_A^\infty\to\R$ is the function defined in \eqref{1MU_2014_09_10}. Now we want to establish a dynamical formula for the Hausdorff dimension of the projection measures $\mu\circ\pi^{-1}$, frequently referred to as volume lemma. We will do in fact more: we will prove exact dimensionality of such measures and will provide a dynamical formula for the local Hausdorff dimension. We start off with the precise definitions of the concepts involved.

\

\begin{definition}\label{d6.4.10}
Let $\mu$ be a non-zero Borel measure on a metric space $(X,\rho)$.
We define:
$$
\HDl(\mu):=\inf\{\HD(Y): \mu( Y)>0\},$$
and
$$\HD(\mu)=\inf\{\HD(Y): \mu(X\sms Y)=0\}.
$$
and call the latter of these two numbers the \index{Hausdorff dimension of measure} 
 {\it Hausdorff dimension} of the measure $\mu$, while the former one is referred to as the \index{lower Hausdorff dimension of measure} 
 {\it lower Hausdorff dimension} of the measure $\mu$.
\end{definition}

 \fr Of course $\HDl(\mu)\le \HD(\mu)$; we will see soon that in the context of conformal GDMS these quantities are frequently equal.

\sp\fr Analogously $\PDl(\mu)$ and $\PD(\mu)$ will respectively denote the {\it lower packing dimension} and the
\index{packing dimension of measure} {\it packing dimension} of the
measure~$\mu$. 

\sp\fr The next definition introduces concepts that form effective tools to
calculate the dimensions introduced above. 

\

\begin{definition}\label{d220130510}
Let $\mu$ be a Borel probability measure on a metric space $(X,\rho)$.
For every point $x\in X$ we define the {\it lower and upper pointwise
dimension} \index{dimension pointwise lower}
\index{dimension pointwise upper} \!\!\!\!\!
of the measure $\mu$ at the point $x\in X$ respectively as
$$
\un d_\mu(x):=\liminf_{r\to 0} {\log\mu(B(x,r))\over \log r}\ \ {\text {and} }\ \
\ov d_\mu(x):=\limsup_{r\to 0} {\log\mu(B(x,r))\over \log r}.
$$
In the case when both numbers $\un d_\mu(x)$ and $\ov d_\mu(x)$ are
equal, we denote their common value by $d_\mu(x)$. We then obviously have
$$
d_\mu(x)=\lim_{r\to 0}{\log\mu(B(x,r))\over \log r},
$$
and we call $d_\mu(x)$ the {\it pointwise dimension} of the measure
$\mu$ at the point $x\in X$. If then in addition the function $X\ni x \mapsto d_\mu(x)$ is $\mu$-a.e. constant (refered to as $d_\mu$ in the sequel), we call the measure $\mu$ {\it dimensional exact} \index{dimensional exact}. 
\end{definition}

\fr The following theorem about Hausdorff and packing dimensions of 
Borel measures is well-known and its proof can be found for example in \cite{PU}. 

\begin{theorem}\label{t6.6.4.}
If $\mu$ is a Borel probability measure on a metric space $(X,\rho)$, then
$$
\HDl(\mu)=\ess\inf \un d_\mu,\  \  \  \HD(\mu)=\ess\sup \un
d_\mu
$$
and
$$
\PDl (\mu)=\ess\inf\ov d_\mu, \  \  \  \PD(\mu)=\ess\sup\ov d_\mu.
$$
\end{theorem}

\fr As an immediate consequence of this theorem we get the following.

\begin{corollary}\label{cor dimexact}
If $\mu$ is a dimensional exact Borel probability measure on a metric space $(X,\rho)$, then all its dimensions are equal:
$$
\HDl(\mu)=\HD(\mu)= \PDl(\mu)=\PD(\mu)=d_\mu.
$$
\end{corollary}

\fr We now shall prove the following main result of this section, versions of which have been established in many contexts of conformal dynamics.

\begin{theorem}\label{t4.3.2}
Let $\cS$ be a boundary regular Carnot conformal GDMS. Suppose that $\mu$ is a Borel  probability shift-invariant ergodic measure on $E_A^\infty$ such that at least one of the numbers $\H_\mu(\a)$ or $\chi_\mu(\sg)$ is finite. Then the measure $mu\circ\pi^{-1}$ is dimensional exact and
$$
\begin{aligned}
\dim_{_\star\cH}(\mu\circ\pi^{-1})&=\dim_\cH(\mu\circ\pi^{-1}) 
=d_{\mu\circ\pi^{-1}}
={{\rm h}_\mu(\sg)\over \chi_\mu(\sg)}
=\PDl(\mu\circ\pi^{-1})
=\PD(\mu\circ\pi^{-1}).
\end{aligned}
$$
\end{theorem}

\begin{proof}
By virtue of Corollary~\ref{cor dimexact} it only suffices to prove the  equality
\begin{equation}\label{1_2016_05_03}
d_{\mu\circ\pi^{-1}}={{\rm h}_\mu(\sg)\over \chi_\mu(\sg)}.
\end{equation}
Suppose first that $\H_\mu(\a)<+\infty$. Since $\H_\mu(\a)<\infty$ and since $\a$ is a generating partition, the entropy
$\hmu(\sg)=\hmu(\sg,\a) \le\H_\mu(\a)$ is finite. Thus,
Birkhoff's Ergodic Theorem \index{Birkhoff's Ergodic Theorem} \index{Breiman-Shannon-McMillan theorem} and the Breiman--Shannon--McMillan Theorem imply that there exists a set $Z\sbt E_A^\infty$ such that $\mu(Z)=1$,
\begin{equation}\label{4.3.2.1}
\lim_{n\to\infty}-{1\over n}\sum_{j=0}^{n-1}\zeta\circ\sg^j(\om)=
\chi_\mu(\sg)
\end{equation}
and
\begin{equation}\label{4.3.2.2}
\lim_{n\to\infty}{-\log(\mu([\om|_n]))\over n}=\hmu(\sg)
\end{equation}
for all $\om\in Z$. Note that \eqref{4.3.2.1} holds even if $\chi_\mu(\sg)=+\infty$ since the function $-\zeta:E_A^\infty\to\R$ is  everywhere positive. Fix now $\om\in Z$ and $\eta>0$. For $r>0$ let
$n=n(\om,r) \ge 0$ be the least integer such that
$\phi_{\om|_n}(X_{t(\om_n)})\sbt
B(\pi(\om),r)$. Then
$$
\log\(\mu\circ\pi^{-1}(B(\pi(\om),r))\)\ge \log\(\mu\circ\pi^{-1}
(\phi_{\om|_n}(X_{t(\om_n)}))\)\ge \log(\mu([\om|_n]))\ge -(\hmu(\sg)+\eta)n
$$
for $r$ small enough (i.e., for $n=n(\om,r)$ is large enough), moreover,
$$
\diam\(\phi_{\om|_{n-1}}(X_{t(\om_{n-1})})\)\ge r.
$$
Using \eqref{4.1.9} and Lemma \ref{l12013_03_11} the latter inequality implies that
$$
\log r \le \log\(\diam(\phi_{\om|_{n-1}}(X_{t(\om_{n-1})}))\)
 \le \log\(\La M K|\phi_{\om|_{n-1}}'(\pi(\sg^{n-1}(\om)))|\).
 $$
Recalling that $\zeta(\om)= \log \|D \f_{\om_1}(\pi( \sigma(\om)))\|$, it follows from \eqref{4.3.2.1} that for arbitrary $N>0$ and sufficiently large $n$,
\begin{equation*}\begin{split}
-\frac{1}{n-1} \sum_{j=0}^{n-2}\log \|D \f_{\om_{j+1}}(\pi(\sigma^{j+1}(\om)))\|
&=-\frac{1}{n-1} \sum_{j=1}^{n-1}\log \|D \f_{\om_{j}}(\pi(\sigma^{j}(\om)))\| \\
&\geq \chi_\mu'(\sg)-\eta
\end{split}\end{equation*}
where $\chi_\mu'(\sg)=\min\{N,\chi_\mu(\sg)\}$. Therefore
\begin{equation*}\begin{split}
\log r
&\le \log (\La M K)+ \sum_{j=1}^{n-1} \log|D\phi_{\om_j}'(\pi(\sg^j(\om)))| \\
&\le \log (\La M K)-(n-1)(\chi_\mu'(\sg)-\eta)
\end{split}\end{equation*}
for all $r>0$ small enough. Hence
\begin{equation*}\begin{split}
{\log\(\mu\circ\pi^{-1}(B(\pi(\om),r))\) \over \log r}
&\le {-(\hmu(\sg)+\eta)n \over \log (\La M K)-(n-1)(\chi_\mu'(\sg)-\eta)} \\
&= {\hmu(\sg)+\eta\over {-\log (\La M K)\over n} +{n-1\over n} (\chi_\mu'(\sg)-\eta)}
\end{split}\end{equation*}
for such $r$. Letting $r\to 0$ (and consequently $n\to\infty$), we obtain
$$
\limsup_{r\to 0} {\log\(\mu\circ\pi^{-1}(B(\pi(\om),r))\) \over \log r}
\le {\hmu(\sg)+\eta \over \chi_\mu'(\sg)-\eta}.
$$
Since $\eta>0$ was arbitrary, we have that
$$
\limsup_{r\to 0} {\log\(\mu\circ\pi^{-1}(B(\pi(\om),r))\) \over \log r}
\le {\hmu(\sg) \over \chi_\mu'(\sg)}
$$
for all $\om\in Z$. Letting $M\to+\infty$, we finally obtain
$$
\ov d_{\mu\circ\pi^{-1}}(\pi(\om))
=\limsup_{r\to 0} {\log\(\mu\circ\pi^{-1}(B(\pi(\om),r))\) \over \log r}
\le {\hmu(\sg) \over \chi_\mu(\sg)}
$$
for all $\om\in Z$. Therefore, as $\mu\circ \pi^{-1}(\pi(Z))=1$, we get that
\begin{equation}\label{2_2016_05_03}
\ess\sup(\ov d_{\mu\circ\pi^{-1}})\le\hmu(\sg)/\chi_\mu(\sg).
\end{equation}
Let us now prove the opposite counterpart of this inequality. If $\chi_\mu(\sg)=+\infty$, then $\hmu(\sg)/\chi_\mu(\sg)=0$ and we are done. For the rest of the proof, we assume that $\chi_\mu(\sg)<+\infty$.
Let $J_1\sbt J_\cS$ be an arbitrary Borel set such that
$\mu\circ\pi^{-1}(J_1)>0$. Fix $\eta>0$. In view of \eqref{4.3.2.2} and
Egorov's Theorem there exist $n_0\ge 1$ and a Borel set $\^J_2\sbt
\pi^{-1}(J_1)$ such that $\mu(\^J_2)>\mu(\pi^{-1}(J_1))/2>0$,
\begin{equation}\lab{4.3.3}
\mu([\om|_n])\le \exp\((-\hmu(\sg)+\eta)n\)
\end{equation}
and $$\|D\phi_{\om|_n}(\pi(\sg^n(\om))\|\ge \exp\((-\chi_\mu(\sg)-\eta)n\)$$
for all $n\ge n_0$ and all $\om\in \^J_2$. In view of \eqref{4.1.10}, the final
inequality implies that there exists $n_1\ge n_0$ such that
\begin{equation}\begin{split}\lab{4.3.4}
\diam\(\phi_{\om|_n}(X_{t(\om_n)})\)
&\ge (\La M K)^{-1}\exp\((-\chi_\mu(\sg)-\eta)n\) \\
&\ge \exp\(-(\chi_\mu(\sg)+2\eta)n\)
\end{split}\end{equation}
for all $n\ge n_1$ and all $\om\in \^J_2$. Given $0<r<\exp(-(\chi_\mu(\sg)+2\eta)n_1\)$ and $\om\in \^J_2$, let $n(\om,r)$
be the least number $n$ such that
\begin{equation}\label{39}
\diam\(\phi_{\om|_{n+1}}
(X_{t(\om_{n+1})})\)<r.
\end{equation}
Using \eqref{4.3.4} we deduce that $n(\om,r)+1>n_1$, hence $n(\om,r)\ge n_1$ and
$$
\diam\(\phi_{\om|_{n(\om,r)}}(X_{t(\om_{n(\om,r)})})\)\ge r.
$$
Using Lemma~\ref{l1j81A} we find a universal constant $\Gamma \ge 1$ such that for every $\om\in \^J_2$
and $0<r<\exp(-(\chi_\mu(\sg)+2\eta)n_1\)$ there exist $k\le \Gamma$ points
$\om^{(1)},\ld,\om^{(k)}\in\^J_2$ such that
$$
\pi(\^J_2)\cap B(\pi(\om),r)
\sbt \bu_{j=1}^k\phi_{\om^{(j)}|_{n(\om^{(j)},r)}}
\lt(X_{t\lt(\om^{(j)}_{n(\om^{(j)},r)}\rt)}\rt).
$$
Let $\^\mu=\mu|_{\^J_2}$ be the restriction of the measure $\mu$ to the set
$\^J_2$. Using \eqref{4.3.1}, \eqref{4.3.3}, \eqref{4.3.4} and \eqref{39} we get
\begin{equation*}\begin{split}
\^\mu\circ\pi^{-1} &(B(\pi(\om),r))
\le \sum_{j=1}^k\mu\circ\pi^{-1}\(\phi_{\om^{(j)}|_{n(\om^{(j)},r)}}
    (X_{t(\om^{(j)}_{n(\om^{(j)},r)}}))\) \\
&=\sum_{j=1}^k\mu\([\om^{(j)}|_{n(\om^{(j)},r)}]\)
 \le \sum_{j=1}^k\exp\((-\hmu(\sg)+\eta)n(\om^{(j)},r)\) \\
&=\sum_{j=1}^k\lt(\exp\(-(\chi_\mu(\sg)+2\eta)(n(\om^{(j)},r)+1)\)\rt)^
{{n(\om^{(j)},r)\over n(\om^{(j)},r)+1}\cdot {-\hmu(\sg)+\eta\over
-(\chi_\mu(\sg)+2\eta) }} \\
&\le \sum_{j=1}^k\diam\lt(\phi_{\om^{(j)}|_{n(\om^{(j)},r)+1}}
  \lt(X_{t\lt(\om^{(j)}_{n(\om^{(j)},r)+1}\rt)}\rt)\rt)^{{
  n(\om^{(j)},
r)\over n(\om^{(j)},r)+1}\cdot
{\hmu(\sg)-\eta\over\chi_\mu(\sg)+2\eta}} \\
&\le \sum_{j=1}^kr^{{n(\om^{(j)},r)\over n(\om^{(j)},r)+1}\cdot
{\hmu(\sg)-\eta\over\chi_\mu(\sg)+2\eta}} \\
&\le Lr^{{\hmu(\sg)-2\eta\over\chi_\mu(\sg)+2\eta}},
\end{split}\end{equation*}
where the last inequality is valid provided $n_1$ is chosen so large that
$${n_1 \over n_1+1}\cdot {\hmu(\sg)-\eta\over\chi_\mu(\sg)+2\eta}\ge
{\hmu(\sg)-2\eta\over\chi_\mu(\sg)+2\eta}.
$$
Hence, $\dim_\cH(J_1) \ge \dim_{\cH}(\pi(\^J_2))\ge
{\hmu(\sg)-2\eta\over\chi_\mu(\sg)+2\eta}$, and, since $\eta$ was
arbitrary, $\dim_{\cH}(J_1)\ge {\hmu(\sg)\over\chi_\mu(\sg)}$. Thus
$$
\dim_{\cH\star}(\mu\circ \pi^{-1})\ge {\hmu(\sg)\over\chi_\mu(\sg)}
$$
By virtue of Theorem~\ref{t6.6.4.} this means that
$$
\ess\inf \un d_{\mu\circ \pi^{-1}}\ge {\hmu(\sg)\over\chi_\mu(\sg)},
$$
and along with \eqref{2_2016_05_03} this completes the proof of \eqref{1_2016_05_03} in the case of finite entropy. 

\sp If $\H_\mu(\a)=\infty$ but
$\chi_\mu(\sg)$ is finite, then it is not difficult to see that there exists a set $Z\sbt E_A^\infty$ such that $\mu(Z)=1$ and
\begin{equation*}
\lim_{n\to\infty}{-\log(\mu([\om|_n]))\over n}=+\infty
\end{equation*}
for all $\om\in Z$. Therefore the above considerations would imply that
$\dim_\cH(\mu)=+\infty$ which is impossible, and the proof is finished.
\end{proof}

\begin{remark}\label{r4.3.3}
Observe that the property
$\mu([\om])=\mu\circ \pi^{-1}(\phi_\om(X_{t(\om)}))$, $\om\in E_A^*$, which is
equivalent with \eqref{4.3.1}, was not used in the proof of the inequality $\dim_{\cH}(\mu\circ \pi^{-1})\le
{\hmu(\sg)\over\chi_\mu(\sg)}$.
\end{remark}

\begin{remark}\label{r4.3.4}
It is worth noting that $\H_\mu(\a)<\infty$ if and only if
$\H_\mu(\a^q)<\infty$ for some $q\ge 1$, and therefore it
suffices to assume, in Theorem~\ref{t4.3.2}, that $\H_\mu(\a^q)<\infty$ for some $q\ge 1$.
\end{remark}

\begin{corollary}\label{c4.3.6}
If the boundary regular Carnot conformal GDMS $S=\{\phi_e\}_{e\in E}$ is strongly regular, or more generally if it is regular and
$\H_{\^\mu_h}(\a)<\infty$, then
$$
\dim_{\cH}(m_h)=\dim_{\cH}(\mu_h)=\dim_{\cH}(J_\cS).
$$
\end{corollary}

Recall that $\mu_h$ is the $\sg$-invariant version of the $h$-conformal measure $m_h$.

\begin{proof}
We remark first that for each strongly regular
system $\cS$, $\H_{\^\mu_h}(\a)<\infty$. Indeed, since $\cS$ is strongly
regular, there exists $\eta>0$
such that $Z_1(h-\eta)<\infty$, which means that $\sum_{e\in E}||D\f_e||_\infty^{h-
\eta}<\infty$. Since by Lemma \ref{limdiam} $\lim_{e \in E}||D\f_e||_\infty=0$ we have that $||D\f_e||_\infty^{-\eta}\ge -h\log(||D\f_e||_\infty)$ for all but finitely many $e\in E$, and therefore the series $\sum_{e\in E}-
h\log(||D\f_e||_\infty)||D\f_e||_\infty^h$ converges. Hence, by virtue of \eqref{2j93},
$$\sum_{e\in E}-\log(\^\mu_h([e])\^\mu_h([e])<\infty$$ which means that
$\H_{\^\mu_h}(\a)<\infty$.

Suppose that $\cS$ is regular and $\H_{\^\mu_h}(\a)<\infty$. Since $m_h$ and $\mu_h$ are equivalent, $\dim_{\cH}(\mu_h)=\dim_{\cH}(m_h)$, and hence by Theorem \ref{t1j97} it suffices to show that $\dim_{\cH}(\mu_h)=h$. Notice that for $\om \in E^\N_A$,
\begin{equation*}\begin{split}
\sum_{j=0}^{n-1} - \zeta \circ \sg^j(\om)&=-\sum_{j=0}^{n-1} \log \|D \f_{\om_{j+1}}(\pi(\sg^{j+1}(\om)))\|\\
&=-\log\left(\|D \f_{\om_{1}}(\pi(\sg^{1}(\om)))\|\,\|D \f_{\om_{2}}(\pi(\sg^{2}(\om)))\|\cdots \|D \f_{\om_{n}}(\pi(\sg^{n}(\om)))\| \right).
\end{split}\end{equation*}
Using \eqref{leibniz} and the fact that  $\pi(\sg^k(\om))= \f_{\om_{k+1}}\circ \dots \circ \f_{\om_n}(\pi(\sg^n(\om)))$ for $1 \leq k \leq n$ we deduce that
\begin{equation}
\label{logleib}
\sum_{j=0}^{n-1} - \zeta \circ \sg^j(\om)=-\log \|D \f_{\om|_n}(\pi(\sg^n(\om)))\|.
\end{equation}
Since $\^\mu_h$ and $\^m_h$ are
equivalent, by \eqref{2j93} there exists $\texttt{const} >0$ such that
$$-\log(\texttt{const}^{-1} \,\^\mu_h([\om|_n])^{1/h})\leq -\log \|D \f_{\om|_n}(\pi(\sg^n(\om)))\| \leq -\log(\texttt{const}\, \^\mu_h([\om|_n])^{1/h}).$$
Therefore,
\begin{equation}
\label{avlimits}
\lim_{n \rightarrow \infty} \frac{-\log(\^\mu_h([\om|_n]))}{n}=h \lim_{n \rightarrow \infty} \frac{-\log \|D \f_{\om|_n}(\pi(\sg^n(\om)))\|}{n}.
\end{equation}
Therefore by \eqref{logleib}, \eqref{avlimits}, Birkhoff's Ergodic Theorem and the Breiman--Shannon--McMillan theorem
\index{Birkhoff Ergodic Theorem} \index{Breiman-Shannon-McMillan theorem} we get that
\begin{equation*}
\h_{\^\mu_h}(\sigma)= \lim_{n \rightarrow \infty}  \frac{-\log(\^\mu_h([\om|_n]))}{n}=h\,\lim_{n \rightarrow \infty}\frac{-\sum_{j=0}^{n-1} \zeta \circ \sg^j(\om)}{n}=h \chi_{\^\mu_h}(\sigma).
\end{equation*}
for a.e.\ $\om \in E_A^\N$. The proof is now completed by invoking Theorem~\ref{t4.3.2}.
\end{proof}

We end this short chapter with the following fact, which shows that $\^\mu_h$ is essentially the only invariant
measure on $E_A^\infty$ whose projection onto $J_\cS$ has maximal dimension $\dim_{\cH}(J_\cS)$.

\begin{theorem}\label{t4.3.7}
Suppose that $\cS=\{\phi_e\}_{e\in E}$ is a finitely irreducible and boundary regular conformal Carnot GDMS. Suppose also that $\mu$ is a Borel ergodic probability shift-invariant measure on $E_A^\infty$ such that $\H_\mu(\a)<+\infty$. If
$$
\dim_{\cH}(\mu\circ\pi^{-1})=h:=\dim_{\cH}(J_\cS),
$$
then the systems $\cS$ is regular and $\mu=\^\mu_{h}$.
\end{theorem}

\begin{proof}
If $\chi_\mu(\sg)=+\infty$, then it follows from Theorem~\ref{t4.3.2} that
$h=\dim_{\cH}(\mu\circ\pi^{-1})=0$ which is a contradiction. So, $\chi_\mu(\sg)<+\infty$ and it follows from Theorem~\ref{t4.3.2} that $\hmu (\sg)-h\chi_\mu(\sg)=0$. Along with the 2nd Variational Principle (Theorem~\ref{t2.1.4}), this implies that $\P(h)\ge 0$. In conjunction with Observation~\ref{phleq0} this entails $\P(h)=0$, meaning that the system $\cS$ is regular. In consequence both shift-invariant measures $\mu$ and
$\^\mu_{h}$ are equilibrium states of the potential $h\zeta:E_A^\N\to\R$. Because of Observation~\ref{o6_2016_01_20} it therefore follows from Theorem~\ref{thm-conformal-invariant} (e) that $\mu=\^\mu_{h}$. The proof is complete.
\end{proof}

\backmatter
%-----------------------------------------------------------------------------
% Beginning of biblio.tex
%-----------------------------------------------------------------------------

\bibliographystyle{acm}
\bibliography{CIFSCTUmemoirs}

\begin{thebibliography}{10}

\bibitem{all:octave-projective}
{\sc Allcock, D.}
\newblock Identifying models of the octave projective plane.
\newblock {\em Geom.\ Dedicata 65}, 2 (1997), 215--217.

\bibitem{all:octave-hyperbolic}
{\sc Allcock, D.}
\newblock Reflection groups on the octave hyperbolic plane.
\newblock {\em J. Algebra 213}, 2 (1999), 467--498.

\bibitem{baez:octonions}
{\sc Baez, J.~C.}
\newblock The octonions.
\newblock {\em Bull.\ Amer.\ Math.\ Soc.\ (N.S.) 39}, 2 (2002), 145--205.

\bibitem{bhit:horizfractals}
{\sc Balogh, Z.~M., Hoefer-Isenegger, R., and Tyson, J.~T.}
\newblock Lifts of {L}ipschitz maps and horizontal fractals in the {H}eisenberg
  group.
\newblock {\em Ergodic Theory Dynam. Systems 26}, 3 (2006), 621--651.

\bibitem{brsc:comparison}
{\sc Balogh, Z.~M., Rickly, M., and Serra~Cassano, F.}
\newblock Comparison of {H}ausdorff measures with respect to the {E}uclidean
  and the {H}eisenberg metric.
\newblock {\em Publ.\ Mat. 47}, 1 (2003), 237--259.

\bibitem{BR}
{\sc Balogh, Z.~M., and Rohner, H.}
\newblock Self-similar sets in doubling spaces.
\newblock {\em Illinois J. Math. 4\/} (2007), 1275--1297.

\bibitem{bt:horizdim}
{\sc Balogh, Z.~M., and Tyson, J.~T.}
\newblock Hausdorff dimensions of self-similar and self-affine fractals in the
  {H}eisenberg group.
\newblock {\em Proc.\ London Math.\ Soc.\ (3) 91}, 1 (2005), 153--183.

\bibitem{btw:announce}
{\sc Balogh, Z.~M., Tyson, J.~T., and Warhurst, B.}
\newblock Gromov's dimension comparison problem on {C}arnot groups.
\newblock {\em C. R. Math.\ Acad.\ Sci.\ Paris 346}, 3-4 (2008), 135--138.

\bibitem{btw:dimcomp}
{\sc Balogh, Z.~M., Tyson, J.~T., and Warhurst, B.}
\newblock Sub-{R}iemannian vs. {E}uclidean dimension comparison and fractal
  geometry on {C}arnot groups.
\newblock {\em Adv.\ Math. 220}, 2 (2009), 560--619.

\bibitem{BG}
{\sc Bandt, C., and Graf, S.}
\newblock A characterization of self-similar fractals with positive {H}ausdorff
  measure.
\newblock {\em Proc.\ Amer.\ Math.\ Soc. 114\/} (1992), 995--1001.

\bibitem{bog:meas}
{\sc Bogachev, V.~I.}
\newblock {\em Measure Theory}.
\newblock Springer-Verlag, 2007.

\bibitem{BLU}
{\sc Bonfiglioli, A., Lanconelli, E., and Uguzzoni, F.}
\newblock {\em Stratified {L}ie groups and potential theory for their
  sub-{L}aplacians}.
\newblock Springer Monographs in Mathematics. Springer, Berlin, 2007.

\bibitem{Bowen_QC}
{\sc Bowen, R.}
\newblock Hausdorff dimension of quasicircles.
\newblock {\em Inst.\ Hautes \'Etudes Sci.\ Publ.\ Math.}, 50 (1979), 11--25.

\bibitem{Bow}
{\sc Bowen, R.}
\newblock {\em Equilibrium states and the ergodic theory of {A}nosov
  diffeomorphisms}, revised~ed., vol.~470 of {\em Lecture Notes in
  Mathematics}.
\newblock Springer-Verlag, Berlin, 2008.
\newblock With a preface by David Ruelle, Edited by Jean-Ren{\'e} Chazottes.

\bibitem{bs:mobius2}
{\sc Buyalo, S., and Schroeder, V.}
\newblock M{\"o}bius characterization of the boundary at infinity of rank one
  symmetric spaces.
\newblock {\em Geom.\ Dedicata 172}, 1 (2014), 1--45.

\bibitem{cns:cKg}
{\sc Cano, A., Navarrete, J.~P., and Seade, J.}
\newblock {\em Complex {K}leinian groups}, vol.~303 of {\em Progress in
  Mathematics}.
\newblock Birkh\"auser Verlag, Basel, 2013.

\bibitem{cap:regularity1}
{\sc Capogna, L.}
\newblock Regularity of quasi-linear equations in the {H}eisenberg group.
\newblock {\em Comm.\ Pure Appl.\ Math. 50}, 9 (1997), 867--889.

\bibitem{cap:regularity2}
{\sc Capogna, L.}
\newblock Regularity for quasilinear equations and {$1$}-quasiconformal maps in
  {C}arnot groups.
\newblock {\em Math.\ Ann. 313}, 2 (1999), 263--295.

\bibitem{cc:conformality}
{\sc Capogna, L., and Cowling, M.}
\newblock Conformality and {$Q$}-harmonicity in {C}arnot groups.
\newblock {\em Duke Math.\ J. 135}, 3 (2006), 455--479.

\bibitem{cdpt:survey}
{\sc Capogna, L., Danielli, D., Pauls, S.~D., and Tyson, J.~T.}
\newblock {\em An introduction to the {H}eisenberg group and the
  sub-{R}iemannian isoperimetric problem}, vol.~259 of {\em Progress in
  Mathematics}.
\newblock Birkh\"auser Verlag, Basel, 2007.

\bibitem{cg:nta}
{\sc Capogna, L., and Garofalo, N.}
\newblock Boundary behavior of nonnegative solutions of subelliptic equations
  in {N}{T}{A} domains for {C}arnot--{C}arath\'eodory metrics.
\newblock {\em J. Fourier Anal.\ Appl. 4}, 4-5 (1998), 403--432.

\bibitem{cg:ahlfors}
{\sc Capogna, L., and Garofalo, N.}
\newblock Ahlfors type estimates for perimeter measures in
  {C}arnot--{C}arath\'eodory spaces.
\newblock {\em J. Geom.\ Anal. 16}, 3 (2006), 455--497.

\bibitem{cdkr:h-type}
{\sc Cowling, M., Dooley, A.~H., Kor{\'a}nyi, A., and Ricci, F.}
\newblock {$H$}-type groups and {I}wasawa decompositions.
\newblock {\em Adv.\ Math. 87}, 1 (1991), 1--41.

\bibitem{co:conformal-carnot}
{\sc Cowling, M., and Ottazzi, A.}
\newblock Conformal maps of {C}arnot groups.
\newblock {\em Ann.\ Acad.\ Sci.\ Fenn.\ Math. 40\/} (2015), 203--213.

\bibitem{dgn:convexity}
{\sc Danielli, D., Garofalo, N., and Nhieu, D.-M.}
\newblock Notions of convexity in {C}arnot groups.
\newblock {\em Comm.\ Anal.\ Geom. 11}, 2 (2003), 263--341.

\bibitem{dgn:memoirs}
{\sc Danielli, D., Garofalo, N., and Nhieu, D.-M.}
\newblock Non-doubling {A}hlfors measures, perimeter measures, and the
  characterization of the trace spaces of {S}obolev functions in
  {C}arnot--{C}arath\'eodory spaces.
\newblock {\em Mem.\ Amer.\ Math.\ Soc. 182}, 857 (2006), x+119 pp.

\bibitem{DKU}
{\sc Denker, M., Keller, G., and Urba{\'n}ski, M.}
\newblock On the uniqueness of equilibrium states for piecewise monotone maps.
\newblock {\em Studia Math. 97\/} (1990), 27--36.

\bibitem{DuS1}
{\sc Dunford, N., and Schwartz, J.~T.}
\newblock {\em Linear Operators, Part I}.
\newblock Interscience Publishers, 1957.

\bibitem{fol}
{\sc Folland, G.~B.}
\newblock {\em Real Analysis: modern techniques and their applications},
  2nd~ed.
\newblock Wiley-Interscience, 1999.

\bibitem{geh:rings}
{\sc Gehring, F.~W.}
\newblock Rings and quasiconformal mappings in space.
\newblock {\em Trans.\ Amer.\ Math.\ Soc. 103\/} (1962), 353--393.

\bibitem{hei:lectures}
{\sc Heinonen, J.}
\newblock {\em Lectures on analysis on metric spaces}.
\newblock Universitext. Springer-Verlag, New York, 2001.

\bibitem{hk:quasi}
{\sc Heinonen, J., and Koskela, P.}
\newblock Quasiconformal maps in metric spaces with controlled geometry.
\newblock {\em Acta Math. 181}, 1 (1998), 1--61.

\bibitem{KaVI}
{\sc K{\"a}enm{\"a}ki, A., and Vilppolainen, M.}
\newblock Separation conditions on controlled {M}oran constructions.
\newblock {\em Fund. Math. 200}, 1 (2008), 69--100.

\bibitem{kr:heisenberg}
{\sc Kor{\'a}nyi, A., and Reimann, H.~M.}
\newblock Quasiconformal mappings on the {H}eisenberg group.
\newblock {\em Invent.\ Math. 80}, 2 (1985), 309--338.

\bibitem{kr:foundations}
{\sc Kor{\'a}nyi, A., and Reimann, H.~M.}
\newblock Foundations for the theory of quasiconformal mappings on the
  {H}eisenberg group.
\newblock {\em Adv.\ Math. 111}, 1 (1995), 1--87.

\bibitem{lau:sep}
{\sc Lau, K.~S., Ngai, S.~M., and Y., W.~X.}
\newblock Separation conditions for conformal iterated functions systems.
\newblock {\em Monatsh. Math. 156}, 4 (2009), 325--355.

\bibitem{lms:convex}
{\sc Lu, G., Manfredi, J.~J., and Stroffolini, B.}
\newblock Convex functions on the {H}eisenberg group.
\newblock {\em Calc.\ Var.\ Partial Differential Equations 19}, 1 (2004),
  1--22.

\bibitem{luvacf}
{\sc Lukyanenko, A., and Vandehey, J.}
\newblock Continued fractions on the {H}eisenberg group.
\newblock {\em Acta Arithmetica 167\/} (2015), 19--42.

\bibitem{mp:octonions}
{\sc Markham, S., and Parker, J.~R.}
\newblock J\o rgensen's inequality for metric spaces with application to the
  octonions.
\newblock {\em Adv.\ Geom. 7}, 1 (2007), 19--38.

\bibitem{mat:geometry}
{\sc Mattila, P.}
\newblock {\em Geometry of sets and measures in {E}uclidean spaces}, vol.~44 of
  {\em Cambridge Studies in Advanced Mathematics}.
\newblock Cambridge University Press, Cambridge, 1995.

\bibitem{MSU}
{\sc Mauldin, R.~D., Szarek, T., and Urba{\'n}ski, M.}
\newblock Graph directed {M}arkov systems on {H}ilbert spaces.
\newblock {\em Math.\ Proc.\ Cambridge Philos.\ Soc. 147}, 2 (2009), 455--488.

\bibitem{MU1}
{\sc Mauldin, R.~D., and Urba{\'n}ski, M.}
\newblock Dimensions and measures in infinite iterated function systems.
\newblock {\em Proc.\ London Math.\ Soc.\ (3) 73}, 1 (1996), 105--154.

\bibitem{MU2}
{\sc Mauldin, R.~D., and Urba{\'n}ski, M.}
\newblock Conformal iterated function systems with applications to the geometry
  of conformal iterated function systems.
\newblock {\em Trans.\ Amer.\ Math.\ Soc. 351\/} (1999), 4995--5025.

\bibitem{MUGDMS}
{\sc Mauldin, R.~D., and Urba{\'n}ski, M.}
\newblock {\em Graph directed {M}arkov systems}, vol.~148 of {\em Cambridge
  Tracts in Mathematics}.
\newblock Cambridge University Press, Cambridge, 2003.
\newblock Geometry and dynamics of limit sets.

\bibitem{mont:tour}
{\sc Montgomery, R.}
\newblock {\em A tour of subriemannian geometries, their geodesics and
  applications}, vol.~91 of {\em Mathematical Surveys and Monographs}.
\newblock American Mathematical Society, Providence, RI, 2002.

\bibitem{mos:hyperbolic}
{\sc Mostow, G.~D.}
\newblock Quasi-conformal mappings in {$n$}-space and the rigidity of
  hyperbolic space forms.
\newblock {\em Inst.\ Hautes \'Etudes Sci.\ Publ.\ Math.}, 34 (1968), 53--104.

\bibitem{mos:rigidity}
{\sc Mostow, G.~D.}
\newblock {\em Strong rigidity of locally symmetric spaces}.
\newblock Princeton University Press, Princeton, N.J., 1973.
\newblock Annals of Mathematics Studies, No. 78.

\bibitem{pan:metriques}
{\sc Pansu, P.}
\newblock M\'etriques de {C}arnot-{C}arath\'eodory et quasiisom\'etries des
  espaces sym\'etriques de rang un.
\newblock {\em Ann.\ of Math.\ (2) 129}, 1 (1989), 1--60.

\bibitem{PRSS}
{\sc Peres, Y., Rams, M., Simon, K., and Solomyak, B.}
\newblock Equivalence of positive {H}ausdorff measure and the open set
  condition for self conformal sets.
\newblock {\em Proc.\ Amer.\ Math.\ Soc. 129}, 9 (2001), 2689--2699.

\bibitem{pla:ptolemaic}
{\sc Platis, I.~D.}
\newblock Cross-ratios and the {P}tolemaean inequality in boundaries of
  symmetric spaces of rank 1.
\newblock {\em Geom.\ Dedicata 169\/} (2014), 187--208.

\bibitem{PU}
{\sc Przytycki, F., and Urba{\'n}ski, M.}
\newblock {\em Conformal fractals: ergodic theory methods}, vol.~371 of {\em
  London Mathematical Society Lecture Notes}.
\newblock Cambridge University Press, Cambridge, 2010.

\bibitem{RV}
{\sc Rajala, T., and Vilppolainen, M.}
\newblock Weakly controlled {M}oran constructions and iterated function systems
  in metric spaces.
\newblock {\em Illinois J. Math. 55}, 3 (2011), 1015--1051.

\bibitem{Ru}
{\sc Ruelle, D.}
\newblock {\em Thermodynamic formalism}.
\newblock Addison-Wesley, 1978.

\bibitem{Sar}
{\sc Sarig, O.}
\newblock Theormodynamic formalism for countable {M}arkov shifts.
\newblock {\em Ergod. Th. Dynam. Sys. 62\/} (1999), 1565--1593.

\bibitem{Schief:separation}
{\sc Schief, A.}
\newblock Separation properties for self-similar sets.
\newblock {\em Proc. Amer. Math. Soc. 122\/} (1994), 111--115.

\bibitem{Schief:complete}
{\sc Schief, A.}
\newblock Self-similar sets in complete metric spaces.
\newblock {\em Proc. Amer. Math. Soc. 124\/} (1996), 481--490.

\bibitem{vand:lagrange}
{\sc Vandehey, J.}
\newblock Lagrange's theorem for continued fractions on the {H}eisenberg group.
\newblock {\em Bull.\ London Math.\ Soc. 47}, 5 (2015), 866--882.

\bibitem{vand:diophantine}
{\sc Vandehey, J.}
\newblock Diophantine properties of continued fractions on the {H}eisenberg
  group.
\newblock {\em Internat.\ J. Number Theory 12}, 2 (2016), 541--560.

\bibitem{Wa}
{\sc Walters, P.}
\newblock {\em An Introduction to Ergodic Theory}.
\newblock Springer, 1982.

\end{thebibliography}

%-----------------------------------------------------------------------------
% End of biblio.tex
%-----------------------------------------------------------------------------

\printindex
\end{document}